\documentclass[reqno]{amsart}
\usepackage{amssymb,amsmath,epsfig,graphics,mathrsfs}

\usepackage{tabularx}
\usepackage{fancyhdr}

\usepackage{hyperref}
\hypersetup{
	colorlinks   = true,
	urlcolor     = blue,
	linkcolor    = blue,
	citecolor   = red ,
	bookmarksopen=true
}

\usepackage{dsfont}

\usepackage{cancel} 
\usepackage{esvect}
\usepackage{tikz}
\usepackage{tikz-3dplot}

\usepackage{pgfplots}
\usepackage{pgfmath}

\pgfplotsset{compat=1.18}

\usepackage{multicol}

\usepackage{amssymb,amsmath,epsfig,graphics,mathrsfs}

\usepackage{tikz}

\usepackage{graphicx}
\usepackage{caption}

\usepackage{bbm}
\usepackage[normalem]{ulem}

\usepackage[dvipsnames,table,xcdraw]{xcolor} 
\usepackage[a4paper,
left=1in,
right=1in,
top=1in,
bottom=1in,
footskip=.25in]{geometry}

\def \R {\mathbb{R}}
\def \l {\lambda}

\numberwithin{equation}{section}

\newcommand{\beq}{\begin{equation}}
	\newcommand{\bea}[1]{\begin{array}{#1} }
		\newcommand{\eeq}{ \end{equation}}
	\newcommand{\ea}{ \end{array}}

\newtheorem{theorem}{Theorem}[section]
\newtheorem{lemma}[theorem]{Lemma}
\newtheorem{proposition}[theorem]{Proposition}
\newtheorem{corollary}[theorem]{Corollary}
\newtheorem{remark}[theorem]{Remark}
\newtheorem{definition}[theorem]{Definition}

\makeatletter
\def\@settitle{\begin{center}%
		\baselineskip14\p@\relax
		\bfseries
		\uppercasenonmath\@title
		\@title
		\ifx\@subtitle\@empty\else
		\\[1ex]\uppercasenonmath\@subtitle
		\footnotesize\mdseries\@subtitle
		\fi
	\end{center}%
}
\def\subtitle#1{\gdef\@subtitle{#1}}
\def\@subtitle{}
\makeatother

\newtheorem*{claim}{Claim}

\usepackage[section]{placeins}

\usepackage{pgfplots}

\begin{document}
	\title[Long-time dynamics for the 2D Keller--Segel equation at critical mass]{Determination of the long-time dynamics for the 2D Keller--Segel equation at critical mass} 

\author[F.~Buseghin]{Federico Buseghin}
\address{\noindent F.~Buseghin: CY Cergy Paris University. Laboratoire AGM, 2 avenue Adolphe Chauvin 95302 Cergy-Pontoise}
\email{federico.buseghin@cyu.fr}

\author[C.~Collot]{Charles Collot}
\address{\noindent C.~Collot: CY Cergy Paris University. Laboratoire AGM, 2 avenue Adolphe Chauvin 95302 Cergy-Pontoise}
\email{charles.collot@cyu.fr}

\setcounter{tocdepth}{2}

\begin{abstract}

We consider the parabolic-elliptic Keller-Segel equation in two dimensions on the whole space. We prove that for arbitrary initial data with critical mass $8\pi$ and finite second momentum, all solutions have the same universal behaviour. They are globally defined and, asymptotically for large times, they converge to a renormalized stationary state of the equation which concentrates around the center of mass of the solution at a universal logarithmic-in-time scale. Our result holds for general solutions without symmetry assumptions, and we furthermore provide explicit convergence rates. In the radial case, by combining our result with previous ones (by Blanchet-Dolbeault-Perthame, Mizoguchi, and related works), this achieves a complete classification of the possible dynamics: for subcritical masses solutions converge to a self-similar expander, at the critical mass they concentrate a stationary state in infinite time at the aforementioned universal scale, and for supercritical masses they blow up in finite time by type II concentration of a stationary state at another universal scale. Our proof starts with soliton resolution, then controls the motion of the stationary state and the remainder in new spaces, and devises new techniques for the multi-scale linearized analysis, which eventually enables the stability and modulation analysis around an approximate solution.

\end{abstract}

\maketitle
\makeatletter
\def\@oddhead{\normalfont\scriptsize
\hfil LONG-TIME DYNAMICS FOR THE 2D KELLER--SEGEL EQUATION AT CRITICAL MASS
\hfil\thepage}
\def\@evenhead{\normalfont\scriptsize
\thepage\hfil F.~BUSEGHIN AND C.~COLLOT\hfil}
\def\@oddfoot{}
\def\@evenfoot{}
\makeatother

\section{Introduction}
\subsection{Dynamics of the Keller-Segel equation on the plane and main result}

We consider the two dimensional Keller-Segel system
\begin{align}\label{KS}
	\begin{cases}
		\partial_{t}u=\Delta u - \nabla \cdot(u\nabla \Phi_{u}), \ \ \ (x,t)\in \mathbb{R}^{2}\times(0,\infty),\\
		-\Delta \Phi_{u}(x)=u(x), \ \ \ x\in \mathbb{R}^{2} \\
		u(0,x)=u_{0} \geq0, \ \ \ \ x \in \mathbb{R}^{2}.
	\end{cases}
\end{align}
The second equation is solved explicitly via the formula $\Phi_u(x)=-(2\pi)^{-1}\log|\cdot|*u$, leading to
\begin{equation} \label{id:definition-nablaPhiu}
\nabla \Phi_u (x)= -\frac{1}{2\pi}\int_{\mathbb R^2} \frac{x-y}{|x-y|^2}u(y)dy.
\end{equation}
While $\log |\cdot|*u$ might not be well-defined, here as in the literature, an abuse is made in that solutions are those of the first equation in \eqref{KS} with $\nabla \Phi_u$ given by \eqref{id:definition-nablaPhiu}.

The system \eqref{KS} is generally regarded as the fundamental mathematical model for describing chemotactic aggregation of certain microorganisms \cite{P,KS}. For recent surveys on the subject, we refer the reader to \cite{H,HP,BBook,SBook}. It describes the evolution of a population density subject to diffusion and chemotactic drift induced by a chemical signal. Moreover, after a suitable rescaling, the model can also be interpreted as describing the distribution of self-gravitating particles \cite{CS}. Since the seminal work of Jäger and Luckhaus \cite{JL}, the mathematical analysis of this system has attracted considerable attention. We refer the interested reader to some representative mathematical studies on the consequences of the competition between diffusion and chemotactic aggregation in \eqref{KS} and related systems \cite{CP,NS,B,HV0,SS,BKLN,BDP,S,BCM,BCC,CHVY}.

It is known that the Cauchy problem for \eqref{KS} is well-posed in $L^1$, see \cite{W} and references therein and Theorem \ref{th:lwp:L1} here. Solutions with $u_0\in L^1$ are instantaneously smooth $u\in C^{\infty}((0,T)\times \mathbb R^2)$ and bounded together with their derivatives $\nabla^k u \in L^\infty_{loc}((0,T),L^\infty(\mathbb R^2))$ for all $k\in \mathbb N$, and their maximal time of existence $T>0$ is finite if and only if $\lim_{t\uparrow T}\| u(t)\|_{L^\infty(\mathbb R^2)}=\infty$. Their mass is preserved,
\begin{equation} \label{id:mass-conservation}
	M[u(t)]=\int_{\mathbb R^2} u(t,x)dx =M[u_0].
\end{equation}
The equation \eqref{KS} admits spatial translation and scaling invariances, in that if $u$ is a solution then so is
$$
(t,x)\mapsto \frac{1}{\lambda^2}u(\frac{t}{\lambda^2},\frac{x-x_0}{\lambda}).
$$
The mass is invariant by this transformation, so that \eqref{KS} is $L^1$-critical. In the present article, the data is assumed to be \emph{localized} in that it further has finite second momentum,
\begin{equation} \label{id:hp-data-2}
	u_0\in L^1(\langle x \rangle^2 dx).
\end{equation}

The Keller-Segel system \eqref{KS} is also well-posed in $L^1(\langle x \rangle^2dx)$, see Theorem \ref{thm:lwp-L1x2}. The first momentum of such solutions is preserved
\begin{equation}\label{first-momentum}
	M^{-1}\int_{\mathbb R^2} x u(t,x)dx = M^{-1}\int_{\mathbb R^2} x u_0(x)dx = x^*[u_0]\in \mathbb R^2
\end{equation}
and their second momentum
$$
\int_{\mathbb R^2} |x-x^*|^2u(t)dx= \mu [u(t)]>0
$$
 is explicitly given by
\begin{equation}\label{second-momentum}
	\mu [u(t)] =\mu[u_0] +4t  \left(1-\frac{M}{8\pi} \right)M.
\end{equation}
Furthermore, the free energy functional
\begin{align}\label{freeenergy}
	\mathcal{F}[u](t)=\int_{\R^2}u(x,t)\log u(x,t)dx - \frac{1}{2}\int_{\R^2}u(x,t)\Phi_u(x,t)dx.
\end{align}
is well-defined for all $t\in (0,T)$. It decreases over time with
\begin{equation}\label{freeenergy-dissipation}
	\mathcal{F}[u](t')=\mathcal{F}[u](t)-\int_{t}^{t'} \int_{\mathbb R^2} u(s,x)|\nabla \log u(s,x)-\nabla \Phi_{u(s)}(x)|^2 dx
\end{equation}
for all $0<t<t'<T$.

The critical mass for solutions is $8\pi$. Indeed, the free energy \eqref{freeenergy} is bounded from below precisely at mass $8\pi$, and its minimizers are exactly the stationary states
\begin{align}\label{steadstate}
U(x)=\frac{8}{(1+|x|^2)^2}, \qquad \int_{\mathbb R^2} U(x),dx=8\pi,
\end{align}
see Section \ref{sec:stationary-states}. This characterization follows from the sharp logarithmic Hardy--Littlewood--Sobolev inequality \cite{CL}, which can be viewed as a limiting case of the sharp Hardy--Littlewood--Sobolev inequality \cite{Lieb}. This variational structure was exploited in the seminal work of Blanchet, Dolbeault and Perthame \cite{BDP} to rigorously establish $8\pi$ as the threshold between global existence and blow-up. When the total mass is strictly below the critical threshold $8\pi$, it is well understood that diffusion dominates the dynamics. Solutions converge to unique self-similar expanders and the corresponding convergence rates were characterized in \cite{BDP,BKLN,N1,EM}. Conversely, when the mass exceeds $8\pi$, the negativity of the production rate of the second moment \eqref{second-momentum} necessarily induces finite-time blow-up. The blow-up is of type II, and the first rigorous construction of a precise example with anomalous self-similarity is due to the pioneering works \cite{HV0, HV1, HV2, V}, which was subsequently shown to be stable in the radial case in \cite{RS}, and in the non-radial case in \cite{CGMN1} where the precise blow-up rate was proved. A complete classification of radial solutions in the supercritical regime was later obtained in \cite{M}. Beyond radial symmetry, the dynamics in the supercritical regime becomes considerably richer. In addition to the formation of multiple blow-up spikes at isolated points \cite{BDdPM}, interactions between singularities may occur, including the collision and merging of two or more concentration points.
This phenomenon was first described through formal matched asymptotic expansions in \cite{SSV}, and a rigorous construction was subsequently provided in \cite{CGMN2}.

We will focus on \emph{threshold solutions}, namely solutions with the \emph{critical mass}
\begin{equation} \label{threshold mass}
M=\int_{\mathbb R^2} u(t,x)\,dx =8\pi.
\end{equation}
Such $L^1$ solutions are known to be globally regular, see \cite{W}, \cite{BCM} for data with $u_0\in L^1(\langle x \rangle^2dx)$ and \cite{BKLN} for an earlier result in the radial case. Note that \cite{W} shows the existence of a unique global mild solution, and that Theorem \ref{th:lwp:L1} shows global mild solutions are globally regular. In this case, we get from \eqref{second-momentum} that the second momentum is constant over time $\int |x-x^*|^2 u(t,x)dx =\mu[u_0]$. Moreover, these solutions converge to a Dirac at their center of mass in the sense that $u(t_n)\to 8\pi \delta_{x^*}$ in the weak-star sense for measures for any sequence  $t_n\to \infty$ \cite{BCM}. An anomalous self-similar asymptotic behaviour had been predicted on a formal level in \cite{C, CS}, see also \cite{Serr}. The first rigorous construction of a radial solution with such behaviour was achieved in \cite{GM}, where its stability within the radial class was also established. This analysis was later extended to the non-radial setting in \cite{DdPDMW}. The solutions of \cite{GM,DdPDMW} have the form
\begin{align} \label{mainth:intro-decomposition}
	u(t,x) \;=\; \frac{1}{\lambda^{2}(t)} \, U\!\left(\frac{x-\xi(t)}{\lambda(t)}\right)+\tilde u, 
	\qquad \lambda(t) = \frac{c}{\sqrt{\ln t}}, 
	\qquad \xi(t) \to x^{\star} \in \mathbb{R}^{2},
\end{align}
for some constant $c>0$, where $\tilde u$ converges to zero in some appropriate sense. The main result of this article is to prove that in fact \emph{this behaviour is universal and holds for all solutions}.

\begin{theorem}[Universal large time asymptotics at critical mass] \label{main:th:main}
	
	Let $u$ be a solution of the Keller-Segel system \eqref{KS} whose initial data satisfies $\int_{\mathbb R^2} u_0(x)(1+|x|^2)dx<\infty $ and which has critical mass $M(u)=8\pi$. Let $x^*$ denote its center of mass and $\mu$ its second momentum as defined in \eqref{first-momentum} and \eqref{second-momentum}.
	
	Then the solution $u$ is global and its second momentum is constant over time. Moreover, it can be decomposed for large times as a stationary state concentrating at its center of mass
	$$
	u(x,t)=\frac{1}{\lambda^2(t)}U\left(\frac{x-x(t)}{\lambda(t)}\right)+\tilde u(x,t)
	$$
	at a scale
	\begin{equation} \label{mainth:scale}
	\lambda(t)= \sqrt{\frac{\mu}{8\pi}} \frac{1+o_{t\to \infty}(1)}{\sqrt{\ln t}},
	\end{equation}
	and with the convergence rates
	\begin{equation} \label{cv-rates}
	\| \tilde u(t)\|_{L^1(\mathbb R^2)}+\lambda(t)^2\| \tilde u(t)\|_{L^\infty(\mathbb R^2)}\leq C t^{-1+\epsilon} \quad \mbox{and} \quad |x(t)-x^*|\leq C t^{-1/2+\epsilon}
	\end{equation}
	for $t \geq 1$, for any $\epsilon>0$ for some constant $C=C(u,\epsilon)$.
		
\end{theorem}

\begin{remark}

The convergence rates \eqref{cv-rates} are almost sharp, in that it is clear from our far field analysis that $t^{-\alpha}$ and $t^{-\alpha+1/2}$ rates are impossible for $\alpha>1$. Obtaining $\alpha=1$ would require us to expand our approximate solution to the next order, as it currently would yield a logarithmic loss.

Our formula for the scale \eqref{mainth:scale} shows that $c=\sqrt{\mu/(8\pi)} $ in \eqref{mainth:intro-decomposition}. This is slightly different from the value $c=\sqrt{\mu}$ found in \cite{GM}, but we believe our constant is the correct one. In the work \cite{DdPDMW} the constant $c>0$ was shown to exist but was not computed explicitly.

\end{remark}

In the next two paragraphs, we explain the role that Theorem~\ref{main:th:main} plays in the current literature.

\medskip
\noindent 
\textbf{Sharpness of the finite second moment assumption.} The large-time dynamics of critical-mass solutions is highly sensitive to the assumptions imposed on the initial datum. As formally observed in \cite{CS} the conserved second momentum should be regarded as the selection rule for the unique admissible asymptotic scale.

The works \cite{BCC,CF,C} (see also \cite{LGNY} for an approach based on rearrangement techniques and PDE methods) show a different behaviour than that of Theorem \ref{main:th:main} for solutions that fail logarithmically to belong to $L^1(\langle x \rangle^2 dx)$. If the initial datum $u_0$ has mass $8\pi$, finite free energy, and finite relative entropy
\[
H_\lambda[u_0]
:=
\int_{\mathbb R^2}
\frac{\bigl(\sqrt{u_0(x)}-\sqrt{U_\lambda(x)}\bigr)^2}
{\sqrt{U_\lambda(x)}}\,dx
<\infty , \quad \text{where} \quad 
U_\lambda(x):=\frac{1}{\lambda^2}\,U\!\left(\frac{x}{\lambda}\right)
\]
then the corresponding solution converges towards the stationary state $U_\lambda$ in $L^1(\mathbb R^2)$. Moreover, quantitative stability estimates for the logarithmic Hardy--Littlewood--Sobolev inequality imply the best rate of convergence currently available in this framework,
\[
\|u(t)-U_\lambda\|_{L^1(\mathbb R^2)}
\lesssim t^{-1/16},
\]

In the absence of a finite relative entropy condition the asymptotic behavior can be drastically different. The work \cite{NaiS} showed that, for suitably chosen critical-mass radial initial data, solutions may oscillate among prescribed stationary states instead of converging to a single equilibrium. Bounded oscillatory solutions and unbounded oscillatory solutions with infinite-time blow-up were constructed. This picture was further clarified in \cite{LGNY1}, where it was proved that any bounded radial solution either converges to a single stationary state or oscillates through a whole interval of stationary states. Moreover, non-radial solutions whose asymptotic dynamics involve a non-trivial continuum of stationary states were constructed.

This discussion highlights how much the behaviour at spatial infinity matters for global dynamics. On bounded domains, solutions with critical mass concentrate in infinite time but with yet a different rate \cite{BKLN,KS,SiCh}

\medskip
\noindent 
\textbf{Minimal (infinite time) blow-up} Minimal solutions are those whose mass (or another conserved quantity) equals that of the ground state of the equation. Non-trivial asymptotic behaviours of such solutions are expected to be threshold dynamics for the equation, and are believed to display rigidity features. Solutions of Theorem \ref{main:th:main} are indeed at the threshold between self-similar diffusion and finite-time blow-up, and have a rigid universal asymptotic behaviour. Critical mass solutions to the parabolic-parabolic Keller-Segel equation, studied in \cite{M2,Ho} for example, could have similar properties. Minimal mass blow-up has been studied for the nonlinear Schr\"odinger equation \cite{Merle,RaSz} and the generalized KdV equation \cite{MMR}. Minimal dynamics associated to unstable manifolds near the ground state have been studied for energy critical equations in \cite{DuMe1,DuMe2,CMR}.

\medskip

\subsection{Classification of the dynamics in the radial setting}

In the radial case, Theorem \ref{main:th:main} finishes the \emph{complete classification of all dynamics} by determining the dynamics at critical mass. In the subcritical case $M<8\pi$, solutions were known to exist globally and converge towards the unique self-similar expander of mass $M$ \cite{BDP,BKLN,N1,EM}. In the supercritical case $M>8\pi$, all radial solutions undergo type II finite-time blow-up and concentrate a stationary state with a universal concentration scale, as obtained in \cite{M} that compares the solution to the known blow-up examples \cite{HV0,RS,CGMN1}. The unified presentation of these asymptotic regimes will appear in a companion note \cite{BC}, yielding the following trichotomy:

\begin{theorem}[Classification of radial dynamics]
	\label{thm:classification}
	Let $u$ be a radially symmetric solution to \eqref{KS} such that $u_0\in L^1(\langle x \rangle^2 dx)$, with maximal time of existence $T$. Then the following alternatives hold.
	
	\begin{enumerate}
		
	\item \textbf{Subcritical case}. If \(M<8\pi\), the solution is global \(T=\infty\) and, denoting by $\Psi_M$ the unique radial self-similar expander with $\int \Psi_M dx=M$, there holds
	\[
	u(t,x)=\frac{1}{t+1} \Psi_M \!\left(\frac{x}{\sqrt{t+1}}\right) + \tilde{u}(x,t), \qquad \quad \lim_{t\to \infty} (1+t)^\alpha (\|\tilde{u}(t)\|_{L^{1}}+t\|\tilde{u}(t)\|_{L^{\infty}})=0
	\]
    for some $\alpha>0$.
\item \textbf{Critical case}. If \(M=8\pi\), the solution is global \(T=+\infty\) and can be decomposed for large times as a stationary state concentrating at the origin
$$
u(x,t)=\frac{1}{\lambda^{2}(t)}U\left(\frac{x-x(t)}{\lambda(t)}\right)+\tilde u(x,t), \qquad \quad \lim_{t\to \infty}(1+t)^{1-\epsilon}(\| \tilde u(t)\|_{L^1}+\lambda(t)^2\| \tilde u(t)\|_{L^\infty})=0,
$$
for any $\epsilon>0$, at scale (where below $\mu$ is the second momentum \eqref{second-momentum})
$$
\lambda(t)= \sqrt{\frac{\mu}{8\pi}} \frac{1+o_{t\to \infty}(1)}{\sqrt{\ln t}}.
$$
		\item \textbf{Supercritical case}. If \(M>8\pi\), then the solution blows up in finite time \(T<\infty\) and can be decomposed as a stationary state concentrating at the origin
		\[
		u(x,t)
		=
		\frac1{\lambda^{2}(t)}
		U\!\left(\frac{x}{\lambda(t)}\right)
		+ \tilde u\!\left(x,t\right), \qquad  \lim_{R\to 0}\sup_{[0,T)} \| \tilde u(t)\|_{L^1(|x|<R)}=\lim_{t\to T} \lambda^2(t)\| \tilde u(t)\|_{L^\infty} =0 
		\]
		at scale (where below \(\gamma\) is the Euler--Mascheroni constant)
		\[
		\lambda(t)
		=
		2e^{-1-\gamma/2}
		\sqrt{T-t}\,
		\exp\!\left(
		-\sqrt{\frac{|\log(T-t)|}{2}}
		\right)(1+o_{t\to T}(1)).
		\]
		
	\end{enumerate}
\end{theorem}

\begin{remark}
Among the three universal asymptotic behaviours above, the one for subcritical masses $M<8\pi$ remains valid in the non-radial case and for $L^1$ data, \cite{N1}. As discussed above, the one for the critical mass $M=8\pi$ remains valid in the non-radial case (Theorem \ref{main:th:main}) but other scale dynamics exist without finite second momentum, and the one for the supercritical mass $M>8\pi$ remains valid for radial $L^1$ data but other singular collisional dynamics exist in the non-radial case.
\end{remark}

\subsection{Strategy, novelties, and organization of the article}

\subsubsection{A roadmap for the classification of anomalous self-similar behaviours}

Our proof of Theorem \ref{main:th:main} relies on a general strategy, applicable to similar problems, that has three main steps.

\smallskip

\noindent \textbf{Step 1-A.} \emph{Soliton resolution in the base scale-invariant space}. The starting point in Section \ref{sec:soliton-resolution} is to show the emergence of a stationary state with time dependent position $x(t)$ and scale $\lambda(t)$ in the critical space $L^1$. Namely, that $u$ can be decomposed as
\begin{equation} \label{intro-sol-res}
u=U_{\lambda(t),x(t)}+\tilde u, \quad \lim_{t\to \infty}\|\tilde u(t)\|_{L^1}=0.
\end{equation}
Here, we rely on the dissipative structure of the Keller-Segel system associated to the free energy functional \eqref{freeenergy}, and on the classification of zero dissipation functions, see Theorem \ref{th:rigidity}.

\smallskip

\noindent \textbf{Step 1-B.} \emph{Soliton resolution in additional scale-invariant spaces}. As the sole $L^1$ topology does not allow to control well the flow, we refine in Section \ref{sec:roughconv} the convergence \eqref{intro-sol-res}:
\begin{itemize}
\item \emph{Modulation slow-down in scale-invariant spaces}. By a standard orthogonal decomposition on the soliton manifold, we show the following bound that controls the motion on timescales $\approx \lambda^2(t)$,
\begin{equation} \label{introduction-modulation}
\lim_{t\to \infty}\lambda \dot \lambda =\lim_{t\to \infty}\lambda \dot x= 0.
\end{equation}
To control over longer timescales, we prove that in the critical H\"older $1/2$ space,
\begin{equation} \label{introduction-holder}
\| x(\cdot)\|_{\dot C^{1/2}([t,\infty)}=\sup_{t'\geq t, \ R>0} R^{-1/2}|x(t'+R)-x(t')|\to 0 \quad \mbox{and}\quad \| \lambda(\cdot)\|_{\dot C^{1/2}([t,\infty)}\to 0
\end{equation}
as $t\to \infty$, which means that the soliton slows down over parabolic scales. This is proved using localized first momenta.
\item \emph{Vanishing of the remainder in scale-invariant weighted $L^\infty$ spaces}. For $\tilde u$ in \eqref{intro-sol-res}, we show
\begin{equation} \label{introduction-pointwise}
|\nabla^k \tilde u(t,x)|=\frac{o_{t\to \infty}(1)}{(\lambda(t)+|x-x(t)|)^{2+k}} \quad \mbox{and} \quad  |\nabla \Phi_{ \tilde u}(t,x)|=\frac{o_{t\to \infty}(1)}{\lambda(t)+|x-x(t)|}
\end{equation}
for $k=0,1$. This is obtained thanks to a new quantitative $\varepsilon$-regularity theorem, a result of independent interest in Appendix \ref{sec:epsilon-regularity}.
\end{itemize}
\smallskip

\noindent \textbf{Step 2-A.} \emph{Local linear analysis}. We next study in Section \ref{section:LinAnalysis} the \emph{linearized dynamics around the moving soliton} \eqref{intro-sol-res}. The linearized operator in parabolic variables $(z,\tau)=(t^{-1/2}(x-x(t)),\ln t)$ is
\begin{equation} \label{def:Lnuxi}
L_{\nu,\xi}\varepsilon(z)= \Delta \varepsilon +(\frac z2-\nabla \Phi_{\hat U_{\nu,\xi}})\cdot\nabla \varepsilon+(2\hat U_{\nu,\xi}+1)\varepsilon -\nabla \hat U_{\nu,\xi}\cdot \nabla \Phi_{\varepsilon}
\end{equation}
where $(\nu,\xi)=t^{-1/2}(\lambda,x-x(t))$ and, anticipating step 3 below, $\hat U_{\nu,\xi}$ is an approximate solution close to $U_{\nu,\xi}$, see \eqref{matchedsol}. As a general feature of anomalous self-similarity, multiple scales arise which we first handle \emph{locally and separately}:
\begin{itemize}
\item In the inner zone $\{|z-\xi|\lesssim \nu\}$ the leading order term is $L_{in}\tilde v=\Delta\tilde v -\nabla \Phi_u\cdot \nabla \tilde v +2U\tilde v-\nabla U\cdot \Phi_{\tilde v}$. It is symmetric for the quadratic form $\langle u,v\rangle_{in}=\int U^{-1}uv-1/2\int \Phi_u v$ induced by the Hessian of the free energy around its minimizer $U$.
\item In the nearby parabolic zone $\{\nu\ll |z-\xi|\ll 1\}$ the leading term $L_{near}=\Delta+4|z-\xi(\tau)|^{-2}(z-\xi(\tau))\cdot \nabla$ is that around a moving Dirac mass. We introduce the new concept of \emph{parabolically averaged distance to the soliton}
\begin{equation} \label{introduction-averaged-distance}
d(z,\xi,\nu,\tau)= "\mbox{Average of }\tau'\mapsto |z-\xi(\tau')|\mbox{ over a }|z-\xi(\tau)|^2 \mbox{ timescale}",
\end{equation}
see Definition \ref{def:ParabolicDistance}.
That is, given a point $z$ at distance $|z-\xi(\tau)|$ from the soliton, we average the distance to the soliton over the forward parabola until it meets the soliton. With this distance we build the weighted space $L^2_{d^4}$, with scalar product $\langle \cdot,\cdot \rangle_{near}$ in which $L_{near}$ is dissipative \emph{and} the time oscillations of $\xi$ are controllable, thanks to \eqref{introduction-holder}.
\item In the parabolic outer zone $\{|z|\approx 1\}$ the operator $L_{\infty}=\Delta+(1/2+4|z|^{-2})z\cdot \nabla +1$ is well known. It is self-adjoint in the Gaussian space $L^2_{|z|^4e^{|z|^2/4}}$, with scalar product $\langle\cdot,\cdot\rangle_{out}$.
\item In the far away outer zone $\{|z|\gg 1\}$, the same operator, now considered on $L^1(|z|^2dz)$, generates a contraction semi-group.
\end{itemize}

\smallskip

\noindent \textbf{Step 2-B.} \emph{Global linear analysis}. The linearized evolution $\partial_\tau \varepsilon=L_{\nu,\xi}\varepsilon$ is then described \emph{globally on $\mathbb R^2$}, relying on the previous local linear analysis in two different ways:
\begin{itemize}
\item \emph{Inner-outer gluing for zones that can be isolated}. The far away outer zone $\{|z|\gg 1\}$ is, to leading order, isolated from the parabolic zone $\{|z|\lesssim 1\}$, as the solution mostly moves from the former to the latter. We thus employ \emph{inner-outer gluing} and look for a solution under the form
\begin{equation} \label{intro-inner-outer-gluing}
\varepsilon=\varepsilon_{in}+(1-\chi)\varepsilon_{out}.
\end{equation}
It transforms the evolution equation into a triangular system, where $\varepsilon_{out}$ evolves to leading order via $L_{\infty}$, and affects $\varepsilon_{in}$ via boundary terms. This is actually directly performed at the nonlinear level, see \eqref{introduction-inner-outer-system}.
\item \emph{Matched scalar product for zones that cannot be isolated}. The inner zone $\{|z-\xi|\lesssim \nu\}$, the nearby parabolic zone $\{\nu\ll |z-\xi|\ll 1\}$ and the parabolic outer zone $\{|z|\approx 1\}$ cannot be isolated one from another due to diffusion, and have to be handled altogether. We construct a \emph{matched scalar product} by gluing the corresponding scalar products together:
\begin{equation} \label{introduction-matched-product}
\langle \cdot,\cdot \rangle_* = "\langle \cdot,\cdot \rangle_{in}\mbox{ matched with }\langle \cdot,\cdot \rangle_{near}\mbox{ in }\{|z-\xi|\approx R\nu\}\mbox{ and matched with }\langle \cdot,\cdot \rangle_{out} \mbox{ in }\{|z|\approx \zeta_*\}",
\end{equation}
see \eqref{def:matchedscalarproduct}. We show for $R$ large and $\zeta_*$ small, that on a finite co-dimensional space it is indeed a scalar product and that the operator $L_{\nu,\xi}$ is dissipative with a precise spectral gap:
\begin{equation} \label{introduction-dissipation}
		\langle L_{\nu,\xi}(\varepsilon),\varepsilon \rangle_*
		\le
		(-2+ \delta_0)\langle \varepsilon,\varepsilon\rangle_* 
		-c_0\|\nabla \varepsilon \|_{L^2_{\omega}}^{2}
		+
		C\left[\nu^{6}
		\langle u\, ,(\Lambda U)_{\nu,\xi}
		\,\chi_\eta\rangle^{2}_{L^2}+\sum_{i=1}^{2}\nu^{8}
		\langle u\, \partial_{z_{i}}U_{\nu,\xi}
		\,\chi_\eta \rangle^{2}_{L^2}\right]
\end{equation}
where $\delta_0>0$ is arbitrarily small, and we refer to Proposition \ref{prop:global_matched_coercivity} for details on $c_0,\omega,C,\eta$.
\end{itemize}

\smallskip
\noindent \textbf{Step 3.} \emph{Justification of linearized dynamics, description of modulation, and rigidity of asymptotics}. To control the full nonlinear evolution in Section \ref{sec:refined-convergence}, we first replace the soliton $U_{\nu,\xi}$ by a better approximate solution, the \emph{matched soliton} $\tilde U_{\nu,\xi}$, see \eqref{eq:ansatz}. The decomposition \eqref{intro-inner-outer-gluing} of the remainder $\varepsilon$ leads to the inner-outer gluing system
\begin{align} \label{introduction-inner-outer-system}
& \left\{ \begin{array}{l l l l} & \partial_\tau\varepsilon_{\mathrm{in}}=L_{\nu,\xi}( \varepsilon_{\mathrm{in}})+E_{\nu,\xi}+A_{\nu,\xi}(\varepsilon_{\mathrm{in}})+\mbox{Mod}(\nu_\tau,\xi_\tau)+Q_{\mathrm{in}}+B[\varepsilon_{\mathrm{out}}], &\\
& \varepsilon_{\mathrm{in}}(\tau_0,z)= \varepsilon_0(z)\chi(z/4),&
\end{array} \right. \\
&\left\{ \begin{array}{l l l l} & \partial_\tau\varepsilon_{\mathrm{out}}=L_\infty \varepsilon_{\mathrm{out}}+b_\varepsilon\cdot\nabla\varepsilon_{\mathrm{out}}+c_\varepsilon\varepsilon_{\mathrm{out}}, & \qquad \mbox{for } |z|\geq 1, \\
& \varepsilon_{\mathrm{out}}(\tau,z)=0, & \qquad \mbox{for }|z|=1,\\
& \varepsilon_{\mathrm{out}}(\tau_0,z)=\varepsilon_0(z)(1-\chi_4(z)).&
\end{array} \right.
\end{align}

Above, the approximate solution $\tilde U_{\nu,\xi}$ generates the error terms $E_{\nu,\xi}$ and its motion generates the modulation terms $\mbox{Mod}(\nu_\tau,\xi_\tau)$. The small linear terms $A_{\nu,\xi}(\varepsilon_{\mathrm{in}})$ arising from the difference $\tilde U_{\nu,\xi}-U_{\nu,\xi}$ are irrelevant. The localized quadratic terms $Q_{\mathrm{in}}$ are small thanks to the pointwise bounds \eqref{introduction-pointwise}. The boundary terms are $B[\varepsilon_{\mathrm{out}}]$, and $b_\varepsilon,c_\varepsilon$ gather residual terms in the outer parabolic zone $\{|z|\geq 1\}$ from the soliton and $\varepsilon$.

Going further than Step 1-B, we now control the evolution in non-scale invariant spaces. First $\varepsilon_{\mathrm{out}}$ is estimated in Subsection \ref{subsec:outer} by combining spectral gap and semi-group type estimates,
\begin{equation} \label{introduction-bd-varepsilon-out}
\| \varepsilon_{\mathrm{out}}\| =o_{\tau\to \infty}(e^{-2\tau})
\end{equation}
in various norms $\| \cdot\|$ thanks to parabolic regularization. Then $\varepsilon_{\mathrm{in}}$ is controlled in Subsection \ref{subsec:inner} by an energy estimate based on the matched scalar product \eqref{introduction-matched-product}, the dissipation \eqref{introduction-dissipation} and eventually written as
\begin{align*}
		\partial_{\tau}\langle \varepsilon_{\mathrm{in}},\varepsilon_{\mathrm{in}} \rangle_*
		\le
		(-4+\delta)\langle \varepsilon_{\mathrm{in}},\varepsilon_{\mathrm{in}} \rangle_*
		-
		C\|\nabla \varepsilon_{\mathrm{in}}\|^{2}_{L^{2}_{\omega}}
		+
		o_{\tau \to \infty}\!\left( \left(\sup_{0<h<1}\frac{1}{h}\int_{\tau-h}^\tau\|\nabla \varepsilon_{\mathrm{in}}(\sigma)\|_{L^{2}_{\omega}}d\sigma\right)^{2}\right)
		+
		o_{\tau \to \infty}(e^{-2\tau}).
\end{align*}
Obtaining this identity is delicate and requires the following. The \emph{matched soliton} is chosen so that the error $E_{\nu,\xi}$ it generates is small, and so that the crossed term $\langle \varepsilon_{in},\mbox{Mod}(\nu_\tau,\xi_\tau)\rangle_*$ produces cancelations from explicit algebraic identities and orthogonality conditions. Precise quantitative versions of the modulation estimates \eqref{introduction-modulation}-\eqref{introduction-holder}-\eqref{introduction-averaged-distance} are needed to further bound $\langle \varepsilon_{in},\mbox{Mod}(\nu_\tau,\xi_\tau)\rangle_*$ and the time derivative of the matched scalar product itself. The nonlocal term above is due to parabolic averaging \eqref{introduction-averaged-distance}. Integrating in time, we obtain
\begin{equation} \label{introduction-bd-varepsilon-in}
\| \varepsilon_{\mathrm{in}}\| =o(e^{-2\tau})
\end{equation}
in various norms $\| \cdot \|$. The rigidity of the large time asymptotics then comes from the conservation of the second momentum \eqref{second-momentum},
$$
\int \tilde U_{\nu,\xi}|z|^2dz+\int \varepsilon |z|^2dz=\mu e^{-2\tau},
$$
as the second term is negligible since $o(e^{-2\tau})$ by \eqref{introduction-bd-varepsilon-out}-\eqref{introduction-bd-varepsilon-in}, so that an explicit computation for the first solitonic term shows $\nu(\tau)= \sqrt{\mu /8\pi} \sqrt{\tau}^{-1}e^{-\tau/2}(1+o_{\tau\to \infty}(1))$. From there, further refinement for the convergence follows from bootstrap-type arguments.

\subsubsection{Novelties and difficulties}

\paragraph{\emph{(i) Completing the roadmap of the classification of soliton-based asymptotic behaviours for data in the large}} A few works on soliton resolution (Step 1-A above) show the convergence of general solutions to a sum of solitons, without describing further their modulation parameters and the remainder (see \cite{CDKM1} for a review). More numerous works on well-prepared initial data construct examples of solutions with precise behaviour (Step 3 above); this the case of \cite{GM,DdPDMW} showing the existence of a solution behaving as in Theorem \ref{main:th:main}. In order to obtain the precise behaviour for general solutions, one has to \emph{develop the bridge between soliton resolution, and the perturbative analysis of well-prepared data}. The main novelty here is the identification of a clear roadmap, where the bridge consists in the new Step 1-B, Step 2, and to finding a robust and flexible framework for Step 3 (in particular we use a simpler approximate solution compared to \cite{GM,DdPDMW}).

For $1d$ or radial parabolic problems, classification of precise asymptotic behaviours for data in the large was obtained thanks to comparison principle and intersection number counting \cite{MaMe,M3,M}. This was not applicable to the present problem since it is non-radial and non-local; the goal of the present roadmap was to generalize these previous approaches. The recent works \cite{JeLa,KiMe,KiMe2} were able to obtain precise behaviours, in cases where the interaction between solitons drives the dynamics. The difference here is that the interaction between the soliton and the remainder drives the dynamics, which is why we believe our new framework to study precisely the remainder is relevant for future works. A recent work \cite{JKKK} for the 1d Calogero-Moser derivative NLS developed another approach to control the remainder, based on the complete integrability of their equation.

\smallskip

\paragraph{\emph{(ii) Initiation of non-radial soliton resolution for $2d$ Keller-Segel}} The present work starts the study of non-radial soliton resolution for Equation \eqref{KS} (Theorem \ref{thn:soliton-resolution}). So far only the emergence of Dirac masses in the sense of measures was known (see \cite{BCM}  for the mass critical case and \cite{SBook} for the mass supercritical case. Our approach relies on ideas for soliton resolution ranging from the seminal work \cite{St} to more recent non-radial works \cite{DJKM}. We expect it to be extendable in other cases to the Keller-Segel equation.

\smallskip

\paragraph{\emph{(iii) Identification of scale-invariant spaces for modulation enabling linearized analysis}}

 A major difficulty is that even if one knows soliton resolution \eqref{intro-sol-res}, which describes the solution in the inner zone $|x-x(t)|\lesssim \lambda(t)$, the determination of the precise dynamics requires to understand the solution at the much larger parabolic scale $|x|\lesssim \sqrt{t}$ and over the much longer parabolic time scale. For previous perturbative constructions \cite{GM,DdPDMW} well-prepared initial data ensure the remainder is localized in the parabolic zone $\{|x|\approx \sqrt{t}\}$ and that the modulation parameters vary slowly $|\dot x(t)|,|\dot \lambda(t)|\lesssim t^{-1}\lambda(t)$. For unprepared initial data however, we only know a priori that $|\dot x(t)|,|\dot \lambda(t)| \lesssim (\lambda(t))^{-1}$ which can allow for arbitrarily fast oscillations! We face the difficulty of \emph{describing linearized dynamics around a possibly fastly oscillating soliton}.

A key novelty of the present work is that such description of the linearized dynamics is possible, by combining the standard modulation estimates \eqref{introduction-modulation} that control the dynamics at the \emph{soliton time scale} $|t'-t|\lesssim \lambda^2$, with estimates in the scale-invariant H\"older $1/2$ space \eqref{introduction-holder} to control the dynamics up to the \emph{parabolic time scale} $|t'-t|\lesssim t$. At a technical level, we introduce the novel concept of parabolic averaging \eqref{introduction-averaged-distance} to control the dynamics near the tail of the soliton. The idea is that at a point $x$ far away from a soliton, its tail is not well described by the far field expansion $\lambda^2|x-x(t)|^{-4}$ of $U_{\lambda,x(t)}(x)$, because $\lambda$ and $x(t)$ are determined from local orthogonalities in the inner zone which is distant from $x$. We thus replace it by a parabolically averaged tail, and this removes all issues related to the possibly fast oscillations of $\lambda$ and $x(t)$! We believe this may have applications to other problems.

\smallskip

\paragraph{\emph{(iv) New techniques for multiscale analysis for the remainder in classification problems}} We develop here for the first time the concept of \emph{matched scalar product} for solutions in the large, see \eqref{introduction-matched-product}. It was previously developed by the second author and collaborators for well-prepared initial data in the context of blow-up \cite{CGMN1,CGMN2}. The idea is that it is a counterpart to \emph{matched asymptotics} which produce solutions to multiscale problems by matching solutions of localized problems, in that we match local functionals (here scalar products) to produce a globally defined functional to control the linearized flow. Due to the poor control of the modulation parameters (see (iii) above), we have to be more careful in the design of the matched product. Its application to a classification problem shows the robustness of this concept.

A second novelty is the use of \emph{inner-outer gluing} for solutions in the large, see \eqref{intro-inner-outer-gluing}. It was previously developed for well-prepared initial data \cite{DdPDMW,DDW}. This enables us to isolate the dynamics in the far away outer region $|x|\gg \sqrt{t}$, and to get our result in the optimal space $L^1(|x|^2dx)$ (see discussion above). This idea can also be applied to other classification problems.

\smallskip 

\noindent \emph{(v) Applications of $\varepsilon$-regularity theory to get scale-invariant weighted $L^\infty$ decay}. In order to show that the effects of some nonlinear terms are negligible, we prove the pointwise bounds \eqref{introduction-pointwise} for the remainder. A novelty of the present work is to obtain such bounds using techniques from $\varepsilon$-regularity, but in a classification context. Namely, we prove a new quantitative $\varepsilon$-regularity result, Theorem \ref{thm:quantitative-epsilon-reg}. Previous results (e.g., \cite{SBook}) did not establish pointwise smallness or uniformity, whereas we establish both. As a result, we get robust scaled $L^1_{loc}\to W^{1,\infty}_{loc}$ estimates for the Keller-Segel equation.

\subsubsection{Organization of the article}

After introducing the notation in Section \ref{sec:notation}, well-posedness results and regularization estimates are given in Section \ref{sec:LWP}. Section \ref{sec:soliton-resolution} shows soliton resolution in $L^1 $, recalling the classification of stationary states in Subsection \ref{sec:stationary-states}, and then showing it along a sequence of times in Subsection \ref{subsec:sequential} and continuously in time in Subsection \ref{subsec:continuous-sol-res}. The convergence in soliton resolution is showed in other scale invariant spaces in Section \ref{sec:roughconv}: for the modulation parameters derivatives are estimated Subsection \ref{subsec:rough-modulation} and H\"older $1/2$ norms is estimated in Subsection \ref{subsec:slowvar-modulation}, and for the remainder pointwise estimate are given in Subsection \ref{subsec:pointwise-remainder}. The linearized evolution is studied in Section \ref{section:LinAnalysis} where the parabolic-outer, far away outer, nearby outer and inner zones are studied separately in Subsections  \ref{subsection:parabolic-outer}, \ref{subsection:far-outer}, \ref{subsection:nearby-outer} and \ref{subsection:inner} respectively. Then the matched scalar product is defined in Subsection \ref{subsection:matched-sp} and the corresponding coercivity of the operator is shown in Subsection \ref{subsection:coercivity}. The refined convergence is then obtained in Section \ref{sec:refined-convergence}. Quantitative versions of the modulation estimates of Subsections \ref{subsec:rough-modulation} and \ref{subsec:slowvar-modulation} are given in Subsection \ref{subsec:quant-mod}. The ansatz for the solution is described in Subsection \ref{subsec:ansatz}, and the outer and inner parts of the remainder are estimated in Subsections \ref{subsec:outer} and \ref{subsec:inner} respectively. The precise asymptotic behaviour is obtained in Subsection \ref{subsec:asymptotics}. The appendix is mostly devoted to the proof of the new $\varepsilon$-regularity result in Section \ref{sec:epsilon-regularity} and contains some standard estimates for the Poisson field in Section \ref{sec:poisson}.

\subsection{Acknowledgements}

Funded by the European Union. Views and opinions expressed are however those of the author(s) only and do not necessarily reflect those of the European Union or the European Research Council Executive Agency. Neither the European Union nor the granting authority can be held responsible for them.

This work is supported by ERC grant (FloWAS, No. 101117820, 
DOI 10.3030/101117820). C. Collot was supported by the CY Initiative of Excellence Grant "Investissements d'Avenir" ANR-16-IDEX-0008 via Labex MME-DII, the grant "Chaire Professeur Junior" ANR-22-CPJ2-0018- 01 of the French National Research Agency, and the grant BOURGEONS ANR-23-CE40-0014-01 of the French National Research Agency.

The authors thank Pierre Rapha\"el for suggesting this problem. They also thank Kleber Carrapatoso, Kihyun Kim, Frank Merle and Van Tien Nguyen for related stimulating discussions.

\section{Notation} \label{sec:notation}

\subsection*{Analysis on $\mathbb R^2$} The Euclidean basis is
$$
e_1=(1,0)^\top \quad \mbox{and} \quad e_2=(0,1)^\top.
$$
The gradient is $\nabla u = (\partial_{x_1}u,\partial_{x_2}u)^\top$ and the divergence is $\nabla \cdot f = \partial_{x_1}f_1+\partial_{x_2}f_2$. For $\alpha= (\alpha_1,\alpha_2)$ we write $\partial^\alpha u =\partial_{x_1}^{\alpha_1}\partial^{\alpha_2}_{x_2}$. For $k\in \mathbb N$, the vector of all $k$-th order derivatives is
$$
\nabla^k u = (\partial^{\alpha} u)_{\alpha_1+\alpha_2=k}.
$$
The ball is $B_R(x_0)=\{x\in \mathbb R^2, \ |x-x_0|\leq R\}$ where $|\cdot|$ denotes the Euclidean distance. The complement of a set $E$ is $E^c=\{x\in \mathbb R^2, \ x\notin E\}$. The Japanese bracket is
$$
\langle x \rangle=\sqrt{1+|x|^2}.
$$

\subsection*{Renormalizations}
Given a function $u$ on $\mathbb R^2$ and parameters $\lambda>0$ and $y\in \mathbb R^2$, we write
$$
u_{\lambda,y}=\frac{1}{\lambda^2} u(\frac{x-y}{\lambda}).
$$
The above notation will be the convention for rescaling and translation of all functions except for the cut-off function $\chi$, for which we will use a different convention, see \eqref{id:notation-cutt-off}. The scaling operator is
$$
\Lambda u = 2u+x\cdot \nabla u.
$$

\subsection*{Function spaces}

For $k\in \mathbb N$ and $p\in [1,\infty]$ the standard Sobolev space $W^{k,p}(\mathbb R^2)$ is the closure of smooth and compactly supported functions for the norm
$$
\| u \|_{W^{k,p}(\mathbb R^2)}= \sum_{l=0}^k \| \nabla^l u \|_{L^p(\mathbb R^2)}.
$$
The notation $C^k$ refers to standard spaces of $k$-times differentiable functions, whose norm is the sum of the supremum of all such derivatives. The H\"older space $\dot C^{1/2}$ is the one associated to the norm
$$
\| f\|_{\dot C^{1/2}([T,\infty))}= \sup_{t,t'\geq T, \ t\neq t'} \frac{|f(t)-f(t')|}{\sqrt{t-t'}}.
$$
The space of smooth compactly supported functions is $C^\infty_c$. The weighted spaces $L^2_w$ are associated to the norm
$$
\| u\|_{L^2_\omega}^2=\int u^2\omega.
$$

\subsection*{The heat equation}

The two-dimensional heat kernel is defined by
\[
G(t,x)=\frac{1}{4\pi t} e^{-\frac{|x|^2}{4t}},
\]
and the associated heat semigroup is given by
\[
S(t)u_0 = G(t)*u_0.
\]

\subsection*{Cut-offs}
We denote by \(\chi\) a smooth cut-off function satisfying $\chi(x)=1 \ \text{for } |x|\le 1$ and $\chi(x)=0 \ \text{for } |x|\ge 2$. In addition, $\chi$ is radially symmetric and radially decreasing. For any \(R>0\) and $y\in \mathbb R^2$, we define the rescaled and translated cut-off
\begin{equation} \label{id:notation-cutt-off}
\chi_{R}(x)=\chi\!\left(\frac{x}{R}\right) \quad \mbox{and} \quad \chi_{R,y}(x)=\chi\!\left(\frac{x-y}{R}\right).
\end{equation}

\subsection*{Inequalities}

We write \(A \lesssim B\) if there exists a universal constant \(C>0\) such that \(A \le C B\).
Similarly, \(A \gtrsim B\) means \(B \lesssim A\), and \(A \approx B\) means that both
\(A \lesssim B\) and \(B \lesssim A\) hold.

When needed, we indicate the dependence of the implicit constant by a subscript.
For instance, \(A \lesssim_{\eta} B\) means that \(A \le C_{\eta} B\), where the constant
\(C_{\eta}>0\) depends only on the parameter \(\eta\).

We write $A=o(B)$ as $t\to \infty $ (or another limit) if $A/B \to 0$.

The notation $o_{\alpha}(1)$ denotes a quantity converging to zero in the relevant limit, with a rate that may depend on the parameter $\alpha$.

\section{Local well-posedness and parabolic regularization} \label{sec:LWP}

\subsection{Local well-posedness in $L^1$}

The purpose of this section is to recall the local well-posedness of the Keller-Segel system \eqref{KS} in the critical space $L^1$ in two dimensions. The arguments are standard, following the approach by Weissler \cite{Weissler1,Weissler2} and Brezis-Cazenave \cite{BrCa}. More details can be found in the companion paper \cite{BC}.

\begin{definition} \label{def:solution}
	
	Given $u_0\in L^1$, a mild solution on $[0,T)$ is a function $u$ such that $u\in C([0,T),L^1(\mathbb R^2))$ and $t^{1/4}u(t)\in L^\infty_{loc}([0,T),L^{4/3})$ with $\lim_{t\downarrow 0}t^{1/4}\|u(t)\|_{L^{4/3}}=0$ for which
	\begin{equation} \label{id:mild-solution}
		u(t)=G(t)*u_0+\int_0^t \nabla G(t-s)* \cdot (u(s)\nabla \Phi_{u(s)})ds 
	\end{equation}
	for $t\in [0,T)$.

\end{definition}

One can verify for a mild solution, using for example the Hardy-Littlewood-Sobolev, Young and H\"older inequalities, that $s\mapsto \nabla G(t-s)* \cdot (u(s)\nabla \Phi_{u(s)})\in L^1([0,t]\times \mathbb R^2)$ for $t\in [0,T)$ from the regularity assumptions on $u$; in particular the right-hand side above in \eqref{id:mild-solution} is well-defined and this identity makes sense as an equality between measurable functions.

In fact the concept of mild solutions is only needed around the initial time $t=0$, as the result below shows that mild solutions are instantaneously smooth and become classical solutions.

\begin{theorem} \label{th:lwp:L1}
	
	Let $u_0\in L^1(\mathbb R^2)$. Then there exists a maximal time of existence $T=T(u_0)$ and a mild solution $u$ to the Keller-Segel system \eqref{KS} on $[0,T)$ such that
	\begin{itemize}
		\item \emph{Continuity of the flow in $L^1$.} We have $u\in C([0,T(u_0)),L^1(\mathbb R^2))$, and the conservation of mass \eqref{id:mass-conservation}. Moreover, the flow is locally continuous in that for any $\delta>0$ and $0<T'<T$, there exists $\epsilon>0$ (all depending on $u_0$) such that if $\|v_0-u_0\|_{L^1}\leq \epsilon$, then $T(v_0)\geq T'$ and $\| v(t)-u(t)\|_{L^1}\leq \delta$ for $t\in [0,T']$.
		\item \emph{Instantaneous regularization}. We have $u,\nabla \Phi_u\in C^\infty((0,T)\times \mathbb R^2)$ and $u$ is a classical solution to \eqref{KS} on $(0,T)$. Moreover, for any $k\in \mathbb N$, there holds $u\in C((0,T),W^{k,\infty} \cap W^{k,1} (\mathbb R^2))$.
		\item \emph{Blow-up criterion}. We have $T(u_0)<\infty$ if and only if $
		\lim_{t\uparrow T} \| u(t,\cdot)\|_{L^\infty(\mathbb R^2)}=\infty 
		$, and in this case
		$$
		\| u(t,\cdot)\|_{L^\infty(\mathbb R^2)}\geq \frac{1}{T-t} .
		$$
		\item \emph{Uniqueness}. The solution is unique in the class of mild solutions in the sense that if $u'$ is another mild solution on $[0,T')$ with the same initial data $u_0$, then $T'\leq T$ and $u=u'$ on $[0,T')$.
	\end{itemize}
	
\end{theorem}

\begin{proof}
The proof relies on the Hardy–Littlewood–Sobolev inequality, and then the standard approach based on a semilinear fixed point argument and heat kernel estimates, as in the works of Weissler, Brezis and Cazenave. We refer the interested reader to the companion note \cite{BC} for the complete proof.

\end{proof}
We point out that the above blow-up criterion is not optimal in the radial setting. Indeed, it is known (see \cite{S, M}) that radial blow-up solutions satisfy
\begin{align}\label{typeIIblow-up}
\limsup_{t\to T}(T-t)|u(t)|_{L^{\infty}(\R^{2})}=\infty.
\end{align}
Furthermore, \cite{NaiSu} established the same conclusion in bounded domains without any radial symmetry assumption, under the additional hypothesis of finite entropy of the initial data. In \cite{SBook}, a proof of \eqref{typeIIblow-up} under the additional assumption of finite second moment of the initial data is provided.

The solutions of Theorem \ref{th:lwp:L1} coincide with all known solutions for $u_0\in L^1$ as in \cite{BCM}. Indeed, this theorem states that the flow map for Schwartz initial data $u_0\in \mathcal S$ (which can be constructed easily by standard arguments) extends continuously and uniquely in $L^1$, in the sense that given any $u_0\in L^1$ and any Schwartz approximating sequence $u_{n,0}\in \mathcal S$ that converges to $u_0$ in $L^1$ we have $\liminf T(u_{n,0})\geq T(u_n)$ and $u_n\to u$ in $C([0,T'],L^1)$ for any $T'<T(u_0)$.

The following results will be useful in the proof of the soliton resolution later on. Both are regularization estimates that are uniform with respect to a set of initial data. The first one is a short time regularization estimate for bounded initial data.

\begin{lemma}\label{Solitoncorollary}
	For any $R>0$, there exists $T_1(R)>0$, and, for any $T_0<T_1(R)$ and $N\in \mathbb N$, there exists a constant $C(R,T_0,N)$ such that if $u_0\in L^1\cap L^\infty$ satisfies
	\begin{align*}
		\|u_{0}\|_{L^{1}\cap L^\infty}\leq R
	\end{align*}
	then the solution to the Keller-Segel system given by Theorem \ref{th:lwp:L1} exists on the time interval $[0,T_1(R)]$. Moreover, it satisfies
	\begin{align}\label{ReguliarityPara}
		\sup_{T_{0}\le t \le T_{1}} \| \partial_t^\alpha \nabla^{\alpha'} u(t,\cdot )\|_{L^{1}(\mathbb{R}^{2})\cap L^{\infty}(\mathbb{R}^{2})}\le C[N,T_{0},R]
	\end{align} 
	for all $\alpha+\alpha'\leq N$.
\end{lemma}

\begin{proof}
		
		Let $p=4/3$. Let $u_0\in L^1\cap L^\infty$ with $\| u_0\|_{L^1\cap L^\infty}\leq R$. The interpolation inequality $\| u_0\|_{L^p}\leq \| u_0\|_{L^1}^{\frac 1p}\| u_0\|_{L^\infty}^{1-\frac 1p}\leq R$  allows to apply Theorem 3.5. in \cite{BC} that shows that the solution exists at least on the time interval $[0,T^*(R)]$ where $T^*(R)$ is independent of $u_0$. Moreover, by Corollary 3.9. in \cite{BC} we have uniform parabolic smoothing estimates. More precisely, for every \(\tau_0>0\) and every \(k\in\mathbb N\),
\begin{align}\label{Corollary2.8smoothing}
\sup_{\tau_0 \leq t \leq T^*}
\|\partial_t^\alpha \nabla^{\alpha'}u(t)\|_{L^\infty\cap L^{4/3}(\mathbb R^2)}
\le C(k,\tau_0,R),
\qquad \alpha+\alpha'\le k.
\end{align}
		
		There remains to show the $L^1$ estimates for $\partial^\alpha \nabla^{\alpha'}u$. We first show the $L^1$ estimate for $\nabla^{k}u$ for any $k\in \mathbb N$. The estimate for $k=0$ follows from the conservation of mass. Let now $i\in \{1,2\}$ and $t_0\in (0,T^*(R))$. We have for all $t\in [t_0,T^*(R)]$ that
		\begin{equation} \label{identity-partialtderivatives2}
		\partial_{x_i} u(t) = \partial_{x_i} G(t-t_0)u(t_0)+\int_{t_0}^t \nabla G(t-s)\cdot (\partial_{x_i}u\nabla \Phi_u (s)+u\nabla \Phi_{\partial_{x_i} u} (s))ds.
		\end{equation}
		Applying the standard heat kernel estimate for the first term and Hardy-Littlewood-Sobolev inequality (see Lemma 3.4. in \cite{BC}) for the second term shows,
		$$
		\| \partial_{x_i} u(t)\|_{L^1}\lesssim \frac{\| u(t_0)\|_{L^1}}{\sqrt{t-t_0}}+\int_{t_0}^t \frac{1}{\sqrt{t-s}} \| \nabla u(s) \|_{W^{1,4/3}(\mathbb R^2)}^2 ds \lesssim \frac{R}{\sqrt{t-t_0}}+C(1,\epsilon,R)^2
		$$
		where $C(l,\epsilon,R)$ is given in \eqref{Corollary2.8smoothing}. This shows the desired $L^1$ bound for $\nabla u$. The bounds for higher order derivatives are proved analogously by applying spatial derivatives to \eqref{identity-partialtderivatives2}, and we get
        \begin{align}\label{Corollary2.8smoothingbis}
        \sup_{\tau_0 \leq t \leq T^*}
\| \nabla^{\alpha'}u(t)\|_{L^1(\mathbb R^2)}
\le C(k,\tau_0,R),
\qquad \alpha'\le k
		\end{align}
        for any integer $k$, up to changing the constant $C(k,\tau_0,R)$. Now that all spatial derivatives \(\nabla^{\alpha'}u\), \(\alpha'\in\mathbb N\), are bounded in \(L^1\cap L^\infty\), we estimate mixed space-time derivatives by induction using the equation. Namely, for $\alpha=1$ we have
        \begin{equation} \label{identity-partialtderivatives}
        \partial_t u = \Delta u +u^2-u\cdot \nabla \Phi_u.
        \end{equation}
        Combining \eqref{Corollary2.8smoothing} and \eqref{Corollary2.8smoothingbis}, and using H\"older's inequality, one easily bounds in $W^{k,1}$ all terms in the right-hand side, showing $\| \partial_t \nabla^{\alpha'}u(t)\|_{L^1(\mathbb R^2)}
\le C(k,\tau_0,R)$ for $\tau_0\leq t \leq T^*(R)$, for all $\alpha'\le k$, up to possibly taking a larger constant $C(k,\tau_0,R)$. Higher order $\partial_t$ derivatives are estimated similarly by induction by differentiating \eqref{identity-partialtderivatives} with respect to $t$. One then obtains bounds for all quantities \(\partial_t^\alpha\nabla^{\alpha'}u\). This finishes the proof.
		
\end{proof}

The second result is a large time regularization estimate in the vicinity of the ground state.

\begin{lemma} \label{lem:uniform-regularization-near-soliton}
	
	For any $\delta>0$, $0<T_0<T_1$ and $N\in \mathbb N$, there exists $\epsilon>0$ such that if $u_0\in L^1$ satisfies
	\begin{align*}
		\|u_{0}-U\|_{L^{1}}\leq \epsilon
	\end{align*}
	then the solution to the Keller-Segel system given by Theorem \ref{th:lwp:L1} exists on the time interval $[0,T_1]$. Moreover, it satisfies
	\begin{align}\label{ReguliarityPara-near-ground-state}
		\sup_{T_{0}\le t \le T_{1}} \| \partial_t^\alpha \nabla^{\alpha'} \left(u(t,\cdot )-U(\cdot)\right)\|_{L^{1}(\mathbb{R}^{2})\cap L^{\infty}(\mathbb{R}^{2})}\le \delta
	\end{align} 
	for all $\alpha+\alpha'\leq N$.
	
\end{lemma}

	\begin{proof}
Writing
$u=U+\tilde u$,
where \(U\) is the stationary solution, the perturbation \(\tilde u\) satisfies
\[
\partial_t\tilde u
=\Delta\tilde u
-\nabla\cdot(\tilde u\nabla\Phi_{\tilde u})
-\nabla\cdot(U\nabla\Phi_{\tilde u})
-\nabla\cdot(\tilde u\nabla\Phi_U).
\]

The proof of the existence a straightforward perturbative modification of the proof of Theorem~3.11. in \cite{BC}. For \(\|\tilde u_0\|_{L^1}\)
sufficiently small, the additional terms involving \(U\) are controlled in the
same functional framework and yield a contraction mapping on a suitable ball.
This gives existence on \([0,T_1]\) together with
\[
\sup_{T_0\le t\le T_1}\|\tilde u(t)\|_{L^1\cap L^\infty}
\le \delta
\]
provided $\epsilon$ is small enough.

The estimates for higher-order spatial and mixed space-time derivatives then
follow by repeating the standard parabolic regularity arguments in the proof of Corollary 3.9. in \cite{BC} and the proof of Lemma \ref{Solitoncorollary} here.
\end{proof}

\subsection{Local well-posedness in $L^1(\langle x \rangle^2 dx)$}

We state below the well-posedness of the Keller-Segel equation \eqref{KS} for the initial data studied in the present article satisfying \eqref{id:hp-data-2}.

\begin{theorem} \label{thm:lwp-L1x2}
	
	For any nonnegative $u_0\in L^1(\langle x \rangle^2 dx)$, let $u$ be the solution to the Keller-Segel system \eqref{KS} provided by Theorem \ref{th:lwp:L1}. Then the following properties hold
	
	\begin{itemize}
		\item \emph{Continuity of the flow in $L^1(\langle x \rangle^2dx)$.} We have $u\in C([0,T(u_0)),L^1(\langle x \rangle^2 dx))$.
		\item \emph{Instantaneous boundedness and decrease of the free energy functional}. For all $t\in (0,T(u_0))$ we have $u(t)|\log u(t)|,u(t)|\Phi_{u(t)}|\in L^1(\mathbb R^2)$, so that the free energy $\mathcal F[u](t)$ \eqref{freeenergy} is well-defined. We also have that $u|\nabla \log u-\nabla \Phi_u|^2\in L^1_{loc}((0,T),L^1(\mathbb R^2))$ and the identity \eqref{freeenergy-dissipation} holds true for all $0<t<t'<T(u_0)$.
		\item \emph{First and second momenta}. There hold the identities \eqref{first-momentum} and \eqref{second-momentum}.
	\end{itemize}

\end{theorem}

\begin{proof}
We refer the interested reader to Theorem 4.1. in the companion note \cite{BC} for the complete proof.
	
\end{proof}

\section{Soliton Resolution} \label{sec:soliton-resolution}

Throughout the rest of the article, we will study solutions $u$ of the Keller-Segel equation such that $u_0\in L^1(\langle x \rangle^2 dx) $ and with critical mass  $\int_{\mathbb R^2} u_0 dx =8\pi$, as provided by Theorems \ref{th:lwp:L1} and \ref{thm:lwp-L1x2}. We will stop mentioning explicitly Theorems \ref{th:lwp:L1} and \ref{thm:lwp-L1x2} in the sequel, unless at key locations, and will always implicitly assume the properties they provide (regularity of the solution, well-definedness of the free energy functional, etc.). We will now start establishing additional properties of these solutions, until we are able to prove the main result of the article. We begin with the properties recalled in the introduction:
\begin{equation} \label{id:property-1}
\mbox{the solution $u$ is global, i.e. it exists for all }t\in [0,\infty),
\end{equation}
and moreover, its mass, center of mass and second momentum are constant over time,
\begin{equation} \label{id:property-2}
\int_{\mathbb R^2} u(t,x) dx =8\pi, \quad (8\pi)^{-1}\int_{\mathbb R^2} x u(t,x) dx =x^* \quad \mbox{and} \quad \int_{\mathbb R^2} |x-x^*|^2 u(t,x) dx =\mu.
\end{equation}
The main result of this section is the continuous resolution in the critical space $L^1$ into a stationary state as $t\to \infty$, which is the following Theorem.

\begin{theorem}\label{thn:soliton-resolution}
	Let $u$ be a solution of the Keller-Segel equation satisfying $u_0\in L^1(\langle x \rangle^2 dx)$ and $\int u_0=8\pi$, and recall it then satisfies \eqref{id:property-1}-\eqref{id:property-2}. There exist continuous functions $\lambda(t)>0$ with $\lim_{t\to \infty} \lambda(t)=0$ and $x(t)\in \mathbb{R}^{2}$ with $\lim_{t\to \infty}x(t)=x^*$ such that $u$ can be decomposed as
	\begin{align*}
		u(x,t)=\frac{1}{\lambda^{2}(t)}U\big( \frac{x-x(t)}{\lambda(t)} \big)+\tilde u(x,t)
	\end{align*}
	where as $t\to \infty$, the remainder satisfies $\tilde{u}(t)\to 0$ in $L^1$, and more generally for any $k\in \mathbb N$,
	\begin{align*}
		 \lim_{t\to \infty}(\lambda^{2+k}(t)\| \nabla^k \tilde{u}(t)\|_{L^\infty}+\lambda^{k}(t)\| \nabla^k \tilde{u}(t)\|_{L^{1}})=0.
	\end{align*}
	
\end{theorem}

\begin{proof}

This will be the content of this section; the result being eventually a direct consequence of Proposition \ref{continuoussolresPR} and Lemma \ref{limpara}.

\end{proof}

\subsection{The Free Energy and classified solutions to the Liouville equation} \label{sec:stationary-states} 

In this section we will introduce some properties of the \emph{free-energy} \eqref{freeenergy} associated to \eqref{KS}. We recall that the free energy is instantaneously well-defined and dissipated for solutions starting in $L^1(\langle x \rangle^2dx)$, see Theorem \ref{thm:lwp-L1x2}.

\subsubsection{The Keller-Segel equation as a gradient flow}

The Keller-Segel equation \eqref{KS} can be written as a gradient flow evolution of the free energy functional \eqref{freeenergy} (see \cite{BlCalCar} and \cite{O}),
$$
\partial_t u = \nabla \cdot \left( u \frac{\delta \mathcal F}{\delta u}\right)
$$
where $\delta \mathcal F/\delta u$ denotes the differential of $\mathcal F$. We now explain this at a formal level. Given an initial datum $u_0\in L^1(\langle x \rangle^2dx)$, the Keller-Segel evolution lies on the manifold $\mathcal M=\{v\geq 0, \ \int v dx=\int u_0 dx, \ \int v\langle x \rangle^2dx<\infty \}$. This manifold can be equipped with the following Riemannian structure. Given a $v\in \mathcal M$, functions of the tangent space $T\mathcal M (v)=\{w\in L^1(\langle x \rangle^2 dx), \ \int w dx=0\}$ can formally be written under the form $w=\nabla \cdot (v\nabla p)$ where $p$ is defined modulo constant. The metric $g$ defined by the following formula at $v\in \mathcal M$,
$$
g(w,w')= \int v \nabla p\cdot \nabla p' =-\int p w',
$$
is the one that induces the Wasserstein distance on $\mathcal M$. The Keller-Segel equation can then be written as
$$
\partial_t u=-\textup{grad} \mathcal F(u)
$$
where $\textup{grad} \mathcal F$ denotes the gradient of $\mathcal F$ for the metric $g$.

\subsubsection{Classification of stationary states}
Recall the free energy dissipation identity \eqref{freeenergy-dissipation}, where the dissipation rate 
$
\int u \, |\nabla \log u - \nabla \Phi_u|$
can be interpreted as a Fisher information. We now invoke a classical rigidity result, stated precisely below, which asserts that all functions with vanishing dissipation rate correspond to the unique stationary state, up to scaling and translation. This rigidity ultimately follows from the classification of solutions to the Liouville equation established by W.~Chen and C.~Li \cite{CL}.

Indeed, by Theorem~1 in \cite{CL}, any solution of  
\begin{align}\label{LiouvilleEq}
	\begin{cases}
		-\Delta u = e^{u}, & \text{in } \mathbb{R}^{2}, \\[0.3em]
		\displaystyle \int_{\mathbb{R}^{2}} e^{u(x)} \, dx < \infty,
	\end{cases}
\end{align}
is radially symmetric and necessarily takes the explicit form  
\begin{equation} \label{LiouvilleEq-2}
\phi_{\lambda,y}(x) 
= \log \!\left( \frac{8 \lambda^{2}}{(\lambda^{2} + |x-y|^{2})^{2}} \right),
\end{equation}
for some $\lambda > 0$ and $y \in \mathbb{R}^2$. 

\begin{theorem}[Classification of zero dissipation rate functions] \label{th:rigidity}
	
	Assume that $f>0$ and $g$ are smooth function on $\mathbb R^2$ with $f\in L^1$ which solve
	$$
	\left\{ \begin{array}{l l}
		\nabla \log f-\nabla g=0 , \\   -\Delta g=f .
	\end{array}\right.
	$$
	on $\mathbb R^2$. Then there exist $\lambda>0$ and $y\in \mathbb R^2$ such that
	$$
	f(x)=\frac{1}{\lambda^2}U(\frac{x-y}{\lambda}).
	$$
	
\end{theorem}

\begin{proof}
	
	By the first equation we know that $\log f-g=C$ for some integration constant $C$. Up to changing $g$ to $g+C$ we can assume $C=0$ without loss of generality. Hence $f=e^g$, so by the second equation the function $g$ solves the Liouville equation \eqref{LiouvilleEq}. By \cite{CL}, $g$ is of the form \eqref{LiouvilleEq-2}. Hence the result using $f=e^g$.
	
\end{proof}

As a corollary, we have the following result which will be used to show soliton resolution later on.

\begin{corollary} \label{cor:rigidity}
	
	Let $f\in C^\infty (\mathbb R^2)\cap L^1(\mathbb R^2)$ be nonnegative and satisfy
	\begin{equation}\label{FundIdent}
			\nabla \log f-\nabla \Phi_{f}=0 \ \ \text{on the set }\{f>0\}, 
	\end{equation} 
	where $\nabla \Phi_f$ is given by \eqref{id:definition-nablaPhiu}, with $ f \le 1$ on $\mathbb R^2$ and $f(0)=1$. Then 
	$$
	f(x)= \frac{1}{8} U \left( \frac{x}{\sqrt{8}} \right).
	$$
	
\end{corollary}

\begin{proof}
	
	We first claim that $\nabla \Phi_f \in C^\infty (\mathbb R^2)$. To show it, we pick an arbitrary $R>0$, and consider the set $\{|y|<R/10\}$. Using that $\chi_R(a)$ is $1$ for $|a|\leq R$ and $0$ for $|a|\geq 2R$ we can write
	$$
	\nabla \Phi_f(y)=  -\frac{1}{2\pi}\int\frac{y-z}{|y-z|^{2}}(1-\chi_{R/10}(z-y))(1-\chi_R(z))f(z)dz -\underbrace{\frac{1}{2\pi}\int_{\R^2}\frac{y-z}{|y-z|^{2}}f(z)\chi_R(z)dz}_{=-\nabla \Phi_{\chi_R} f(y)}.
	$$
	The first term is $C^\infty$ since the underlying convolution kernel is smooth, and the second term is smooth since it is the gradient of $-\frac{1}{2\pi}\log(|\cdot|)*(\chi_R f)$ which is smooth by elliptic regularity. Hence $\nabla \Phi_f $ is $C^\infty$ on the set $\{|y|<R/10\}$, so that $\nabla \Phi_f$ is smooth on $\mathbb R^2$ as $R$ is arbitrary.
	
	Let $\Omega=\{f>0\}$. Since $-\nabla \cdot \nabla \Phi_{f} = f$, by setting $H_{f} := \log f$ \eqref{FundIdent} yields
	\[
	-\Delta H_{f}(y) = e^{H_{f}(y)} \quad \text{in } \Omega.
	\]
	We claim that $\Omega= \mathbb R^2$. Suppose by contradiction $\Omega \neq \mathbb{R}^2$. Then there exists a boundary point $y^*\in \partial \Omega$ where $f(y_n) \to 0$ along a sequence $y_n$ of points $\Omega$ that converge to $y^*$, so that $H_{f}(y_n) \to -\infty$. Since $\nabla H_{f} = \nabla \Phi_{f}$ in $\Omega$ and $\nabla \Phi_{f}$ is a continuous function, we deduce that $\nabla H_{f}$ is continuous on $\bar \Omega$. Hence $H_f$ is also continuous on $\bar \Omega$, which contradicts $H_{f}(y_n) \to -\infty$. Hence $\Omega=\mathbb R^2$.
	
	Now, a direct computation shows
	$$
	\partial_{y_2}\left(\int_{\R^2}\frac{y_1-z_1}{|y-z|^{2}}f(z)dz \right)= \partial_{y_1}\left(\int_{\R^2}\frac{y_2-z_2}{|y-z|^{2}}f(z)dz \right).
	$$
	By Poincar\'e's Lemma, this implies that there exists a function $g$ defined on $\mathbb R^2$ such that $\nabla g =\nabla \Phi_f$. It then satisfies $-\Delta g=- \nabla \cdot \nabla \Phi_f=f$ on $\mathbb R^2$. Since $f>0$ on $\mathbb R^2$, \eqref{FundIdent} shows that it also satisfies $\nabla \log f-\nabla g=0$ on $\mathbb R^2$. By applying Theorem \ref{th:rigidity} we have
	\[
	f(y) = \frac{1}{\lambda^2} \frac{8}{\big(1 + \frac{|y - y_0|^2}{\lambda^2}\big)^2}
	\]
	for some $\lambda > 0$ and $y_0 \in \mathbb{R}^2$. The above function takes all values in $(0,8/\lambda^2)$ so that the constraints $0\leq f \leq 1$ and $\| f\|_{L^\infty}=1$ imply $\lambda=\sqrt{8}$. The only point at which it takes the value $1$ is $y_0$, so that $y_0=0$ since $f(0)=1$. Hence the result.
	
\end{proof}

\subsection{Sequential soliton resolution} \label{subsec:sequential}

We now begin our proof of Theorem \ref{thn:soliton-resolution}. As a starting point we recall the result established by \cite{BCM}, slightly restated here in view of the preliminary observations made in the previous sections.

\begin{theorem}[Theorem~1.3 in \cite{BCM}]\label{BCMThm}
	Let $u$ be a solution to the Keller--Segel system \eqref{KS} satisfying $u_0 \in L^1(\langle x \rangle^2 dx)$ and $\int_{\mathbb{R}^{2}} u(x,t) \, dx = 8\pi$, and recall it satisfies \eqref{id:property-1}-\eqref{id:property-2}. Then, for any sequence $t_n \to \infty$, the sequence $u(t_n)$ converges, in the weak-$\ast$ sense of measures, to a Dirac delta of mass $8\pi$ concentrated at the center of mass $x^*$.
\end{theorem}

The goal of this subsection is to establish the \emph{sequential soliton resolution}, which is stated precisely in the following proposition.

\begin{proposition}\label{SeqSolResPR}
	Let $u$ be a solution to the Keller--Segel system \eqref{KS} satisfying $u_0 \in L^1(\langle x \rangle^2 dx)$ and $\int_{\mathbb{R}^{2}} u(x,t) \, dx = 8\pi$, and recall it satisfies \eqref{id:property-1}-\eqref{id:property-2}. Then, there exists a sequence of times $t_n\to \infty$, scales $\lambda_n>0$ and positions $x_n\in \mathbb R^2$ such that
	$$
	u(t_n,x)= \frac{1}{\lambda_n^2}U\left(\frac{x-x_n}{\lambda_n}\right)+\tilde u_n (x)
	$$
	with
	$$
	\| \tilde u_n\|_{L^1} \to 0 \quad \mbox{as }n\to \infty.
	$$
\end{proposition}

\begin{proof}
	
	It is a direct consequence of combining Proposition \ref{prop:primary-sequential-soliton-resolution} with Lemma \ref{lem:cv-L1-sequential}.
\end{proof}

\subsubsection{Primary sequential soliton resolution}

In this subsection we prove a first sequential soliton resolution result. The remaining features of this resolution in order to show Proposition \ref{SeqSolResPR} will be obtained in the next subsubsection.

\begin{proposition} \label{prop:primary-sequential-soliton-resolution}
	Let $u_0$ satisfy the assumptions of Proposition \ref{SeqSolResPR}. Then there exists $t_{n}\to \infty$, $\lambda_{n}\to 0 $, $x_{n}\in \mathbb R^2$  such that
	\begin{align}\label{Decomposition2}
		u(t_{n},x)=\frac{1}{\lambda_{n}^{2}}U \left(\frac{x-x_n}{\lambda_n}\right) +\frac{1}{\lambda_{n}^{2}}\tilde{v}_{n}(\frac{x-x_{n}}{\lambda_{n}}),
	\end{align}
	and as $n\to \infty$ we know that $\tilde{v}_{n}$ converges to $0$ in $C^{\infty}_{\text{loc}}(\mathbb{R}^{2})$.
\end{proposition}

The proof of Proposition~\ref{prop:primary-sequential-soliton-resolution} is given at the end of this subsection. We first establish a profile decomposition result. We also recall, for clarity, that here we use the representation \eqref{id:definition-nablaPhiu} when speaking of solutions to the Keller-Segel system \eqref{KS}.

\begin{lemma}\label{FirstDecoLemma}
	Let $u_0$ satisfy the assumptions of Proposition \ref{SeqSolResPR}, then there exist $t_n\to \infty$, $\lambda_{n}\to 0 $, $x_{n}\in \mathbb R^{2}$ such that
	\begin{align}\label{Decomposition}
		u(t_{n},x)=\frac{1}{\lambda_{n}^{2}}f_{\infty}(\frac{x-x_{n}}{\lambda_{n}})+\frac{1}{\lambda_{n}^{2}}\tilde{v}_{n}(\frac{x-x_{n}}{\lambda_{n}}),
	\end{align}
	where $f_{\infty}\in C^{\infty}(\mathbb{R}^{2})$ satisfies
	\begin{equation} \label{eq:finfty}
		\nabla \log f_{\infty}-\nabla \Phi_{f_{\infty}}=0 \ \ \text{in }\{f_\infty>0\},
	\end{equation}
	where $\nabla \Phi_{f_\infty}$ is given by \eqref{id:definition-nablaPhiu}, $\int_{\R^2}f_{\infty}(y)dy\le 8\pi$ and that $ f_{\infty} \le 1$ on $\mathbb R^2$ and $f_\infty(0)=1$, and as $n\to \infty$ we know that $\tilde{v}_{n}(y)$ converges to $0$ in $C^{\infty}_{\text{loc}}(\mathbb{R}^{2})$.
\end{lemma}
\begin{proof}
	We split the proof in six steps. 
	
    \smallskip 
	\noindent\textbf{Step 1.} \emph{Divergence in $L^\infty$.} We claim that
	\[
	\limsup_{t\to+\infty}\|u(t)\|_{L^\infty(\mathbb R^2)}=+\infty.
	\]
	Indeed, this follows immediately from the weak convergence towards a Dirac mass established in Theorem~\ref{BCMThm}.

    \smallskip 
    \noindent\textbf{Step 2.} \emph{Definition of the sequence of renormalizations.} Define the sequence of times
	\[
	t_n := \inf \{ t > 0 \mid \|u(t)\|_{L^\infty} \ge n \}.
	\]
	By Theorem \ref{th:lwp:L1} and Step 1, $t_n$ is well-defined and $t_n\to \infty$. By definition, we have $\|u(t)\|_{L^\infty} < n$ for $t < t_n$. We claim that for all large enough $n$,
	\begin{equation} \label{id:sequential-soliton-claim}
	\|u(t_n)\|_{L^\infty} = n  \mbox{ and there exists }x_n\in \mathbb R^2 \mbox{ such that } u(t_n)=n.
	\end{equation}
	To prove the claim, we first prove that
	\begin{equation} \label{id:sequential-soliton-claim-localization}
	u(x,t)\to 0 \quad \mbox{as }|x|\to \infty, \quad \mbox{uniformly for }t\geq 1.
	\end{equation}
	To show this, fix $\delta>0$, and let $\sigma$ and $\varepsilon$ be as in Theorem~\ref{thm:L1Linfsmall} and define $R=\sigma^{-1/2}$. Pick a time $t\geq 1$ and let
	\[
	t_0:=t-\sigma R^2\ge 0.
	\]
	For $x_0\in \mathbb R^2$ with $|x_0|>1$, by the boundedness of the second moment \eqref{id:property-2}, we have
	\[
	\int_{B_{8R}(x_0)} u(x,t_0)\,dx
	\le
	\frac{1}{|x_0|^2}\int_{\mathbb{R}^2} u(x,t_0)\,|x|^2\,dx
	\lesssim
	\frac{1}{|x_0|^2} \leq \varepsilon
	\]
	provided $|x_0|$ is sufficiently large depending on $\varepsilon$ but independently of $t$. Applying Theorem~\ref{thm:L1Linfsmall}, we obtain
	\[
	|u(x_0,t)|\lesssim \delta.
	\]
	Since $\delta>0$ is arbitrary, the uniform convergence \eqref{id:sequential-soliton-claim-localization} follows.
	
	By Theorem \ref{th:lwp:L1} we have $u\in C([t_n-1,t_n+1],W^{2,\infty} \cap W^{2,1} (\mathbb R^2))$. This implies as $\partial_t u$ is given by \eqref{KS} that $\partial_t u\in C([t_n-1,t_n+1],L^\infty (\mathbb R^2))$. This fact and \eqref{id:sequential-soliton-claim-localization} imply the claim \eqref{id:sequential-soliton-claim}.

	We then define the rescaled profile
	\[
	f_n(y) := \lambda_n^2 \, u(t_n, x_n + \lambda_n y), \qquad \lambda_n := 1/\sqrt{n}.
	\]
	Let $\delta > 0$ small to be fixed in the next step and define
	\[
	v_{n,0}(y) = \lambda_n^2 \, u(t_n - \delta \lambda_n^2, x_n + \lambda_n y).
	\]
	Let $v_n$ denote the solution to the Keller-Segel equation \eqref{KS} with initial data $v_{n,0}$. Then we have
	\[
	u(t_n+(s-\delta)  \lambda_n^2, x) = \frac{1}{\lambda_n^2} \, v_n\Big(s, \frac{x - x_n}{\lambda_n}\Big) \quad \mbox{and} \quad f_n(y) = v_n(\delta, y).
	\]
	\noindent \textbf{Step 3.} \emph{Extraction of a limit and first properties.} We claim that there exist two times $0<T_0<\delta<T_1$ such that, up to extracting a subsequence, $v_n$ converges in $C^{\infty}_{\mathrm{loc}}([T_0, T_1] \times \mathbb{R}^2)$ to a function $f$, and $f_n$ converges in $C^{\infty}_{\mathrm{loc}}(\mathbb R^2)$ to a function $f_{\infty}$, which are both nonnegative, satisfy $f(y,\delta)=f_\infty(y)$ and
	\begin{align} 
	\label{fintyprops-1}
		& \int_{\mathbb R^2} f(s,y)dy\leq 8\pi \mbox{ for all }s\in [T_0,T_1], \\
		\label{fintyprops-2}
		&f_{\infty}(0) = 1, \quad 0\leq f_\infty(y)\leq 1 \mbox{ for all }y\in \mathbb R^2, \quad \mbox{and} \quad \int_{\mathbb R^2} f_\infty(y)dy\leq 8\pi.
	\end{align}
	We now prove this claim. Observe that $\|v_{n,0}\|_{L^1} = 8\pi$ and $\|v_{n,0}\|_{L^{\infty}} < 1$. By Lemma~\ref{Solitoncorollary}, if $\delta$ is sufficiently small (smaller than the time up to which we have uniform control of Sobolev norms), taking $0<T_0 < \delta < T_1$, we obtain
	\begin{equation} \label{intermediate-bd-vn}
		\sup_{T_0\leq t \leq T_1}\| \partial^\alpha_s \nabla_y^{\alpha'} v_n\|_{L^{\infty}(\mathbb R^2)\cap L^1(\mathbb R^2)} \le C[N, T_0] \quad \text{for all } n
	\end{equation}
	for all $\alpha+\alpha'\leq N$.
	This implies that $v_n$ is uniformly bounded in $W^{N,\infty}([T_0, T_1] \times \mathbb{R}^2)$ for any $N \in \mathbb{N}$.  
	Applying the Arzel\`a--Ascoli theorem, we extract a subsequence converging in $C^{\infty}_{\mathrm{loc}}([T_0, T_1] \times \mathbb{R}^2)$:
	\[
	v_n(y, s) \xrightarrow{C^{\infty}_{\mathrm{loc}}(\mathbb{R}^2 \times (T_0, T_1))} f(y, s).
	\]
	In particular, we have 
	\[
	f_n(y)=v_n(\delta, y) \to f_{\infty}(y) \quad \text{in } C^{\infty}_{\mathrm{loc}}(\mathbb{R}^2).
	\]
	Moreover, we have $v_n(\delta,0)= 1$ and $0 \le v_n(\delta, y) \le 1$ for all $y\in \mathbb R^2$ by Step 2, and $\int v_n dy=8\pi$. The first two properties in \eqref{fintyprops-2} follow from the convergence in $C^\infty_{loc}$, while \eqref{fintyprops-1} and the third property in \eqref{fintyprops-2} follow from Fatou's Lemma.
	
	\smallskip
	\noindent \textbf{Step 4.} \emph{Vanishing free energy dissipation.}
	We claim that    
	\begin{align}
		\int_{T_0}^{T_1} \int_{\mathbb{R}^2} v_n(s,y) \, \big|\nabla \log v_n(s,y) - \nabla \Phi_{v_n}(s,y)\big|^2 \, dy \, ds \to 0 \quad \text{as } n \to \infty.
	\end{align}
	
	To prove it, first notice that by the free energy dissipation identity \eqref{freeenergy-dissipation} we know $\mathcal F[u(t)]$ decreases over time, and is bounded from below by the lower bound on the free energy. Thus it converges to some limit $\lim_{t\to \infty}\mathcal F[u(t)]$. By \eqref{freeenergy-dissipation} we obtain
	$$
	\int_1^{\infty}  \int_{\mathbb{R}^2} u(t) \, \big|\nabla \log u(t) - \nabla \Phi_u(t)\big|^2 \, dx \, dt = \mathcal F[u(1)]- \lim_{t\to \infty}\mathcal F[u(t)].
	$$
	We introduce
	$$
	I_n = \int_{t_n + \lambda_n^2 (T_0 - \delta)}^{t_n + \lambda_n^2 (T_1 - \delta)} \int_{\mathbb{R}^2} 
	u(x,t) \, \big|\nabla \log u(x,t) - \nabla \Phi_u(x,t)\big|^2 \, dx \, dt .
	$$
	Note that $I_n\to 0$ as $n\to\infty$, since the time interval of integration shrinks to zero and the integrand is integrable in time.
	 After performing a change of variables in the identity \eqref{id:definition-nablaPhiu} we obtain
	$$
	\nabla \Phi_u (t_n+t,x)= \frac{1}{\lambda_n}\nabla \Phi_{v_n(\frac{t+\delta \lambda_n^2}{\lambda_n^2})} (\frac{x-x_n}{\lambda_n})
	$$
	Therefore, we can rewrite $I_n$ as
	\begin{align*}
		I_n &= \int_{\lambda_n^2 (T_0 - \delta)}^{\lambda_n^2 (T_1 - \delta)} \frac{1}{\lambda_n^4} \int_{\mathbb{R}^2} 
		v_n\Big(\frac{x-x_n}{\lambda_n}, \frac{t + \delta \lambda_n^2}{\lambda_n^2}\Big) 
		\Big|\nabla_y \log v_n\Big(\frac{x-x_n}{\lambda_n}, \frac{t + \delta \lambda_n^2}{\lambda_n^2}\Big) - \nabla_y \Phi_{v_n}\Big(\frac{x-x_n}{\lambda_n}, \frac{t + \delta \lambda_n^2}{\lambda_n^2}\Big)\Big|^2 dx \, dt \\
		&= \int_{T_0}^{T_1} \int_{\mathbb{R}^2} v_n(s,y) \, \big|\nabla \log v_n(s,y) - \nabla \Phi_{v_n}(s,y)\big|^2 \, dy \, ds.
	\end{align*}
	Combining the above identity with the convergence $I_n\to 0$ then gives the desired claim.
	
	\smallskip
	\noindent \textbf{Step 5.} \emph{Zero free energy dissipation rate for the limit.} In this step, we show that $f$ satisfies
	\begin{equation} \label{final equation}
		\nabla \log f(y,s) - \nabla \Phi_{f}(y,s) = 0 \quad \text{in } \{ (y,s)\in \mathbb R^2 \times (T_0,T_1), \ f(y,s)>0\}, 
	\end{equation}
	where $\nabla \Phi_f$ is given by \eqref{id:definition-nablaPhiu}. Before proceeding, let us make some remarks on $\nabla \Phi_f$. We observe that
	\[
	|\nabla \Phi_{v_n}(y,s) - \nabla \Phi_f(y,s)| 
	\le \underbrace{\int_{B_R(y)} \frac{|v_n(z,s) - f(z,s)|}{|z-y|} \, dz}_{\text{arbitrarily small, since } v_n \to f \text{ in } C^\infty_{\mathrm{loc}}} 
	+ \underbrace{\int_{B_R^c(y)} \frac{|v_n(z,s) - f(z,s)|}{|z-y|} \, dz}_{\text{arbitrarily small for large } R \text{ using } L^1\text{ bound}}.
	\]
	Hence, $\nabla \Phi_{v_n}(y,s) \to \nabla \Phi_f(y,s)$ for all $(y,s) \in \mathbb{R}^2 \times [T_0, T_1]$. In addition, by \eqref{intermediate-bd-vn} and standard estimates, we have that
	$$
	\sup_{T_0\leq t \leq T_1} \| \partial^\alpha \nabla^{\alpha'}\nabla \Phi_{v_n}\|_{W^{N,\infty}(\mathbb R^2)}\lesssim C(N,T_0) \quad \mbox{for all }n
	$$
	for all $\alpha+\alpha'\leq N$. Thus, $\nabla \Phi_{v_n}$ is uniformly bounded in $W^{N,\infty}([T_0,T_1]\times \mathbb R^2)$. By the Arzel\`a-Ascoli Theorem, we obtain that $\lim_{n\to \infty}\nabla \Phi_{v_n}=\nabla \Phi_f \in C^\infty([T_0,T_1]\times \mathbb R^2)$.
	
	Now, assume that $f(\bar{y},\bar{s}) > 0$ for some $(\bar{y},\bar{s}) \in \mathbb{R}^2 \times (T_0, T_1)$. By the smoothness of $f$, there exists $\varepsilon > 0$ such that $f > 0$ in the neighborhood $B_\varepsilon(\bar{y},\bar{s})$. Since $v_n \ge 0$, we have
	\[
	\int_{B_\varepsilon(\bar{y},\bar{s})} v_n(s,y) \, \big|\nabla \log v_n(s,y) - \nabla \Phi_{v_n}(s,y)\big|^2 \, dy \, ds
	\le \int_{T_0}^{T_1} \int_{\mathbb{R}^2} v_n(s,y) \, \big|\nabla \log v_n(s,y) - \nabla \Phi_{v_n}(s,y)\big|^2 \, dy \, ds,
	\]
	which tends to $0$ as $n \to \infty$. Passing to the limit, we obtain $f(s,y) \, \big|\nabla \log f(s,y) - \nabla \Phi_{f}(s,y)\big|^2=0$ on $B_\varepsilon(\bar{y},\bar{s})$. Since $f(y,s) > 0$ in $B_\varepsilon(\bar{y},\bar{s})$ this implies $\nabla \log f(\bar{y},\bar{s}) = \nabla \Phi_f(\bar{y},\bar{s})$ on this set. As $(\bar y,\bar s)$ is arbitrary, this concludes the proof of \eqref{final equation}.
	
	We can now end the proof of the Proposition. By Step 3 we have $f_n(y)=f_\infty+\tilde v_n(y)$ where $\tilde v_n\to 0$ in $C^\infty_{loc}(\mathbb R^2)$. By Step 2 this gives $u(t_{n},x)=\lambda_{n}^{-2}f_{\infty}((x-x_{n}0/\lambda_{n})+\lambda_{n}^{-2}\tilde{v}_{n}((x-x_{n})/\lambda_{n})$. The function $f_\infty$ satisfies all the desired properties of the Proposition by \eqref{fintyprops-2}, and by \eqref{eq:finfty} since $f(\delta,y)=f_\infty(y)$.
	
	\end{proof}

We can now give the proof of Proposition \ref{prop:primary-sequential-soliton-resolution} as a consequence of the profile decomposition result of Lemma \ref{FirstDecoLemma} and of the rigidity result of Corollary \ref{cor:rigidity}.

\begin{proof}[Proof of Proposition \ref{prop:primary-sequential-soliton-resolution}]
	
	We start by applying Lemma \ref{FirstDecoLemma} and keep the same notation as in its statement. As $f_{\infty}\in C^{\infty}$ and satisfies \eqref{eq:finfty} with $\int_{\R^2}f_{\infty}(y)dy\le 8\pi$ and $ f_{\infty} \le 1$ on $\mathbb R^2$ and $f_\infty(0)=1$, applying Corollary \ref{cor:rigidity} shows $f_\infty (y)=(1/8)U\left(y/\sqrt{8}\right)$. The desired result then follows up to multiplying $\lambda_n$ by $\sqrt{8}$.
	
\end{proof}

\subsubsection{Additional feature of the sequential soliton resolution}
In this subsection we recover additional information on the decomposition \eqref{Decomposition2}.
\begin{lemma} \label{lem:cv-L1-sequential}
	Let $u$ be a solution as in Proposition \ref{prop:primary-sequential-soliton-resolution} and consider the decomposition \eqref{Decomposition2}. We have 
	\begin{align*}
		\|\tilde{v}_{n}\|_{L^{1}(\mathbb{R}^{2})}\to 0 \ \ \text{as }n\to \infty.
	\end{align*}
\end{lemma}
\begin{proof}
Changing variables $x=x_{n}+\lambda_{n}y$, we introduce
	\begin{align*}
		v_{n}(y):=\lambda_{n}^{2}u(x_{n}+\lambda_{n}y,t_{n})=U(y)+\tilde{v}_{n}(y).
	\end{align*}
	Let $\epsilon>0$. Pick $R>1$ large enough so that $\int_{B_R^c} Udy<\epsilon/4$, so in particular $|\int_{B_R} Udy-8\pi|<\epsilon/4$. Since $\tilde{u}_{n}\to 0$ in $C^{\infty}_{\text{loc}}$, one has
	$$
	\int_{B_R}| \tilde u_n| dy < \epsilon/4
	$$
	for all large enough $n$. Then $|\int_{B_R} v_n dy - 8\pi|<\epsilon/2$. Since $\int_{\mathbb R^2} v_n dy =8\pi$, this implies $\int_{B_R^c} v_n dy<\epsilon/2$. Hence
	$$
	\int_{B_R^c}| \tilde u_n|dy \leq \int_{B_R^c} U dy+\int_{B_R^c}v_n dy <\epsilon/4+\epsilon/2=3\epsilon/4.
	$$
	Therefore, $\int_{\mathbb R^2}| \tilde u_n|dy<\epsilon$ which, since $\epsilon$ is arbitrary, implies the result.
	\end{proof}

\subsection{Continuous in time soliton resolution} \label{subsec:continuous-sol-res}

In this section we prove Theorem \ref{thn:soliton-resolution}. It is a direct consequence of Proposition \ref{continuoussolresPR} and Lemma \ref{limpara} we show below.

\subsubsection{Primary continuous in time soliton resolution}

In this subsection, we aim to prove the continuous in time soliton resolution in $L^{1}(\mathbb{R}^{2})$ with uniform bounds in $W^{k,\infty}\cap W^{k,1}(\R^{2})$ for any $k\in \mathbb N$.

\begin{proposition}\label{continuoussolresPR}
	Let $u$ be a solution to the Keller--Segel system \eqref{KS} satisfying $u_0 \in L^1(\langle x \rangle^2 dx)$ and $\int_{\mathbb{R}^{2}} u(x,t) \, dx = 8\pi$, and recall it satisfies \eqref{id:property-1}-\eqref{id:property-2}. Let $k\in \mathbb N$. For any constant $K>0$, there exist functions $\lambda(t)>0$, $x(t)\in \mathbb{R}^{2}$ such that 
	\begin{align*}
		u(x,t)=\frac{1}{\lambda^{2}(t)}U\big( \frac{x-x(t)}{\lambda(t)} \big)+ \frac{1}{\lambda^{2}(t)}\tilde{v}\big( t, \frac{x-x(t)}{\lambda(t)}\big)
	\end{align*}
	with 
	\begin{align*}
		\lim_{t\to \infty}\|\tilde{v}(t)\|_{L^{1}}=0, \ \ \ \limsup_{t\to \infty}\|\tilde{v}(t)\|_{W^{k,\infty}\cap W^{k,1}}< K.
	\end{align*}
\end{proposition}

To prove Proposition \ref{continuoussolresPR}, we will need some results that provide further information on the sequential soliton resolution we previously investigated. The first lemma, through parabolic regularization, allows us to gain better control of the remainder.
\begin{lemma}\label{ImprovedSequentialLemma}
	Let be a solution as in Proposition \ref{continuoussolresPR}. There exist $t_{n}\to \infty$, $\lambda_{n}>0$, $x_{n}\in \mathbb{R}^{2}$ such that for any $t_n\leq t \leq t_n+\lambda_n^2$,
	\begin{align*}
		u(t_{n}+\lambda_n^2 s,x)=\frac{1}{\lambda_{n}^{2}}
		U\big(\frac{x-x_{n}}{\lambda_{n}}\big)+\frac{1}{\lambda_{n}^{2}}\tilde{v}\big(t,\frac{x-x_{n}}{\lambda_{n}} \big),
	\end{align*}
	where the function $\tilde u$ satisfies for all $k\in \mathbb N$ that
	\begin{align*}
		\lim_{n\to \infty} \sup_{t_n\leq t \leq t_{n}+\lambda_n^2}\|\tilde{v}(t,\cdot)\|_{W^{k,\infty}\cap W^{k,1}}=0.
	\end{align*}
\end{lemma}
\begin{proof}
	
	Starting with the sequential soliton resolution of Proposition \ref{SeqSolResPR}
	\begin{align*}
		u(t_{n},x)=\frac{1}{\lambda_{n}^{2}}U\big(\frac{x-x_{n}}{\lambda_{n}}\big)+o_{L^1}(1) \quad \mbox{as }n\to \infty,
	\end{align*}
	we let $v_{n,0} (y)= \lambda_n^2 u(x_n+\lambda_n y)$ and $v$ be the corresponding solution of \eqref{KS}, so that
	$$
	u(t_n+\lambda_n^2 s,x)= \frac{1}{\lambda_n^2}v_n\left(s,\frac{x-x_n}{\lambda_n} \right).
	$$
	Then we have $v_{n,0}=U+\tilde v_{n,0}$ with $\|\tilde v_{n,0}\|_{L^1}\to 0$ as $n\to \infty$. Applying the uniform regularization estimates near the ground state of Lemma \ref{lem:uniform-regularization-near-soliton} we obtain that the associated solutions $v_n$ can be decomposed for $s\in [1,2]$ as
	$$
	v_n(s,\cdot)= U(\cdot) + \tilde v_n (s,\cdot), \qquad \mbox{with} \quad \sup_{1\leq s \leq 2} \| \tilde v_n (s,\cdot)\|_{W^{k,1}\cap W^{k,\infty}} \to 0
	$$
	as $n\to \infty$ for any $k\in \mathbb N$. We get the desired result by redefining the sequence $t_n$ to be $t_n+\lambda_n^2$.
	
\end{proof}

We now fix $k\in \mathbb N$ arbitrarily large. Define for $K,\delta>0$ the adapted distance to the set of solitons
\begin{align*}
	d_{\delta,K}(u)
	= \inf_{\lambda>0,\,x^*\in\R^{2}} 
	\left( \delta^{-1} \left\| 
	y\mapsto \lambda^2 u(x^*+\lambda y) -U(y)
	\right\|_{L^1} + K^{-1} \left\| 
	y\mapsto \lambda^2 u(x^*+\lambda y) -U(y)
	\right\|_{W^{k,1}\cap W^{k,\infty}} \right)
\end{align*}
and consider the tube
$$
\mathcal T_{\delta,K}= \{ u\in W^{k,1}(\mathbb R^2)\cap W^{k,\infty}(\mathbb R^2), \ \ | \ d_{\delta,K}(u)<1\}.
$$

We now denote
$$
t_{n*} \mbox{ the sequence provided by Lemma } \ref{ImprovedSequentialLemma}.
$$
For $K,\eta>0$, we define
$$
t_n^*=t_n^*[\eta,K] = \sup \{ t^*\geq t_{n*},\ \ | u(t)\in \mathcal T_{\eta,K} \ \mbox{for all } t\in [t_{n*},t^*]\}
$$
Note that by Lemma \ref{ImprovedSequentialLemma}, we have $t_n^*>t_{n*}$ for all $n$ large enough, and possibly $t_n^*=\infty$. The next Lemma shows that if the result of Proposition \ref{continuoussolresPR} does not hold true, then the solution enters and exits the tube $\mathcal T_{\eta,K}$ infinitely many times.

\begin{lemma} \label{lem:set-up-contradiction-continuous}
	
	Assume the result of Proposition \ref{continuoussolresPR} does not hold true. Then there exists $K>0$, such that for all $\eta>0$ small enough, one has
	$$
	t_{n}^*<\infty \mbox{ for all large enough }n\in \mathbb N.
	$$
	
\end{lemma}

\begin{proof}
	
	Assume by contradiction the result of the Lemma \ref{lem:set-up-contradiction-continuous} does not hold true. Then, for all $K>0$, there exists a sequence $\eta_m\to 0 $ such that
	$$
	t_{n_{m,l}}^*[K,\eta_m]=\infty \mbox{ for a sequence }(n_{m,l})_{l\in \mathbb N} \mbox{ with } n_{m,l}\to \infty .
	$$
	In particular, setting $n_m=n_{m,1}$, we have
	$$
	t_{n_m}^*[K,\eta_m]=\infty \qquad \mbox{for all }m\in \mathbb N.
	$$
	By definition of $\mathcal T_{\eta_m,K}$, this means that for all $t\geq t_{n_m*}$, the solution $u$ can be decomposed as
	$$
	u(t,x)=\frac{1}{\lambda^2(t)}U\left(\frac{x-x^*(t)}{\lambda(t)}\right)+\frac{1}{\lambda^2(t)} \tilde u\left(t,\frac{x-x^*(t)}{\lambda(t)}\right)
	$$
	for some $\lambda(t),x^*(t)$ with $\tilde u$ satisfying
	$$
	\| \tilde u (t) \|_{L^1}\leq \eta_m \quad \mbox{and} \quad \| \tilde u (t) \|_{W^{k,1}\cap W^{k,\infty}}\leq K.
	$$
	But since $\eta_m\to 0$ as $m\to \infty$, this implies that the statement of Proposition \ref{continuoussolresPR} does hold true, which we assumed was not the case. Hence the desired contradiction.
	
\end{proof}

We shall now use the following lemma, which provides a \emph{geometric decomposition} for any function that is sufficiently close in $L^{1}$ to a steady state. This lemma yields a \emph{canonical choice} for the parameters $\lambda(t)$ and $x^{\star}(t)$ on the time interval $[t_{n*},t_n^*]$.

\begin{lemma}\label{GeometricalDeco}
	For all $K>0$ the following hold for any sufficiently small $\delta>0$. For each $u\in \mathcal T_{\delta,K}$ there exists a unique decomposition
	\begin{align*}
		u 
		= \frac{1}{\lambda[u]^{2}} 
		U\!\left( \frac{x - x^{\star}[u]}{\lambda[u]} \right)
		+ \frac{1}{\lambda[u]^{2}} 
		\tilde{v}\!\left( \frac{x - x^{\star}[u]}{\lambda[u]} \right),
	\end{align*}
	where the perturbation $\tilde{u}$ satisfies the orthogonality conditions
	\begin{align}\label{orthogonalitycond}
		\int_{\R^{2}} \tilde{v}(y)\,\Lambda U(y)\,dy 
		= \int_{\R^{2}} \tilde{v}(y)\,\partial_{y_{1}} U(y)\,dy
		= \int_{\R^{2}} \tilde{v}(y)\,\partial_{y_{2}} U(y)\,dy
		= 0.
	\end{align}
	The parameters $\lambda[u]$ and $x^{\star}[u]$ depend in a Fréchet differentiable way on $u$ in $\mathcal{T}_{\delta}$ and moreover
	\[
	\|\tilde{v}\|_{L^{1}} \lesssim \delta \quad \mbox{and} \quad \|\tilde{v} \|_{W^{k,1}\cap W^{k,\infty}} \lesssim K.
	\]
\end{lemma}

\begin{proof}
	The proof follows from a standard application of the implicit function theorem. We mention for example \cite{CZ}, \cite{CRS} as references for the proofs of analogue results.
	\end{proof}

We are now ready to present the proof of the continuous-in-time soliton resolution stated in Proposition \ref{continuoussolresPR}.

\begin{proof}[Proof of Proposition \ref{continuoussolresPR}]
		
	\smallskip
	\noindent 
	\textbf{Step 1.} \emph{Setting up the argument.} We proceed by contradiction.  
	Assume that Proposition~\ref{continuoussolresPR} does not hold.  
	Then by Lemma \ref{lem:set-up-contradiction-continuous}, we know that we can pick any $K>0$ and then any $\eta>0$ small enough, such that (up to extracting subsequences) there exist two sequences $t_{n_{\star}}$ and $t_{n}^*$ with $t_{n*}<t_n^*<t_{n+1}^*$ and on the interval $[t_{n*},t_n^*]$ functions $\lambda(t),x(t)$ such that the solution $u$ can be decomposed as
	\begin{align} \label{id:proof-continuous-decomposition-1}
		u(t,x) 
		= \frac{1}{\lambda(t)^{2}} 
		U\left( \frac{x - x(t)}{\lambda(t)} \right)
		+ \frac{1}{\lambda(t)^{2}} 
		\tilde{u}\!\left( t, \frac{x - x(t)}{\lambda(t)} \right),
	\end{align}
	with $\| \tilde{u}(t_{n_{\star}}) \|_{L^{1}} \to 0$ and $\| \tilde{u}(t_{n_{\star}}) \|_{W^{k,\infty}\cap W^{k,1}} \to 0$, and
	\begin{equation} \label{id:proof-continuous-decomposition-2}
		\| \tilde{u}(t) \|_{L^{1}} < \eta, 
		\quad \| \tilde{u}(t) \|_{W^{k,1}\cap W^{k,\infty}} < K,
		\qquad \text{for all } t \in [t_{n_{\star}}, t_{n}^{\star}),
	\end{equation}
	and at time $t_{n}^{\star}$
	\begin{equation} \label{id:proof-continuous-decomposition-3}
		d_{\eta,K}(u(t^*_n))= 1.
	\end{equation}
	Note that to justify \eqref{id:proof-continuous-decomposition-3} above, on the one hand the inequality $d_{\eta,K}(u(t^*_n))\geq 1$ follows from the definition of $t^*_n$ and Theorem \ref{th:lwp:L1}. Indeed, if $d_{\eta,K}(u(t^*_n))< 1$ then we would have $d_{\eta,K}(u(t))< 1$ on an interval of the form $[t_n^*,t_n^*+\delta]$ by the regularity of solutions provided by Theorem \ref{th:lwp:L1}, which would contradict the definition of $t_n^*$. On the other hand, the inequality $d_{\eta,K}(u(t^*_n))\leq 1$ follows from the fact that $d_{\eta,K}(u(t))< 1$ for all $t\in [t_{n*},t_n^*)$, and again the regularity of the solution provided by Theorem \ref{th:lwp:L1}. Hence $d_{\eta,K}(u(t^*_n))= 1$.

	By Lemma~\ref{GeometricalDeco}, in the interval $[t_{n_{\star}}, t_{n}^{\star}]$ we have a canonical choice of parameters $\lambda(t)$ and $x(t)$.  
	Indeed, by \eqref{id:proof-continuous-decomposition-1} and \eqref{id:proof-continuous-decomposition-2} we can apply Lemma~\ref{GeometricalDeco} and obtain that, up to relabelling $\lambda(t)$, $x(t)$, $\tilde v(t)$, we have
	\begin{align*}
		u(t,x) 
		= \frac{1}{\lambda(t)^{2}} 
		U\!\left( \frac{x - x(t)}{\lambda(t)} \right)
		+ \frac{1}{\lambda(t)^{2}} 
		\tilde{v}\!\left( t, \frac{x - x(t)}{\lambda(t)} \right), \qquad \| \tilde v(t)\|_{L^1}\lesssim \eta, \qquad \| \tilde v(t)\|_{W^{k,1}\cap W^{k,\infty}}\lesssim K
	\end{align*}
	where now $\lambda(t)$ and $x(t)$ are continuous on $[t_{n*},t_n^*]$ (the parameters are Fr\'echet differentiable with respect to $u$, which is continuous in time). It is convenient to introduce an additional sequence of times $t_{0_{n}}\in (t_{n*},t_n^*)$ satisfying 
	\[
	t_{0_{n}} + \lambda^{2}(t_{0_{n}}) = t_{n}^{\star}.
	\]
	This is always possible because $\lambda(t)$ is continuous and strictly positive on $[t_{n*},t_n^*]$, and that $t_n^*>t_{n*}+\lambda(t_{n*})^2$ by Lemma \ref{ImprovedSequentialLemma}. The situation described above can be illustrated by the following diagram.
	
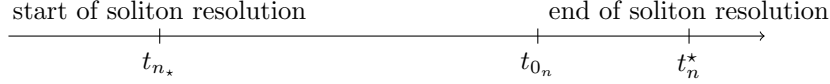
\begin{center}
		\begin{tikzpicture}
			\draw[->] (0,0) -- (10,0);
			
			\draw (2,0.1) -- (2,-0.1) node[below] {$t_{n_{\star}}$};
			\node[above] at (2,0.1) {start of soliton resolution};
			
			\draw (7,0.1) -- (7,-0.1) node[below] {$t_{0_{n}}$};
			
			\draw (9,0.1) -- (9,-0.1) node[below] {$t_{n}^{\star}$};
			\node[above] at (9,0.1) {end of soliton resolution};
		\end{tikzpicture}
	\end{center}
	
	\noindent \textbf{Step 2.} \emph{Improving the $W^{k,1}\cap W^{k,\infty}$ bound at time $t_n^*$.} Let $\epsilon>0$ that will be chosen small enough depending on $K$ but independent of $\eta$. We know that $\|\tilde{v}(t_{0_{n}})\|_{L^{1}}\lesssim\eta$ for some small $\eta$. 

    We change variables $y = (x - x(t_{0_{n}}))/\lambda(t_{0_{n}})$ and consider the solution $v_n(s,y)$ to the Keller--Segel equation with initial data 
	$v_{n,0}(y)=\lambda^{2}(t_{0_{n}})u(x(t_{0_n})+\lambda (t_{0_{n}} )y,t_{0_{n}})$. By the regularization estimate near the ground state from Lemma \ref{lem:uniform-regularization-near-soliton}, if $\eta$ is chosen small enough depending on $\epsilon$, there holds the decomposition
	\begin{equation} \label{continuous-resolution-id-1}
		v_n(s,y)= U +\tilde v_n(s,y), \quad \mbox{with }\sup_{1\leq s \leq 2}\|\tilde{v}_n(\cdot,t)\|_{W^{k,1}\cap W^{k,\infty}} \leq \epsilon.
	\end{equation}
	This implies in particular the following decomposition at time $t_n^*$,
	\begin{equation} \label{continuous-resolution-id-3}
		u(t_n^*,x)= \frac{1}{\lambda(t_{0_n})^2}U \left(\frac{x-x(t_{0_n})}{\lambda(t_{0_n})}\right)+\frac{1}{\lambda(t_{0_n})^2}\tilde v_n \left(1,\frac{x-x(t_{0_n})}{\lambda(t_{0_n})}\right).
	\end{equation}
	
	\smallskip
	
	\noindent \textbf{Step 3.} \emph{Improving the $L^1$ bound at time $t_n^*$.} By the bounds on $\tilde v_n$ from Step 2 and standard compactness arguments, there exists a subsequence such that
	\begin{align*}
		v_{n}(y,s) \to f(y,s)
		\quad \text{in } C^{k-1}_{\mathrm{loc}}([1,2]\times \mathbb R^2).
	\end{align*}
	Introducing $f_\infty(y)= f(1,y)$, one can repeat exactly the same arguments as in the proof of Lemma \ref{FirstDecoLemma} to show that $f_{\infty}\in C^{\infty}(\mathbb{R}^{2})$ is a nonnegative function that satisfies
	\begin{align*} 
			\nabla \log f_{\infty}-\nabla \Phi_{f_{\infty}}=0 \ \ \text{in }\{f_\infty>0\}, \ \ \ 
	\end{align*} 
	with $\int_{\R^2}f_{\infty}(y)dy\le 8\pi$. Moreover, since $f_\infty$ is the limit of $v_{n}(1,y)=U+O_{W^{k,1}\cap W^{k,\infty}}(\epsilon)$ from Step 2, we get $f_\infty\neq 0$. Appealing to the rigidity result of Corollary \ref{cor:rigidity}, we deduce $f_\infty (y)=\mu^{-2}U\left( (y-z)/\mu\right)$ for some $\mu>0$ and $z\in \mathbb R^2$. This shows
	$$
	v_n(1,y) = \frac{1}{\mu^2}U\left(\frac{y-z}{\mu}\right)+\tilde w_n(y) , \quad \mbox{with} \ \tilde w_n\to 0 \ \mbox{in }C^{k-1}_{loc}.
	$$
	Reasoning exactly as in the proof of Lemma \ref{lem:cv-L1-sequential}, this implies $\tilde w_n\to 0$ in $L^1$. Hence
	\begin{equation} \label{continuous-resolution-id-2}
		v_n(1,y) = \frac{1}{\mu^2}U\left(\frac{y-z}{\mu}\right)+\tilde w_n(y) , \quad \mbox{with} \ \tilde w_n\to 0 \ \mbox{in }C^{k-1}_{loc}\cap L^1.
	\end{equation}
	Combining \eqref{continuous-resolution-id-1} and \eqref{continuous-resolution-id-2} gives
	$$
	\frac{1}{\mu^2}U\left(\frac{\cdot-z}{\mu}\right)-U= \lim_{n\to \infty} \tilde v_n(1,\cdot) \quad \mbox{with} \|\tilde v_{n}(1,\cdot)\|_{W^{k,1}\cap W^{k,\infty}}\leq \epsilon,
	$$
	implying $|\mu-1|+|z|\lesssim \epsilon$. Thus, by \eqref{continuous-resolution-id-1} and \eqref{continuous-resolution-id-2} again we deduce
	$$
	\| \tilde w_n \|_{W^{k,1}\cap W^{k,\infty}} = \left\| U- \frac{1}{\mu^2}U\left(\frac{\cdot-z}{\mu}\right)+\tilde v_n(1,\cdot) \right\|_{W^{k,1}\cap W^{k,\infty}} \lesssim \epsilon.
	$$
	By the above inequality and \eqref{continuous-resolution-id-2} we deduce that for all $n$ large enough,
	$$
	\eta^{-1}\| \tilde w_n \|_{L^1}+K^{-1}\| \tilde w_n \|_{W^{k,1}\cap W^{k,\infty}}\lesssim \epsilon K^{-1}.
	$$
	Injecting \eqref{continuous-resolution-id-2} in \eqref{continuous-resolution-id-3} shows
	$$
	u(t_n^*,x)= \frac{1}{(\mu \lambda(t_{0_n}))^2}U\left(\frac{x-(x(t_{0_n})+\lambda(t_{0_n})z)}{(\mu \lambda(t_{0_n}))^2}\right)+\frac{1}{ \lambda(t_{0_n})^2} \tilde w_n \left(\frac{x-x(t_{0_n})}{ \lambda(t_{0_n})^2} \right)
	$$
	but then
	$$
	d_{\eta,K}(u(t_n^*,\cdot )\lesssim  \epsilon K^{-1} <1
	$$
	for all large enough $n$, provided $\epsilon$ has been chosen small enough. This contradicts \eqref{id:proof-continuous-decomposition-3} and ends the proof of Proposition \ref{continuoussolresPR}.

\end{proof}

\subsubsection{Additional features of the soliton resolution}
In this section, we establish the limiting behavior of the parameters 
$\lambda(t)$ and $x(t)$ introduced in the previous decomposition, and the convergence of $\tilde u$ to $0$ in $W^{k,1}\cap W^{k,\infty}$.

\begin{lemma}\label{limpara}
	Let $u$ be as in Proposition \ref{continuoussolresPR}. There exists continuous functions of time $\lambda(t),x(t)$ such that
	\begin{align*}
		u(x,t) 
		= \frac{1}{\lambda^{2}(t)} 
		U\!\left( \frac{x - x(t)}{\lambda(t)} \right) 
		+ \frac{1}{\lambda^{2}(t)} 
		\tilde{v}\!\left( t, \frac{x - x(t)}{\lambda(t)} \right),
	\end{align*}
	with $\lim_{t\to \infty} \|\tilde{v}(t)\|_{W^{k,1}\cap W^{k,\infty}} = 0$, $\lim_{t\to \infty}\lambda(t)=0$ and $\lim_{t\to \infty}x(t)=x^*$.
\end{lemma}

\begin{proof}
	
	\textbf{Step 1.} \emph{Control of the remainder and continuity.} We remark that  Proposition \ref{continuoussolresPR} is valid for all $k\in \mathbb N$ for any constant $K>0$. Thus, by applying this theorem with a sequence $K_n\to 0$ we conclude that there exist functions $\lambda(t)>0$, $x(t)\in \mathbb{R}^{2}$ such that 
	\begin{align} \label{Decomposition1000}
		u(x,t)=\frac{1}{\lambda^{2}(t)}U\big( \frac{x-x(t)}{\lambda(t)} \big)+ \frac{1}{\lambda^{2}(t)}\tilde{u}\big( t, \frac{x-x(t)}{\lambda(t)}\big)
	\end{align}
	with $\lim_{t\to \infty} \|\tilde{u}(t)\|_{W^{k,\infty}\cap W^{k,1}}=0$ for any $k\in \mathbb N$. Next, one can apply Lemma \ref{GeometricalDeco} and, up to redefining $\tilde u$, $\lambda$ and $x$, we obtain that $\lambda$ and $x$ are continuous functions and that the above convergence remains valid.
	
	\smallskip
	
	\noindent \textbf{Step 2.} \emph{Control of the scale.} We now examine the limiting behavior of $\lambda(t)$.  
	Assume, for contradiction, that $\lambda(t) \not\to 0$.  
	Then there exists a sequence $t_{n} \to \infty$ and a constant $\delta > 0$ such that $\lambda(t_{n}) \ge \delta$ for all $n$.  
	In this case,
	\begin{align*}
		\frac{1}{\lambda^{2}(t_{n})}
		\left( U + \tilde{u}(t_{n}, \cdot) \right)
		\!\left( \frac{x - x(t_{n})}{\lambda(t_{n})} \right)
		\lesssim \frac{1}{\lambda^{2}(t_{n})}
		\le \frac{1}{\delta^{2}}
		\quad \text{for all } n.
	\end{align*}
	However, this contradicts the result of \cite{BCM}, see Theorem \ref{BCMThm} here, which shows that the limiting profile must be a Dirac mass.
	
	\smallskip
	
	\noindent \textbf{Step 3.} \emph{Control of the position.} Assume by contradiction that there exists $\epsilon>0$ and a sequence $t_n\to \infty$ such that $x_n=x(t_n)$ satisfies $|x_n-x^*|>\epsilon$. Letting $\lambda_n=\lambda(t_n)$, we know that by \eqref{Decomposition1000} we have
	\begin{align*}
		u(t_{n},x)=\frac{1}{\lambda_{n}^{2}} \frac{8}{(1+\frac{|x-x_{n}|^{2}}{\lambda_{n}^{2}})^{2}}+\tilde{u}_{n}(x)
	\end{align*}
	and that $\|\tilde{u}_{n}\|_{L^{1}(\R^{2})}\to 0$ as $n\to \infty$. As $\lambda_n \to 0$ while $|x_n-x^*|>\epsilon$, we have that on the set $\{|x-x^*|<\epsilon/2\}$, the function $x\mapsto  8\lambda_{n}^{-2} (1+\frac{|x-x_{n}|^{2}}{\lambda_{n}^{2}})^{-2}\lesssim_\epsilon \lambda_n^2$ converges uniformly to $0$ as $n\to \infty$. Therefore,
	$$
	\| u(t_n) \|_{L^1(B_{\epsilon/2}(x^*))}\to 0
	$$
	as $n\to \infty$. But this contradicts the result of \cite{BCM}, see Theorem \ref{BCMThm}, that states that $u(t_n)$ must converge to a Dirac mass of $8\pi$ centred at $x^*$.	
\end{proof}

\section{Rough convergence in additional scale invariant spaces}\label{sec:roughconv}

In the soliton resolution provided by Theorem~\ref{thn:soliton-resolution}, we have shown that a solution $u$ of the Keller-Segel equation \eqref{KS} emanating from an initial data $u_0\in L^1(\langle x \rangle^2 dx)$ with $\int u_0=8\pi$ can be written asymptotically as $t \to \infty$ in the form
\begin{align} \label{id:property-3}
	u(x,t)
	&= U_{\lambda(t),x_{\star}(t)}+\tilde u(x,t) .
\end{align}
The modulation parameters $\lambda(t)$ and $x_{\star}(t)$ are uniquely determined by the orthogonality conditions
\begin{equation} \label{id:property-4}
\int_{\mathbb R^2} \tilde{v}\, \Lambda U \,dy = 0,
\qquad
\int_{\mathbb R^2} \tilde{v}\, \partial_{y_i}U \,dy = 0,
\end{equation}
for $i=1,2$, where $\tilde u(x) = \lambda^{-2} \tilde v((x-x_\star)/\lambda)$. Moreover, they satisfy
\begin{equation} \label{id:property-4}
\lambda(t)\to 0,
\qquad
x_{\star}(t)\to x^*,
\end{equation}
and the remainder term obeys
\begin{equation} \label{id:property-5}
\|\tilde u\|_{L^1}=o(1),
\qquad
\|\nabla^k \tilde u\|_{L^\infty}=o\!\left(\lambda(t)^{-2-k}\right) \mbox{ for }k=0,1,2,
\end{equation}
as $t\to \infty$; more generally $\| \tilde v\|_{W^{k,1}\cap W^{k,\infty}}\to 0$ as $t \to \infty$ for any $k\in \mathbb N$. We call this convergence rough because firstly it is only measured in the scale invariant space $L^1$ (and in other norms but as a direct consequence of parabolic regularization), and because it does not provide a convergence rate neither for the remainder $\tilde u$ nor for the parameters $\lambda $ and $x_\star$. The goal of this section is to show a rough convergence in other scale invariant spaces, which will be used to control more precisely the flow in subsequent sections.

These scaling critical norms are the following. For the remainder $\tilde u$, to control nonlinear effects the bound $\| \tilde u(t) \|_{L^\infty}=o(\lambda^{-2}(t))$ in \eqref{id:property-5} is not sufficient as it diverges and is not localized; we will thus bound $\| |x-x_\star(t)|^2\tilde u(t)\|_{L^\infty}$ and similar norms. For the parameters of the soliton we will bound $\lambda \dot \lambda $ and $\lambda \dot x_\star$, which is a standard approach that amounts to bound time derivative of modulation parameters. A novelty of the present work is the convergence in the H\"older norms $\| \lambda\|_{\dot C^{1/2}([T,\infty)}$ and $\| x_\star \|_{\dot C^{1/2}([T,\infty)}$ which will be the key to control the linearized flow in the zone $|x|\lesssim \sqrt{t}$ in the next section. For $\mu>0$ the rescaled solution $\mu^{-2}u(x/\mu,t/\mu^2)$ can also be decomposed as \eqref{id:property-3} with parameters $\mu \lambda(t/\mu^2)$, $\mu x_\star(t/\mu^2)$ and remainder $\mu^{-2}\tilde u(x/\mu,t/\mu^2)$. The aforementioned norms are left unchanged by this transformation, hence they are scale invariant norms.

The main result of this Section is the following Proposition, that gathers all estimates proved hereafter.

\begin{proposition}[Qualitative convergence in scale invariant spaces] \label{pr:cv-scale-invariant}

Assume $u$ is a solution of the Keller-Segel equation \eqref{KS} with $u_0\in L^1(\langle x \rangle^2 dx)$ and $\int u_0=8\pi$. Recall it satisfies \eqref{id:property-1}-\eqref{id:property-2} and can be decomposed as \eqref{id:property-3}-\eqref{id:property-5}. Then the following hold.

\smallskip

\noindent \emph{(i) Control of the parameters}. We have the modulation estimate
$$
|\dot \lambda(t)|+|\dot x_\star(t)|\leq o(\lambda^{-1}(t))
$$
as $t\to \infty$. Moreover, $\| \lambda \|_{\dot C^{1/2}([T,\infty))}+\| x_\star \|_{\dot C^{1/2}([T,\infty))}\to 0$ as $T\to \infty$, i.e.
$$
|\lambda(t+R^2)-\lambda (t)|+|x_\star(t+R^2)-x_{\star}(t)|\leq o(R)
$$
for all $R\geq 0$, uniformly as $t\to \infty$.

\smallskip

\noindent \emph{(ii) Control of the remainder}. In the \emph{nearby parabolic region} $\{|x-x_\star(t)|\le \sqrt t\}$, 
	\begin{equation} \label{bd:pointwise-tildeu-interior}
	(\lambda(t)+|x-x_\star(t)|)\,|\nabla \widetilde u(x,t)|
	+
	|\widetilde u(x,t)|
	=
	o\!\left(\frac{1}{\lambda(t)^2+|x-x_\star(t)|^2}\right).
	\end{equation}
	
	\medskip
	
	\noindent In the \emph{far away parabolic region} $\{|x-x_\star(t)|\ge \sqrt t\}$,
	\begin{equation} \label{bd:pointwise-tildeu-exterior}
	\sqrt{t}\,|\nabla \widetilde{u}(x,t)|
	+
	|\widetilde u(x,t)|
	\lesssim
	\frac{1}{t}\frac{1}{|x-x_\star(t)|^2}.
	\end{equation}
	
	\noindent For the Poisson field, for all $x\in \mathbb R^2$,
	\begin{equation} \label{bd:pointwise-Phitildeu}
	|\nabla \Phi_{\widetilde u}(x,t)|
	=
	o\!\left(\frac{1}{\lambda(t)+|x-x_\star(t)|}\right).
	\end{equation}
	
\end{proposition}

\begin{remark}

We call such estimates \emph{qualitative} because they are expressed with $o()$'s and not with explicit norms of the remainder $\tilde u$. We will establish \emph{quantitative} versions of these estimates, useful for the full nonlinear analysis later on, in Section \ref{sec:refined-convergence}.

\end{remark}

\begin{proof}

These results are proved in Lemmas \ref{lem:modest}, \ref{lem:cv-holder-scale} and \ref{lemma:QualSlowVariation}, Proposition \ref{Prop:Pointest} and Corollary \ref{Coro:PointPoisson} respectively.

\end{proof}

\subsection{The renormalized evolution in inner variables}

We introduce the inner variables
\begin{align} \label{id:inner-variables}
	u(x,t) = \frac{1}{\lambda^2} v(y,s),
	\qquad
	y = \frac{x-x_\star}{\lambda},
	\qquad
	s = \int_1^t \lambda^{-2}(t')dt'.
\end{align}
A direct computation shows that \(v\) solves the renormalized Keller-Segel equation
\begin{align}\label{innerKS}
	v_s-\frac{\partial_s \lambda}{\lambda}\Lambda v -\frac{\partial_s x_{\star}}{\lambda}\cdot \nabla v = \Delta v -\nabla \cdot (v\nabla \Phi_v).
\end{align}
The decomposition \eqref{id:property-3}-\eqref{id:property-5} becomes
\begin{align} \label{id:property-3bis}
	v(y,s) = U(y) + \tilde v(s,y), \quad \mbox{with} \quad \| \tilde v(s)\|_{W^{k,1}\cap W^{k,\infty}}\to 0 \mbox{ for all }k\in \mathbb N.
\end{align}
As $U$ is a stationary state, the renormalized remainder $\tilde v$ solves
	\begin{align}\label{innerKS-remainder}
	\tilde v_s &= \frac{\partial_s \lambda}{\lambda}\Lambda (U+\tilde v) +\frac{\partial_s x_\star}{\lambda}\cdot \nabla (U+\tilde v )+ \Delta \tilde v -\nabla \cdot (\tilde v\nabla \Phi_U)-\nabla \cdot (U\nabla \Phi_{\tilde v})-\nabla (\tilde v\nabla \Phi_{\tilde v}).
	\end{align}

\subsection{Rough modulation estimates} \label{subsec:rough-modulation}

We here bound $\dot \lambda $ and $\dot x_\star$. We first record a basic estimate on the field generated by the remainder term.
\begin{lemma}\label{lem:grad_phi_tilde_u}
	Assume that $\|\tilde u\|_{L^\infty}=o\!\left(\lambda^{-2}\right)$ and $\|\tilde u\|_{L^1}=o(1)$ as $t \to \infty$. Then
	$$
	|\nabla  \Phi_{\tilde u} |\leq o\left(\frac{1}{\lambda(t)}\right).
	$$
	
\end{lemma}

\begin{proof}
	For any $\rho>0$, we decompose
	\begin{equation} \label{bd:Poisson-field-Linfty}
	|\nabla \Phi_{\tilde u}(x)|
	\lesssim
	\int_{|x-y|\le \rho}\frac{|\tilde u(y)|}{|x-y|}\,dy
	+
	\int_{|x-y|\ge \rho}\frac{|\tilde u(y)|}{|x-y|}\,dy \lesssim \rho\,\|\tilde u\|_{L^\infty}+\rho^{-1}\|\tilde u\|_{L^1}.
	\end{equation}
	We can optimize in $\rho$ by choosing $\rho=\|\tilde u\|_{L^\infty}^{-1/2}\|\tilde u\|_{L^1}^{1/2}$, which yields the result.
\end{proof}

We can now derive the qualitative modulation estimates.

\begin{lemma}\label{lem:modest}
	Let $u$ be a solution of the Keller-Segel equation \eqref{KS} as in Proposition \ref{pr:cv-scale-invariant}. Then $\lim_{t\to \infty}\lambda \dot \lambda = \lim_{t\to \infty} \lambda \dot x_\star=0$. 
	\end{lemma}

\begin{proof}
	Since $\nabla \Phi_{\tilde v}(s,y)=\lambda \nabla \Phi_{\tilde u}(t,x)$, we have $\| \Phi_{\tilde v}\|_{L^\infty}=o(1)$ by \eqref{bd:Poisson-field-Linfty}. Injecting this estimate and the estimates \eqref{id:property-3bis} in \eqref{innerKS-remainder} shows
	$$
	\tilde v_s = \frac{\partial_s \lambda}{\lambda}\Lambda (U+\tilde v) +\frac{\partial_s x_{\star}}{\lambda}\cdot \nabla (U+\tilde v )+\mathcal R,
	$$
	where $\|\mathcal R(s)\|_{L^\infty(\mathbb R^2)}\to 0$ as $s\to \infty$. Differentiating the first orthogonality condition \eqref{id:property-4} and using the above equation for $\tilde v$, we obtain
	\begin{align}
	\nonumber0 & =\frac{\partial_s \lambda}{\lambda} \int \Lambda (U+\tilde v)\Lambda U dy +\frac{\partial_s x_{\star}}{\lambda}\cdot \left( \int \nabla \tilde v  \Lambda U dy \right)+\int \mathcal R \Lambda U dy, \\
	&\label{id:modulation-estimate-techI}= \frac{\partial_s \lambda}{\lambda} \left( \int (\Lambda U)^2 dy +o(1)\right) +\frac{\partial_s x_{\star}}{\lambda}\cdot o(1)+o(1), 
	\end{align}
	where for the second term we used that $\int \Lambda U \partial_{y_i} Udy=0$ for $i=1,2$. Differentiating the second orthogonality condition \eqref{id:property-4} for $i=1,2$, yields similarly
	\begin{equation} \label{id:modulation-estimate-techII}
	0 =  \frac{\partial_s \lambda}{\lambda} o(1)+ \frac{\partial_{s}x_{\star,i}}{\lambda} \left(\int (\partial_{y_i}U)^2+o(1)\right)+ \frac{\partial_{s}x_{\star,i'}}{\lambda} o(1) +o(1).
	\end{equation}
	where $i'=2$ if $i=1$ and $i'=1$ if $i=1$, and where we used $\int \partial_{y_1}U\partial_{y_2}Udy=0$. Combining \eqref{id:modulation-estimate-techI} and \eqref{id:modulation-estimate-techII} for $i=1,2$ shows $ \partial_s \lambda =o(\lambda)$ and $\lambda \partial_s x_{\star}=o(\lambda)$ as $s\to \infty$. Going back to the original variable $t$, this is  $\lambda \dot \lambda =o(1)$ and $\lambda \dot x_\star=o(1)$ which is the desired result.
	
	\end{proof}

\subsection{Slow variation of the soliton over parabolic scales} \label{subsec:slowvar-modulation}

In this subsection we show that the H\"older $1/2$ norm of $\lambda $ and $x_\star$ converges to $0$ in time. Brownian motion, modelled by the linear heat equation, travels on average at a distance $R$ over a timescale of order $R^2$. The results of Lemmas \ref{lem:cv-holder-scale} and \ref{lemma:QualSlowVariation} show that the stationary state, asymptotically over time, only moves (by distance and scaling) at a distance $o(R)$ over a timescale of order $R^2$. This shows it \emph{varies slowly over parabolic scales}.

\begin{lemma} \label{lem:cv-holder-scale}

Let $u$ be a solution of the Keller-Segel equation \eqref{KS} as in Proposition \ref{pr:cv-scale-invariant}. Then $\| \lambda \|_{\dot C^{1/2}([T,\infty))}\to 0$ as $T\to \infty$.

\end{lemma}

\begin{proof}

Pick $\delta>0$. Let $R>0$ and $t\geq T$. Consider first the case when both $\lambda(t+R^2)\leq \delta R$ and $\lambda(t)\leq \delta R$. Then we have $|\lambda(t+R^2)-\lambda(t)|\leq \delta R$.

Second, consider the case when either $\lambda(t+R^2)\geq \delta R$ or $\lambda(t)\geq \delta R$. By Lemma \ref{lem:modest} we have $\frac{d}{dt}(\lambda^2)=o(1)$. Integrating gives $|\lambda^2(t+R^2)-\lambda^2(t)|\leq o(R^2)$. Using $|a-b|=|a^2-b^2|/|a+b|$ this implies $|\lambda(t+R^2)-\lambda(t)|\leq \delta^{-1} o( R)$ as $T\to \infty$ in this second case. By combining the two cases, since $\delta$ is arbitrary, one gets $|\lambda(t_0+R^2)-\lambda(t)|=o(R)$ as desired.

\end{proof}

In order to prove the slow variation of the position $x_\star(t)$ over parabolic scales, we record a localized first-moment variation estimate.

\begin{lemma}\label{lem:locfirstmom}
	
Let $u$ be a solution of the Keller-Segel equation \eqref{KS} as in Proposition \ref{pr:cv-scale-invariant}. Then, for any $\xi\in\mathbb{R}^2$ and $R>0$, one has
	\begin{align*}
		\left| \frac{d}{dt}\left(\int_{\mathbb{R}^2} u(x,t)(x-\xi)\chi_{R,\xi}\,dx  \right) \right|
		\lesssim  \int_{|x-\xi|\ge R} \frac{u(x,t)}{|x-\xi|}\,dx.
	\end{align*}
	
\end{lemma}
\begin{proof}
We recall the weak formulation of the Keller--Segel equation (see Section~3.6 of \cite{SBook} for further details): for every $\phi\in C_c^\infty(\mathbb{R}^2)$,
\begin{align*}
	\frac{d}{dt}\int_{\mathbb{R}^2} u(x,t)\phi(x)\,dx
	=
	\int_{\mathbb{R}^2} u(x,t)\Delta\phi(x)\,dx
	+
	\iint_{\mathbb{R}^2\times\mathbb{R}^2}
	\rho_\phi(x,y)\,u(x,t)u(y,t)\,dx\,dy,
\end{align*}
where
\begin{align*}
	\rho_\phi(x,y)
	=
	-\frac{1}{4\pi}
	\frac{x-y}{|x-y|^2}
	\cdot
	\bigl(\nabla\phi(x)-\nabla\phi(y)\bigr).
\end{align*}
	Fix \(i\in\{1,2\}\) and choose $\phi(x)=(x-\xi)_i\chi_R(x)$. Then $\nabla\phi(x)=e_i\chi_R(x)+(x-\xi)_i\nabla\chi_R(x)$. Using $\Delta (x-\xi)_i=0$ and $|\nabla^k \chi_{R,\xi}|\lesssim R^{-k}\mathbbm 1(R<|x-\xi|<2R)$ the first term is
	$$
	\int_{\mathbb{R}^2} u(x,t)\Delta\phi(x)\,dx=O\left( \int_{R<|x-\xi|<2R} \frac{u(x,t)}{|x-\xi|}\,dx \right).
	$$
	For the second, we decompose	
	\begin{align*}
		I_R:=B_R(\xi),\qquad
		M_R:=B_{2R}(\xi)\setminus B_R(\xi),\qquad
		E_R:=\mathbb{R}^2\setminus B_{2R}(\xi).
	\end{align*}
	Since \(\nabla\phi=e_i\) on \(I_R\) and \(\nabla\phi=0\) on \(E_R\), the bilinear kernel $\rho_\phi$ vanishes on \(I_R\times I_R\) and \(E_R\times E_R\). Thus only the regions $\mathbb R^2\times M_R\cup M_R\times \mathbb R^2$ and $I_R \times E_R\cup E_R\times I_R$ contribute. 
	
	If at least one of the variables lies in \(M_R\), then $|\rho_\phi(x,y)|\lesssim R^{-1}$. Using symmetry and conservation of mass, all such contributions are bounded by
	\begin{align*}
		\iint_{M_R\times \mathbb R^2\cup \mathbb R^2\times M_R}\rho_\phi(x,y)u(x)u(y)\,dx\,dy \lesssim \frac1R\int_{M_R}u(x,t)\,dx
		\lesssim
		\int_{|x-\xi|\ge R}\frac{u(x,t)}{|x-\xi|}\,dx.
	\end{align*}

	Finally, for the outer--inner contribution, if \(x\in E_R\) and \(y\in I_R\), then
		$\nabla\phi(x)=0$ and
		$\nabla\phi(y)=e_i$,
	so that $|\rho_\phi(x,y)|\lesssim |x-y|^{-1}$.Using again conservation of mass, this yields
	\begin{align*}
		\iint_{E_R\times I_R} |\rho_\phi(x,y)|u(x)u(y)\,dx\,dy
		&\lesssim
		\left(\int_{|x-\xi|\ge 2R}\frac{u(x)}{|x-\xi|}\,dx\right)
		\left(\int_{|y-\xi|\le R}u(y)\,dy\right) \lesssim
		\int_{|x-\xi|\ge 2R}\frac{u(x)}{|x-\xi|}\,dx
	\end{align*}
	and similarly for the integration over $I_R\times E_R$. Collecting the above estimates, we obtain
	\begin{align*}
		\left| \frac{d}{dt}\int_{\mathbb{R}^2} u(x,t)(x-\xi)_i\chi\!\left(\frac{x-\xi}{R}\right)\,dx \right|
		\lesssim
		\int_{|x-\xi|\ge R}\frac{u(x,t)}{|x-\xi|}\,dx.
	\end{align*}
	Since this holds for each component \(i=1,2\), the vectorial estimate follows.
\end{proof}

\begin{lemma}\label{lemma:QualSlowVariation}
	Let $u$ be a solution of the Keller-Segel equation \eqref{KS} as in Proposition \ref{pr:cv-scale-invariant}. Then $\| x_\star \|_{\dot C^{1/2}([T_0,\infty))}\to 0$ as $T_0\to \infty$.
	\end{lemma}

\begin{proof}
	
	Let $T_0>0$. We need to show that
	\begin{equation} \label{bd:slow-variation-xstar-def}
	\frac{|x_{\star}(T+R^2)-x_\star(T)|}{R}\to 0
	\end{equation}
	as $T_0\to \infty$, uniformly in $T,R\in [T_0,\infty)\times (0,\infty)$. We divide between three timescales.
	
	\smallskip
	
	\noindent \textbf{Step 1.} \emph{Large timescales.} We claim that \eqref{bd:slow-variation-xstar-def} holds as $T_0\to \infty$, uniformly in $T,R\in [T_0,\infty)\times (1,\infty)$. Indeed, in that case since $R>1$ we use \eqref{id:property-4} to bound
	$$
	\frac{|x_{\star}(T+R^2)-x_\star(T)|}{R}\leq |x_{\star}(T+R^2)-x^*|+|x_{\star}(T)-x^*|\to 0 \quad \mbox{as }t\to \infty.
	$$

	\noindent \textbf{Step 2.} \emph{The soliton timescale.} We claim that for any $M>0$, \eqref{bd:slow-variation-xstar-def} holds as $T_0\to \infty$, uniformly for $T\geq T_0$ and $0<R<M\lambda(T)$. Indeed, in this case by Lemma \ref{lem:cv-holder-scale} for all $t\in [T,T+R^2]$ there holds $|\lambda(t)-\lambda(T)|\leq  o(\lambda(T))$ so that $\lambda(t)=\lambda(T)(1+o(1))$. Lemma \ref{lem:modest} then shows $\dot x_{\star}(t)=o(1/\lambda(T))$ for $t\in [T,T+R^2]$. Integrating over time shows
	$$
	\frac{|x_{\star}(T+R^2)-x_\star(T)|}{R}= \frac{\int_{T}^{T+R^2} o((\lambda(T))^{-1}) dt}{R} =o(R(\lambda(T))^{-1})=o_M(1) \quad \mbox{as }t\to \infty.
	$$
	
	\noindent \textbf{Step 3.} \emph{Intermediate timescales.} In this last step we show that \eqref{bd:slow-variation-xstar-def} holds as $T_0\to \infty$, uniformly in $T\geq T_0$ and $\lambda(T)\leq R \leq 1$. Fix $\delta>0$ small. For each $T\gg1$, define
	\begin{equation} \label{slow-variation-technical10}
	\widetilde R_\star[T]
	:=
	\sup\Bigl\{
	\widetilde R\in(0,1)\ :\
	|x_\star(T+R^2)-x_\star(T)|<\delta R
	\ \text{for all }R\in[0,\widetilde R)
	\Bigr\}.
	\end{equation}
	We claim that, for $T$ sufficiently large depending on $\delta$,
	\[
	\widetilde R_\star[T]=1.
	\]
	As $\delta>0$ is arbitrary, this immediately implies the statement of the lemma. The proof below relies on a bootstrap argument. Pick $T$ large and write $\widetilde R_{\star}=\widetilde R_\star[T]$ for simplicity. First, by Step 2, we have 
	\begin{equation}\label{eq:prelim-lower-bound}
		\widetilde R_\star \ge M\lambda(T)
	\end{equation}
	for any $M\gg 1$ if $T_0$ is large enough. By Lemma \ref{lem:cv-holder-scale}, we have for $t\in[T,T+\widetilde R_\star^2]$ that $\lambda(t)\leq \lambda(T)+o(\widetilde R_\star)$. By \eqref{eq:prelim-lower-bound} and since $M\gg 1$ is arbitrary, it follows that
		\begin{equation}\label{eq:lambda-small-on-window}
		\lambda(t)=o(\widetilde R_\star)
		\qquad\text{for all }t\in[T,T+\widetilde R_\star^2]
	\end{equation}
	provided $T_0$ is sufficiently large. We now consider the solution over the time interval $[T,T+\widetilde R_\star]$. On this time interval we have by definition
	\begin{equation}\label{eq:saturation}
		|x_\star(t)-x_\star(T)|<\delta \widetilde R_\star
		\quad\text{for all }t\in[T,T+\widetilde R_\star^2).
	\end{equation}
	We consider the localized center of mass
	\[
	I(t):=\int_{\mathbb R^2}u(x,t)(x-x_\star(T))\chi_{\widetilde R_\star,x_\star(T)}(x)\,dx.
	\]
	Let us first relate $I(T)$ and $I(T+\widetilde R_\star^2)$ to $x_\star$ and $x_\star(T+\widetilde R_\star^2)$. By the soliton resolution decomposition \eqref{id:property-3}-\eqref{id:property-5}, and since $\int y_iU \chi_{\widetilde R_\star}=0$ for $i=1,2$ by radiality of $U$ and $\chi$, we have
	\begin{align}
	\nonumber I(T) & = \int_{\mathbb R^2}\left(U_{\lambda(T),x_\star(T)}(x)+\tilde u(x,T)\right)(x-x_\star(T))\chi_{\widetilde R_\star,x_\star(T)}(x)\,dx\\
	\label{slow-variation-technical3} & = 0+O(\widetilde R_\star\| \tilde u(t)\|_{L^1}) \ = o(\widetilde R_\star).
	\end{align}
	Next, pick $t\in [T,T+\widetilde R_\star^2]$. Since $|x_{\star}(t) -x_{\star}(T)|\leq \delta \widetilde R_\star$ by \eqref{eq:saturation}, we have if $|x-x_\star(T)|\ge \widetilde R_\star$ that
	\begin{equation} \label{slow-variation-technical1}
	|x-x_\star(t)|
	\ge |x-x_\star(T)|-|x_\star(t)-x_\star(T)|
	\ge (1-\delta)\widetilde R_\star
	\ge \frac12\widetilde R_\star,
	\end{equation}
	provided $\delta$ is chosen small enough. Using \eqref{slow-variation-technical1}, the bound $U_{\lambda,x_\star}(x)\lesssim \lambda^2|x-x_\star|^{-4}$ and \eqref{eq:lambda-small-on-window}, we obtain the following technical bound:
	\begin{align}
	 \label{slow-variation-technical2} \int_{|x-x_{\star}(T)|>\widetilde R_\star} U_{\lambda(t),x_\star(t)}(x) |x-x_\star(T)| dx& \lesssim \int_{|x-x_{\star}(T)|>\widetilde R_\star} \frac{\lambda^2(t)}{|x-x_\star(T)|^3}dx  \lesssim \frac{\lambda^2(t)}{\widetilde R_\star}=o (\widetilde R_\star).
	\end{align}
	Using the explicit formula $\int U_{\lambda,y}(x-y')dx=8\pi(y-y')$ and \eqref{slow-variation-technical2}, we find similarly to \eqref{slow-variation-technical3} that
	\begin{align}
	\nonumber I(T+\widetilde R_\star^2) & = \int_{\mathbb R^2}\left(U_{\lambda(T+\widetilde R_\star^2),x_\star(T+\widetilde R_\star^2)}(x)+\tilde u(x,T+\widetilde R_\star^2)\right)(x-x_\star(T))\chi_{\widetilde R_\star,x_\star(T)}(x)\,dx\\
	\nonumber & =  \int_{\mathbb R^2} U_{\lambda(T+\widetilde R_\star^2),x_\star(T+\widetilde R_\star^2)}(x) (x-x_\star(T))dx\\
	\nonumber & \qquad + \int_{\mathbb R^2} U_{\lambda(T+\widetilde R_\star^2),x_\star(T+\widetilde R_\star^2)}(x) (x-x_\star(T)) (\chi_{\widetilde R_\star,x_\star(T)}(x)-1)\,dx+ o(\widetilde R_\star)\\
	\label{slow-variation-technical4} & = 8\pi (x_\star(T+\widetilde R_\star^2)-x_\star(T))+o(\widetilde R_\star).
	\end{align}
	Combining \eqref{slow-variation-technical3} and \eqref{slow-variation-technical4} we see
	$$
	x_\star(T+\widetilde R_\star^2)-x_\star(T)= \frac{1}{8\pi}\int_T^{T+\widetilde R_\star^2} \dot I(t)dt+o(\widetilde R_\star).
	$$
	We now estimate the time derivative of $I(t)$. Applying successively Lemma \ref{lem:locfirstmom} with $\xi=x_\star(T)$ and $R=\widetilde R_\star[T]$, then the decomposition \eqref{id:property-3}-\eqref{id:property-5} and \eqref{slow-variation-technical2} we infer
	\begin{align*}
	& \left|\dot I(t)\right|
	\lesssim
	\int_{|x-x_\star(T)|\ge \widetilde R_\star}
	\frac{u(x,t)}{|x-x_\star(T)|}\,dx = \int_{|x-x_\star(T)|\ge \widetilde R_\star}
	\frac{U_{\lambda(t),x_\star(t)}+\tilde u(x,t)}{|x-x_\star(T)|}\,dx \\
	&\qquad \qquad  \leq \widetilde R^{-2}_\star \int_{|x-x_\star(T)|\ge \widetilde R_\star} U_{\lambda(t),x_\star(t)} |x-x_\star(T)| \,dx+ \widetilde R^{-1}_\star \int_{|x-x_\star(T)|\ge \widetilde R_\star} |\tilde u(x,t)| \,dx = o(\widetilde R^{-1}_\star ).
	\end{align*}
	Therefore, $|x_\star(T+\widetilde R_\star^2)-x_\star(T)|=o(\widetilde R_\star)<\delta \widetilde R_\star /2$, provided $T$ has been chosen large enough depending on $\delta$. By the very definition \eqref{slow-variation-technical10} of $\widetilde R_\star$ and a standard continuity argument, this implies $\widetilde R_\star=1$ as desired.
	
\end{proof}

\medskip

\noindent 
\subsection{Weighted pointwise estimates for the remainder} \label{subsec:pointwise-remainder}
 
In this section we derive pointwise estimates for $\tilde u$, $\nabla \tilde{u}$ and $\nabla \Phi_{\tilde u}$. These estimates rely on standard parabolic regularization estimates near the soliton, and on the new $\varepsilon$-regularity result of Theorem \ref{thm:L1Linfsmall} away from the soliton. The time dependence of the modulation parameter $x_{\star}(t)$ introduces additional difficulties, which will be handled using its slow variation established in Lemma \ref{lemma:QualSlowVariation}.

\begin{proposition}\label{Prop:Pointest}
	Let $u$ be a solution of the Keller-Segel equation \eqref{KS} as in Proposition \ref{pr:cv-scale-invariant}. Then asymptotically as $t\to \infty$ the following hold for all $x\in\mathbb{R}^2$.
	
	\medskip
	
	\noindent\emph{Interior region.} If $|x-x_\star(t)|\le \sqrt t$, then
	\begin{equation} \label{bd:pointwise-tildeu}
	(\lambda(t)+|x-x_\star(t)|)\,|\nabla \widetilde u(x,t)|
	+
	|\widetilde u(x,t)|
	=
	o\!\left(\frac{1}{\lambda(t)^2+|x-x_\star(t)|^2}\right).
	\end{equation}
	
	\medskip
	
	\noindent\emph{Exterior region.} If $|x-x_\star(t)|\ge \sqrt t$, then
	\[
	\sqrt{t}\,|\nabla \widetilde{u}(x,t)|
	+
	|\widetilde u(x,t)|
	\lesssim
	\frac{1}{t}\frac{1}{|x-x_\star(t)|^2}.
	\]
\end{proposition}

\begin{proof}
	We consider separately the interior and exterior regions.
	
	\noindent\emph{Interior region.}
	Fix $\delta>0$. We claim that for $t$ sufficiently large, if $|x-x_{\star}(t)|\le \sqrt{t}$, then
	\begin{equation}\label{eq:intgoal}
		(\lambda(t)^2+|x-x_\star(t)|^2)^{3/2}|\nabla \widetilde u(x,t)|
		+
		(\lambda(t)^2+|x-x_\star(t)|^2)|\widetilde u(x,t)|
		\le \delta.
	\end{equation}
	Since $\delta>0$ is arbitrary, this will yield the desired estimate. \newline 
	We first consider the intermediate region
	\begin{equation}  \label{pointwise-bounds20}
	M_\delta \lambda(t)\le |x-x_\star(t)|\le \sqrt t,
	\end{equation}
	where $M_\delta\gg 1$ will be fixed later. Let $\sigma=\sigma(\frac{\delta}{2})$ and $\varepsilon=\varepsilon(\sigma,\delta/2)$ be as in Theorem \ref{thm:L1Linfsmall}. Set
	\begin{equation}  \label{pointwise-bounds10}
	R^2=\beta |x-x_\star(t)|^2,
	\qquad
	t_0:=t-\sigma R^2,
	\end{equation}
	where $\beta>0$ is a sufficiently small universal constant. Choosing $\beta$ small enough, we have $0<\sigma R^2<t$ and $t_0\gtrsim t$. \newline 
	By the soliton resolution and the $L^1$ smallness of the remainder \eqref{id:property-3}-\eqref{id:property-5},
	\[
	\int_{B_{8R}(x)}u(t_0,z)\,dz
	=
	\int_{B_{8R}(x)}
	\frac{1}{\lambda(t_0)^2}
	U\!\left(\frac{z-x_\star(t_0)}{\lambda(t_0)}\right)\,dz
	+
	o_{t\to\infty}(1).
	\]
	Moreover, Lemma \ref{lemma:QualSlowVariation} yields
	\begin{equation}  \label{pointwise-bounds1}
	|x_\star(t_0)-x_\star(t)|=o(R)=o(|x-x_\star(t)|).
	\end{equation}
	Similarly, by Lemma \ref{lem:cv-holder-scale} we have
	\begin{equation}\label{pointwise-bounds2}
	\lambda(t_0)=\lambda(t)+o(R).
	\end{equation}
	Now let $z\in B_{8R}(x)$. For $\beta$ sufficiently small and $t$ sufficiently large, by \eqref{pointwise-bounds20}, \eqref{pointwise-bounds10}, \eqref{pointwise-bounds1} and \eqref{pointwise-bounds2}:
	\begin{align*}
		|z-x_\star(t_0)|\approx  |x-x_\star(t)|\gtrsim M_\delta \lambda(t_0).
	\end{align*}
	Using the decay $U(y)\lesssim |y|^{-4}$, we infer
	\[
	\int_{B_{8R}(x)}
	\frac{1}{\lambda(t_0)^2}
	U\!\left(\frac{z-x_\star(t_0)}{\lambda(t_0)}\right)\,dz
	\lesssim
	\int_{|y|\gtrsim M_\delta}U(y)\,dy
	\lesssim
	\frac{1}{M_\delta^2}.
	\]
	Therefore
	\[
	\int_{B_{8R}(x)}u(t_0,z)\,dz
	\le
	C M_\delta^{-2}+o_{t\to\infty}(1)\leq \varepsilon
	\]
	where the last inequality holds by choosing first $M_\delta$ sufficiently large and then $t$ sufficiently large. We may then apply Theorem \ref{thm:L1Linfsmall}, which yields
		\[
	R^3|\nabla u(x,t)|+R^2|u(x,t)|\le \frac{\delta}{2}.
	\]
	Since $u=U_{\lambda,x_\star}+\widetilde u$ and $|\nabla^k U_{\lambda(t),x_{\star}(t)}(x)|\lesssim |x-x_\star(t)|^{-2-k}M_\delta^{-2}$ by \eqref{pointwise-bounds20}, we can absorb the soliton contribution by taking $M_\delta$ sufficiently large, and obtain
	\[
	|x-x_\star(t)|^3|\nabla \widetilde u(x,t)|
	+
	|x-x_\star(t)|^2|\widetilde u(x,t)|.
	\le \delta 
	\]
	We recall the above bound was proved in the region $\{M_\delta\lambda(t)\le |x-x_\star(t)|\le \sqrt t\}$. In the complementary region $|x-x_\star(t)|\le M_\delta\lambda(t)$, the soliton resolution bound \eqref{id:property-5} directly yields
	\[
	\lambda(t)^3|\nabla \widetilde u(x,t)|+\lambda(t)^2|\widetilde u(x,t)|=o(1),
	\]
	uniformly in $x$. Combining the last two estimates, since $\delta>0$ is arbitrary, we obtain \eqref{eq:intgoal}.
	
	\smallskip

\noindent\emph{Exterior region.}
We now consider the region $|x-x_\star(t)|\ge \sqrt t$. We take $\delta=1$, and let $\sigma$ and $\varepsilon$ be given by Theorem \ref{thm:L1Linfsmall}. Fix $s\in[t/8,t]$, and set
$R=\beta \sqrt t$,
$t_0:=s-\sigma R^2$,
where $\beta>0$ is a sufficiently small universal constant. Then $t_0\gtrsim t$. Moreover, in this region one has $|x-x_\star(s)|\approx |x|$, and therefore
\[
\int_{B_{8R}(x)} u(z,t_0)\,dz
\lesssim
\frac{1}{|x-x_\star(t)|^2}\int_{\mathbb R^2} u(z,t_0)\,|z|^2\,dz
\lesssim
\frac{1}{|x-x_\star(t)|^2},
\]
where we used that the second moment is conserved and bounded. Since $|x-x_\star(t)|\ge \sqrt t$, the right-hand side is $o_{t\to\infty}(1)$. Hence, for $t$ sufficiently large,
\[
\int_{B_{8R}(x)} u(z,t_0)\,dz\le \varepsilon.
\]
Applying Theorem \ref{thm:L1Linfsmall}, we obtain
\[
R^3\|\nabla u(s,\cdot)\|_{L^\infty(B_{R/32}(x))}
+
R^2\|u(s,\cdot)\|_{L^\infty(B_{R/32}(x))}
\le 1
\qquad\text{for all } s\in[t/8,t].
\]
In particular, both $u$ and $\nabla \Phi_u$ are uniformly bounded in
$B_{R/64}(x)\times [t/8,t]$.
We may therefore apply Theorem \ref{thm:Lieberman}. Choose $r\lesssim \sqrt t$ sufficiently small, with universal implicit constant, so that for every $(x_0,s)\in B_{R/128}(x)\times [t/4,t]$ the parabolic cylinder $Q_{2r}(x_0,s)$ is contained in $B_{R/64}(x)\times [t/8,t]$. Then
\[
\iint_{Q_{2r}(x_0,s)} u(z,a)\,dz\,da
\lesssim
r^2\,\frac{1}{|x-x_\star(t)|^2}.
\]
Indeed, for each fixed $a\in[s-4r^2,s]$, the ball $B_{2r}(x_0)$ remains at distance comparable to $|x-x_\star(t)|$ from the origin, and the boundedness of the second moment gives
\[
\int_{B_{2r}(x_0)} u(z,a)\,dz
\lesssim
\frac{1}{|x-x_\star(t)|^2}.
\]
Integrating in time yields the desired bound on the cylinder. Applying Theorem \ref{thm:Lieberman} at the point $(x,t)$, we conclude that
\[
u(x,t)
\lesssim
r^{-4}\iint_{Q_{2r}(x,t)}u(z,a)\,dz\,da
\lesssim
\frac{1}{r^2}\frac{1}{|x-x_\star(t)|^2}
\lesssim
\frac{1}{t}\frac{1}{|x-x_\star(t)|^2}.
\]
After absorbing the contribution of the soliton this proves the desired pointwise bound for the solution in the exterior region. The estimate for the gradient is an immediate consequence of interior parabolic estimates and Sobolev embedding.
	
\end{proof}

We conclude with pointwise estimates for the potential generated by $\widetilde u$.

\begin{corollary}\label{Coro:PointPoisson}
	Assume $\tilde u$ is a function that satisfies $\| \tilde u\|_{L^1}\lesssim 1$ and the second bound in \eqref{bd:pointwise-tildeu} (i.e. only the inequality for $\tilde u$). Then, asymptotically as $t\to \infty$ we have for all $x\in\mathbb{R}^2$,
	\[
	|\nabla \Phi_{\widetilde u}(x,t)|
	=
	o\!\left(\frac{1}{\lambda(t)+|x-x_\star(t)|}\right).
	\]
\end{corollary}

\begin{proof}
	Recalling the representation formula,
	\[
	|\nabla \Phi_{\widetilde u}(x,t)|
	\lesssim
	\int_{\mathbb{R}^2}\frac{|\widetilde u(w,t)|}{|x-w|}\,dw,
	\]
	we estimate the right-hand side by distinguishing two cases.
	
	\medskip
	
	\noindent\emph{Case 1: $|x-x_\star(t)|\le \lambda(t)$.}
	We split the integral into $|x-w|\le \lambda(t)$ and $|x-w|>\lambda(t)$. For the far contribution, by \eqref{id:property-5},
	\[
	\int_{|x-w|>\lambda(t)}\frac{|\widetilde u(w,t)|}{|x-w|}\,dw
	\lesssim
	\frac{1}{\lambda(t)}\|\widetilde u(t)\|_{L^1}
	=
	o\!\left(\frac{1}{\lambda(t)}\right).
	\]
	For the near contribution, we use the pointwise bound on $\widetilde u$ of Proposition \ref{Prop:Pointest}:
	\[
	\int_{|x-w|\le \lambda(t)}\frac{|\widetilde u(w,t)|}{|x-w|}\,dw
	\lesssim
	o\!\left(\frac{1}{\lambda(t)^2}\right)
	\int_{|x-w|\le \lambda(t)}\frac{dw}{|x-w|}
	\lesssim
	o\!\left(\frac{1}{\lambda(t)}\right).
	\]
	
	\medskip
	
	\noindent\emph{Case 2: $|x-x_\star(t)|\ge \lambda(t)$.}
	Fix $\beta>0$ sufficiently small and split into
	\[
	|x-w|\le \beta\bigl(\lambda(t)+|x-x_\star(t)|\bigr),
	\qquad
	|x-w|>\beta\bigl(\lambda(t)+|x-x_\star(t)|\bigr).
	\]
	For the far contribution, by \eqref{id:property-5},
	\[
	\int_{|x-w|>\beta(\lambda(t)+|x-x_\star(t)|)}\frac{|\widetilde u(w,t)|}{|x-w|}\,dw
	\lesssim
	\frac{\|\widetilde u(t)\|_{L^1}}{\lambda(t)+|x-x_\star(t)|}
	=
	o\!\left(\frac{1}{\lambda(t)+|x-x_\star(t)|}\right).
	\]
	For the near contribution, note that for $\beta$ sufficiently small one has
	\[
	|w-x_\star(t)|\gtrsim |x-x_\star(t)|
	\quad\text{whenever}\quad |x-w|\le \beta(\lambda(t)+|x-x_\star(t)|).
	\]
	Hence, $|\widetilde u(w,t)|=o\!\left(\lambda(t)^2+|x-x_\star(t)|^2\right)^{-2}$ by proposition \ref{Prop:Pointest}, and therefore
	\begin{align*}
		\int_{|x-w|\le \beta(\lambda(t)+|x-x_\star(t)|)}\frac{|\widetilde u(w,t)|}{|x-w|}\,dw
		&\lesssim
		o\!\left(\frac{1}{\lambda(t)^2+|x-x_\star(t)|^2}\right)
		\int_{|x-w|\le \beta(\lambda(t)+|x-x_\star(t)|)}\frac{dw}{|x-w|} \\
		&\lesssim
		o\!\left(\frac{1}{\lambda(t)+|x-x_\star(t)|}\right).
	\end{align*}
	
	Combining the above estimates yields the result.
\end{proof}

\section{Linearized analysis in parabolic self-similar variables}\label{section:LinAnalysis}

By the soliton resolution provided by Theorem~\ref{thn:soliton-resolution} and the additional estimates for this convergence provided by Proposition \ref{pr:cv-scale-invariant}, we now know that a solution $u$ of the Keller-Segel equation \eqref{KS} with $u_0\in L^1(\langle x \rangle^2 dx)$ and $\int u_0=8\pi$ can be written asymptotically as $t \to \infty$ in the form
\begin{equation} \label{id:linearized-analysis-decomposition}
u(x,t)= U_{\lambda(t),x_{\star}(t)}+\tilde u(x,t)
\end{equation}
where the modulation parameters $\lambda(t)$ and $x_{\star}(t)$ are determined by the orthogonality conditions \eqref{id:property-4} and satisfy $\lambda(t)\to 0$ and $x_{\star}(t)\to x^*$ with
\begin{align} \label{id:property-7}
& |\dot \lambda|+|\dot x_{\star}|=o(\lambda), \\
\label{id:property-8}&  |\lambda(t+R^2)-\lambda(t)|+| x_{\star}(t+R^2)-x_\star(t)|=o(R) \qquad \mbox{for all }R\geq 0,
\end{align}
as $t\to \infty$, where the $o()$ in the second estimate is uniform. The remainder $\tilde u$ satisfies
\begin{align} \label{id:property-9}
	|\tilde u(x,t)|
	=
	o\!\left(\frac{1}{\lambda^{2} + |x-x_\star|^{2}}\right),
	\qquad
	|\nabla \Phi_{\tilde u}(x,t)|
	=
	o\!\left(\frac{1}{\lambda + |x-x_\star|}\right).
\end{align}
and more generally the estimates \eqref{id:property-5} and \eqref{bd:pointwise-tildeu-interior}-\eqref{bd:pointwise-tildeu-exterior}-\eqref{bd:pointwise-Phitildeu}. To refine further the description of the solution, we need to understand the linearized dynamics driving the evolution of $\tilde u$. This is the purpose of this section.

\subsection{The renormalized Keller-Segel equation}

We introduce the self-similar variables
\begin{equation} \label{self-similarKS2}
	u(x,t) = \frac{1}{t} w(z,\tau),
	\qquad
	z = \frac{x-x^*}{\sqrt{t}}.
	\qquad
	\tau = \ln t,
\end{equation}
A direct computation shows that \(w\) solves
\begin{align}\label{self-similarKS}
	\partial_\tau w
	=
	\Delta w - \nabla\cdot(w\nabla \Phi_w)
	+ \frac{1}{2}\Lambda w.
\end{align}
The decomposition \eqref{id:linearized-analysis-decomposition} becomes
\begin{align*}
	w(z,\tau) = U_{\nu,\xi} + \tilde w(z,\tau),
\end{align*}
where
\begin{align*}
	\xi(\tau) = \frac{x_\star(t)-x^*}{\sqrt{t}},
	\qquad
	\nu(\tau) = \frac{\lambda(t)}{\sqrt{t}}, \qquad \tilde w(z,\tau)=t \tilde u(x,t).
\end{align*}
The orthogonality conditions \eqref{id:property-4} to determine \(\lambda\) and \(x_\star\) can be written as
\begin{align}\label{eq:ortoselfsimilar}
	\int_{\mathbb R^2} \tilde w  (\Lambda U)_{\nu,\xi}\,dz = 0,
	\quad \mbox{and} \quad
	\int_{\mathbb R^2} \tilde w  (\partial_{y_i} U)_{\nu,\xi}\,dz = 0 \quad \mbox{ for }i=1,2.
\end{align}
The properties of the parameters and the remainder in self-similar variables are as follows.

\begin{lemma} \label{lem:selfsimilarquali}
	Assume that $\lambda$, $x_{\star}$ and $\tilde u$ satisfy  \eqref{id:property-7}, \eqref{id:property-8} and \eqref{id:property-9}. Then asymptotically as $\tau\to \infty$ we have
	\begin{align*}
		|\tilde w(z,\tau)|
		=
		o\!\left(\frac{1}{\nu^2 + |z-\xi|^2}\right),
		\qquad
		|\nabla \Phi_{\tilde w}(z,\tau)|
		=
		o\!\left(\frac{1}{\nu + |z-\xi|}\right),
	\end{align*}
	as well as
	\begin{align}
	\label{id:qualitative-cv-nu-xi} &|\nu| +|\xi|  = o\!\left(e^{-\tau/2}\right),\\
		\label{id:qualitative-modulation-nu-xi} &|\nu_\tau| +|\xi_\tau|  = o\!\left(\nu^{-1}\right),\\
		\label{id:slow-variation-nu-xi} &|\nu(\tau+R^2) - \nu(\tau)|+|\xi(\tau+R^2) - \xi(\tau)| = o(R),
		\qquad
		\forall R\geq 0.
	\end{align}
\end{lemma}

\begin{proof}
	The pointwise bounds for $\tilde w$ follow directly from rescaling the
	estimates on \(\tilde u\) and \(\nabla\Phi_{\tilde u}\). As $\nu=\lambda e^{-\tau/2}$ and $\xi=(x_{\star}-x^*)e^{-\tau/2}$, the bound \eqref{id:qualitative-cv-nu-xi} follows from the convergences $\lambda (t)\to 0$ and $x_{\star}(t)\to x^*$ as $t\to \infty$. \newline 
	We next consider $\xi$ and compute first its evolution. Since
	\(\partial_\tau t = t\), we obtain by \eqref{id:property-7}
	\begin{align*}
		\partial_\tau \xi
		&=
		t\,\partial_t\!\left(\frac{x_\star-x^*}{\sqrt{t}}\right)
		=
		t\left(
		\frac{\partial_t x_\star}{\sqrt{t}}
		-
		\frac{1}{2}\frac{x_\star-x^*}{t^{3/2}}
		\right) =
		o\!\left(\frac{\sqrt{t}}{\lambda}\right)
		+
		o\!\left(\frac{|x_\star-x^*|}{\sqrt{t}}\right)
		=
		o\!\left(\frac{1}{\nu}\right).
	\end{align*}
	This shows the estimate \eqref{id:qualitative-modulation-nu-xi} for $\xi$. Next, writing $\xi(\tau) = e^{-\tau/2}(x_\star(e^\tau)-x^*)$, we have for $R\geq 1$ that
	$$
	\xi(\tau+R^2)=\xi(\tau)+\frac{x_\star(e^\tau e^{R^2})-x^*}{e^{\tau/2}e^{R^2/2}}-\frac{x_\star(e^\tau )-x^*}{e^{\tau/2}}=\xi(\tau)+o(R)
	$$
	since $x_{\star}-x^*\to 0$. We have for $0\leq R\leq 1$, using $x_{\star}-x^*\to 0$ and then \eqref{id:property-9}, that
	\begin{align*}
		\xi(\tau+R^2)
		&=
		\frac{x_\star(e^\tau e^{R^2})-x^*}{e^{\tau/2}e^{R^2/2}}
		=
		\frac{x_\star(e^\tau e^{R^2})-x^*}{e^{\tau/2}}
		+
		O(R^2)\,e^{-\tau/2}(x_\star(e^\tau e^{R^2})-x^*)\\
		& \quad = \frac{x_\star(e^\tau )-x^*}{e^{\tau/2}}+ \frac{x_\star(e^\tau e^{R^2/2} )-x_\star(e^\tau) }{e^{\tau/2}}
		+o(R) = \xi(\tau)+ \frac{o(e^{\tau/2}R)  }{e^{\tau/2}}+o(R)=\xi(\tau)+o(R).
	\end{align*}
	Combining, this yields the estimate \eqref{id:slow-variation-nu-xi} for $\xi$. Noting that the assumptions and the conclusions of the lemma for $\nu$ are exactly the same as those for $\xi$, the same reasoning shows \eqref{id:qualitative-modulation-nu-xi}-\eqref{id:slow-variation-nu-xi} for $\nu$.
	
\end{proof}

The purpose of this section is to establish coercivity properties for the
linearized evolution. Writing
\[
w = W[\nu,\xi] + \varepsilon
\]
in \eqref{self-similarKS}, we obtain
\begin{align*}
	\partial_\tau \varepsilon
	=
	\Delta \varepsilon
	-
	\nabla\cdot(W\nabla\Phi_\varepsilon)
	-
	\nabla\cdot(\varepsilon\nabla\Phi_W)+ \Lambda \varepsilon
	+
	E[W,\nu,\xi]
	+
	Q[\varepsilon],
\end{align*}
where \(E\) collects the error terms arising from inserting the profile \(W\)
into the equation, and \(Q\) denotes quadratic terms in \(\varepsilon\).
The choice of the profile \(W\) is crucial. While the natural candidate would
be the rescaled soliton \(U_{\nu,\xi}\), this choice is not accurate far away from its center because this function does not have finite second moment. We thus introduce the \emph{matched soliton}
\begin{align}\label{matchedsol}
	\hat U_{\nu,\xi}
	=
	U_{\nu,\xi}
	\Bigl[
	\chi_{\zeta_*}(z)
	+
	\Bigl(1-\chi_{\zeta_*}(z)\Bigr)
	e^{-\frac{|z|^2}{4}}
	\Bigr],
\end{align}
where \(\zeta_*>0\) is a fixed parameter, chosen sufficiently small (independently of the dynamics). This profile matches a soliton in the inner region with a Gaussian at infinity. The exact form of the localization in \eqref{matchedsol} is important: it will permit coercivity estimates in the present section, will generate a suitably small error in the evolution equation and will produce cancellations in energy-type estimates. We denote by
\begin{align}\label{def:linoperator}
	L_{\nu,\xi}[u]
	&=
	\Delta u
	-
	\nabla\cdot(u\nabla\Phi_{\hat{U}_{\nu,\xi}})
	-
	\nabla\cdot(\hat{U}_{\nu,\xi}\nabla\Phi_u)
	+
	\frac{1}{2}\Lambda u
\end{align}
the linearized operator around \(\hat{U}_{\nu,\xi}\). This will be the main object of
study in this section. As $\tau \to \infty$, since $\nu\to 0$ and $\xi \to 0$, we face a singular limit. Four zones arise:
\begin{itemize}
\item The \emph{inner zone} $|z-\xi(\tau)|\lesssim \nu(\tau)$ corresponds to the core of the soliton. Its complement is the \emph{outer zone} $|z-\xi(\tau)|\gg \nu(\tau)$. There, the linearized operator $L_{\nu,\xi}$ is well captured by the \emph{interior operator}
\begin{align}\label{def:intlinoperator}
	L_0[u]
	&=
	\Delta u
	-
	\nabla\cdot(u\nabla U_{\nu,\xi})
	-
	\nabla\cdot(U_{\nu,\xi}\nabla\Phi_u).
\end{align}
\item The \emph{nearby outer zone} $\nu(\tau)\ll |z-\xi(\tau)|\ll 1$ corresponds to the tail of the soliton. There, the leading order part of $L_{\nu,\xi}$ is the \emph{linearization around a Dirac mass centered at \(\xi\)} 
\begin{align}\label{def:shiftedextlinoperator}
	L_{\infty,\xi}[u]
	&=
	\Delta u
	+
	4\frac{z-\xi}{|z-\xi|^2}\cdot\nabla u
\end{align}
\item The \emph{parabolic outer zone} $|z|\approx 1$ is the parabolic region surrounding the soliton, where the effect of the soliton ressembles that of a Dirac mass at the origin, corresponding to the \emph{limiting exterior operator}
\begin{align}\label{def:extlinoperator}
	L_\infty[u]
	&=
	\Delta u
	+
	4\frac{z}{|z|^2}\cdot\nabla u
	+
	\frac{1}{2}\Lambda u.
\end{align}
\item The \emph{far away outer zone} $|z|\gg 1$ is where the limiting operator remains $L_\infty$, but the parabolic part $\Delta u+4|z|^{-2}z\cdot\nabla u$ is of lower order compared to the scaling operator $\Lambda$.
\end{itemize}

Observe that \(L_{\infty,\xi}\to L_\infty\) as \(\xi\to 0\), but the shifted
operator captures the correct position in the nearby outer region. In particular, the usual dichotomy between interior and exterior
regimes is not sufficient in the present setting involving unprepared initial data. Indeed, so far we know very little about the modulation parameters $\xi$ and $\nu$; whether the soliton is close to the origin $|\xi|\lesssim \nu$ is unknown, and the estimates \eqref{id:qualitative-modulation-nu-xi} do not rule out arbitrarily fast oscillations $|\nu_\tau|,|\xi_\tau|\to \infty$! In the analysis of well-prepared initial data, usually the behaviour of modulation parameters is far more under control. Note that one may think that an alternative approach to the present one would consist in first refining first even more the control of the modulation parameters. However, their behaviour and that of the remainder is $\varepsilon$ are intimately linked, so that a more refined description for $\xi$, $\nu$ and $\varepsilon$ can only be obtained altogether.

One main novelty of the present work is thus a framework for obtaining decay estimates for $L_{\nu,\xi}$, the linearization around a soliton with very little control over its modulation parameters. A first new idea is for the nearby outer region, where we will introduce a suitable framework based on \emph{parabolic averaging} that will enable a dissipative estimate. A second one is the use of a \emph{matched scalar product} to produce a global Lyapunov functional without boundary terms between the inner zone, the nearby outer zone and the parabolic outer zone. A third one is a \emph{gluing-type decomposition} between the parabolic zone $|z|\lesssim 1$ and the far away outer zone $|z|\gg 1$ that will be explained in the next section.

\subsection{The parabolic outer zone} \label{subsection:parabolic-outer}

 The properties of the Fokker-Planck type operator $L_\infty$ are well-known. We first prove first a dissipative estimate with precise spectral gap for the operator for functions with Gaussian decay. The choice of the weight reflects the structure
of the exterior operator, which, in radial coordinates, can be viewed as a
six-dimensional Laplacian. Set
\begin{equation} \label{def:omega-infty}
\omega_\infty(z):=|z|^{4}e^{|z|^{2}/4} \quad \mbox{and} \quad L^2_\infty(\mathbb R^2)=L^2_{\omega_\infty}(\mathbb R^2).
\end{equation}
We introduce the scalar product
\[
\langle u,v\rangle_{\infty}= \int_{\mathbb R^2} u(z)v(z)\omega_\infty(z)dz.
\]
Note that it is the scalar product on $L^2_{\infty}$. By a direct inspection, the operator $L_\infty $ is self-adjoint on $L^2_\infty$. Note that this weighted space corresponds to functions with Gaussian localization, which we interpret as being localized in the parabolic region $|z|\lesssim 1$ but not in the far away outer region $|z|\gg 1$.

\begin{proposition}[Coercivity of the exterior operator]\label{prop:ExteriorCoercivity}
	For any \(\delta>0\) sufficiently small, for all \(u \in C^\infty_c (\R^2)\),
	\[
	\langle L_\infty u,u\rangle_{\infty}
	\le
	(-2+3\delta)\|u\|_{L^2_\infty}^2
	-
	\delta\|\nabla u\|_{L^2_\infty}^2.
	\]
\end{proposition}

\begin{proof}
	We first perform an exponential conjugation. Set $u(z)=e^{-|z|^2/4}v(z)$. A direct computation gives
	\begin{align*}
	& \nabla u
	=
	e^{-|z|^2/4}
	\left(\nabla v-\frac{z}{2}v\right), \\
	& \Delta u
	=
	e^{-|z|^2/4}
	\left(
	\Delta v
	-
	z\cdot\nabla v
	+
	\left(\frac{|z|^2}{4}-1\right)v
	\right).
	\end{align*}
	It follows that
	\begin{align*}
	L_\infty u
	& =
	e^{-|z|^2/4}
	\left(
	\Delta v
	+
	4\frac{z}{|z|^2}\cdot\nabla v
	-
	\frac{z}{2}\cdot\nabla v
	-
	2v
	\right) = e^{-|z|^2/4}[\mu^{-1}\nabla\cdot(\mu \nabla v)-2 v]
	\end{align*}
	where $\mu(z)=|z|^4 e^{-|z|^2/4}$. Therefore, by integration by parts, we get
	\begin{align}\label{eq:coerc1}
		\langle L_\infty u,u\rangle_{L^2_\infty}
		=
		-2\int_{\R^2} v^2\,\mu(z)dz
		-
		\int_{\R^2} |\nabla v|^2\,\mu (z)dz
		\le
		-2\|u\|_{L^2_\infty}^2.
	\end{align}
	
	On the other hand, we observe that
	\[
	L_\infty u
	=
	\omega_\infty^{-1}\nabla\cdot(\omega_\infty \nabla u) + u.
	\]
	Hence, by integration by parts,
	\begin{align}\label{eq:coerc2}
		\langle L_\infty u,u\rangle_{L^2_\infty}
		=
		\|u\|_{L^2_\infty}^2
		-
		\|\nabla u\|_{L^2_\infty}^2.
	\end{align}
	
	Combining \eqref{eq:coerc1} and \eqref{eq:coerc2}, for \(0<\delta\ll1\),
	\[
	\langle L_\infty u,u\rangle_{L^2_\infty}
	=
	(1-\delta)\langle L_\infty u,u\rangle
	+
	\delta\langle L_\infty u,u\rangle\le
	-2(1-\delta)\|u\|_{L^2_\infty}^2
	+
	\delta\|u\|_{L^2_\infty}^2
	-
	\delta\|\nabla u\|_{L^2_\infty}^2,
	\]
	which gives the desired estimate.
\end{proof}

We next record a Hardy-type inequality associated to the weight $\omega_\infty$.

\begin{lemma}[Exterior Hardy inequality]\label{lem:exterior_hardy_weighted}
	We have for all $u\in C^\infty_c(\mathbb R^2)$,
	\begin{align*}
		\int_{\R^{2}}\frac{u^{2}(z)}{|z|^{2}}\omega_\infty(z)\,dz
		\lesssim\int_{\R^{2}}|\nabla u(z)|^{2}\omega_\infty(z)\,dz.
	\end{align*}
\end{lemma}

\begin{proof}
	We use the vector field $F(z):=|z|^{2}e^{|z|^{2}/4}z$. A direct computation gives
	\begin{align*}
		\nabla\cdot F
		&=
		4|z|^{2}e^{|z|^{2}/4}
		+
		\frac{|z|^{4}}{2}e^{|z|^{2}/4}
		\ge
		4|z|^{2}e^{|z|^{2}/4}.
	\end{align*}
	Integrating by parts, we obtain
	\begin{align*}
		\int_{\R^{2}}u^{2}|z|^{2}e^{|z|^{2}/4}\,dz
		&\le
		\frac14\int_{\R^{2}}(\nabla\cdot F)u^{2}\,dz \le
		\frac12\int_{\R^{2}}|u|\,|\nabla u|\,|F|\,dz.
	\end{align*}
	Since \(|F|=|z|^{3}e^{|z|^{2}/4}\), Cauchy--Schwarz yields
	\begin{align*}
		\int_{\R^{2}}u^{2}|z|^{2}e^{|z|^{2}/4}\,dz
		&\lesssim
		\left(
		\int_{\R^{2}}u^{2}|z|^{2}e^{|z|^{2}/4}\,dz
		\right)^{1/2}
		\left(
		\int_{\R^{2}}|\nabla u|^{2}|z|^{4}e^{|z|^{2}/4}\,dz
		\right)^{1/2}.
	\end{align*}
	Cancelling the common factor gives the claim.
\end{proof}

\subsection{The far away outer zone}\label{subsection:far-outer}

The previous weighted space $L^2_\infty$ is not enough to study functions with only finite second moment, as they may be localized in the far away outer zone $|z|\gg 1$. There, we shall use weighted $L^1$ semigroup estimates.

\begin{lemma} \label{lem:far-away-outer}

Denote by $e^{\tau L_\infty}u_0$ the solution to $\partial_\tau u=L_\infty u$ on the set $\{|z|>1\}$ with homogeneous Dirichlet boundary condition on $\{|z|=1\}$ and initial data with $u(0)=u_0\in L^1(|z|^2dz)$. Then
$$
\int_{|z|>1} (e^{\tau L_\infty}u_0)|z|^2 dz \leq e^{-\tau} \int_{|z|>1} u_0|z|^2 dz \quad \mbox{and} \quad \int_{|z|>1} (e^{\tau L_\infty}u_0)|z|^2 dz=o_{u_0}(e^{-\tau})\mbox{ as }\tau \to \infty.
$$

\end{lemma}

\begin{remark}

The above Proposition is stated here for clarity. It is a particular case of Proposition \ref{prop:ExteriorEvolutionContraction} that we shall use instead for the full nonlinear evolution in the next section.

\end{remark}

\begin{proof}

The first inequality is a particular case of Proposition \ref{prop:ExteriorEvolutionContraction}. To obtain the second, we decompose $u_0=u_{0,1}+u_{0,2}$ where $u_{0,1}=u_0\chi_M$ with $M>1$ and let $u_i=e^{\tau L_\infty}u_{0,i}$ for $i=1,2$. The first inequality implies $\int_{|z|>1} u_2 |z|^2 dz\leq C(M)e^{-\tau}$ where $C(M)\to 0$ as $M\to \infty$. Due to the homogeneous Dirichlet boundary condition, applying Proposition \ref{prop:ExteriorCoercivity} thanks to a standard density argument shows $\partial_\tau (\int_{|z|>1} u_1^2\omega_\infty dz)< -3 \int_{|z|>1} u_1^2\omega_\infty dz$ provided $\delta<1/6$. Hence $ \int_{|z|>1} u_1^2\omega_\infty dz\lesssim_M e^{-3\tau}$ by Gronwall, which implies $ \int_{|z|>1} u_1|z|^2 dz\lesssim_M e^{-3\tau/2}$ thanks to Cauchy-Schwarz. Since $M$ is arbitrary and $3/2>1$, combining the two inequalities for $u_1$ and $u_2$ implies the second inequality.
	
\end{proof}

\subsection{The nearby outer zone} \label{subsection:nearby-outer}

We consider the following formal linear evolution,
\begin{equation} \label{id:formal-evolution-nearby-outer}
\partial_\tau u =L_{\infty,\xi} u=\Delta u+4\frac{z-\xi(\tau)}{|z-\xi(\tau)|^2}\cdot\nabla u.
\end{equation}
Knowing by Proposition \ref{prop:ExteriorCoercivity} the coercivity of $L_\infty$ for functions with Gaussian decay, we want to find an analogue dissipative estimate for \eqref{id:formal-evolution-nearby-outer}.

\subsubsection{Naive choices of scalar products}

For a weight $\gamma$ to be determined, we have formally if $u$ satisfies \eqref{id:formal-evolution-nearby-outer} and is localized in the zone $|z-\xi|\gtrsim \nu$ that
\begin{equation} \label{id:formal-evolution-nearby-outer-adjoint}
\partial_\tau \left(\frac 12 \int  u^2 \gamma dz \right)= - \int |\nabla u|^2\gamma dz+\frac 12 \int u^2(\partial_{\tau}\gamma -L_{\infty,\xi}^*\gamma)dz,
\end{equation}
where the formal adjoint operator is $L_{\infty,\xi}^* =-\Delta+4|z-\xi|^{-2}(z-\xi)\cdot\nabla $.

A first attempt is to consider whether the choice $\gamma_\infty=|z|^4$ makes a suitably small error. Since $L_{\infty,0}^*\gamma_\infty=0$, such error is $L_{\infty,\xi}^*\gamma_\infty=16|z|^{2}(|z-\xi|^{-2}(z-\xi)-|z|^{-2}z)\cdot z$. This cannot be absorbed by the first term in \eqref{id:formal-evolution-nearby-outer-adjoint} thanks to a Hardy-type inequality (note that by a more refined inspection a leading cancellation between the two terms could be used, but only if we knew $|\xi|\lesssim \nu$ which is unknown so far). A second attempt is to consider the choice $\gamma_{\xi}(z)= |z-\xi|^4$ that satisfies $L_{\infty,\xi}^*\gamma_\xi=0$. In that case the error is $\partial_{\tau}\gamma_{\xi}=-4\partial_\tau \xi \cdot (z-\xi)|z-\xi|^2$. As we only currently have $\partial_\tau \xi=o(\nu^{-1})$ which may diverge as $\tau \to \infty$, this error also cannot be absorbed. 

A more abstract attempt would then be to find a solution $\gamma$ of the adjoint evolution $\partial_{\tau}\gamma -L_{\infty,\xi}^*\gamma =0$. However, this reverse parabolic equation would cause technical problems; we will rather solve this equation approximately. Noticing that it is a parabolic equation, that the distance to the soliton should play a role, and that the possibly fast oscillations of its center $\xi$ should be eliminated by an averaging of some kind, we are lead to the \emph{parabolically averaged distance to the soliton} developed below.

\subsubsection{The parabolic averaging based scalar product}

We begin by fixing a smooth cut-off function \(\bar\chi\).
Let \(\bar\chi \in C_c^\infty(\mathbb R)\) be nonnegative, supported in $(1/3,2/3)$
and normalized so that
\[
\int_{\R} \bar\chi(s)\,ds = 1.
\]
For any $\mu \in \mathbb R$ with $|\mu|$ small enough, its smoothness implies that \(\bar\chi\) satisfies
\begin{align}
	|\bar\chi(t) - \bar\chi(t(1+\mu))|
	&= O(\mu\,\mathbf{1}_{\{|t|\le 1\}}), \label{barChi_Prop1}\\
	|\bar\chi'(t) - \bar\chi'(t(1+\mu))|
	&= O(\mu\,\mathbf{1}_{\{|t|\le 1\}}). \label{barChi_Prop2}
\end{align}

For $\tau\geq 0$ and $z\in \mathbb R^2$, the soliton is at distance $|z-\xi(\tau)|$. We will average this distance over the parabolic timescale $|z-\xi(\tau)|^2$.

\begin{definition}[Parabolically averaged distance to the soliton]\label{def:ParabolicDistance}
	For \((z,\tau)\in \R^2\times\R\), define
	\[
	I(\tau,z)
	:=
	[\tau - |z-\xi(\tau)|^2,\ \tau],
	\]
	and set
	\begin{align*}
		d (z,\xi,\tau)
		:=
		\int_{I(\tau,z)}
		|z-\xi(\tau')|\,
		\frac{1}{ |z-\xi(\tau')|^2}\,
		\bar\chi\!\left(
		\frac{\tau-\tau'}{ |z-\xi(\tau')|^2}
		\right)\,d\tau'.
	\end{align*}
\end{definition}

By Lemma~\ref{lem:selfsimilarquali}, if $|z|\lesssim 1$, then for every
\(\tau' \in I(\tau,z)\) we have $\xi(\tau')=\xi(\tau)+o\bigl(|z-\xi(\tau)|\bigr)$. Consequently
\begin{equation} \label{id:parabolic-stationarity-distance}
|z-\xi(\tau')|
=
|z-\xi(\tau)|\,(1+o(1)) \qquad \mbox{for all }\tau' \in I(\tau,z).
\end{equation}
Above, the $o()$ is as $\tau \to \infty$, and is uniform in $|z|\lesssim 1$ and $\tau' \in I(\tau,z)$. In particular, the distance is essentially constant along \(I(\tau,z)\) at
leading order.

\begin{lemma}[Basic properties of the parabolically averaged distance]\label{Lemma:dProp}
	One has the following estimates as $\tau \to \infty$, for all \(0< |z-\xi(\tau)|\le1\):
	\begin{align}
		d (z,\xi,\tau)
		&=
		|z-\xi(\tau)|(1+o(1)), \label{Lemma:expd}\\
		\partial_{\tau}d (z,\xi,\tau)
		&=
		o\!\left(\frac{1}{|z-\xi(\tau)|}\right), \label{Lemma:expdtd}\\
		\nabla_{z}d(z,\xi,\tau)
		&=
		\frac{z-\xi(\tau)}{|z-\xi(\tau)|}(1+o(1)) \quad \mbox{and} \quad \Delta_z d(z,\xi,\tau)=\frac{1}{|z-\xi(\tau)|}(1+o(1)). \label{Lemma:expdzd}
	\end{align}
\end{lemma}

\begin{proof}
	We start with \eqref{Lemma:expd}. Applying successively \eqref{id:parabolic-stationarity-distance}, then \eqref{barChi_Prop1} and finally changing variables using that $\int \bar \chi=1$, we find
	\begin{align}
		\label{id:averaged-distance-1}d(z,\xi,\tau)
		&=
		\int_{I(\tau,z)}
		|z-\xi(\tau')|
		\frac{1}{ |z-\xi(\tau')|^{2}}
		\bar{\chi}\!\left(
		\frac{\tau-\tau'}{ |z-\xi(\tau')|^{2}}
		\right)d\tau' \\
		\nonumber & \quad =
		\int_{I(\tau,z)}
		\frac{1}{ |z-\xi(\tau)|}
		\bar{\chi}\!\left(
		\frac{\tau-\tau'}{ |z-\xi(\tau)|^{2}}(1+o(1))
		\right)d\tau'
		+
		o(|z-\xi(\tau)|) \\
		\nonumber &\qquad =
		\int_{\tau-|z-\xi(\tau)|^{2}}^{\tau}
		\frac{1}{ |z-\xi(\tau)|}
		\bar{\chi}\!\left(
		\frac{\tau-\tau'}{ |z-\xi(\tau)|^{2}}
		\right)d\tau'
		+
		o(|z-\xi(\tau)|) = |z-\xi(\tau)|+o(|z-\xi(\tau)|).
	\end{align}
	This is \eqref{Lemma:expd}. We now prove \eqref{Lemma:expdtd}. Since $\bar \chi$ is supported in $(1/3,2/3)$, by \eqref{id:parabolic-stationarity-distance} the integrand vanishes at the endpoints of \(I(\tau,z)\). Therefore no boundary terms appear and
	\begin{align*}
		\partial_{\tau}d(z,\xi,\tau)
		&=
		\int_{I(\tau,z)}
		|z-\xi(\tau')|
		\frac{1}{|z-\xi(\tau')|^{4}}
		\bar{\chi}'\!\left(
		\frac{\tau-\tau'}{ |z-\xi(\tau')|^{2}}
		\right)d\tau'.
	\end{align*}
	Proceeding as above we obtain
	\begin{align} \label{id:averaged-distance-2}
		\partial_{\tau}d(z,\xi,\tau)
		&=
		\int_{\tau-|z-\xi(\tau)|^{2}}^{\tau}
		\frac{1}{|z-\xi(\tau)|^{3}}
		\bar{\chi}'\!\left(
		\frac{\tau-\tau'}{ |z-\xi(\tau)|^{2}}
		\right)d\tau'
		+
		o\!\left(\frac{1}{|z-\xi(\tau)|}\right).
	\end{align}
	Since $\int_0^{1} \bar\chi'(s)\,ds = 0$, the first term vanishes so that $\partial_{\tau}d(z,\xi,\tau)=o(|z-\xi(\tau)|^{-1})$. This is \eqref{Lemma:expdtd}. Finally, we prove \eqref{Lemma:expdzd}. As before, no boundary terms appear, and we compute
	\begin{align*}
		\nabla_{z}d(z,\xi,\tau)
		&=
		\int_{I(\tau,z)}
		\nabla_{z}(|z-\xi(\tau')|)
		\frac{1}{ |z-\xi(\tau')|^{2}}
		\bar{\chi}\!\left(
		\frac{\tau-\tau'}{ |z-\xi(\tau')|^{2}}
		\right)d\tau' \\
		&\quad
		+
		\int_{I(\tau,z)}
		|z-\xi(\tau')|
		\nabla_{z}\!\left(
		\frac{1}{ |z-\xi(\tau')|^{2}}
		\bar{\chi}\!\left(
		\frac{\tau-\tau'}{ |z-\xi(\tau')|^{2}}
		\right)
		\right)d\tau' \\
		& = \frac{z-\xi(\tau)}{|z-\xi(\tau)|}(1+o(1))+\frac{z-\xi(\tau)}{|z-\xi(\tau)|^{3}}
	\int_0^{1}
	\big(\bar\chi(s)+2s\bar\chi'(s)\big)\,ds+o(1)= \frac{z-\xi(\tau)}{|z-\xi(\tau)|}+o(1),
	\end{align*}
	where the first term was estimated similarly to \eqref{id:averaged-distance-1} and the second similarly to \eqref{id:averaged-distance-2}, and where we used the cancellation $\int_0^{1}(\bar\chi(s)+2s\bar\chi'(s))ds=0$ for the last equality. The function $\Delta_{z}d(z,\xi,\tau)$ is given analogously by
	\begin{align*}
		&\int_{I}
		\Delta_{z}(|z-\xi(\tau')|)
		\frac{1}{ |z-\xi(\tau')|^{2}}
		\bar{\chi}\!\left(
		\frac{\tau-\tau'}{ |z-\xi(\tau')|^{2}}
		\right)+
		2\int_{I}
		\nabla_z |z-\xi(\tau')|\cdot
		\nabla_{z}\!\left(
		\frac{1}{ |z-\xi(\tau')|^{2}}
		\bar{\chi}\!\left(
		\frac{\tau-\tau'}{ |z-\xi(\tau')|^{2}}
		\right)
		\right) \\
		&+\int_{I}
		|z-\xi(\tau')|\Delta_{z}\left(
		\frac{1}{ |z-\xi(\tau')|^{2}}
		\bar{\chi}\!\left(
		\frac{\tau-\tau'}{ |z-\xi(\tau')|^{2}}
		\right)\right)
		 =  \frac{1}{|z-\xi(\tau)|}(1+o(1))+o(\frac{1}{|z-\xi(\tau)|})+o(\frac{1}{|z-\xi(\tau)|})
	\end{align*}
	where the first term is estimated similarly to \eqref{id:averaged-distance-1}, the second similarly to \eqref{id:averaged-distance-2} and for the third one we used the cancellation $\int_0^{1}(\bar\chi(s)+3s\bar\chi'(s)+s^2\bar \chi''(s))ds=0$. This concludes the proof of \eqref{Lemma:expdzd}.
	
\end{proof}

We introduce the following weight and associated scalar product in the nearby zone
$$
\langle u,v\rangle_{\infty,\xi}= \int \gamma_{\infty,\xi} uvdz \quad \mbox{with} \quad \gamma_{\infty,\xi}(z)=d^4(z,\xi,\tau) .
$$
It solves approximately the adjoint equation appearing in the error term in the energy estimate \eqref{id:formal-evolution-nearby-outer-adjoint}.

\begin{lemma}

There holds $\partial_{\tau}\gamma_{\infty,\xi} -L_{\infty,\xi}^*\gamma_{\infty,\xi}=o(|z-\xi|^{-2}\gamma_{\infty,\xi})$ as $\tau \to \infty$ for all \(0< |z-\xi(\tau)|\le1\).

\end{lemma}

\begin{remark}

The above Lemma is stated here for clarity. In the next section we shall rely on a different but intimately related leading order cancellation.

\end{remark}

\begin{proof}

We have $\partial_{\tau}\gamma_{\infty,\xi}=o(|z-\xi|^{-2}\gamma_{\infty,\xi})$ by \eqref{Lemma:expdtd}. We compute using \eqref{Lemma:expd} and \eqref{Lemma:expdzd}:
\begin{align*}
&\nabla \gamma_{\infty,\xi}= 4d^3\nabla d=[4|z-\xi|^2(z-\xi)+o(d^3)],\\
& \Delta \gamma_{\infty,\xi}=4d^3\Delta d+12 d^3|\nabla d|^2= 16|z-\xi|^2+o(d^2).
\end{align*}
Using these two identities and then \eqref{Lemma:expd} we infer that $L_{\infty,\xi}^* \gamma_{\infty,\xi}$ equals
\begin{align*}
&-[16|z-\xi|^2+o(d^2)] +4\frac{z-\xi}{|z-\xi|^{2}}\cdot [4|z-\xi|^2(z-\xi)+o(d^3)]+o(d^2)=o(|z-\xi|^{-2})\gamma_{\infty,\xi}.
\end{align*}
\end{proof}

\subsection{The inner zone} \label{subsection:inner}

Recall the inner variables \eqref{id:inner-variables}. For simplicity, we keep the notation \(L\) for the operator obtained from
\(L_0\), defined in \eqref{def:intlinoperator}, in the variable \(y\). Thus
\[
L[u]
=
\Delta u
-
\nabla\cdot(U\nabla\Phi_u)
-
\nabla\cdot(u\nabla\Phi_U).
\qquad
\]
We recall below results proved by Rapha\"el and Schweyer in \cite{RS} and extended in \cite{CGMN2}, intimately linked to the linearization of the free energy around its minimizer $U$. Since $\nabla\Phi_U=U^{-1}\nabla U$ we can write
\begin{align} \label{id:formula-LU}
	L[u]
	&= \nabla\cdot\Bigl(U\nabla \mathcal M[u] \Bigr) \quad \mbox{where} \quad \mathcal M[u]:=\frac{u}{U}-\Phi_u.
\end{align}

\begin{proposition}[See Lemma 2.1 and Proposition 2.3 in \cite{RS}]\label{prop:RS}
	The following algebraic identities hold:
	\[
	\mathcal M[\Lambda U]=-2,
	\qquad
	\mathcal M[\nabla U]=0.
	\]
	Moreover, for any \(u\in\mathcal S(\R^2)\), there exist constants
	\(\delta,C>0\) such that
	\begin{align}\label{prop:coercivityproduct}
		\langle u,\mathcal M[u]\rangle_{L^2}
		\ge
		\delta\|u\|_{L^2_{1/U}}^2
		-
		C\left[
		\langle u,1\rangle_{L^2}^2
		+
		\langle u,\Lambda U\rangle_{L^2}^2
		+
		\sum_{i=1}^{2}
		\langle u,\partial_i U\rangle_{L^2}^2
		\right].
	\end{align}
\end{proposition}

\begin{proposition}[See Lemma 3.5 in \cite{CGMN2}]\label{prop:coercivityoperator}
	For all \(u\in \mathcal S(\R^2)\), there exist universal constants
	\(\delta,C>0\) such that
	\[
	\langle L[u],\mathcal M[u]\rangle_{L^2}=\int_{\R^2} U|\nabla\mathcal M[u]|^2\,dy
	\ge
	\delta\|\nabla u\|_{L^2_{1/U}}^2
	-
	C\left[
	\langle u,\Lambda U\rangle_{L^2}^2
	+
	\sum_{i=1}^{2}
	\langle u,\partial_i U\rangle_{L^2}^2
	\right].
	\]
\end{proposition}

We now introduce a quadratic form, obtained from $\langle f,\mathcal M[g]\rangle_{L^2}$ by localizing the contribution of the Poisson fields. Let \(R>0\) be a large parameter
to be fixed later, and define
\begin{align}\label{innerquadratic}
	\langle f,g\rangle_0
	:=
	\int_{\R^2} fg\,\frac{1}{U}\,dy
	-
	\int_{\R^2}
	\nabla\Phi_f\cdot\nabla\Phi_g\,
	\chi\!\left(\frac{y}{R}\right)\,dy.
\end{align}

\begin{proposition}\label{prop:innercoerproduct}
	For any \(u\in\mathcal S(\R^2)\), there exist constants \(\delta,C>0\)
	such that
	\[
	\langle u,u\rangle_0
	\ge
	\delta\|u\|_{L^2_{1/U}}^2
	-
	C\left[
	\langle u,1\rangle_{L^2}^2
	+
	\langle u,\Lambda U\rangle_{L^2}^2
	+
	\sum_{i=1}^{2}
	\langle u,\partial_i U\rangle_{L^2}^2
	\right].
	\]
\end{proposition}

\begin{proof}
	It follows directly from Proposition~\ref{prop:RS} since
	\[
	\langle u,u\rangle_0
	=
	\int_{\R^2}\frac{u^2}{U}\,dy
	-
	\int_{\R^2}|\nabla\Phi_u|^2
	\chi\!\left(\frac{y}{R}\right)\,dy \ge
	\int_{\R^2}\frac{u^2}{U}\,dy
	-
	\int_{\R^2}|\nabla\Phi_u|^2\,dy=\langle u ,\mathcal M u  \rangle_{L^2}.
	\]
\end{proof}

We now recall some inequalities that will be used in the sequel.

\begin{lemma}[Interior Hardy-type inequality, Lemma C.1 in \cite{CGMN2}]\label{lemma:hardy}
	There exists a constant \(C>0\) such that for all \(u\in \mathcal{S}(\R^{2})\),
	\begin{align*}
		\int_{\R^{2}}\frac{u^{2}}{1+|y|^{2}}\frac{1}{U}\,dy
		\le
		C \int_{\R^{2}}|\nabla u|^{2}\frac{1}{U}\,dy.
	\end{align*}
\end{lemma}

The coercivity of Proposition \ref{prop:coercivityoperator} remains true for the localized quadratic form $\langle \cdot,\cdot \rangle_0$:

\begin{proposition}\label{InnCoercivity1}
	There exist constants \(\delta,C>0\) such that for any sufficiently large \(R>0\), for all \(u\in \mathcal{S}(\R^{2})\),
	\begin{align*}
		\langle L[u],u\rangle_{0}
		\le
		-\delta \int_{\R^{2}} \frac{|\nabla u|^{2}}{U}
		+
		C\left[
		\langle u,\Lambda U\rangle_{L^{2}}^{2}
		+
		\sum_{i=1}^{2}\langle u,\partial_{i}U\rangle_{L^{2}}^{2}
		\right].
	\end{align*}
\end{proposition}

\begin{proof}
	We first isolate the error introduced by the localization in the quadratic
	form. By \eqref{id:formula-LU} we have
	\begin{align*}
		\langle L[u],u\rangle_{0}
		&=
		\int_{\R^2} L[u]\frac{u}{U}
		-
		\int_{\R^2}\nabla\Phi_{L[u]}\cdot\nabla\Phi_u\,
		\chi\!\left(\frac{y}{R}\right)  \\
		&=
		-\int_{\R^2}U|\nabla\mathcal M[u]|^2
		-
		\int_{\R^2}\nabla\Phi_{L[u]}\cdot\nabla\Phi_u
		\left(\chi\!\left(\frac{y}{R}\right)-1\right).
	\end{align*}
	Therefore, by Proposition~\ref{prop:coercivityoperator}, it is enough to
	prove that
	\begin{align}\label{eq:inner_tail_error}
		\left|
		\int_{\R^2}\nabla\Phi_{L[u]}\cdot\nabla\Phi_u
		\left(\chi\!\left(\frac{y}{R}\right)-1\right)
		\right|
		=
		o_{R\to\infty}(1)
		\int_{\R^2}\frac{|\nabla u|^2}{U}.
	\end{align}
	
	Fix \(0<\varepsilon<1\). By Cauchy--Schwarz,
	\begin{align*}
		\left|
		\int \nabla\Phi_{L[u]}\cdot\nabla\Phi_u
		\left(\chi\!\left(\frac{y}{R}\right)-1\right)
		\right|
		&\le
		\left(
		\int |\nabla\Phi_{L[u]}|^2
		\left(1-\chi\!\left(\frac{y}{R}\right)\right)
		\langle y\rangle^\varepsilon
		\right)^{1/2} \left(
		\int |\nabla\Phi_u|^2\langle y\rangle^{-\varepsilon}
		\right)^{1/2}.
	\end{align*}
	Applying Lemma~\ref{lem:weightL2pote} to the second factor, together with the interior Hardy-type inequality (Lemma~\ref{lemma:hardy}), gives
	\begin{align*}
		\left(
		\int |\nabla\Phi_u|^2\langle y\rangle^{-\varepsilon}
		\right)^{1/2}
		\lesssim
		\left(
		\int \frac{|\nabla u|^2}{U}
		\right)^{1/2}.
	\end{align*}
	It remains to estimate the localized norm of \(\nabla\Phi_{L[u]}\).
	Using the equation defining \(L\), we decompose
	\[
	\Phi_{L[u]}
	=
	u
	-
	\Phi_{\nabla\cdot(U\nabla\Phi_u)}
	-
	\Phi_{\nabla\cdot(u\nabla\Phi_U)}.
	\]
	The contribution of the first term satisfies
	\begin{align*}
		\int |\nabla u|^2
		\left(1-\chi\!\left(\frac{y}{R}\right)\right)
		\langle y\rangle^\varepsilon
		\lesssim
		o_{R\to\infty}(1)
		\int \frac{|\nabla u|^2}{U},
	\end{align*}
	since \(U^{-1}\sim \langle y\rangle^4\) at infinity. We now estimate the two Poisson terms. First, $\int \nabla\cdot(U\nabla\Phi_u)=0$ and $\int \nabla\cdot(u\nabla\Phi_U)=0$, so that the zero-mass part of Lemma~\ref{lem:weightL2pote} applies. For every \(2<\alpha<4\),
	\begin{align}\label{Prop:coercivityClaim1}
		\int
		|\nabla\cdot(u\nabla\Phi_U)|^2
		\langle y\rangle^\alpha\,dy
		\lesssim \int u^2 U^2\langle y\rangle^\alpha\,dy
		+
		\int |\nabla u|^2\langle y\rangle^{\alpha-2}\,dy \lesssim
		\int \frac{|\nabla u|^2}{U},
	\end{align}
	where for the last inequality we used that both terms are controlled by Lemma~\ref{lemma:hardy}. Similarly, for every \(2<\alpha<4\),
	\begin{align}\label{Prop:coercivityClaim2}
		\int
		|\nabla\cdot(U\nabla\Phi_u)|^2
		\langle y\rangle^\alpha\,dy
		\lesssim \int |\nabla U|^2|\nabla\Phi_u|^2
		\langle y\rangle^\alpha\,dy
		+
		\int u^2U^2\langle y\rangle^\alpha\,dy \lesssim
		\int \frac{|\nabla u|^2}{U}.
	\end{align}
	where for the last inequality the second term is handled as above, and for the first one, we use
	Lemma~\ref{lem:weightL2pote} with
	\(\omega(y)=|\nabla U(y)|^2\langle y\rangle^\alpha\),
	and then apply the interior Hardy-type inequality (Lemma~\ref{lemma:hardy}). \newline
	We now choose, for instance, \(\varepsilon=\frac12\) and \(\alpha=3\).
	For
	\[
	\omega_R(y)
	=
	\left(1-\chi\!\left(\frac{y}{R}\right)\right)
	\langle y\rangle^\varepsilon,
	\]
	the constant \(K_{\omega_R,\alpha}\) satisfies
	\[
	K_{\omega_R,\alpha}
	\lesssim
	\int_R^\infty \rho^{1+\varepsilon-\alpha}\,d\rho
	=
	o_{R\to\infty}(1),
	\]
	since \(1+\varepsilon-\alpha<-1\). Applying Lemma~\ref{lem:weightL2pote}, together with
	\eqref{Prop:coercivityClaim1}--\eqref{Prop:coercivityClaim2}, we obtain
	\[
	\int |\nabla\Phi_{L[u]}|^2
	\left(1-\chi\!\left(\frac{y}{R}\right)\right)
	\langle y\rangle^\varepsilon
	\le
	o_{R\to\infty}(1)
	\int \frac{|\nabla u|^2}{U}.
	\]
	This proves \eqref{eq:inner_tail_error} and finishes the proof of the Proposition.

\end{proof}

\subsection{The matched scalar product} \label{subsection:matched-sp}

We found that $L_{\infty}$, $L_{\infty,\xi}$ and $L_{0}$, the leading order expressions of the linearized operator $L_{\nu,\xi}$ in the parabolic outer zone $|z|\approx 1$, the nearby outer zone $\nu\ll |z-\xi|\ll 1$ and the inner zone $|z-\xi|\lesssim \nu$ respectively, were associated to the scalar products $\langle \cdot,\cdot \rangle_\infty$, $\langle \cdot,\cdot \rangle_{\infty,\xi}$ and $\langle \cdot ,\cdot \rangle_0$ for which they satisfy a symmetric-type formula and a dissipation estimate. Moreover, one notices that these scalar products match in the transition zones $\{|z-\xi|\sim R\nu, \ R\gg1\}$ and $\{|z|\sim \zeta_*, \ \zeta_*\ll1\}$. In this section we construct a \emph{matched scalar product} that is a global quadratic form that matches these scalar products together.

\subsubsection{Construction of the matched scalar product}

We first introduce the matched parabolically averaged distance, which is obtained by matching the parabolically averaged distance with the true distance to the center of the soliton in the inner zone $|z-\xi|\leq \nu$.

\begin{definition}[Matched parabolically averaged distance]\label{def:adapted_distance}
	For \(|z-\xi(\tau)|\le1\), we define
	\begin{align}\label{def_Dist}
		\mathfrak d(z,\xi,\tau)
		:=
		|z-\xi(\tau)|\,\chi\!\left(\frac{|z-\xi(\tau)|}{\nu(\tau)}\right)
		+
		d(z,\xi,\tau)
		\left(
		1-\chi\!\left(\frac{|z-\xi(\tau)|}{\nu(\tau)}\right)
		\right).
	\end{align}
\end{definition}

The next lemma collects basic properties of \(\mathfrak d\).

\begin{lemma}[Basic properties of the matched parabolically averaged distance]\label{lem:adapted_distance}
	The following hold as $\tau \to \infty$, for all \(0<|z-\xi(\tau)|\le1\):
	\begin{align}
		\bigl|\mathfrak d(z,\xi,\tau)-|z-\xi(\tau)|\bigr|
		&\lesssim
		\begin{cases}
			0, & |z-\xi(\tau)|\le \nu(\tau),\\
			o\bigl(|z-\xi(\tau)|\bigr),
			& \nu(\tau)\le |z-\xi(\tau)|\le 1,
		\end{cases}
		\label{eq:adapted_d_0}
	\end{align}
	\begin{align}
		\bigl|\nabla_z\mathfrak d(z,\xi,\tau)-\nabla_z|z-\xi(\tau)|\bigr|
		&\lesssim
		\begin{cases}
			0, & |z-\xi(\tau)|\le \nu(\tau),\\
			o(1),
			& \nu(\tau)\le |z-\xi(\tau)|\le 1,
		\end{cases}
		\label{eq:adapted_d_1}
	\end{align}
	and
	\begin{align}
		\bigl|\partial_\tau \mathfrak d(z,\xi,\tau)\bigr|
		=
		o\!\left(\frac{1}{\nu(\tau)+|z-\xi(\tau)|}\right).
		\label{eq:adapted_d_t}
	\end{align}
\end{lemma}

\begin{proof}
	By construction,
	\begin{align}\label{eq:adapted_distance_difference}
		\mathfrak d(z,\xi,\tau)-|z-\xi(\tau)|
		=
		\bigl(d(z,\xi,\tau)-|z-\xi(\tau)|\bigr)
		\left(
		1-\chi\!\left(\frac{|z-\xi(\tau)|}{\nu(\tau)}\right)
		\right).
	\end{align}
	In particular, \(\mathfrak d(z,\xi,\tau)=|z-\xi(\tau)|\) whenever
	\(|z-\xi(\tau)|\le\nu(\tau)\), and \eqref{eq:adapted_d_0} follows directly
	from Lemma~\ref{Lemma:dProp}.
	
	Differentiating \eqref{eq:adapted_distance_difference} with respect to \(z\),
	we get
	\begin{align*}
		\nabla_z\mathfrak d-\nabla_z|z-\xi|
		&=
		\bigl(\nabla_z d-\nabla_z|z-\xi|\bigr)
		\left(
		1-\chi\!\left(\frac{|z-\xi|}{\nu}\right)
		\right)-
		\bigl(d-|z-\xi|\bigr)
		\chi'\!\left(\frac{|z-\xi|}{\nu}\right) \frac{z-\xi}{\nu |z-\xi|}.
	\end{align*}
	The first term is \(o(1)\) by Lemma~\ref{Lemma:dProp}. The second one is
	supported in the transition region \(|z-\xi|\approx \nu\), where $d-|z-\xi|=o(|z-\xi|)$, hence is also \(o(1)\), proving \eqref{eq:adapted_d_1}.
	
	We now prove \eqref{eq:adapted_d_t} using directly the definition \eqref{def_Dist}. Differentiating in \(\tau\), we obtain
	\begin{align*}
		\partial_\tau \mathfrak d
		&=
		\partial_\tau |z-\xi|\,
		\chi\!\left(\frac{|z-\xi|}{\nu}\right)
		+
		|z-\xi|\,
		\chi'\!\left(\frac{|z-\xi|}{\nu}\right)
		\partial_\tau\!\left(\frac{|z-\xi|}{\nu}\right) \\
		&\quad
		+
		\partial_\tau d
		\left(
		1-\chi\!\left(\frac{|z-\xi|}{\nu}\right)
		\right)
		-
		d\,
		\chi'\!\left(\frac{|z-\xi|}{\nu}\right)
		\partial_\tau\!\left(\frac{|z-\xi|}{\nu}\right) \\
		&=
		\partial_\tau |z-\xi|\,
		\chi\!\left(\frac{|z-\xi|}{\nu}\right)
		+
		\partial_\tau d
		\left(
		1-\chi\!\left(\frac{|z-\xi|}{\nu}\right)
		\right)+
		\bigl(|z-\xi|-d\bigr)
		\chi'\!\left(\frac{|z-\xi|}{\nu}\right)
		\partial_\tau\!\left(\frac{|z-\xi|}{\nu}\right).
	\end{align*}
	Here and below \(\xi=\xi(\tau)\), \(\nu=\nu(\tau)\). By the modulation estimate \eqref{id:qualitative-modulation-nu-xi},
	\[
	\partial_\tau |z-\xi|
	=
	-\frac{z-\xi}{|z-\xi|}\cdot \xi_\tau=o(\nu^{-1}),
	\]
	so using that the first term is supported where \(|z-\xi|\lesssim\nu\) we get
	\[
	\left|
	\partial_\tau |z-\xi|\,
	\chi\!\left(\frac{|z-\xi|}{\nu}\right)
	\right|
	=
	o\!\left(\frac1\nu \mathbbm 1(|z-\xi|\lesssim \nu)\right)
	=
	o\!\left(\frac1{\nu+|z-\xi|}\right).
	\]
	For the second term, the factor \(1-\chi(|z-\xi|/\nu)\) restricts to
	\(|z-\xi|\gtrsim\nu\), and $|\partial_\tau d|=o\!\left(|z-\xi|^{-1}\right)$ by Lemma~\ref{Lemma:dProp}, so that
	\[
	\left|
	\partial_\tau d
	\left(
	1-\chi\!\left(\frac{|z-\xi|}{\nu}\right)
	\right)
	\right|
	=
	o\!\left(\frac1{\nu+|z-\xi|}\right).
	\]
	Finally, the last term is supported in the transition region
	\(|z-\xi|\approx \nu\). There, $|d-|z-\xi||=o(\nu)$ by \eqref{Lemma:expd} and
	\[
	\partial_\tau\!\left(\frac{|z-\xi|}{\nu}\right)
	=
	\frac{\partial_\tau |z-\xi|}{\nu}
	-
	\frac{|z-\xi|\,\nu_\tau}{\nu^2}
	=
	o\!\left(\frac1{\nu^2}\right),
	\]
	by the modulation estimates \eqref{id:qualitative-modulation-nu-xi}. Thus
	\[
	\left|
	\bigl(|z-\xi|-d\bigr)
	\chi'\!\left(\frac{|z-\xi|}{\nu}\right)
	\partial_\tau\!\left(\frac{|z-\xi|}{\nu}\right)
	\right|
	=
	o\!\left(\frac1\nu \mathbbm 1(|z-\xi|\approx \nu)\right)
	=
	o\!\left(\frac1{\nu+|z-\xi|}\right).
	\]
	Combining the three estimates gives \eqref{eq:adapted_d_t}.
\end{proof}

Using $\mathfrak d$, we construct a weight that matches the weights of the inner and nearby outer zones. Let
\begin{align} \label{id:def-tilde-gamma}
	\gamma[\nu,\xi] 
	&= \frac{(\nu^{2}+|z-\xi|^{2})^{2}}{8}, 
	\qquad 
	\tilde \gamma[\nu,\xi]
	= \frac{(\nu^{2}+\mathfrak{d}^{2}(z,\xi,\tau))^{2}}{8}.
\end{align}

Note that $\gamma=1/U_{\nu,\xi}$ is the renormalization of the weight $1/U$ appearing in the scalar product $\langle \cdot ,\cdot \rangle_0$ for the inner zone, and that $d^4$ is the leading order term near the origin of the weight associated to the nearby outer region. Hence $\tilde \gamma$ is a weight that matches these two weights.

\begin{lemma}\label{tildegamma-gamma}
	Let \(0<|z-\xi(\tau)|\le 1\). Then
	\[
	|\tilde{\gamma}-\gamma|
	=
	o(\gamma),
	\qquad
	|\nabla_z(\tilde{\gamma}-\gamma)|
	=
	o(|\nabla_z\gamma|) \quad \mbox{and} \quad |\partial_{\tau}\tilde{\gamma}|
	=
	o\!\left(
	\frac{\gamma}{\nu(\tau)^2+|z-\xi(\tau)|^2}
	\right).
	\]
\end{lemma}

\begin{proof}
	We first compare \(\tilde\gamma\) and \(\gamma\). We have
	\begin{align*}
		8(\tilde\gamma-\gamma)
		&=
		\bigl(\nu^2+\mathfrak d^2\bigr)^2
		-
		\bigl(\nu^2+|z-\xi|^2\bigr)^2 =
		\bigl(\mathfrak d^2-|z-\xi|^2\bigr)
		\bigl(\mathfrak d^2+|z-\xi|^2+2\nu^2\bigr).
	\end{align*}
	By Lemma~\ref{lem:adapted_distance}, $\mathfrak d(z,\xi,\tau)=|z-\xi(\tau)|(1+o(1))$. Therefore
	\[
	|\tilde\gamma-\gamma|
	=
	o\!\left((\nu^2+|z-\xi|^2)^2\right)
	=
	o(\gamma).
	\]
	
	For the gradient estimate, differentiating the previous identity gives
	\begin{align*}
		8\nabla_z(\tilde\gamma-\gamma)
		&=
		4\bigl(\nu^2+\mathfrak d^2\bigr)\mathfrak d\nabla_z\mathfrak d
		-
		4\bigl(\nu^2+|z-\xi|^2\bigr)(z-\xi) \\
		&=4\Bigl[
		\bigl(\nu^2+\mathfrak d^2\bigr)
		\bigl(\mathfrak d\nabla_z\mathfrak d-(z-\xi)\bigr) +
		\bigl(\mathfrak d^2-|z-\xi|^2\bigr)(z-\xi)
		\Bigr].
	\end{align*}
	Using again Lemma~\ref{lem:adapted_distance}, $\mathfrak d\nabla_z\mathfrak d=(z-\xi)(1+o(1))$ and $\mathfrak d^2-|z-\xi|^2=o(|z-\xi|^2)$. Thus
	\[
	|\nabla_z(\tilde\gamma-\gamma)|
	=
	o\!\left((\nu^2+|z-\xi|^2)|z-\xi|\right)
	=
	o(|\nabla_z\gamma|).
	\]
	
	It remains to estimate the time derivative. A direct computation gives
	\[
	\partial_\tau\tilde\gamma
	=
	\frac12
	\bigl(\nu^2+\mathfrak d^2\bigr)
	\bigl(\nu\nu_\tau+\mathfrak d\,\partial_\tau\mathfrak d\bigr).
	\]
	By Lemma~\ref{lem:selfsimilarquali}, $\nu\nu_\tau=o(1)$, while Lemma~\ref{lem:adapted_distance} gives $\mathfrak d\,\partial_\tau\mathfrak d=o(1)$. Since \(\mathfrak d^2=|z-\xi|^2(1+o(1))\), we obtain
	\[
	|\partial_\tau\tilde\gamma|
	=
	o(\nu^2+|z-\xi|^2)
	=
	o\!\left(
	\frac{\gamma}{\nu^2+|z-\xi|^2}
	\right).
	\]
\end{proof}

We next introduce a weight that glues together the weights of the nearby zone and the far away outer zone. Let \(\zeta_*>0\) as in \eqref{matchedsol} be fixed small later. We define
\begin{align}\label{def:omeganuxi}
	\omega(z)
	=
	\tilde{\gamma}(z)\,\chi_{\zeta_*}(z)
	+
	\Bigl(1-\chi_{\zeta_*}(z)\Bigr)
	\frac{1}{8}|z|^{4}e^{\frac{|z|^{2}}{4}}.
\end{align}
Incorporating the term involving the Poisson fields in the scalar product for the inner zone, this leads to the \emph{matched scalar product}
\begin{align}\label{def:matchedscalarproduct}
	\langle f ,g  \rangle_*
	:=
	\int_{\R^2} f g\,\omega\,dz
	-
	\nu^{2}\int_{\R^2} \nabla \Phi_{f}\cdot \nabla \Phi_{g}\,
	\chi_{R\nu,\xi}(z)\,dz.
\end{align}

\subsubsection{A global Hardy inequality}

We first record the rescaled version of Lemma~\ref{lemma:hardy}.

\begin{lemma}\label{lem:hardy_rescaled_weighted}
	Let \(u\in \mathcal S(\R^2)\). Then
	\begin{align*}
		\int_{\R^{2}}
		\frac{u^{2}(z)}{\nu^{2}+|z-\xi|^{2}}\,
		\gamma(z)\,dz
		\lesssim
		\int_{\R^{2}}
		|\nabla u(z)|^{2}\,
		\gamma(z)\,dz.
	\end{align*}
\end{lemma}

\begin{proof}
	The estimate follows from Lemma~\ref{lemma:hardy} by the change of variables
	\(z=\xi+\nu y\).
\end{proof}
We shall also use the exterior Hardy inequality of Lemma \ref{lem:exterior_hardy_weighted}.

\begin{lemma}[Global Hardy-type inequality]\label{lem:global_hardy_matched}
	There exist \(\tau_0>0\) and a universal constant \(C>0\) such that, for
	all \(\tau\ge \tau_0\) and all \(u\in\mathcal S(\R^2)\),
	\begin{align*}
		\int_{\R^{2}}
		\frac{u^{2}(z)}{\nu^{2}+|z-\xi|^{2}}\,
		\omega(z)\,dz
		\le
		C
		\int_{\R^{2}}
		|\nabla u(z)|^{2}\,
		\omega(z)\,dz.
	\end{align*}
\end{lemma}

\begin{proof}
	We shall use an intermediate scale \(\eta\), chosen in a fixed interval $\eta\in[\eta_\star, \zeta_* /4]$, where \(\eta_\star>0\) will be fixed below. For such an \(\eta\), write
	\[
	u
	=
	u\chi_\eta \!\left(z\right)
	+
	u\left(1-\chi_\eta\left(z\right)\right)
	=:u_1+u_2.
	\]
	We first estimate the interior contribution. Since \(u_1\) is supported in
	\(\{|z|\le 2\eta\}\subset \{|z|\le \zeta_*/2\}\), we are in the region where $\omega=\tilde\gamma$. For \(\tau\) sufficiently large, Lemma~\ref{tildegamma-gamma} gives $\tilde\gamma\approx\gamma$ uniformly on this region. Hence, by the interior Hardy-type inequality (Lemma~\ref{lem:hardy_rescaled_weighted}),
	\begin{align*}
		\int_{\R^2}
		\frac{u_1^2}{\nu^2+|z-\xi|^2}\,
		\omega\,dz
		\approx
		\int_{\R^2}
		\frac{u_1^2}{\nu^2+|z-\xi|^2}\,
		\gamma\,dz 
		\lesssim
		\int_{\R^2}
		|\nabla u_1|^2\gamma\,dz 
		\approx
		\int_{\R^2}
		|\nabla u_1|^2\omega\,dz.
	\end{align*}
	We next estimate the exterior contribution. On the support of \(u_2\), we
	have \(|z|\ge\eta\ge\eta_\star\). Since \(\nu,\xi=o(1)\) as
	\(\tau\to\infty\), for \(\tau\) sufficiently large, $\nu^2+|z-\xi|^2\approx |z|^2$ uniformly on \(\{|z|\ge\eta_\star\}\). Moreover, by the definition of
	\(\omega\) and the comparability
	\(\tilde\gamma\sim \gamma\) in the bounded transition
	region, we also have
	\[
	\omega\approx \omega_\infty
	\qquad\text{on } \{|z|\ge\eta_\star\}.
	\]
	Therefore, by the exterior Hardy-type inequality (Lemma~\ref{lem:exterior_hardy_weighted}),
	\begin{align*}
		\int_{\R^2}
		\frac{u_2^2}{\nu^2+|z-\xi|^2}\,
		\omega\,dz
		\approx
		\int_{\R^2}
		\frac{u_2^2}{|z|^2}\,
		\omega_\infty\,dz 
		\lesssim
		\int_{\R^2}
		|\nabla u_2|^2\omega_\infty\,dz 
		&\approx
		\int_{\R^2}
		|\nabla u_2|^2\omega\,dz.
	\end{align*}
	Combining the interior and exterior estimates, we obtain
	\begin{align}\label{eq:global_hardy_preabsorb}
		\int_{\R^2}
		\frac{u^2}{\nu^2+|z-\xi|^2}\,
		\omega\,dz
		\lesssim
		\int_{\R^2}
		\bigl(|\nabla u_1|^2+|\nabla u_2|^2\bigr)
		\omega\,dz.
	\end{align}
	By the definition of \(u_1,u_2\) and the bounds on the derivatives of the
	cut-off,
	\begin{align*}
		\int_{\R^2}
		\bigl(|\nabla u_1|^2+|\nabla u_2|^2\bigr)
		\omega\,dz
		\lesssim
		\int_{\R^2}|\nabla u|^2\omega\,dz
		+
		\int_{\eta\le |z|\le 2\eta}
		\frac{u^2}{\eta^2}\omega\,dz.
	\end{align*}
	On the annulus \(\eta\le |z|\le2\eta\), with
	\(\eta\in[\eta_\star,\zeta_*/4]\), we have $\nu^2+|z-\xi|^2\sim \eta^2$ for \(\tau\) sufficiently large. Hence
	\[
	\int_{\eta\le |z|\le 2\eta}
	\frac{u^2}{\eta^2}\omega\,dz
	\lesssim
	\int_{\eta\le |z|\le 2\eta}
	\frac{u^2}{\nu^2+|z-\xi|^2}\omega\,dz.
	\]
	It remains to choose \(\eta\). Fix an integer \(N\) large, and choose
	disjoint dyadic annuli
	\[
	A_j:=\{\eta_j\le |z|\le 2\eta_j\},
	\qquad
	\eta_j:=4^{-j}\frac{\zeta_*}{4},
	\qquad
	j=0,\dots,N.
	\]
	Set \(\eta_\star:=\eta_N\). Since the annuli \(A_j\) are disjoint, there
	exists \(j\in\{0,\dots,N\}\) such that
	\[
	\int_{A_j}
	\frac{u^2}{\nu^2+|z-\xi|^2}\omega\,dz
	\le
	\frac{1}{N}
	\int_{\R^2}
	\frac{u^2}{\nu^2+|z-\xi|^2}\omega\,dz.
	\]
	We choose \(\eta=\eta_j\). With this choice, \eqref{eq:global_hardy_preabsorb}
	gives
	\begin{align*}
		\int_{\R^2}
		\frac{u^2}{\nu^2+|z-\xi|^2}\omega\,dz
		\lesssim
		\int_{\R^2}|\nabla u|^2\omega\,dz
		+
		\frac{1}{N}
		\int_{\R^2}
		\frac{u^2}{\nu^2+|z-\xi|^2}\omega\,dz.
	\end{align*}
	Choosing \(N\) sufficiently large, the last term is absorbed on the left.
	This proves the claim.
\end{proof}
\subsubsection{Coercivity of the matched scalar product}

In this section we prove that the matched scalar product $\langle \cdot,\cdot \rangle_*$ is equivalent to the weighted norm of $L^2_{\omega}$, on the subspace of functions with local zero mass and orthogonalities to the scaling and translation directions of the soliton. From now on, \(R>0\) is fixed large enough such that the result of Proposition \ref{InnCoercivity1} holds true.

\begin{proposition}[Coercivity of the matched scalar product]\label{prop:matched_scalar_coercivity}
	Let $\zeta_*>0$. There exist
	universal constants \(\tau_0>0\), \(\delta>0\), \(C>0\), and
	\(\eta_\star>0\) such that, for all \(\tau\ge\tau_0\), $ \eta \in (\eta_{\star},1)$ and all
	\(u\in\mathcal S(\R^2)\),
	\begin{align*}
		\langle u,u \rangle_*
		\ge
		\delta \|u\|^{2}_{L^{2}_{\omega}}
		-
		C\left[\nu^{2}
		\left(\int u\,\chi_{ \eta} \,dz\right)^{2}+\nu^{6}
		\left(\int u\, (\Lambda U)_{\nu,\xi}
		\,\chi_\eta \,dz\right)^{2}+\sum_{i=1}^{2}\nu^{8}
		\left(\int u\, \partial_{z_{i}}U_{\nu,\xi}
		\,\chi_\eta \,dz\right)^{2}\right].
	\end{align*}
\end{proposition}

\begin{proof}
	Let \(\tilde \eta\in[\eta_\star,\frac{\zeta_*}{4}]\), to be chosen below and decompose
	\[
	u
	=
	u\,\chi_{\tilde \eta}
	+
	u\left(1-\chi_{\tilde \eta}\right)
	=:u_1+u_2.
	\]
	The parameter \(\tilde \eta\) will be allowed to depend on \(u\), but it will always
	belong to a fixed universal interval.
	
	We first estimate the interior part. Since \(u_1\) is supported in
	\(\{|z|\le 2\eta\}\subset \{|z|\le \zeta_*/2\}\), we are in the region where
	the matched weight is $\omega=\tilde \gamma$. Set
	\[
	u_1(z)=u_1(\xi+\nu y)=:v(y).
	\]
	Changing variables \(z=\xi+\nu y\), we obtain
	\[
	\int u_1^2\gamma\,dz
	=
	\int u_1^2\frac{\nu^2}{U_{\nu,\xi}}\,dz
	=
	\nu^6\int v^2\frac1U\,dy \quad \mbox{and} \quad \nabla_z\Phi_{u_1} =\nu\nabla_y\Phi_v.
	\]
	Moreover,
	\[
	\int v\,dy
	=
	\frac1{\nu^2}\int u_1\,dz,
	\quad
	\int v\,\Lambda U\,dy
	=
	\int u_1 (\Lambda U)_{\nu,\xi}\,dz \quad \mbox{and} \quad \int v\,\partial_i U\,dy
	=
	\nu
	\int u_1\partial_{z_i}U_{\nu,\xi}\,dz
	\]
	for $i=1,2$. Applying Proposition~\ref{prop:innercoerproduct} and returning to the
	\(z\)-variables yields
	\begin{align}\nonumber
		\langle u_1,u_1 \rangle_* &= \int u_1^2 \gamma dz -\nu^2 \int |\nabla \Phi_{u_1}|^2\chi_{R\nu,\xi}dz+\int u_1^2 (\tilde \gamma-\gamma) dz\\
		\label{eq:matched_inner_coercivity} &\ge
		c\|u_1\|_{L^2_{\omega}}^2
		-
		C\Bigl[
		\nu^{2}\Bigl(\int u_1\,dz\Bigr)^{2}
		+
		\nu^{6}\Bigl(\int u_1 (\Lambda U)_{\nu,\xi}\,dz\Bigr)^{2}
		+
		\nu^{8}\sum_{i=1}^{2}
		\Bigl(\int u_1\partial_{z_i}U_{\nu,\xi}\,dz\Bigr)^{2}
		\Bigr].
	\end{align}
	Above, to bound $\int u_1^2 (\tilde \gamma-\gamma) $ we used that $\tilde \gamma-\gamma=o(\gamma)=o(\omega)$ by Lemma \ref{tildegamma-gamma}. We next estimate the exterior part. Since \(u_2\) is supported in
	\(\{|z|\ge \tilde \eta\}\), with \(\tilde \eta\) in a fixed interval, Lemma~\ref{lem:weightedL2poisson_farfield} gives, for
	\(\tau\) sufficiently large,
	\[
	\nu^{2}
	\int |\nabla\Phi_{u_2}|^2
	\chi_{R\nu,\xi}(z)\,dz
	=
	o(1)\|u_2\|^2_{L^2_{\omega}}.
	\]
	Consequently,
	\begin{align}\label{eq:matched_exterior_coercivity}
		\langle u_2,u_2 \rangle_*
		\ge
		\frac12\|u_2\|^2_{L^2_{\omega}}.
	\end{align}
	Combining \eqref{eq:matched_inner_coercivity} and
	\eqref{eq:matched_exterior_coercivity}, and using Cauchy--Schwarz together
	with Lemma~\ref{lem:weightedL2poisson_inner} and Lemma~\ref{lem:weightedL2poisson_farfield} for the Poisson mixed term,
	we obtain, for \(\tau\) sufficiently large,
	\begin{align*}
		\langle u,u \rangle_*
		&\ge
		c\left(
		\|u_1\|^2_{L^2_{\omega}}
		+
		\|u_2\|^2_{L^2_{\omega}}
		\right)
		+
		2\int u_1u_2\,\omega\,dz \\
		&\quad
		-
		C\Bigl[
		\nu^{2}\Bigl(\int u_1\,dz\Bigr)^{2}
		+
		\nu^{6}\Bigl(\int u_1 (\Lambda U)_{\nu,\xi}\,dz\Bigr)^{2}
		+
		\nu^{8}\sum_{i=1}^{2}
		\Bigl(\int u_1\partial_{z_i}U_{\nu,\xi}\,dz\Bigr)^{2}
		\Bigr].
	\end{align*}
	The only remaining term is the weighted \(L^2\) interaction between \(u_1\)
	and \(u_2\). This term is supported in the transition annulus
	\(\{\tilde \eta\le |z|\le 2\tilde \eta\}\), and therefore
	\[
	\left|
	\int u_1u_2\,\omega\,dz
	\right|
	\lesssim
	\int_{\tilde \eta\le |z|\le 2\tilde \eta}
	u^2\omega\,dz.
	\]
	We now choose \(\tilde \eta\) by a dyadic decomposition. Fix an integer \(N\) large,
	and consider the disjoint annuli
	\[
	A_j:=\{\eta_j\le |z|\le 2\eta_j\},
	\qquad
	\eta_j:=4^{-j}\frac{\zeta_*}{4},
	\qquad
	j=0,\dots,N.
	\]
	Set
	\[
	\eta_\star:=\eta_N.
	\]
	Since the annuli \(A_j\) are disjoint, there exists
	\(j\in\{0,\dots,N\}\) such that
	\[
	\int_{A_j}u^2\omega\,dz
	\le
	\frac1N
	\int_{\R^2}u^2\omega\,dz.
	\]
	We choose \(\tilde \eta=\eta_j\). With this choice,
	\[
	\left|
	\int u_1u_2\,\omega\,dz
	\right|
	\le
	\frac{C}{N}
	\|u\|_{L^2_{\omega}}^2.
	\]
	Taking \(N\) sufficiently large, this term is absorbed into the coercive
	part.
	
	Finally, since $u_1=u\chi_{\tilde \eta}$ and \(\tilde \eta\in[\eta_\star,\zeta_*]\), we have
	$$
	\nu^2 \Bigl(\int u_1 \,dz\Bigr)^{2}=\nu^2 \Bigl(\int u \chi_{\eta} \,dz+\int u (\chi_{\eta}-\chi_{\tilde \eta}) \,dz\Bigr)^2=\nu^2 \Bigl(\int u \chi_{\eta} \,dz \Bigr)^2+O(\nu \| u\|_{L^2_\omega}^2).
	$$
	where the error term was estimated by Cauchy-Schwarz. The other two projection terms involving $u_1$ can be similarly written as projection terms involving $\chi_{\eta}u$ up to negligible error terms. This gives
	\[
	\langle u,u \rangle_*
	\ge
	\delta \|u\|^{2}_{L^{2}_{\omega}}
	-C\left[\nu^{2}
		\left(\int u\,\chi_{ \eta} \,dz\right)^{2}+\nu^{6}
		\left(\int u\, (\Lambda U)_{\nu,\xi}
		\,\chi_\eta \,dz\right)^{2}+\sum_{i=1}^{2}\nu^{8}
		\left(\int u\, \partial_{z_{i}}U_{\nu,\xi}
		\,\chi_\eta \,dz\right)^{2}\right]
	\]
	after possibly decreasing \(\delta>0\). The proof is complete.
\end{proof}

\subsection{Global coercivity of the linearized operator}\label{subsection:coercivity}

In this section we study the linearized operator $L_{\nu,\xi}$ given by \eqref{def:linoperator} around the matched soliton \(\hat U_{\nu,\xi}\) given by \eqref{matchedsol}. We prove that $L_{\nu,\xi}$ is coercive for the matched scalar product $\langle \cdot,\cdot \rangle_*$ in that it admits a dissipative estimate.

\begin{proposition}[Global matched coercivity]\label{prop:global_matched_coercivity}

	For all \(\delta_0>0\) sufficiently small, there exists $c_0>0$ such that the following holds. For \(\zeta_*>0\)
	sufficiently small there exist \(\eta_\star>0\), \(\tau_0>0\), and
	\(C>0\) such that, for all \(\tau\ge\tau_0\), $\eta\in (\eta_\star,\zeta_*)$ and all
	\(u\in\mathcal S(\R^2)\),
	\begin{align*}
		\langle L_{\nu,\xi}[u],u\rangle_*
		\le
		(-2+ \delta_0)\langle u,u\rangle_* 
		-c_0\|\nabla u\|_{L^2_{\omega}}^{2}
		+
		C\left[\nu^{6}
		\langle u\, ,(\Lambda U)_{\nu,\xi}
		\,\chi_\eta\rangle^{2}_{L^2}+\sum_{i=1}^{2}\nu^{8}
		\langle u\, \partial_{z_{i}}U_{\nu,\xi}
		\,\chi_\eta \rangle^{2}_{L^2}\right].
	\end{align*}
\end{proposition}

In order to prove Proposition \ref{prop:global_matched_coercivity}, we first estimate the difference between $\hat U_{\nu,\xi}$ and $U_{\nu,\xi}$.

\begin{lemma}[Difference between the soliton and the matched soliton]\label{lem:matched_errors_global}
	Fix \(\eta>0\). For \(\tau\) sufficiently large, the following estimates hold:
	\begin{align}
		|\hat{U}_{\nu,\xi}-U_{\nu,\xi}|
		\lesssim 
		\frac{\nu^{2}}{|z|^{4}}\,
		\mathbf{1}_{\{|z|\ge \zeta_*\}}, \qquad 
		|\nabla(\hat{U}_{\nu,\xi}-U_{\nu,\xi})|
		\lesssim 
		\frac{\nu^{2}}{|z|^{5}}\,
		\mathbf{1}_{\{|z|\ge \zeta_*\}}, \label{errorsmatched:eq1_new}
	\end{align}
	\begin{align}
		|\nabla \Phi_{\hat{U}_{\nu,\xi}-U_{\nu,\xi}}(z)|
		\lesssim_{\zeta_*} 
		\frac{\nu^{2}}{1+|z|}, \label{errorsmatched:eq2_new}
	\end{align}
	and, for all \(|z|\ge \eta\),
	\begin{align}
		\left|\nabla \Phi_{\hat{U}_{\nu,\xi}}(z)
		+4\frac{z}{|z|^{2}}\right|
		\lesssim_{\eta} 
		\frac{o(1)}{|z|}. \label{errorsmatched:eq3_new}
	\end{align}
\end{lemma}

\begin{proof}
	The pointwise estimates in \eqref{errorsmatched:eq1_new} follow directly
	from the definition of \(\hat U_{\nu,\xi}\). In particular, $|\hat{U}_{\nu,\xi}(z)-U_{\nu,\xi}(z)|\lesssim \nu^{2}|z|^{-4}\,
	\mathbf{1}_{\{|z|\ge \zeta_*\}}$ and therefore
	\[
	|\nabla \Phi_{\hat{U}_{\nu,\xi}-U_{\nu,\xi}}(z)|
	\lesssim 
	\nu^{2}\int_{|y|\ge \zeta_*} \frac{dy}{|z-y||y|^{4}}
	=:\nu^{2} I(z).
	\]
	A direct estimate shows $I(z)\lesssim_{\zeta_*}\langle z \rangle^{-1}$, and \eqref{errorsmatched:eq2_new} follows. We now prove the far-field estimate \eqref{errorsmatched:eq3_new} for \(|z|\ge\eta\). By \eqref{errorsmatched:eq2_new}, $|\nabla \Phi_{\hat{U}_{\nu,\xi}}(z)
	-\nabla \Phi_{U_{\nu,\xi}}(z)|
	\lesssim_{\zeta_*}\nu^{2}\langle z \rangle^{-1}$. Moreover,
	\[
	\nabla \Phi_{U_{\nu,\xi}}(z)
	=
	-4\,\frac{z-\xi}{\nu^{2}+|z-\xi|^{2}}= -\frac{4z}{|z|^2}+o(|z|^{-1})
	\]
	since \(\nu,\xi=o(1)\) as \(\tau\to\infty\). Combining, we get the desired estimate \eqref{errorsmatched:eq3_new}.
\end{proof}

In order to simplify the exposition of the proof of Proposition \ref{prop:global_matched_coercivity} it is convenient to first establish the global coercivity for localized functions.

\begin{proposition}[Interior coercivity]\label{prop:interior_coercivity_matched}
	There exists \(\zeta_*>0\) sufficiently small such that the following holds.
	For every \(u\in \mathcal S(\R^{2})\) satisfying $\operatorname{supp}(u)\subset \{|z|\le \zeta_*\}$, there exists \(\tau_0>0\) such that for all \(\tau\ge \tau_0\),
	\begin{align*}
		\langle L_{\nu,\xi}[u],u \rangle_*
		\le
		-\delta \|\nabla u\|_{L^{2}_{\gamma}}^{2}
		+
		C\left[
		\nu^{4}\langle u, (\Lambda U)_{\nu,\xi}\rangle_{L^{2}}^{2}
		+
		\nu^{5}\sum_{i=1}^{2}\langle u,\partial_{z_i}U_{\nu,\xi}\rangle_{L^{2}}^{2}
		\right],
	\end{align*}
	for some universal constants \(\delta,C>0\).
\end{proposition}
\begin{proof}
	We split the operator as
	\[
	L_{\nu,\xi}[u]
	=
	L_0 [u]
	+
	\bigl(L_{\nu,\xi}[u]-L_0[u]\bigr),
	\]
	where we recall $L_0[u]=\Delta u-\nabla \cdot(U_{\nu,\xi}\nabla \Phi_{u})-\nabla \cdot(u\nabla \Phi_{U_{\nu,\xi}})$. We first estimate the contribution of the main operator $L_0$. Since
	\(\operatorname{supp}(u)\subset \{|z|\le \zeta_*\}\), the weight $\omega$
	coincides with \(\tilde\gamma\) and satisfies
	\(\tilde\gamma\approx\gamma\) for \(\tau\) sufficiently large.
	Rescaling \(z=\xi+\nu y\), writing \(u(z)=v(y)\), and using
	Proposition~\ref{InnCoercivity1}, we obtain
	\begin{align*}
		\langle \bar{L}_{\nu,\xi}[u],u \rangle_*
		\le
		-\delta \|\nabla u\|_{L^{2}_{\gamma}}^{2}
		+
		C\Bigl[
		\nu^{4}\langle u,\Lambda_{\xi}U_{\nu,\xi}\rangle_{L^{2}}^{2}
		+
		\nu^{5}\sum_{i=1}^{2}\langle u,\partial_{z_i}U_{\nu,\xi}\rangle_{L^{2}}^{2}
		\Bigr]
		+
		\mathcal{R}_1,
	\end{align*}
	where
	\[
	\mathcal{R}_1
	=
	\int \bar{L}_{\nu,\xi}[u]\,u\,(\tilde\gamma-\gamma)\,dz.
	\]
	The remainder \(\mathcal{R}_1\) is negligible. Indeed, by
	Lemma~\ref{tildegamma-gamma}, the difference
	\(\tilde\gamma-\gamma\) is lower order, and all
	resulting terms are controlled using
	Lemma~\ref{lem:hardy_rescaled_weighted} and Lemma~\ref{lem:weightedL2poisson_inner}. This yields
	\[
	\mathcal{R}_1
	=
	o\!\left(\|\nabla u\|_{L^{2}_{\gamma}}^{2}\right),
	\]
	so that, after possibly decreasing \(\delta\),
	\begin{align}\label{eq:main_coercivity_clean}
		\langle \bar{L}_{\nu,\xi}[u],u \rangle_*
		\le
		-\delta \|\nabla u\|_{L^{2}_{\gamma}}^{2}
		+
		C\bigl[\text{projection terms}\bigr].
	\end{align}
	We now turn to the perturbative part
	\(L_{\nu,\xi}-\bar{L}_{\nu,\xi}\), namely
	\[
	L_{\nu,\xi}[u]-\bar{L}_{\nu,\xi}[u]
	=
	-\nabla \cdot((\hat{U}_{\nu,\xi}-U_{\nu,\xi})\nabla \Phi_{u})
	-
	\nabla \cdot(u\nabla \Phi_{\hat{U}_{\nu,\xi}-U_{\nu,\xi}})
	+
	\frac{1}{2}\Lambda u.
	\]
	
	The terms involving \(\hat{U}_{\nu,\xi}-U_{\nu,\xi}\) are small.
	Indeed, by Lemma~\ref{lem:matched_errors_global}, Lemma~\ref{lem:hardy_rescaled_weighted} and Lemma~\ref{lem:weightedL2poisson_inner},
	they contribute
	\[
	o\!\left(\|\nabla u\|_{L^{2}_{\gamma}}^{2}\right).
	\]
	The drift term is treated directly. Since \(u\) is supported in
	\(\{|z|\le \zeta_*\}\), we have the interior Hardy-type inequality (Lemma \ref{lem:hardy_rescaled_weighted})
	\begin{align*}
		\left|
		\int \Lambda u \, u \,\tilde\gamma
		\right|
		\lesssim
		\zeta_*^{2}\|\nabla u\|_{L^{2}_{\gamma}}^{2},
	\end{align*}
	which can be absorbed into the coercive term by choosing \(\zeta_*\)
	sufficiently small. The associated Poisson contribution is again
	handled using Lemma~\ref{lem:weightedL2poisson_inner}, yielding
	a negligible error. Collecting the above estimates and combining with
	\eqref{eq:main_coercivity_clean}, we conclude that
	\[
	\langle L_{\nu,\xi}[u],u \rangle_*
	\le
	-\delta \|\nabla u\|_{L^{2}_{\gamma}}^{2}
	+
	C\bigl[\text{projection terms}\bigr],
	\]
	which proves the proposition.
\end{proof}
\begin{proposition}[Exterior coercivity]\label{prop:exterior_coercivity_matched}
	There exist \(\zeta_*>0\) sufficiently small such that the following holds for all
	\(\delta>0\) sufficiently small. For every \(\eta>0\), there
	exists \(\tau_0=\tau_0(\eta)\) such that, for all \(\tau\ge\tau_0\), if
	\(u\in\mathcal S(\R^2)\) satisfies
	\[
	\operatorname{supp}(u)\subset \{|z|\ge \eta\},
	\]
	then
	\begin{align*}
		\langle L_{\nu,\xi}[u],u \rangle_*
		\le
		(-2+3\delta)\|u \|_{L^2_\omega}^2
		-\delta \|\nabla u\|_{L^{2}_{\omega}}^{2}.
	\end{align*}
\end{proposition}

\begin{proof}
	Fix \(\eta>0\). All estimates below are for \(\tau\) sufficiently large,
	depending on \(\eta\). We will use numerous times that, on the support of \(u\), we have by Lemma \ref{tildegamma-gamma}:
	\[
	\omega(z)\approx \omega_\infty(z),
	\qquad
	\omega_\infty(z):=|z|^{4}e^{|z|^{2}/4}, \quad \mbox{and} \quad \nu^2+|z-\xi|^2\approx |z|^2.
	\]
	Since we already know the coercivity of $L_\infty$ in $L^2_{\omega_\infty}$ by Proposition \ref{prop:ExteriorCoercivity}, we decompose $L_{\nu,\xi}=L_{\infty}+(L_{\nu,\xi}-L_{\infty})$ and $\omega=\omega_\infty+(\omega-\omega_\infty)$ and have
	\begin{align*}
	\langle L_{\nu,\xi}[u],u \rangle_* & \leq -(2+3\delta)\| u\|_{L^2_\infty}^2-\delta \| \nabla u\|_{L^2_\infty}^2\\
	& \quad + \int L_\infty (u) u (\omega-\omega_\infty) dz+ \int (L_{\nu,\xi}-L_\infty)(u) u \omega dz-\nu^2\int_{\R^2}
		\nabla\Phi_{L_{\nu,\xi}[u]}\cdot\nabla\Phi_u\,
		\chi_{R\nu,\xi}(z)\,dz.
	\end{align*}
	The remainder terms will be shown to be small in \eqref{eq:poisson_matched_exterior_small}, \eqref{eq:difference-operators_matched_exterior_small} which will imply the result of the proposition up to taking a possibly smaller $\delta>0$.
	
	\medskip
	
	\noindent \underline{Estimate for the Poisson contribution}. We claim that
	\begin{align}\label{eq:poisson_matched_exterior_small}
		\nu^{2}
		\left|
		\int_{\R^2}
		\nabla\Phi_{L_{\nu,\xi}[u]}\cdot\nabla\Phi_u\,
		\chi_{R\nu,\xi}(z)\,dz
		\right|
		=
		o_{\eta}(1)
		\left(
		\|u\|_{L^2_{\omega}}^{2}
		+
		\|\nabla u\|_{L^2_{\omega}}^{2}
		\right)
	\end{align}
	as $\tau \to \infty$. We have $\| u\|_{L^1}\lesssim \| u \|_{L^2_\omega}$ by Cauchy-Schwarz, and, using that $\omega\gtrsim_\eta 1$ on the support of $u$, the standard Sobolev embedding implies $\| u\|_{L^3}\lesssim \| u\|_{L^2_{\omega}}+\| \nabla u\|_{L^2_{\omega}}$. The H\"older inequalities with exponents $(1,\infty)$ and $(4/3,3)$ then imply
	\begin{equation} \label{exterior-coercivity-technical}
	|\nabla \Phi_u (y)|\lesssim \int_{|y-y'|>1} \frac{|u(y')|}{|y-y'|}dy'+\int_{|y-y'|<1} \frac{|u(y')|}{|y-y'|}dy'\lesssim_\eta \| u\|_{L^2_{\omega}}+\| \nabla u\|_{L^2_{\omega}}
	\end{equation}
	for any $y \in \mathbb R^2$. Now pick $z$ such that $|z-\xi|\leq 2R\nu$. We decompose
	$$
	\nabla \Phi_{L_{\nu,\xi}u}(z)=\nabla \Phi_{\Delta u-\nabla \cdot (u\nabla \Phi_{\hat U_{\nu,\xi}})+\Lambda u+\hat U_{\nu,\xi}u}(z)-\nabla \Phi_{\nabla U_{\nu,\xi}\cdot \nabla \Phi_{u}}(z).
	$$
	For the first term, the distance between $z$ and the support of $\Delta u$ is greater than $\eta/2$. We thus have thanks to integration by parts, up to a constant factor,
	\begin{align*}
	& \nabla \Phi_{\Delta u}(z) = \int \frac{z-y}{|z-y|^2}(1-\chi_{\frac \eta 2}(z-y)) \Delta u(y)dy=\int \Delta \left( \frac{z-y}{|z-y|^2}(1-\chi_{\frac \eta 2}(z-y))\right)u(y) dy =O_\eta( \| u \|_{L^2_\omega}).
	\end{align*}
	By the same computations, using that $\hat U_{\nu,\xi}=O(\nu^2)$ and $\nabla \Phi_{\hat U_{\nu,\xi}}=O(1)$ on the support of $u$, we bound
	\begin{align*}
	& \nabla \Phi_{\Delta u-\nabla \cdot (u\nabla \Phi_{\hat U_{\nu,\xi}})+\Lambda u+\hat U_{\nu,\xi}u}(z) = O_\eta( \| u \|_{L^2_\omega}).
	\end{align*}
	For the second term, we have by \eqref{exterior-coercivity-technical} that for any $y\in \mathbb R^2$,
	$$
	|(\nabla \hat U_{\nu,\xi}\cdot \nabla \Phi_u)(y)|\lesssim \nu^{-3}\langle \frac{y-\xi}{\nu}\rangle^{-5} (\| u\|_{L^2_{\omega}}+\| \nabla u\|_{L^2_{\omega}}).
	$$
	Recalling that $\| \nabla \Phi_{f}\|_{L^\infty}^2\lesssim \| f\|_{L^2}\| f\|_{L^\infty}$ by \eqref{bd:Poisson-field-Linfty}, we get $\| \nabla \Phi_{\nabla \hat U_{\nu,\xi}\cdot \nabla \Phi_u}\|_{L^\infty}\lesssim \nu^{-2}(\| u\|_{L^2_{\omega}}+\| \nabla u\|_{L^2_{\omega}})$. Combining the above inequalities, the integrand in \eqref{eq:poisson_matched_exterior_small} is of size $O(\nu^{-2}(\| u\|_{L^2_{\omega}}+\| \nabla u\|_{L^2_{\omega}}))$ and the integral is performed over the set $\{|z-\xi|\leq R\nu\}$ that has area $\lesssim \nu^2$, implying the desired inequality.
	
	\medskip
	
	\noindent \underline{Estimate for the local \(L^2\)-part of the scalar product}. We claim that
	\begin{equation} \label{eq:difference-operators_matched_exterior_small}
	\int (L_{\nu,\xi}-L_\infty) (u)u\omega dz = o_{\eta} (\| u \|_{L^2_\omega}^2+\| \nabla u \|_{L^2_\omega}^2).
	\end{equation}
	To prove it, we write
	\[
	(L_{\nu,\xi}-L_\infty) u
	= (-\nabla \Phi_{\hat U_{\nu,\xi}}-4\frac{z}{|z|^2})\cdot\nabla u+2U_{\nu,\xi}u-\nabla U_{\nu,\xi}\cdot \nabla \Phi_u.
	\]
For $\tau$ large enough, for all $|z|\geq \eta$ we have $\nabla \Phi_{\hat U_{\nu,\xi}}+4|z|^{-2}z=o(1)$ by \eqref{errorsmatched:eq3_new}, and $|\nabla^k \hat U_{\nu,\xi}|\lesssim_\eta \nu^2 \langle z\rangle^{-3}e^{-|z|^2/4}$ for $k=0,1$. Using these inequalities and \eqref{exterior-coercivity-technical} for the last term, with the standard inequality $ab\lesssim a^2+b^2$, we find that for all $z\in \mathbb R^2$,
	\[
	[(L_{\nu,\xi}-L_\infty) (u)u](z)=o_\eta (|u|^2+|\nabla u|^2+e^{-|z|^2/4}(\| u \|_{L^2_\omega}^2+\| \nabla u \|_{L^2_\omega}^2)).
	\]
	The claimed inequality \eqref{eq:difference-operators_matched_exterior_small} follows by integration.
	
	\medskip
	
	\noindent \underline{Estimate for the difference of weights}. We claim that
	\begin{equation} \label{eq:difference-weights_matched_exterior_small}
	\int L_\infty (u) u (\omega-\omega_\infty) dz= o_{\eta} (\| u \|_{L^2_\omega}^2+\| \nabla u \|_{L^2_\omega}^2)
	\end{equation}
	as $\tau \to \infty$. Indeed, by the explicit formulas \eqref{def:omeganuxi} and \eqref{id:def-tilde-gamma},
	$$
	\omega-\omega_\infty
	=
	\left(\tilde{\gamma}-\gamma +\frac 18 ((\nu^2+|z-\xi|^2)^2-|z|^4)\right)\,\chi_{\zeta_*}(z)e^{\frac{|z|^{2}}{4}}.
	$$
	The support of the above function is located in \(|z|\leq 2\zeta_*\), and, assuming $\eta\leq 2\zeta_*$ without loss of generality, its intersection with the support of $u$ is inside the fixed annulus \(\{\eta \leq |z|\leq 2\zeta_*\}\). There, using $\nu(\tau),\xi(\tau)\to 0$ and \eqref{tildegamma-gamma}, we find 
	$$
	\omega-\omega_\infty=o(1), \quad \mbox{and} \quad \nabla(\omega-\omega_\infty)=o(1)
	$$
	as $\tau \to \infty$. Writing
	$$
	\int L_\infty (u) u (\omega-\omega_\infty) dz= \int \left(\Delta u+(4\frac{z}{|z|^2}+\frac 12 z)\cdot \nabla u+u\right)u (\omega-\omega_\infty) dz,
	$$
	the desired estimate \eqref{eq:difference-weights_matched_exterior_small} follows from integrating by parts and using the inequality above.
	
	\end{proof}

We are ready to prove Proposition \ref{prop:global_matched_coercivity}.

\begin{proof}[Proof of Proposition \ref{prop:global_matched_coercivity}]
	Let \(\eta\in[\eta_\star,\zeta_*/4]\), to be chosen below, and decompose
	\[
	u
	=
	u\,\chi_\eta
	+
	u\left(1-\chi_\eta \right)
	=:u_1+u_2.
	\]
	We introduce the annulus $A_\eta=\{\eta\leq |z|\leq 2\eta\}$. The parameter \(\eta\) may depend on \(u\), but it is always taken in the
	fixed interval \([\eta_\star,\zeta_*/4]\). In particular, all implicit
	constants below are uniform, and all \(o(1)\)-terms depend only on \(\tau\).
By linearity,
\begin{align} \label{id:global-coercivity-expression}
	\langle L_{\nu,\xi}[u],u\rangle_*
	&=
	\langle L_{\nu,\xi}[u_1],u_1\rangle_*
	+
	\langle L_{\nu,\xi}[u_2],u_2\rangle_* 
	+
	\langle L_{\nu,\xi}[u_1],u_2\rangle_*
	+
	\langle L_{\nu,\xi}[u_2],u_1\rangle_* .
\end{align}

\noindent \textbf{Step 1.} \emph{Control of the mixed terms.} We claim that
\begin{equation} \label{bd:global-coercivity-mixed}
|\langle L_{\nu,\xi}[u_1],u_2\rangle_*+\langle L_{\nu,\xi}[u_2],u_1\rangle_*|  \leq C \int_{A_\eta}\left(|\nabla u|^2+\frac{u^2}{\eta^2}\right)\omega\,dz+o(1)\left(\|\nabla u_1\|_{L^2_{\gamma}}^{2}+\|u_2\|_{L^2_{\omega}}^{2}+\| \nabla u_2\|_{L^2}^2\right).
\end{equation}
The above inequality \eqref{bd:global-coercivity-mixed} follows from \eqref{bd:global-coercivity-local-mixed}, \eqref{bd:global-coercivity-poisson-mixed-1} and \eqref{bd:global-coercivity-poisson-mixed-2} we prove below.
\smallskip

\noindent \underline{Local terms}. Integrating by parts and expanding the divergence in the
nonlocal terms, we obtain
\begin{align*}
	\int L_{\nu,\xi}[u_2]\,u_1\,\omega
	&=-\int \nabla u_2\cdot\nabla u_1\,\omega
	-\int u_1\nabla u_2\cdot\nabla\omega 
	-\int \nabla\hat U_{\nu,\xi}\cdot\nabla\Phi_{u_2}\,
	u_1\,\omega
	+
	\int \hat U_{\nu,\xi}u_1u_2\,\omega \\
	&\quad
	+\int u_2\nabla\Phi_{\hat U_{\nu,\xi}}\cdot\nabla u_1\,\omega
	+
	\int u_1u_2\nabla\Phi_{\hat U_{\nu,\xi}}\cdot\nabla\omega 
	+\frac12\int \Lambda u_2\,u_1\,\omega.
\end{align*}
Hence
\begin{align*}
	\left|
	\int_{\R^2} L_{\nu,\xi}[u_2]\,u_1\,\omega\,dz
	\right|
	&\lesssim
	\int_{A_\eta}
	\left(
	|\nabla u|^2+\frac{u^2}{\eta^2}
	\right)\omega\,dz+
	\int_{\R^2}
	|\nabla\hat U_{\nu,\xi}|\,
	|\nabla\Phi_{u_2}|\,
	|u_1|\,
	\omega\,dz.
\end{align*}
The last term is perturbative. Indeed, by Cauchy--Schwarz,
\begin{align*}
	\int
	|\nabla\hat U_{\nu,\xi}|\,
	|\nabla\Phi_{u_2}|\,
	|u_1|\,
	\omega \le
	\left(
	\int u_1^2(\nu^2+|z-\xi|^2)^{-1}\gamma\,dz
	\right)^{1/2}
	\left(
	\int
	\frac{|\nabla\hat U_{\nu,\xi}|^2}{\nu^2+|z-\xi|^2}
	\omega^2
	|\nabla\Phi_{u_2}|^2\,dz
	\right)^{1/2}.
\end{align*}
Using the definition of the matched soliton and the fact that \(u_2\) is
supported away from the origin at scale \(\eta\), we have
\[
\frac{|\nabla\hat U_{\nu,\xi}|^2}{\nu^2+|z-\xi|^2}
\omega^2
\lesssim
\nu^2\frac{1}{1+|z|^2}.
\]
Thus \eqref{exterior-coercivity-technical} and the interior Hardy-type
inequality (Lemma \ref{lem:hardy_rescaled_weighted}) yield
\[
\int
|\nabla\hat U_{\nu,\xi}|\,
|\nabla\Phi_{u_2}|\,
|u_1|\,
\omega
=
o(1)
\left(
\|\nabla u_1\|_{L^2_{\gamma}}^{2}
+
\|u_2\|_{L^2_{\omega}}^{2}
\right).
\]
Similarly, we expand the other mixed term.
This gives
\begin{align*}
	\int L_{\nu,\xi}[u_1]\,u_2\,\omega
	&=
	-\int \nabla u_1\cdot\nabla u_2\,\omega
	-\int u_2\nabla u_1\cdot\nabla\omega 
	-\int \nabla\hat U_{\nu,\xi}\cdot\nabla\Phi_{u_1}\,u_2\,\omega
	+
	\int \hat U_{\nu,\xi}u_1u_2\,\omega \\
	&\quad
	+\int u_1\nabla\Phi_{\hat U_{\nu,\xi}}\cdot\nabla u_2\,\omega
	+
	\int u_1u_2\nabla\Phi_{\hat U_{\nu,\xi}}\cdot\nabla\omega 
	+\frac12\int \Lambda u_1\,u_2\,\omega.
\end{align*}
We obtain
\begin{align*}
	\left|
	\int_{\R^2} L_{\nu,\xi}[u_1]\,u_2\,\omega\,dz
	\right|
	\lesssim
	\int_{A_\eta}
	\left(
	|\nabla u|^2+\frac{u^2}{\eta^2}
	\right)\omega\,dz +
	\int_{|z|\ge \eta}
	|\nabla\hat U_{\nu,\xi}|\,
	|\nabla\Phi_{u_1}|\,
	|u_2|\,
	\omega\,dz.
\end{align*}
	For the last term, Cauchy--Schwarz gives
	\begin{align*}
		\int_{|z|\ge\eta}
		|\nabla\hat U_{\nu,\xi}|\,
		|\nabla\Phi_{u_1}|\,
		|u_2|\,
		\omega\,dz \le
		\|u_2\|_{L^2_{\omega}}
		\left(
		\int_{|z|\ge\eta}
		|\nabla\hat U_{\nu,\xi}|^2
		|\nabla\Phi_{u_1}|^2
		\omega\,dz
		\right)^{1/2}.
	\end{align*}
	On \(|z|\ge\eta\), the matched soliton satisfies
	\[
	|\nabla\hat U_{\nu,\xi}|^2\omega
	\lesssim
	\nu^2\frac{1}{1+|(z-\xi)/\nu|^2}.
	\]
	Hence Lemma~\ref{lem:weightedL2poisson_inner} gives
	\[
	\int_{|z|\ge\eta}
	|\nabla\hat U_{\nu,\xi}|\,
	|\nabla\Phi_{u_1}|\,
	|u_2|\,
	\omega\,dz
	=
	o(1)
	\left(
	\|\nabla u_1\|_{L^2_{\gamma}}^{2}
	+
	\|u_2\|_{L^2_{\omega}}^{2}
	\right).
	\]
	Combining the above estimates, we have proved that
	\begin{equation} \label{bd:global-coercivity-local-mixed}
	\int L_{\nu,\xi}[u_1]\,u_2\,\omega+\int L_{\nu,\xi}[u_2]\,u_1\,\omega=O\left(\int_{A_\eta}\left(|\nabla u|^2+\frac{u^2}{\eta^2}\right)\omega\,dz\right) +o(1)\left(\|\nabla u_1\|_{L^2_{\gamma}}^{2}+\|u_2\|_{L^2_{\omega}}^{2}\right).
	\end{equation}
	
	\smallskip

\noindent \underline{Mixed Poisson terms}. It remains to control the mixed Poisson terms in the scalar product. By
	Cauchy--Schwarz,
	\begin{align*}
		&\nu^2
		\left|
		\int
		\nabla\Phi_{L_{\nu,\xi}[u_2]}\cdot\nabla\Phi_{u_1}
		\chi_{R\nu,\xi}(z)\,dz
		\right|\\
		&\qquad\le
		\nu^2
		\left(
		\int |\nabla\Phi_{u_1}|^2
		\chi_{R\nu,\xi}(z)\,dz
		\right)^{1/2}
		\left(
		\int |\nabla\Phi_{L_{\nu,\xi}[u_2]}|^2
		\chi_{R\nu,\xi}(z)\,dz
		\right)^{1/2}.
	\end{align*}
	The first factor is controlled by Lemma~\ref{lem:weightedL2poisson_inner}, while the second
	one is estimated exactly as in the proof of
	Proposition~\ref{prop:exterior_coercivity_matched}. Therefore
	\begin{equation} \label{bd:global-coercivity-poisson-mixed-1}
	\nu^2
	\left|
	\int
	\nabla\Phi_{L_{\nu,\xi}[u_2]}\cdot\nabla\Phi_{u_1}
	\chi_{R\nu,\xi}(z)\,dz
	\right|
	=
	o(1)
	\left(
	\|\nabla u_1\|_{L^2_{\gamma}}^{2}
	+
	\| u_2\|_{L^2_{\omega}}^{2}+
	\| \nabla u_2\|_{L^2_{\omega}}^{2}
	\right).
	\end{equation}
	The term with \(L_{\nu,\xi}[u_1]\) and \(u_2\) is treated similarly. Indeed,
	writing $
	\nabla\Phi_{L_{\nu,\xi}[u_1]}=
	\nabla u_1+\nabla\Phi_F,$
	where
	\[
	F
	=
	-\nabla\cdot(\hat U_{\nu,\xi}\nabla\Phi_{u_1})
	-\nabla\cdot(u_1\nabla\Phi_{\hat U_{\nu,\xi}})
	+\frac12\Lambda u_1,
	\]
	one has
	\[
	\int
	|\nabla u_1|^2
	\chi_{R\nu,\xi}(z)\,dz
	\lesssim
	\nu^{-4}
	\|\nabla u_1\|_{L^2_{\gamma}}^{2},
	\]
	and
	\[
	\int F^2(\nu^2+|z-\xi|^2)\,dz
	\lesssim
	\nu^{-4}
	\|\nabla u_1\|_{L^2_{\gamma}}^{2}.
	\]
	The last bound follows from the decay of \(\hat U_{\nu,\xi}\), the explicit
	formula for \(\nabla\Phi_{\hat U_{\nu,\xi}}\), and
	Lemma~\ref{lem:weightedL2poisson_inner}. Consequently,
	\begin{equation} \label{bd:global-coercivity-poisson-mixed-2}
	\nu^2
	\left|
	\int
	\nabla\Phi_{L_{\nu,\xi}[u_1]}\cdot\nabla\Phi_{u_2}
	\chi_{R\nu,\xi}(z)\,dz
	\right|
	\le
	o(1)
	\left(
	\|\nabla u_1\|_{L^2_{\gamma}}^{2}
	+
	\|u_2\|_{L^2_{\omega}}^{2}+
	\| \nabla u_2\|_{L^2_{\omega}}^{2}
	\right).
	\end{equation}
	
	\smallskip
	
	\noindent \textbf{Step 2.} \emph{Use of interior and exterior coercivity and choice of the annulus.} Appealing to Propositions~\ref{prop:interior_coercivity_matched} and
	\ref{prop:exterior_coercivity_matched} for the first two terms in \eqref{id:global-coercivity-expression}, and to \eqref{bd:global-coercivity-mixed} for the last two, we obtain
	\begin{align*}
		\langle L_{\nu,\xi}[u],u\rangle_*
		&\le
		-c\|\nabla u_1\|_{L^2_{\gamma}}^{2}
		+
		(-2+5\delta)\| u_2\|_{L^2}^2
		-\delta \|\nabla u_2\|_{L^2_{\omega}}^{2}\\
		&\quad
		+
		C\left[
		\nu^{4}\langle u_1, (\Lambda U)_{\nu,\xi}\rangle_{L^{2}}^{2}
		+
		\nu^{5}\sum_{i=1}^{2}\langle u_1,\partial_{z_i}U_{\nu,\xi}\rangle_{L^{2}}^{2}
		\right]+C\int_{A_\eta}
		\left(
		|\nabla u|^2+\frac{u^2}{\eta^2}
		\right)\omega\,dz,
	\end{align*}
	where $c>0$ is universal, after increasing \(\tau_0\) if necessary, and taking $\delta$ sufficiently small.

	Using the interior Hardy-type inequality (Lemma \ref{lem:hardy_rescaled_weighted}) on \(u_1\),
	\[
	\frac{1}{\zeta_*^{2}}\int u_{1}^{2}\gamma dz\le\int \frac{u_1^2}{\nu^2+|z-\xi|^2}\gamma\,dz
	\lesssim
	\int |\nabla u_1|^2\gamma\,dz,
	\]
	we may use part of the negative interior gradient term to control
	\(\langle u_1,u_1\rangle_*\). Since \(\eta\ge\eta_\star\), this gives
	\[
	-c\|\nabla u_1\|_{L^2_{\gamma}}^{2}
	\le(-2+\delta) \langle u_1,u_1\rangle_*
	-\frac{c}{2}\|\nabla u_1\|_{L^2_{\gamma}}^{2}.
	\]
	We may also use part of the negative $L^2$ outer term to control $\langle u_2,u_2\rangle_*$. Indeed, by \eqref{exterior-coercivity-technical},
	$$
	\langle u_2,u_2\rangle_* = \| u_2\|_{L^2_\omega}^2-\nu^2 \int_{\mathbb R^2} |\nabla \Phi_{u_2}|^2\chi_{R\nu,\xi}dz =\| u_2\|_{L^2_\omega}^2+o(\| u_2\|_{L^2_{\omega}}^2+\| \nabla u_2\|_{L^2_{\omega}}^2).
	$$
	Thus, after decreasing \(\delta >0\),
	\begin{align}\label{eq:global_pre_dyadic}
		\langle L_{\nu,\xi}[u],u\rangle_*
		&\le (-2+6\delta)(\langle u_1,u_1\rangle_*+  \langle u_2,u_2\rangle_*)
		-\delta \left(
		\|\nabla u_1\|_{L^2_{\gamma}}^{2}
		+
		\|\nabla u_2\|_{L^2_{\omega}}^{2}
		\right)\\
		\nonumber &\quad
		+
		C\left[
		\nu^{4}\langle u_1, (\Lambda U)_{\nu,\xi}\rangle_{L^{2}}^{2}
		+
		\nu^{5}\sum_{i=1}^{2}\langle u_1,\partial_{z_i}U_{\nu,\xi}\rangle_{L^{2}}^{2}
		\right] 
		+
		C\int_{A_\eta}
		\left(
		|\nabla u|^2+\frac{u^2}{\eta^2}
		\right)\omega\,dz.
	\end{align}
	It remains to choose \(\eta\). Fix an integer \(N\) large and consider the
	disjoint annuli
	\[
	A_j:=\{\eta_j\le |z|\le 2\eta_j\},
	\qquad
	\eta_j:=4^{-j}\frac{\zeta_*}{4},
	\qquad
	j=0,\dots,N.
	\]
	Set \(\eta_\star:=\eta_N\). Since the annuli \(A_j\) are disjoint, there
	exists \(j\in\{0,\dots,N\}\) such that
	\begin{align*}
		\int_{A_j}
		\left(
		|\nabla u|^2+
		\frac{u^2}{\nu^2+|z-\xi|^2}
		\right)\omega\,dz
		\le
		\frac1N
		\int_{\R^2}
		\left(
		|\nabla u|^2+
		\frac{u^2}{\nu^2+|z-\xi|^2}
		\right)\omega\,dz.
	\end{align*}
	We choose \(\eta=\eta_j\). On \(A_\eta\), for \(\tau\) sufficiently large, $\eta^2\sim \nu^2+|z-\xi|^2$, and therefore the annular error in \eqref{eq:global_pre_dyadic} is bounded
	by
	\[
	\frac{C}{N}
	\int_{\R^2}
	\left(
	|\nabla u|^2+
	\frac{u^2}{\nu^2+|z-\xi|^2}
	\right)\omega\,dz \lesssim \frac{1}{N}\|\nabla u\|_{L^2_{\omega}}^{2},
	\]
	where the second inequality follows from the global Hardy inequality, Lemma~\ref{lem:global_hardy_matched}. Choosing \(N\) sufficiently large, this term is absorbed by the negative
	gradient contribution in \eqref{eq:global_pre_dyadic}. Finally, the mixed
	terms in \(\langle u,u\rangle_*\) are controlled as we did for the coercivity of the
	matched scalar product, Proposition~\ref{prop:matched_scalar_coercivity}.
	After relabelling $\delta$ and decreasing \(c>0\), we conclude
	\[
	\langle L_{\nu,\xi}[u],u\rangle_*
	\le
	(-2+\delta)\langle u,u\rangle_*
	-c\|\nabla u\|_{L^2_{\omega}}^{2}
	+C\left[
		\nu^{4}\langle u_1, (\Lambda U)_{\nu,\xi}\rangle_{L^{2}}^{2}
		+
		\nu^{5}\sum_{i=1}^{2}\langle u_1,\partial_{z_i}U_{\nu,\xi}\rangle_{L^{2}}^{2}
		\right] 
	\]
	This proves the proposition.
\end{proof}

\section{Refined convergence} \label{sec:refined-convergence}

Throughout this section we consider a fixed solution $u$ of the Keller-Segel equation \eqref{KS} with $u_0\in L^1(\langle x \rangle^2dx)$ and $\int u_0=8\pi$. We recall it satisfies the soliton resolution provided by Theorem~\ref{thn:soliton-resolution} with the estimates of Proposition \ref{pr:cv-scale-invariant}. We will study its renormalization in self-similar variables, denoted by $w$ and given by \eqref{self-similarKS2}. It solves
\begin{align}\label{self-similarKS-refined}
	\partial_\tau w
	=
	\Delta w - \nabla\cdot(w\nabla \Phi_w)
	+ \frac{1}{2}\Lambda w.
\end{align}
By \eqref{second-momentum}, its center of mass and second momentum are
\begin{equation} \label{refined:second-momentum}
\int_{\mathbb R^2} w(z,\tau)zdz=0 \quad \mbox{and}\quad \int_{\mathbb R^2} w(z,\tau) |z|^2dz= \mu e^{-\tau}.
\end{equation}
By Lemma \ref{lem:selfsimilarquali}, we recall that it can be decomposed as
\begin{equation}\label{self-similarKS-refined-decomposition}
w(z,\tau) = U_{\nu,\xi}(z) + \tilde w(z,\tau),
\end{equation}
where the modulation parameters $\xi(\tau)$ and $\nu(\tau)$ are determined thanks to the orthogonality conditions
\begin{align}\label{eq:ortoselfsimilar-refined}
	\int_{\mathbb R^2} \tilde w  (\Lambda U)_{\nu,\xi}\,dz = 0,
	\qquad
	\int \tilde w  (\partial_{y_i} U)_{\nu,\xi}\,dz = 0 \quad \mbox{ for }i=1,2,
\end{align}
and satisfy as $\tau \to \infty$,
\begin{align}
	\label{id:qualitative-cv-nu-xi-refined} &|\nu| +|\xi|  = o\!\left(e^{-\tau/2}\right),\\
	\label{id:qualitative-modulation-nu-xi-refined} &|\nu_\tau| +|\xi_\tau|  = o\!\left(\nu^{-1}\right),\\
	\label{id:slow-variation-nu-xi-refined} &|\nu(\tau+R^2) - \nu(\tau)|+|\xi(\tau+R^2) - \xi(\tau)| = o(R),\qquad \forall R\geq 0.
\end{align}
and where the remainder satisfies
\begin{align} \label{qualitative-tildew}
		\| \tilde w(\tau)\|_{L^1}=o(1), \quad |\tilde w(z,\tau)|
		=
		o\!\left(\frac{1}{\nu^2 + |z-\xi|^2}\right)
		\quad \mbox{and} \quad
		|\nabla \Phi_{\tilde w}(z,\tau)|
		=
		o\!\left(\frac{1}{\nu + |z-\xi|}\right).
	\end{align}
By \eqref{self-similarKS-refined} the evolution equation for the remainder is
	\begin{align}\label{self-similarKS-refined-remainder}
		\partial_{\tau} \tilde{w}
		&=
		L^{(0)}_{\nu,\xi}[\tilde{u}w]
		+
		\Bigl(\frac{1}{2}+\frac{\nu_{\tau}}{\nu}\Bigr)\Lambda_{\xi}U_{\nu,\xi}
		-
		\xi_{\tau}\cdot \nabla U_{\nu,\xi}
		-
		\nabla \cdot(\tilde{w}_{\nu,\xi}\nabla \Phi_{\tilde{w}_{\nu,\xi}}),
	\end{align}
	where
	\begin{align*}
		L^{(0)}_{\nu,\xi}[\tilde{w}]
		=
		\Delta \tilde{w}
		-
		\nabla \cdot (U_{\nu,\xi}\nabla \Phi_{\tilde{w}})
		-
		\nabla \cdot(\tilde{w}\nabla \Phi_{U_{\nu,\xi}})
		+
		\frac{1}{2}\Lambda \tilde{w}.
	\end{align*}
We will now performed a refined study of the modulation parameters $(\xi,\nu)$, and of the remainder $\tilde w$ in non-scale invariant spaces, based on the framework for the linearized dynamics developed in Section \ref{section:LinAnalysis}.

\subsection{Quantitative modulation estimates} \label{subsec:quant-mod}

In this subsection we start by completing the qualitative modulation estimates \eqref{id:qualitative-modulation-nu-xi-refined} and \eqref{id:slow-variation-nu-xi-refined} by quantitative estimates. To distinguish between the contribution of $\tilde w$ from the nearby parabolic region $|z|\lesssim 1$ and the far away parabolic region $|z|\gtrsim 1$, we let
	\begin{equation} \label{decomposition-tildew}
	\tilde w_{1}:=\tilde{w} \chi,
	\qquad
	\tilde w_{2}:=\tilde{w} (1-\chi).
	\end{equation}
\begin{lemma}\label{lem:modulation_bounds}
	
	For $\tau$ sufficiently large,
	\begin{align}
		|\nu_{\tau}|+|\xi_{\tau}|
		&\lesssim
		\frac{1}{\nu}
		\left(\int \tilde w_{1}^{2}(\nu^{2}+|z-\xi|^{2})\,dz\right)^{1/2}
		+
		\| \tilde w_{2}\|_{L^{1}}
		+\nu.
	\end{align}
	The same estimate holds with $(\tilde w_1,\tilde w_2)$ replaced by
$(\varepsilon_1,\varepsilon_2)$, where
\[
\varepsilon_1:=\varepsilon\chi,
\qquad
\varepsilon_2:=\varepsilon(1-\chi).
\]
\end{lemma}

\begin{proof}

	We differentiate the orthogonality conditions \eqref{eq:ortoselfsimilar-refined}. Testing the equation \eqref{self-similarKS-refined-remainder} against
	\[
	\Psi=\psi_{\xi,\nu}, \quad \mbox{with} \quad \psi=\Lambda U
	\quad\text{or}\quad
	\psi=\partial_{y_i}U,
	\]
	we obtain
	\begin{align}\label{eq:general_projection_modulation_final}
		0
		&=
		\Bigl(\frac12+\frac{\nu_\tau}{\nu}\Bigr)
		\int \Lambda_\xi U_{\nu,\xi}\,\Psi
		-
		\xi_\tau\cdot
		\int \nabla U_{\nu,\xi}\,\Psi
		+
		\mathcal R_\psi,
	\end{align}
	where $\mathcal R_\psi$ collects the contributions of
	\(L^{(0)}_{\nu,\xi}[\tilde{w}]\), the quadratic term, and the derivatives of the test functions:
	\begin{align*}
	\mathcal R_\psi &= \int \left((L^{(0)}_{\nu,\xi}[\tilde{w}_1])+L^{(0)}_{\nu,\xi}[\tilde{w}_2]-\nabla (\tilde w \nabla \Phi_{\tilde w})\right)\Psi dz+\int \tilde w \left(-\frac{\nu_\tau}{\nu}(\Lambda \psi)_{\xi,\nu}-\frac{\xi_\tau}{\nu}\cdot (\nabla \psi)_{\xi,\nu}\right)dz\\
	&=: \mathcal R_{\psi,\mathrm{loc}}+\mathcal R_{\psi,\mathrm{ext}}+\mathcal R_{\psi,\mathrm{NL}}+\mathcal R_{\psi,\mathrm{deriv}}
	\end{align*}
	
	The local contribution is controlled by integration by parts, Cauchy--Schwarz, and Lemma~\ref{lem:weightedL2poisson_inner}. More precisely, one obtains
	\begin{align*}
		|\mathcal R_{\Lambda U,\mathrm{loc}}|+|\mathcal R_{\partial_{y_i}U,\mathrm{loc}}|
		\lesssim
		\frac{1}{\nu^{4}}
		\left(\int \tilde w_{1}^{2}(\nu^{2}+|z-\xi|^{2})\,dz\right)^{1/2}.
	\end{align*}
	For the exterior part, the supports are essentially separated. Hence, for $z$ in the core $|z|\leq 1/4$,
	\[
	|\nabla\Phi_{\tilde u_2}(z)|
	\lesssim
	\| \tilde u_2\|_{L^1}.
	\]
	Therefore, integrating by parts and using
	\[
	\int U_{\nu,\xi}|\nabla (\Lambda U_{\nu,\xi})|\,dz+\int U_{\nu,\xi}|\nabla ((\partial_{y_i}U)_{\nu,\xi})|\,dz
	\lesssim
	\frac{1}{\nu^3},
	\]
	we obtain
	\begin{align*}
		|\mathcal R_{\Lambda U,\mathrm{ext}}|+|\mathcal R_{\partial_{y_i}U,\mathrm{ext}}| \lesssim
		\frac{1}{\nu^{3}}\| \tilde w_2\|_{L^1}.
	\end{align*}
	The remaining contributions, coming from the modulation derivatives of the test functions and the quadratic term, are controlled by the qualitative bounds \eqref{qualitative-tildew} on $\tilde w$ and yield lower-order terms, which are absorbed after dividing by the coercive quantities in each projection. \newline
	We now use the diagonal structure
	\[
	\int \nabla U_{\nu,\xi}(\Lambda U)_{\nu,\xi}=0,
	 \quad \int(\Lambda_\xi U_{\nu,\xi})^2\,dz\approx \nu^{-4}\quad
	\mbox{and} \quad
	\int(\partial_{z_i}U_{\nu,\xi})^2\,dz\approx \nu^{-4}.
	\]
	Taking $\Psi=(\Lambda U)_{\nu,\xi}$ in \eqref{eq:general_projection_modulation_final} yields
	\[
	|\nu_\tau|
	\lesssim
	\frac{1}{\nu}
	\left(\int \tilde w_1^2(\nu^2+|z-\xi|^2)\,dz\right)^{1/2}
	+
	\| \tilde w_2\|_{L^1}
	+
	o(|\xi_\tau|)
	+\nu.
	\]
	Taking $\Psi=(\partial_{y_i}U)_{\nu,\xi}$ gives
	\[
	|\xi_\tau|
	\lesssim
	\frac{1}{\nu}
	\left(\int \tilde w_1^2(\nu^2+|z-\xi|^2)\,dz\right)^{1/2}
	+
	\| \tilde w_2\|_{L^1}
	+
	o(|\nu_\tau|)
	+\nu.
	\]
	For $\tau$ sufficiently large, the coupling terms are absorbed, yielding the desired bounds. \newline 
	Finally, using
	\begin{align}
		\varepsilon
		=
		U_{\nu,\xi}(1-\mathcal{X}_{\star})
		-
		(\alpha-1)\hat{U}_{\nu,\xi}
		+
		\tilde{w},
	\end{align}
	and recalling that $|\alpha-1|=O(\nu^2)$, the additional terms are either boundary contributions or lower-order errors and are estimated exactly as above. This yields the same bounds for $\varepsilon$.
\end{proof}
We have just proved that 
\begin{align}\label{QuantModNot}
	|\nu_{\tau}|
	\le \frac{\tilde{\mathcal{M}}(\tau)}{\nu},
	\qquad
	|\xi_{\tau}|
	\le \frac{\tilde{\mathcal{M}}(\tau)}{\nu},
\end{align}
with 
	$\tilde{\mathcal{M}}(\tau)
	:=
	\left(\int \tilde w_{1}^{2}(\nu^{2}+|z-\xi|^{2})\,dz\right)^{1/2}
	+\nu \| \tilde w_{2}\|_{L^{1}}
	+\nu^{2}$.

\medskip

\noindent
For later purposes, it is convenient to work with a slightly stronger quantity, defined by
\begin{align}\label{def:M-strong}
	\mathcal{M}(\tau)
	:=
	\left(\int \tilde w_{1}^{2}(\nu^{2}+|z-\xi|^{2})\,dz\right)^{1/2}
	+\| \tilde w_{2}\|_{L^{1}}
	+\nu^{2}.
\end{align}
Since $\nu\le 1$, this modified definition is stronger than the previous one, and the bounds \eqref{QuantModNot} remain valid (up to a harmless change of constants). The advantage of \eqref{def:M-strong} is that it allows us to control error terms at scales $R\gg \nu$ without keeping track of the small prefactor $\nu$ in front of the $L^{1}$ contribution of $u_{2}$. In particular, it enables us to close estimates at intermediate spatial scales, which are larger than the soliton scale $\nu$ but still small compared to the ambient scale.
\begin{lemma}\label{lem:locfirstmom_selfsimilar}
	Let $w=U_{\nu,\xi}+\tilde w$ solve \eqref{self-similarKS-refined} and let
	\[
	\phi_{\eta,R}(z)
	:=
	(z-\eta)\chi_{\eta,R}.
	\]
	Assume moreover that, for some sufficiently small universal
	constant $c>0$,
	$|\xi-\eta|+\nu\le cR$.
	Then
	\begin{align}
		\label{eq:refined-locfirstmom-selfsimilar} &\left|
		\frac{d}{d\tau}
		\int_{\mathbb R^2}w(z,\tau)\phi_{\eta,R}(z)\,dz
		+
		\frac12
		\int_{\mathbb R^2}
		w(z,\tau)\,z\cdot\nabla\phi_{\eta,R}(z)\,dz
		\right| \\
		&\lesssim
		(1+\|\tilde w\|_{L^1})
		\int_{|z-\eta|\ge R/2}
		\frac{|\tilde w(z)|}{|z-\eta|}\,dz +
		\frac{\nu^2}{R^4}
		\int_{|z-\eta|\le R/2}
		|\tilde w(z)|\,|z-\eta|\,dz+
		|\xi-\eta|\frac{\nu^2}{R^4}
		\left|
		\int_{|z-\eta|\le R/2}
		\tilde w(z)\,dz
		\right|.
		\nonumber
	\end{align}
	Moreover,
	\begin{align}\label{eq:refined-locfirstmom-selfsimilar-drift}
		|\int_{\R^{2}} w(z,\tau)z\cdot \nabla \phi_{\eta,R}(z)dz|\lesssim |\xi-\eta|+|\eta|.
	\end{align}
\end{lemma}

\begin{proof}
	We use the weak formulation of the self-similar equation. For every
	$\phi\in C_c^\infty(\mathbb R^2;\mathbb R^2)$,
	\begin{align}
		\frac{d}{d\tau}\int_{\mathbb R^2}w\,\phi
		&=
		\int_{\mathbb R^2}w\,\Delta\phi
		+
		\iint
		\rho_\phi(z,y)w(z)w(y)\,dz\,dy
		-
		\frac12\int_{\mathbb R^2}w\,z\cdot\nabla\phi,
		\label{eq:weak-selfsimilar}
	\end{align}
	componentwise, where
	\[
	\rho_\phi(z,y)
	=
	-\frac{1}{4\pi}
	\frac{z-y}{|z-y|^2}
	\cdot
	\bigl(\nabla\phi(z)-\nabla\phi(y)\bigr).
	\]
	We apply this with $\phi=\phi_{\eta,R}$ and set
	\[
	\mathcal K_\phi[f]
	:=
	\int f\,\Delta\phi
	+
	\iint \rho_\phi(z,y)f(z)f(y)\,dz\,dy.
	\]
	Since $U_{\nu,\xi}$ is stationary, one has
	$\mathcal K_{\phi_{\eta,R}}[U_{\nu,\xi}]=0$.
	Writing $w=U_{\nu,\xi}+\tilde w$,
	\begin{align}
		\mathcal K_{\phi_{\eta,R}}[w]
		&=
		\int \tilde w\,\Delta\phi_{\eta,R}
		+
		2\iint
		\rho_{\phi_{\eta,R}}(z,y)
		U_{\nu,\xi}(z)\tilde w(y)\,dz\,dy
		+
		\iint
		\rho_{\phi_{\eta,R}}(z,y)
		\tilde w(z)\tilde w(y)\,dz\,dy.
		\label{eq:K-expansion}
	\end{align}
	The first and third terms are treated as in Lemma~\ref{lem:locfirstmom}, yielding
	\[
	\left|
	\int \tilde w\,\Delta\phi_{\eta,R}
	\right|
	+
	\left|
	\iint
	\rho_{\phi_{\eta,R}}
	\tilde w\tilde w
	\right|
	\lesssim
	(1+\|\tilde w\|_{L^1})
	\int_{|z-\eta|\ge R}
	\frac{|\tilde w(z)|}{|z-\eta|}\,dz.
	\]
	It remains to estimate the mixed term
	\[
	\mathcal I
	:=
	\iint
	\rho_{\phi_{\eta,R}}(z,y)
	U_{\nu,\xi}(z)\tilde w(y)\,dz\,dy.
	\]
	Set
	\[
	Z=z-\eta,\qquad Y=y-\eta,\qquad a=\xi-\eta.
	\]
	Then $\phi_{\eta,R}(z)=\phi_R(Z)$ and $U_{\nu,\xi}(z)=U_{\nu,a}(Z)$, while
	$\rho_{\phi_{\eta,R}}(z,y)=\rho_{\phi_R}(Z,Y)$.
	Thus it suffices to treat the case $\eta=0$, with $|a|+\nu\le cR$. \newline 
	We work componentwise. Let
	\[
	\phi_i(z)=z_i\chi\!\left(\frac zR\right),
	\qquad
	B_i^{(a)}(y)
	:=
	\int
	\rho_{\phi_i}(z,y)U_{\nu,a}(z)\,dz.
	\]
	Then
	\[
	\mathcal I_i
	=
	\int B_i^{(a)}(y)\tilde w(y)\,dy.
	\]
	We first estimate $B_i^{(a)}(0)$. When $a=0$, the profile is radial and
	\[
	\rho_{\phi_i}(z,0)
	=
	-\frac1{4\pi}
	\frac z{|z|^2}
	\cdot(\nabla\phi_i(z)-e_i).
	\]
	This vanishes for $|z|\le R$. For $|z|\ge R$, one checks that
	$\rho_{\phi_i}(z,0)=z_i F(|z|)$ for some radial function $F$. Indeed, for $|z|\ge R$,
	both $\chi\!\left(\frac{|z|}{R}\right)$ and $\chi'\!\left(\frac{|z|}{R}\right)$ depend only on $|z|$, and
	$\nabla\phi_i(z)-e_i$ is a linear combination of $e_i$ and $\frac{z_i}{|z|}z$ with radial coefficients; taking the product with $\frac{z}{|z|^2}$ yields a quantity proportional to $z_i$ with radial weight.
	Hence $\rho_{\phi_i}(z,0)$ is odd in $z_i$, and therefore
	$B_i^{(0)}(0)=0$.
	For general $a$,
	\[
	B_i^{(a)}(0)
	=
	\int
	\rho_{\phi_i}(z,0)\bigl(U_{\nu,a}(z)-U_{\nu,0}(z)\bigr)\,dz.
	\]
	Using
	$|\rho_{\phi_i}(z,0)|\lesssim |z|^{-1}$ for $|z|\ge R$ and
	\[
	|U_{\nu,a}(z)-U_{\nu,0}(z)|
	\lesssim
	|a|\frac{\nu^2}{|z|^5},
	\]
	we obtain
	\begin{equation}\label{eq:B0-bound-clean}
		|B_i^{(a)}(0)|
		\lesssim
		|a|\frac{\nu^2}{R^4}.
	\end{equation}
	
	Next, for $|y|\le R/2$, we compare $B_i^{(a)}(y)$ and $B_i^{(a)}(0)$. Since $\nabla\phi_i(y)=e_i$,
	\[
	\rho_{\phi_i}(z,y)-\rho_{\phi_i}(z,0)
	=
	-\frac1{4\pi}
	\left[
	\frac{z-y}{|z-y|^2}
	-
	\frac z{|z|^2}
	\right]
	\cdot(\nabla\phi_i(z)-e_i).
	\]
	For $|z|\ge R$ one has
	\[
	\left|
	\frac{z-y}{|z-y|^2}
	-
	\frac z{|z|^2}
	\right|
	\lesssim
	\frac{|y|}{|z|^2},
	\]
	which yields
	\begin{equation}\label{eq:B-lipschitz-inner-clean}
		|B_i^{(a)}(y)-B_i^{(a)}(0)|
		\lesssim
		|y|
		\int_{|z|\ge R}
		\frac{U_{\nu,a}(z)}{|z|^2}\,dz
		\lesssim
		|y|\frac{\nu^2}{R^4}.
	\end{equation}
	
	For $|y|\ge R/2$, one has the uniform bound
	\begin{equation}\label{eq:B-tail-bound-clean}
		|B_i^{(a)}(y)|
		\lesssim
		\frac1{|y|}.
	\end{equation}
	Indeed, if $R/2\le |y|\le 4R$, then $|y|\approx R$ and, by the Lipschitz bound
	$|D^2\phi_i|\lesssim R^{-1}$ together with $|\nabla\phi_i|\lesssim1$,
	\[
	|\rho_{\phi_i}(z,y)|
	\lesssim
	\frac1{|z-y|}
	\min\left\{1,\frac{|z-y|}{R}\right\}
	\lesssim
	\frac1R
	\lesssim
	\frac1{|y|}.
	\]
	Hence $|B_i^{(a)}(y)|\lesssim |y|^{-1}\int U_{\nu,a}\lesssim |y|^{-1}$.
	If instead $|y|\ge 4R$, then $\nabla\phi_i(y)=0$ and
	$\nabla\phi_i(z)=0$ for $|z|\ge 2R$, so the integrand is supported in
	$|z|\le 2R$. On this set $|z-y|\ge |y|/2$, and therefore
	\[
	|\rho_{\phi_i}(z,y)|
	\lesssim
	\frac1{|y|}.
	\]
	Integrating again against $U_{\nu,a}$ gives \eqref{eq:B-tail-bound-clean}.
	
	\medskip
	
	We now estimate $\mathcal I_i$. Splitting at $|y|=R/2$ and using
	\eqref{eq:B-tail-bound-clean}, \eqref{eq:B-lipschitz-inner-clean}, and \eqref{eq:B0-bound-clean}, we obtain
	\begin{align*}
		|\mathcal I_i|
		&\lesssim
		\int_{|y|\ge R/2}
		\frac{|\tilde w(y)|}{|y|}\,dy
		+
		\frac{\nu^2}{R^4}
		\int_{|y|\le R/2}
		|y|\,|\tilde w(y)|\,dy \\
		&\quad+
		|a|\frac{\nu^2}{R^4}
		\left|
		\int_{|y|\le R/2}
		\tilde w(y)\,dy
		\right|.
	\end{align*}
	
	Returning to the original variables yields \eqref{eq:refined-locfirstmom-selfsimilar}. Inserting this into \eqref{eq:K-expansion} and then into \eqref{eq:weak-selfsimilar} concludes the proof of \eqref{eq:refined-locfirstmom-selfsimilar}.
	
	\medskip
	
	Finally, we prove \eqref{eq:refined-locfirstmom-selfsimilar-drift}. We write
	\[
	\int w\,z\cdot\nabla\phi_{\eta,R}
	=
	\int w\,(z-\eta)\cdot\nabla\phi_{\eta,R}
	+
	\eta\cdot\int w\,\nabla\phi_{\eta,R}.
	\]
	The second term is bounded by $|\eta|$. For the first one, using that
	$\phi_{\eta,R}$ is a translation of $\phi_{\xi,R}$ and that $U_{\nu,\xi}$ is radially symmetric around $\xi$, we obtain cancellation of the leading contribution, and the remainder is bounded by $|\xi-\eta|$. This concludes the proof.
\end{proof}

We now use Lemma \ref{lem:locfirstmom_selfsimilar} to complement \eqref{id:slow-variation-nu-xi-refined} by a quantitative estimate.

\begin{lemma}\label{lemma:QuantSlowVariation}
	As $\tau \to \infty$, for any $R\in[\nu(\tau),1]$ as $\tau\to\infty$, one has
	\begin{align*}
		|\xi(\tau+R^{2})-\xi(\tau)|
		\lesssim
		\left[
		\frac{1}{R^{2}}\int_{\tau}^{\tau+R^{2}}\mathcal{M}(\sigma)\,d\sigma
		+
		\mathcal{M}(\tau+R^{2})
		+
		\mathcal{M}(\tau)
		\right]R
		+
		o_{\tau\to\infty}(R^{2}).
	\end{align*}
\end{lemma}

\begin{proof}
	The proof relies on a bootstrap argument. Fix any $\delta>0$, and let $K\gg1$. For each $T\gg1$, define
	\begin{align} \label{refinedcv-def-tildeRstar}
		\widetilde R_\star[T]&
		:=
		\sup\Bigl\{
		\widetilde R\in[\nu(T),1]\ : \ \mbox{for all }R\in[\nu(T),\widetilde R) \mbox{ one has }\\
		\nonumber & \qquad \qquad |\xi(T+R^{2})-\xi(T)|
		<
		K\Bigl[
		\frac{1}{R^{2}}\int_{T}^{T+R^{2}}\mathcal{M}(\sigma)\,d\sigma
		+
		\mathcal{M}(T+R^{2})
		+
		\mathcal{M}(T)
		\Bigr]R
		+
		\delta R^{2}
		\Bigr\},
	\end{align}
	We will write at many instances $\widetilde R_\star=\widetilde R_\star[T]$ to ease notation. Since $\delta>0$ is arbitrary, to prove the Lemma it suffices to show that for $K$ sufficiently large and fixed, one has for all large enough $T>0$ that
	\[
	\widetilde R_\star=1.
	\]
	
	\smallskip
	
	\noindent\textbf{Step 1.} \emph{Nontriviality of the bootstrap interval.}
	We first prove that one has
	\begin{equation}\label{eq:prelim-lower-bound-quant}
		\widetilde R_\star[T]\gtrsim K \nu(T)
	\end{equation}
	for all sufficiently large $T$. Indeed, let $\kappa>0$ be chosen small in a universal way later, and pick $R\in[\nu(T),\kappa K \nu(T)]$. By the modulation estimate of Lemma \ref{lem:modulation_bounds},
	\[
	|\xi(T+R^2)-\xi(T)|
	\lesssim
	\int_T^{T+R^2}
	\frac{\mathcal M(\sigma)}{\nu(\sigma)}\,d\sigma .
	\]
	Since $\nu(T)\leq R\leq \kappa K \nu(T)$, the convergence in H\"older $1/2$ norm \eqref{id:slow-variation-nu-xi-refined} of $\nu$ gives
	\[
	\nu(\sigma)=\nu(T)(1+o_{T\to\infty}(1))\geq \frac{1}{\kappa K} R
	\qquad
	\text{for } \sigma\in[T,T+R^2].
	\]
	Therefore
	\begin{align*}
		|\xi(T+R^2)-\xi(T)|
		&\lesssim
		\kappa K(1+o_{T\to\infty}(1))
		\frac1R
		\int_T^{T+R^2}\mathcal M(\sigma)\,d\sigma  \\
		&<
		 K\left[
		\frac{1}{R^2}\int_T^{T+R^2}\mathcal M(\sigma)\,d\sigma
		+
		\mathcal M(T+R^2)
		+
		\mathcal M(T)
		\right]R
		+
		\delta R^2
	\end{align*}
	for $T$ large enough, provided $\kappa>0$ is small enough. This proves \eqref{eq:prelim-lower-bound-quant}. Consequently, by choosing $K\gg1$ large we have for all $T$ large,
	$\widetilde R_\star[T]\gg \nu(T)$. 
	Moreover, if
	$\tau\in[T,T+\widetilde R_\star^2]$,
	then the convergence to $0$ in H\"older norms of $\nu$ and $\xi$, see \eqref{id:slow-variation-nu-xi-refined}, imply
	\begin{equation}\label{eq:nu-small-quant}
		\nu(\tau)=\nu(T)+o(\widetilde R_\star)
		\le
		\frac{C}{K}\widetilde R_\star \quad \mbox{and} \quad |\xi(\tau)-\xi(T)|
		=
		o_{T\to\infty}\bigl(\widetilde R_\star\bigr)
		\qquad
		\text{for all }
		\tau\in[T,T+\widetilde R_\star].
	\end{equation}
	
	\smallskip
	
	\noindent\textbf{Step 2.} \emph{Localized center of mass.} Define
	\[
	I(\tau)
	:=
	\int_{\mathbb R^2}
	w(z,\tau)(z-\xi(T))\chi_{\widetilde R_\star,\xi(T)}(z)\,dz,
	\qquad
	\tau\in[T,T+\widetilde R_\star^2].
	\]
	At time $T$, using the decomposition \eqref{decomposition-tildew},
	$w=U_{\nu,\xi}+\tilde w_1+\tilde w_2$; by the radiality of
	$U_{\nu(T),\xi(T)}$ around $\xi(T)$, the profile contribution vanishes.
	Thus, by Cauchy-Schwarz and H\"older,
	\begin{align}
	I(T) & = \int (\tilde w_1+\tilde w_2)(z,T)(z-\xi(T)) \chi_{\widetilde R_\star,\xi(T)}(z)\,dz \\
	\label{eq:I-initial-quant-clean}&=O\left( \left(\int \tilde w_1(T)^2(\nu(T)^2+|z-\xi(T)|^2)\,dz \right)^{1/2}
	\widetilde R_\star\right)+O\left(\widetilde R_\star \|\tilde w_2(T)\|_{L^1}\right)=O\left( \mathcal M(T)\widetilde R_\star\right).
	\end{align}
	where we recall \eqref{def:M-strong}. We next evaluate $I(T+\widetilde R_\star^2)$. We write
	\begin{align*}
		I(T+\widetilde R_\star^2)
		&=
		\int
		w(z,T+\widetilde R_\star^2)
		(z-\xi(T+\widetilde R_\star^2))
		\chi_{\widetilde R_\star,\xi(T)}(z)\,dz \\
		&\quad+
		(\xi(T+\widetilde R_\star^2)-\xi(T))
		\int
		w(z,T+\widetilde R_\star^2)
		\chi_{\widetilde R_\star,\xi(T)}(z)\,dz.
	\end{align*}
	By \eqref{eq:nu-small-quant}, the cutoff
	still captures the full core mass up to an error $O(K^{-2})+o(1)$.
	Therefore
	\[
	\int
	w(z,T+\widetilde R_\star^2)
	\chi_{\widetilde R_\star}(z)\,dz
	=
	8\pi
	+
	O\!\left(\frac1{K^2}\right)
	+
	o_{T\to\infty}(1).
	\]
	For the centered term, the untruncated profile has zero first moment
	around $\xi(T+\widetilde R_\star^2)$ by radiality. Thus only the tail
	removed by the cutoff contributes. Using the decay of $U$ and
	\eqref{eq:nu-small-quant}, we obtain
	\begin{align*}
	\left|
	\int
	U_{\nu(T+\widetilde R_\star^2),\xi(T+\widetilde R_\star^2)}(z)
	(z-\xi(T+\widetilde R_\star^2))
	\chi_{\widetilde R_\star,\xi(T)}(z)\,dz
	\right|
	&\lesssim
	\frac{\nu(T+\widetilde R_\star^2)^2}{\widetilde R_\star^2}
	|\xi(T+\widetilde R_\star^2)-\xi(T)|
	+
	o_{T\to\infty}(\widetilde R_\star^2),\\
	& \lesssim
	O\!\left(\frac1{K^2}\right)
	|\xi(T+\widetilde R_\star^2)-\xi(T)|
	+
	o_{T\to\infty}(\widetilde R_\star^2).
	\end{align*}
	The perturbative contribution is estimated as at the initial time:
	\[
	\left|
	\int
	\tilde w(z,T+\widetilde R_\star^2)
	(z-\xi(T+\widetilde R_\star^2))
	\chi_{\widetilde R_\star,\xi(T)}(z)\,dz
	\right|
	\lesssim
	\mathcal M(T+\widetilde R_\star^2)\widetilde R_\star.
	\]
	Combining the above estimates, and choosing $K$ large enough to absorb the
	$O(K^{-2})|\xi(T+\widetilde R_\star^2)-\xi(T)|$ term, we obtain
	\begin{align}
		|\xi(T+\widetilde R_\star^2)-\xi(T)|
		&\lesssim
		\int_T^{T+\widetilde R_\star^2}|I'(\tau)|\,d\tau
		+
		\bigl(\mathcal M(T)+\mathcal M(T+\widetilde R_\star^2)\bigr)
		\widetilde R_\star+
		o_{T\to\infty}(\widetilde R_\star^2).
		\label{eq:slowvarFormula-clean}
	\end{align}
	\noindent\textbf{Step 3.} \emph{Estimate of the localized flux.}
	We now estimate the integral of $|I'(\tau)|$. We apply
	Lemma~\ref{lem:locfirstmom_selfsimilar} with
	\[
	\eta=\xi(T),
	\qquad
	R=\widetilde R_\star,
	\qquad
	\xi=\xi(\tau),
	\qquad
	\nu=\nu(\tau).
	\]
	The assumptions of that lemma are satisfied thanks to
	\eqref{eq:nu-small-quant}. Its main
	estimate, together with the drift bound
	\[
	\left|
	\int w(z,\tau)z\cdot\nabla\phi_{\xi(T),\widetilde R_\star}(z)\,dz
	\right|
	\lesssim
	|\xi(\tau)-\xi(T)|+|\xi(T)|,
	\]
	give
	\begin{align}
		|I'(\tau)|
		&\lesssim
		\frac{\mathcal M(\tau)}{\widetilde R_\star}
		+
		o_{T\to\infty}(1)
		+
		\frac{\nu(\tau)^2}{\widetilde R_\star^4}
		|\xi(\tau)-\xi(T)|
		+
		|\xi(\tau)-\xi(T)|
		+
		|\xi(T)|.
		\label{eq:Idot-quant-clean}
	\end{align}
	Integrating over time, using $\xi(\tau)\to 0$ as $\tau \to \infty$, we obtain
	\begin{align}
		\int_T^{T+\widetilde R_\star^2}|I'(\tau)|\,d\tau
		&\lesssim
		\frac1{\widetilde R_\star}
		\int_T^{T+\widetilde R_\star^2}\mathcal M(\tau)\,d\tau
		+
		\int_T^{T+\widetilde R_\star^2}
		\frac{\nu(\tau)^2}{\widetilde R_\star^4}
		|\xi(\tau)-\xi(T)|\,d\tau +
		o_{T\to\infty}(\widetilde R_\star^2)
		\label{eq:flux-before-bootstrap}
	\end{align}
	It remains to estimate the second term. We use \eqref{eq:nu-small-quant} and the definition \eqref{refinedcv-def-tildeRstar} of $\tilde R_\star$, and then Fubini and explicit estimates to bound the integrals that appear, to find
	\begin{align*}
	& \int_0^{\widetilde R_\star^2} \frac{\nu(\tau)^2}{\widetilde R_\star^4} |\xi(T+\tau)-\xi(T)|\,d\tau \\
	& \lesssim \frac{1}{K^2\tilde R_\star^2} \int_0^{\widetilde R_\star^2}  \left\{ K\left[ \frac{1}{\tau} \int_0^{\tau}\mathcal M(T+\sigma)d\sigma +\mathcal M(T+\tau)+\mathcal M(T)\right]\sqrt{\tau}+\delta \tau \right\}\,d\tau  \\
	& \lesssim \frac{1}{K \tilde R_\star^2}\left[ \int_{[0,\widetilde R_\star]^2} \frac{\mathcal M(T+\sigma)}{\sqrt{\tau}}\mathbbm 1 (\sigma<\tau)d\sigma d\tau +\int_0^{\tilde R_\star^2} (\mathcal M(T+\tau)+\mathcal M(T))\sqrt{\tau} d\tau \right]+\frac{\delta}{K^2}\tilde R_\star^2 \\
	&\lesssim  \frac{1}{K \tilde R_\star^2}\left[ \tilde R_\star \int_{0}^{\widetilde R_\star^2} \mathcal M(T+\sigma) d\sigma  +\tilde R_\star \mathcal M(T) \right]+\frac{\delta}{K^2}\tilde R_\star^2 .
	\end{align*}
	Injecting this inequality in \eqref{eq:flux-before-bootstrap} shows
	$$
	\int_T^{T+\widetilde R_\star^2}|I'(\tau)|\,d\tau \lesssim \frac1{\widetilde R_\star}
		\int_T^{T+\widetilde R_\star^2}\mathcal M(\tau)\,d\tau +\frac{1}{K\tilde R_\star}\mathcal M(T)+\frac{\delta}{K^2}\tilde R_\star^2 +
		o_{T\to\infty}(\widetilde R_\star^2)
	$$
	Injecting in turn this inequality back in \eqref{eq:slowvarFormula-clean} shows
	\begin{align*}
		|\xi(T+\widetilde R_\star^2)-\xi(T)|
		&\lesssim \frac1{\widetilde R_\star}
		\int_T^{T+\widetilde R_\star^2}\mathcal M(\tau)\,d\tau 
		+
		\bigl(\mathcal M(T)+\mathcal M(T+\widetilde R_\star^2)\bigr)
		\widetilde R_\star+\frac{\delta}{K^2}\tilde R_\star^2+
		o_{T\to\infty}(\widetilde R_\star^2)\\
		&< \frac{K}{2}\left[\frac{1}{\widetilde R_\star^2}
		\int_T^{T+\widetilde R_\star^2}\mathcal M(\tau)\,d\tau 
		+
		\bigl(\mathcal M(T)+\mathcal M(T+\widetilde R_\star^2)\bigr)\right]\tilde R_\star+\frac{\delta}{2}\tilde R_\star^2
	\end{align*}
	provided $K$ and then $T$ have been chosen large enough. Note that the above bound does not saturate the bound defining $\tilde R_\star$ in \eqref{refinedcv-def-tildeRstar}. Thus, by a continuity argument, we have $\tilde R_\star =1$. This ends the proof.
	
\end{proof}

Now we proceed by providing a quantitative bound for the difference between the parabolically averaged weight $\tilde{\gamma}$ and the unaveraged weight $\gamma$. It will involve the localized one-sided Hardy-Littlewood maximal function of $\mathcal M$.

\begin{lemma}\label{tildegamma-gamma_quant}
	Let $\tau$ be sufficiently large and assume
	$0\le |z-\xi(\tau)|\le 1$. Then
	\[
	\left|1-\frac{\tilde{\gamma}}{\gamma}\right|
	\lesssim \sup_{0<h<1}\frac{1}{h}\int_{\tau-h}^{\tau} \mathcal M(\sigma)\,d\sigma
	+
	o(|z-\xi(\tau)|).
	\]

\end{lemma}

\begin{proof}
	To ease notation in the proof, we will write
	$$
	A_{\mathcal{M}}(\tau)= \sup_{0<h<1}\frac{1}{h}\int_{\tau-h}^{\tau} \mathcal M(\sigma)\,d\sigma .
	$$
	We start from the definition
	\[
	d(z,\xi,\tau)
	=
	\int_{\tau-|z-\xi(\tau)|^{2}}^{\tau}
	|z-\xi(\tau')|
	\frac{1}{|z-\xi(\tau')|^{2}}
	\bar{\chi}\!\left(
	\frac{\tau-\tau'}{|z-\xi(\tau')|^{2}}
	\right)\,d\tau'.
	\]
	Our goal is to prove for $\nu\le |z-\xi(\tau)|\le 1$
	\begin{equation}\label{key}
		|d(z,\xi,\tau)-|z-\xi(\tau)|| 
		\lesssim
		\left( \sup_{0<h<1}\frac{1}{h}\int_{\tau-h}^{\tau} \mathcal M(\sigma)\,d\sigma\right)
		|z-\xi(\tau)|
		+
		o(|z-\xi(\tau)|^{2}).
	\end{equation}
	The conclusion then follows as in Lemma~\ref{lem:adapted_distance} and
	Lemma~\ref{tildegamma-gamma}. \newline 
	We write
	\begin{align*}
		&|d(z,\xi,\tau)-|z-\xi(\tau)|| \\
		&\quad\lesssim
		\frac{1}{|z-\xi(\tau)|^{2}}
		\int_{\tau-|z-\xi(\tau)|^{2}}^{\tau}
		\bigl||z-\xi(\tau')|-|z-\xi(\tau)|\bigr|\,d\tau' \\
		&\qquad+
		|z-\xi(\tau)|
		\left|
		\int_{\tau-|z-\xi(\tau)|^{2}}^{\tau}
		\left(
		\frac{1}{|z-\xi(\tau')|^{2}}
		\bar{\chi}\!\left(
		\frac{\tau-\tau'}{|z-\xi(\tau')|^{2}}
		\right)
		-
		\frac{1}{|z-\xi(\tau)|^{2}}
		\bar{\chi}\!\left(
		\frac{\tau-\tau'}{|z-\xi(\tau)|^{2}}
		\right)
		\right)d\tau'
		\right|.
	\end{align*}
	Using the smoothness of $\bar\chi$ and the fact that
	$|z-\xi(\tau')|\approx |z-\xi(\tau)|$ on the integration region by \eqref{id:slow-variation-nu-xi-refined}, the second
	term is controlled by the first one, and therefore
	\[
	|d(z,\xi,\tau)-|z-\xi(\tau)|| 
	\lesssim
	\frac{1}{|z-\xi(\tau)|^{2}}
	\int_{\tau-|z-\xi(\tau)|^{2}}^{\tau}
	|\xi(\tau')-\xi(\tau)|\,d\tau'.
	\]
	Setting $h=\tau-\tau'$, we obtain
	\[
	|d(z,\xi,\tau)-|z-\xi(\tau)|| 
	\lesssim
	\frac{1}{|z-\xi(\tau)|^{2}}
	\int_{0}^{|z-\xi(\tau)|^{2}}
	|\xi(\tau-h)-\xi(\tau)|\,dh.
	\]
	We split the integral according to $h\le \nu(\tau)^2$ and
	$h\ge \nu(\tau)^2$.
	
	If $h\le \nu(\tau)^2$, the modulation estimate (Lemma \ref{lem:modulation_bounds}) gives
	\[
	|\xi(\tau-h)-\xi(\tau)|
	\lesssim
	\frac{1}{\nu(\tau)}
	\int_{\tau-h}^{\tau}\mathcal{M}(\sigma)\,d\sigma,
	\]
	and therefore
	\begin{align*}
		&\frac{1}{|z-\xi(\tau)|^{2}}
		\int_{0}^{\nu(\tau)^2}
		|\xi(\tau-h)-\xi(\tau)|\,dh
		\lesssim
		\frac{1}{|z-\xi(\tau)|^{2}}
		\frac{1}{\nu(\tau)}
		\int_{0}^{\nu(\tau)^2}
		\int_{\tau-h}^{\tau}\mathcal{M}(\sigma)\,d\sigma\,dh \\
		&\qquad \qquad \qquad \qquad \lesssim
		\frac{\nu(\tau)}{|z-\xi(\tau)|^{2}}
		\int_{\tau-\nu(\tau)^2}^{\tau}\mathcal{M}(\sigma)\,d\sigma \lesssim
		A_{\mathcal{M}}(\tau)\,|z-\xi(\tau)|.
	\end{align*}
	
	If $h\ge \nu(\tau)^2$, we apply the quantitative slow variation estimate (Lemma \ref{lemma:QuantSlowVariation}):
	\[
	|\xi(\tau-h)-\xi(\tau)|
	\lesssim
	\left[
	\frac{1}{h}\int_{\tau-h}^{\tau}\mathcal{M}(\sigma)\,d\sigma
	+
	\mathcal{M}(\tau)
	+
	\mathcal{M}(\tau-h)
	\right]\sqrt{h}
	+
	o(h).
	\]
	Each contribution is estimated directly. The nonlocal term yields
	\[
	\frac{1}{|z-\xi(\tau)|^{2}}
	\int_{\nu(\tau)^2}^{|z-\xi(\tau)|^{2}}
	\frac{1}{\sqrt{h}}
	\int_{\tau-h}^{\tau}\mathcal{M}(\sigma)\,d\sigma\,dh
	\lesssim
	A_{\mathcal{M}}(\tau)\,|z-\xi(\tau)|,
	\]
the contribution of the term $\mathcal{M}(\tau)$ gives
	\[
	\frac{1}{|z-\xi(\tau)|^{2}}
	\int_{\nu(\tau)^2}^{|z-\xi(\tau)|^{2}}
	\sqrt{h}\,\mathcal{M}(\tau)\,dh
	\lesssim
	\mathcal{M}(\tau)\,|z-\xi(\tau)|.
	\]
	Similarly,
	\[
	\frac{1}{|z-\xi(\tau)|^{2}}
	\int_{\nu(\tau)^2}^{|z-\xi(\tau)|^{2}}
	\mathcal{M}(\tau-h)\sqrt{h}\,dh
	\lesssim
	A_{\mathcal{M}}(\tau)\,|z-\xi(\tau)|,
	\]
	and finally
	\[
	\frac{1}{|z-\xi(\tau)|^{2}}
	\int_{\nu(\tau)^2}^{|z-\xi(\tau)|^{2}} o(h)\,dh
	=
	o(|z-\xi(\tau)|^{2}).
	\]
	Collecting all contributions yields \eqref{key}, which concludes the proof.
\end{proof}

\subsection{The ansatz} \label{subsec:ansatz}

In our current decomposition \eqref{self-similarKS-refined-decomposition}, we want to replace the soliton $U_{\nu,\xi}$ par by a more suitable approximate solution. We recall the definition \eqref{matchedsol} of the \emph{matched soliton}
\begin{align}
	\hat{U}_{\nu,\xi} = U_{\nu,\xi}\,\mathcal{X}_{\star}(z), \qquad \mathcal{X}_{\star}(z)= \chi_{\zeta_*}(z)+ \Bigl(1-\chi_{\zeta_*}(z)\Bigr) e^{-\frac{|z|^2}{4}}
\end{align}
and introduce the ansatz
\begin{align}\label{eq:ansatz}
	w = \tilde{U}_{\nu,\xi} + \varepsilon ,
	\qquad
	\tilde{U}_{\nu,\xi} = \alpha(\tau)\,U_{\nu,\xi}\,\mathcal{X}_{\star},
\end{align}
where $\alpha(\tau)$ is adjusted so that $\int \tilde{U}_{\nu,\xi} =8\pi$, see below. Since $\mathcal{X}_{\star}$ is independent of $\tau$, we have
\begin{align}
	\partial_\tau \tilde{U}_{\nu,\xi}
	=
	\alpha_\tau \hat{U}_{\nu,\xi}
	-
	\alpha \frac{\nu_\tau}{\nu}\Lambda_\xi U_{\nu,\xi}\mathcal{X}_{\star}
	-
	\alpha\,\xi_\tau \cdot \nabla U_{\nu,\xi}\mathcal{X}_{\star},
\end{align}
where $\Lambda_\xi u = 2u + (z-\xi)\cdot \nabla u$.

Substituting \eqref{eq:ansatz} into \eqref{self-similarKS-refined} yields
\begin{align}\label{eq:eps0}
	\partial_\tau \varepsilon
	&= \Delta \varepsilon
	- \nabla \cdot (\varepsilon \nabla \Phi_{\tilde{U}_{\nu,\xi}})
	- \nabla \cdot (\tilde{U}_{\nu,\xi} \nabla \Phi_\varepsilon)
	+ \frac{1}{2}\Lambda \varepsilon \nonumber\\
	&\quad
	+ \frac{1}{2}\Lambda \tilde{U}_{\nu,\xi}
	- \partial_\tau \tilde{U}_{\nu,\xi}
	+ \Delta \tilde{U}_{\nu,\xi}
	- \nabla \cdot (\tilde{U}_{\nu,\xi}\nabla \Phi_{\tilde{U}_{\nu,\xi}})
	- \nabla \cdot (\varepsilon \nabla \Phi_\varepsilon).
\end{align}
We now expand each term. Using $\tilde{U}_{\nu,\xi} = \alpha \hat{U}_{\nu,\xi}$, we write
\begin{align*}
	-\nabla \cdot (\varepsilon \nabla \Phi_{\tilde{U}_{\nu,\xi}})
	- \nabla \cdot (\tilde{U}_{\nu,\xi} \nabla \Phi_\varepsilon)
	=
	-\nabla \cdot (\varepsilon \nabla \Phi_{\hat{U}_{\nu,\xi}})
	- \nabla \cdot (\hat{U}_{\nu,\xi} \nabla \Phi_\varepsilon)
	+ A[\varepsilon],
\end{align*}
where $A[\varepsilon]$ is defined below. Next,
\begin{align*}
	\frac{1}{2}\Lambda \tilde{U}_{\nu,\xi}
	&=
	\frac{\alpha}{2}\Lambda_\xi U_{\nu,\xi}\mathcal{X}_{\star}
	+ \frac{\alpha}{2}\,\xi \cdot \nabla U_{\nu,\xi}\mathcal{X}_{\star} 
	+ \frac{\alpha}{2}U_{\nu,\xi}(z-\xi)\cdot \nabla \mathcal{X}_{\star}
	+ \frac{\alpha}{2}U_{\nu,\xi}\xi \cdot \nabla \mathcal{X}_{\star}.
\end{align*}
Moreover
\begin{align*}
	\Delta \tilde{U}_{\nu,\xi}
	=
	\alpha \Delta U_{\nu,\xi} \mathcal{X}_{\star}
	+
	\alpha\Bigl[
	2 \nabla \mathcal{X}_{\star} \cdot \nabla U_{\nu,\xi}
	+ U_{\nu,\xi} \Delta \mathcal{X}_{\star}
	\Bigr].
\end{align*}
For the nonlinear term, using
$\nabla \cdot (f \nabla \Phi_f)
=
\nabla f \cdot \nabla \Phi_f - f^2$,
we obtain
\begin{align*}
	\nabla \cdot (\tilde{U}_{\nu,\xi} \nabla \Phi_{\tilde{U}_{\nu,\xi}})
	=
	\alpha^2
	\left[
	\nabla \hat{U}_{\nu,\xi} \cdot \nabla \Phi_{\hat{U}_{\nu,\xi}}
	- \hat{U}_{\nu,\xi}^2
	\right].
\end{align*}
Decomposing $\hat{U}_{\nu,\xi}=U_{\nu,\xi}\mathcal{X}_{\star}$ yields
\begin{align*}
	\nabla \cdot (\tilde{U}_{\nu,\xi} \nabla \Phi_{\tilde{U}_{\nu,\xi}})
	&=
	\alpha
	\left[
	\nabla U_{\nu,\xi} \cdot \nabla \Phi_{U_{\nu,\xi}}
	- U_{\nu,\xi}^{2}
	\right]\mathcal{X}_{\star} 
	+ \alpha(\alpha-1)
	\left[
	\nabla U_{\nu,\xi} \cdot \nabla \Phi_{U_{\nu,\xi}}
	- U_{\nu,\xi}^{2}
	\right]\mathcal{X}_{\star} \nonumber\\
	&\quad
	- \alpha^{2} \nabla \hat{U}_{\nu,\xi} \cdot \nabla \Phi_{U_{\nu,\xi}(1-\mathcal{X}_{\star})}
	+ \alpha^{2} \hat{U}_{\nu,\xi} U_{\nu,\xi} (1-\mathcal{X}_{\star})
	+ \alpha^{2} U_{\nu,\xi} \nabla \mathcal{X}_{\star} \cdot \nabla \Phi_{U_{\nu,\xi}}.
\end{align*}
Using
$\Delta U_{\nu,\xi}
-
\nabla \cdot (U_{\nu,\xi} \nabla \Phi_{U_{\nu,\xi}}) = 0$,
the leading terms cancel, and we obtain
\begin{align}\label{eq:finaleps}
	\partial_{\tau} \varepsilon-L_{\nu,\xi}\varepsilon
	&=  \Bigl[
	\alpha \Bigl( \frac{\nu_{\tau}}{\nu} + \frac{1}{2} \Bigr)\Lambda_{\xi} U_{\nu,\xi}
	- \alpha_{\tau} U_{\nu,\xi}
	+\alpha \Bigl( \xi_\tau+ \frac{\xi}{2} \Bigr) \cdot \nabla U_{\nu,\xi}
	\Bigr]\mathcal{X}_{\star} \nonumber\\
	&\quad
	+ A[\varepsilon]
	+ A_{\nu,\xi}^{[0]}
	+ B_{\nu,\xi}^{[0]}
	+ B_{\nu,\xi}^{[1]}
	- \nabla \cdot (\varepsilon \nabla \Phi_\varepsilon)
\end{align}
where we recall the linearized operator is
\begin{align*}
	L_{\nu,\xi}[\varepsilon]
	=
	\Delta \varepsilon
	-
	\nabla \cdot(\varepsilon \nabla \Phi_{\hat U_{\nu,\xi}})
	-
	\nabla \cdot(\hat U_{\nu,\xi}\nabla \Phi_\varepsilon)
	+
	\frac12 \Lambda \varepsilon.
\end{align*}
The first line in \eqref{eq:finaleps} corresponds to the linearized operator around the truncated profile $\hat{U}_{\nu,\xi}$, while the second line collects the modulation terms depending only on $(\alpha,\nu,\xi)$. The error terms have been organized according to their origin and symmetry.

The terms denoted by $A$ arise from the fact that $\alpha \neq 1$, and are therefore small in the regime $\alpha \to 1$. The terms denoted by $B$ correspond to boundary contributions generated by the cutoff $\mathcal{X}_{\star}$ and its derivatives, and are thus localized in the transition region away from the core.

We further distinguish radial and non-radial contributions through the superscripts $[0]$ and $[1]$: when $\xi=0$, the terms carrying the superscript $[0]$ are radial, whereas the terms carrying the superscript $[1]$ vanish. More precisely, 
\begin{align}
	A[\varepsilon]
	&=
	(1-\alpha)
	\Bigl[
	\nabla \cdot (\varepsilon \nabla \Phi_{\hat{U}_{\nu,\xi}})
	+ \nabla \cdot (\hat{U}_{\nu,\xi} \nabla \Phi_{\varepsilon})
	\Bigr],
	\label{eq:Aeps}
	\\[0.4em]
	A_{\nu,\xi}^{[0]}
	&=
	\alpha(1-\alpha)
	\left[
	\nabla U_{\nu,\xi} \cdot \nabla \Phi_{U_{\nu,\xi}}
	- U_{\nu,\xi}^{2}
	\right]\mathcal{X}_{\star},
	\label{eq:A0}
	\\[0.4em]
	B_{\nu,\xi}^{[0]}
	&=
	\alpha \Bigl[
	\frac{1}{2} U_{\nu,\xi}\cdot \nabla \mathcal{X}_{\star}
	+ 2 \nabla \mathcal{X}_{\star} \cdot \nabla U_{\nu,\xi}
	+ U_{\nu,\xi} \Delta \mathcal{X}_{\star}
	\nonumber\\
	&\hspace{3.5em}
	+ \nabla \tilde{U}_{\nu,\xi} \cdot \nabla \mathcal{R}_{\nu,\xi}
	- \tilde{U}_{\nu,\xi} U_{\nu,\xi} (1-\mathcal{X}_{\star})
	- \alpha U_{\nu,\xi} \nabla \mathcal{X}_{\star} \cdot \nabla \Phi_{U_{\nu,\xi}}
	\Bigr],
	\label{eq:B0}
	\\[0.4em]
	B_{\nu,\xi}^{[1]}
	&=
	\frac{\alpha}{2}U_{\nu,\xi}\,\xi \cdot \nabla \mathcal{X}_{\star}.
	\label{eq:B1}
\end{align}
where we used the notation
\begin{equation}\label{eq:Rdef}
	\nabla\mathcal{R}_{\nu,\xi}
	=
	\nabla \Phi_{U_{\nu,\xi}-\hat{U}_{\nu,\xi}}
\end{equation}
the error induced by truncating the profile.
\subsubsection{The $8\pi$ normalization}

We fix the parameter $\alpha$ by imposing the mass constraint
\begin{align}\label{normalization-alpha}
	\int \tilde{U}_{\nu,\xi} = 8\pi.
\end{align}

\begin{lemma}\label{lem:alpha}
	Under the above normalization one has 
	\begin{align}
		|\alpha-1|\lesssim\nu^2=o(e^{-\tau}) \quad \mbox{and}  \quad |\alpha_{\tau}|\lesssim |\nu \nu_{\tau}|+\nu^{2}|\xi_{\tau}|=o(1). 
	\end{align}
\end{lemma}

\begin{proof}
	As $\int U_{\nu,\xi}=8\pi$ and $U_{\nu,\xi}(z)\lesssim \nu^2 |z|^{-4}$ for $|z|\gtrsim 1$, we infer $\int U_{\nu,\xi}\mathcal{X}_{\star}=8\pi + O(\nu^2)$. It follows that $|\alpha-1| = O(\nu^2)$. Differentiating the constraint, we obtain
	\begin{align*}
		\alpha_\tau \int U_{\nu,\xi}\mathcal{X}_{\star}
		+
		\alpha \int \left(-\frac{\nu_\tau}{\nu}\Lambda_\xi U_{\nu,\xi}
		-
		\xi_\tau \cdot \nabla U_{\nu,\xi}\right)\mathcal{X}_{\star} = 0.
	\end{align*}
	As $\int\Lambda_\xi U_{\nu,\xi}=\int \partial_{z_i}U_{\nu,\xi}= 0$ for $i=1,2$ and $|\Lambda_\xi U_{\nu,\xi}|\lesssim \nu^2 |z|^{-4}$ and $|\partial_{z_i}U_{\nu,\xi}|\lesssim \nu^2|z|^{-5}$ for $|z|\gtrsim 1$, we infer $|\alpha_\tau|\lesssim|\nu \nu_\tau| + \nu^2 |\xi_\tau|$. The conclusion follows from Lemma~\ref{lem:selfsimilarquali} and the estimate $\nu = o(e^{-\tau/2})$.
\end{proof}

In energy estimates for $\varepsilon$ later on, a crossed term will involve the motion of the soliton by scaling. A cancellation will be obtained from the fact that the weight in the matched scalar product \eqref{def:omeganuxi} is to leading order the inverse of the tail of the matched soliton.

\begin{lemma}\label{lem:propmathfrakm}
	Let
	\begin{align}
		\mathfrak{w}_{\star}(z)
		=
		\mathcal{X}_{\star}(z)
		\Bigl[
		\chi_{\zeta_*}(z)
		+
		\Bigl(1 - \chi_{\zeta_*}(z)\Bigr)
		e^{\frac{|z|^{2}}{4}}
		\Bigr].
	\end{align}
	Then if $\zeta_*$ is sufficiently small we have
	\begin{align}
		|\mathfrak{w}_{\star}(z) - 1|
		=
		\begin{cases}
			0, & \text{if } |z| \le \zeta_* \ \text{or} \ |z| \ge 2\zeta_*, \\
			O(\zeta_*^{2}), & \text{otherwise}.
		\end{cases}
	\end{align}
\end{lemma}

\begin{proof}
	Recalling that $\mathcal{X}_{\star}(z)=\chi_{\zeta_\star}+( 1 - \chi_{\zeta_\star} )e^{-|z|^2/4}$, the claim is immediate in the regions $|z| \le \zeta_*$ and $|z| \ge 2\zeta_*$. In the transition region $\zeta_* \le |z| \le 2\zeta_*$, it follows from the Taylor expansion of the Gaussian.
	
\end{proof}
We now show that the pointwise bounds obtained in Lemma~\ref{lem:selfsimilarquali} transfer to the perturbation $\varepsilon$.

\begin{lemma}\label{selfsimilarqualieps}
	One has, asymptotically as $\tau \to \infty$,
	\begin{align} \label{bd:pointwise-varepsilon}
		\| \varepsilon\|_{L^1}=o(1), \quad |\varepsilon(z,\tau)|
		=
		o\!\left(\frac{1}{\nu^2 + |z-\xi|^2}\right)
		\quad \mbox{and} \quad
		|\nabla \Phi_{\varepsilon}(z,\tau)|
		=
		o\!\left(\frac{1}{\nu + |z-\xi|}\right).
	\end{align}
\end{lemma}

\begin{proof}
	Comparing the two expressions \eqref{self-similarKS-refined-decomposition} and \eqref{eq:ansatz}, we obtain
	\begin{align}\label{eq:eps_decomp}
		\varepsilon
		=
		U_{\nu,\xi}(1-\mathcal{X}_{\star})
		-
		(\alpha-1)\hat{U}_{\nu,\xi}
		+
		\tilde{w}.
	\end{align}
	The first term in \eqref{eq:eps_decomp} is supported in the exterior region $|z|\gtrsim 1$ and satisfies the same decay as $U_{\nu,\xi}$, which is $O(\nu^2|z|^{-4})=o(|z|^{-4})$. The second term is pointwise bounded by $\nu^4(\nu^2+|z-\xi|^2)^{-4}$ thanks to the estimate on $\alpha$ given in Lemma~\ref{lem:alpha}. The third term satisfies $\| \tilde w\|_{L^1}=o(1)$ and was estimated pointwise in Lemma~\ref{lem:selfsimilarquali}. These bounds directly imply the first and second estimates in \eqref{bd:pointwise-varepsilon}. Moreover, the second estimate in \eqref{bd:pointwise-varepsilon} was proved to imply the third one in Corollary \ref{Coro:PointPoisson}, concluding the proof of the Lemma.
\end{proof}

\subsection{Decay of the remainder in the far away outer zone}\label{subsec:outer}

\subsubsection{Inner--outer gluing system}
We set
\begin{equation}\label{eq:F_nuxi_def}
	\begin{aligned}
		F_{\nu,\xi}
		:=
		&\Biggl[
		\alpha\left(\frac{\nu_\tau}{\nu}+\frac12\right)\Lambda_\xi U_{\nu,\xi}
		-
		\alpha_\tau U_{\nu,\xi}
		+
		\alpha\left(\xi_\tau+\frac{\xi}{2}\right)\cdot\nabla U_{\nu,\xi}
		\Biggr]\mathcal X_\star
		+
		A_{\nu,\xi}^{[0]}
		+
		B_{\nu,\xi}^{[0]}
		+
		B_{\nu,\xi}^{[1]}.
	\end{aligned}
\end{equation}
so by \eqref{eq:finaleps} the perturbation equation can be written as
\begin{equation}\label{eq:eps_simplified}
	\partial_\tau \varepsilon
	=
	L_{\nu,\xi}[\varepsilon]
	+
	A[\varepsilon]
	+
	F_{\nu,\xi}
	-
	\nabla \cdot (\varepsilon \nabla \Phi_\varepsilon).
\end{equation}
We recall that there are four zones at stake: the inner region $|z-\xi|\lesssim \nu$, the nearby parabolic region $ \nu \ll |z-\xi| \l 1$, the outer parabolic region $|z|\approx 1$ and the far away outer region $|z|\gg 1$. Each of the first three zones cannot be treated separately, because diffusion spreads the solution across these zones. At a technical level, localizing the evolution in such a zone would generate uncontrollable boundary terms. To overcome this issue, we use the matched scalar product which defines a globally defined Lyapunov functional over these three zones, thus covering the full parabolic zone $|z|\lesssim 1$. However, the matched scalar product cannot be extended up to the far away outer region $|z|\gg 1$; this is because the functional frameworks for the outer parabolic zone (Proposition \ref{lem:far-away-outer}) and that for the far away outer zone (Lemma \ref{lem:far-away-outer}) cannot be matched.

We remark that, by diffusive effects, in the far away outer region $|z|\gg 1$ the solution moves towards the parabolic region $|z|\lesssim 1$, but that the far away outer region remains isolated to leading order. We will therefore localize the evolution in the far away outer zone $|z|\gg 1$, thanks to an \emph{inner-outer gluing scheme}. We look for a solution under the form
\begin{equation}\label{eq:inner_outer_decomp}
	\varepsilon
	=
	\varepsilon_{\mathrm{in}}
	+
	(1-\chi_2)\varepsilon_{\mathrm{out}}.
\end{equation}
Above, $\varepsilon_{\mathrm{in}}$ can be interpreted as the remainder in the parabolic region $|z|\lesssim 1$, and $\varepsilon_{\mathrm{out}}$ as that in the far away outer region $|z|\gg 1$. This will transform \eqref{eq:inner_outer_decomp} below into a triangular system, where $\varepsilon_{\mathrm{in}}$ is forced by boundary terms generated by $\varepsilon_{\mathrm{out}}$, but $\varepsilon_{\mathrm{out}}$ does not receive boundary terms from $\varepsilon_{\mathrm{in}}$. 

Direct computation give
\begin{align*}
	& \Delta\bigl[(1-\chi_2)\varepsilon_{\mathrm{out}}\bigr]
	=
	(1-\chi_2)\Delta\varepsilon_{\mathrm{out}}
	-
	2\nabla\chi_2\cdot\nabla\varepsilon_{\mathrm{out}}
	-
	\varepsilon_{\mathrm{out}}\Delta\chi_2, \\
	& \Lambda\bigl[(1-\chi_2)\varepsilon_{\mathrm{out}}\bigr]
	=
	(1-\chi_2)\Lambda\varepsilon_{\mathrm{out}}
	-
	(z\cdot\nabla\chi_2)\varepsilon_{\mathrm{out}},\\
		&-\nabla\cdot\bigl((1-\chi_2)\varepsilon_{\mathrm{out}}
		\nabla\Phi_{\tilde U_{\nu,\xi}}\bigr)
		=
		-(1-\chi_2)\nabla\cdot
		\bigl(\varepsilon_{\mathrm{out}}\nabla\Phi_{\tilde U_{\nu,\xi}}\bigr)
		+
		\varepsilon_{\mathrm{out}}\nabla\chi_2\cdot\nabla\Phi_{\tilde U_{\nu,\xi}}.
\end{align*}
Thus
\begin{equation}\label{eq:outer_linear_expansion}
	\begin{aligned}
		&L_{\nu,\xi}\bigl[(1-\chi_2)\varepsilon_{\mathrm{out}}\bigr]
		+
		A\bigl[(1-\chi_2)\varepsilon_{\mathrm{out}}\bigr]
		\\
		&\qquad
		=
		(1-\chi_2)
		\left[
		\Delta\varepsilon_{\mathrm{out}}
		+
		\frac12\Lambda\varepsilon_{\mathrm{out}}
		-
		\nabla\cdot
		\bigl(\varepsilon_{\mathrm{out}}\nabla\Phi_{\tilde U_{\nu,\xi}}\bigr)
		\right]
		+
		B_{\mathrm{lin}}[\varepsilon_{\mathrm{out}}]
		-
		\nabla\cdot
		\bigl(\tilde U_{\nu,\xi}
		\nabla\Phi_{(1-\chi_2)\varepsilon_{\mathrm{out}}}\bigr),
	\end{aligned}
\end{equation}
where $B_{\mathrm{lin}}[\varepsilon_{\mathrm{out}}]$ gathers boundary terms produced by $\varepsilon_{\mathrm{out}}$,
\begin{equation}\label{eq:Blin_out_def}
	B_{\mathrm{lin}}[\varepsilon_{\mathrm{out}}]
	=
	-2\nabla\chi_2\cdot\nabla\varepsilon_{\mathrm{out}}
	-
	\varepsilon_{\mathrm{out}}\Delta\chi_2
	-
	\frac12(z\cdot\nabla\chi_2)\varepsilon_{\mathrm{out}}
	+
	\varepsilon_{\mathrm{out}}\nabla\chi_2\cdot\nabla\Phi_{\tilde U_{\nu,\xi}}.
\end{equation}
We now extract the limiting outer operator. Since the mass of the profile is $8\pi$,
the far-field drift is $\nabla\Phi_{\tilde U_{\nu,\xi}}=-4|z|^{-2}z+o\left(|z|^{-1}\right)$. Then
\begin{equation}\label{eq:outer_operator_decomp}
	\begin{aligned}
		&\Delta\varepsilon_{\mathrm{out}}
		+
		\frac12\Lambda\varepsilon_{\mathrm{out}}
		-
		\nabla\cdot
		\bigl(\varepsilon_{\mathrm{out}}\nabla\Phi_{\tilde U_{\nu,\xi}}\bigr)
		=
		L_\infty[\varepsilon_{\mathrm{out}}]
		-
		\nabla\varepsilon_{\mathrm{out}}\cdot
		\left(
		\nabla\Phi_{\tilde U_{\nu,\xi}}
		+
		4\frac{z}{|z|^2}
		\right)
		+
		\tilde U_{\nu,\xi}\varepsilon_{\mathrm{out}}.
	\end{aligned}
\end{equation}
where we recall $L_\infty$ is given by \eqref{def:shiftedextlinoperator}. Furthermore,
\begin{equation}\label{eq:nonlocal_outer_split}
	\begin{aligned}
		\nabla\cdot
		\bigl(\tilde U_{\nu,\xi}
		\nabla\Phi_{(1-\chi_2)\varepsilon_{\mathrm{out}}}\bigr)
		&=
		\nabla\tilde U_{\nu,\xi}\cdot
		\nabla\Phi_{(1-\chi_2)\varepsilon_{\mathrm{out}}}
		-
		(1-\chi_2)\tilde U_{\nu,\xi}\varepsilon_{\mathrm{out}}.
	\end{aligned}
\end{equation}
Combining \eqref{eq:outer_operator_decomp} and \eqref{eq:nonlocal_outer_split}, the outer contribution becomes
\begin{align*}
	\begin{aligned}
		L_{\nu,\xi}\bigl[(1-\chi_2)\varepsilon_{\mathrm{out}}\bigr]
		+
		A\bigl[(1-\chi_2)\varepsilon_{\mathrm{out}}\bigr] =&(1-\chi_2)
		\Biggl[
		L_\infty[\varepsilon_{\mathrm{out}}]
		-
		\nabla\varepsilon_{\mathrm{out}}\cdot
		\left(
		\nabla\Phi_{\tilde U_{\nu,\xi}}
		+
		4\frac{z}{|z|^2}
		\right)
		+
		2\tilde U_{\nu,\xi}\varepsilon_{\mathrm{out}}\Biggr]\\
		&+B_{\mathrm{lin}}[\varepsilon_{\mathrm{out}}]-
		\nabla\tilde U_{\nu,\xi}\cdot
		\nabla\Phi_{(1-\chi_2)\varepsilon_{\mathrm{out}}}.
	\end{aligned}
\end{align*}
We split the quadratic term according to the inner--outer decomposition. Since
\[
-\nabla\cdot(\varepsilon\nabla\Phi_\varepsilon)
=
-\nabla\varepsilon\cdot\nabla\Phi_\varepsilon
+
\varepsilon^2,
\]
we have
\begin{align*}
	-\nabla\cdot(\varepsilon\nabla\Phi_\varepsilon)
	=
	Q_{\mathrm{in}}[\varepsilon]
	+
	(1-\chi_2)
	\left[
	-\nabla\varepsilon_{\mathrm{out}}\cdot\nabla\Phi_\varepsilon
	+
	\varepsilon_{\mathrm{out}}\varepsilon
	\right],
\end{align*}
where
\begin{equation}\label{eq:Qin_def}
	Q_{\mathrm{in}}[\varepsilon]
	=
	-\nabla\varepsilon_{\mathrm{in}}\cdot\nabla\Phi_\varepsilon
	+
	\varepsilon_{\mathrm{out}}\nabla\chi_2\cdot\nabla\Phi_\varepsilon
	+
	\varepsilon_{\mathrm{in}}\varepsilon.
\end{equation}
We collect the boundary terms in
\begin{equation}\label{eq:Bout_def}
	\begin{aligned}
		B[\varepsilon_{\mathrm{out}}]
		:=
		&-2\nabla\chi_2\cdot\nabla\varepsilon_{\mathrm{out}}
		-
		\varepsilon_{\mathrm{out}}\Delta\chi_2
		-
		\frac12(z\cdot\nabla\chi_2)\varepsilon_{\mathrm{out}}
		+
		\varepsilon_{\mathrm{out}}\nabla\chi_2\cdot\nabla\Phi_{\tilde U_{\nu,\xi}}
		-\nabla\tilde U_{\nu,\xi}\cdot
		\nabla\Phi_{(1-\chi_2)\varepsilon_{\mathrm{out}}}.
	\end{aligned}
\end{equation}
Finally, define
\begin{equation}\label{eq:bc_e_def}
	b_\varepsilon(z,\tau)
	:=
	-
	\left(
	\nabla\Phi_{\tilde U_{\nu,\xi}}
	+
	4\frac{z}{|z|^2}
	+
	\nabla\Phi_\varepsilon
	\right),
	\qquad
	c_\varepsilon(z,\tau)
	:=
	2\tilde U_{\nu,\xi}
	+
	\varepsilon.
\end{equation}
With this notation, the equation becomes
\begin{equation}\label{eq:gluing_identity}
	\begin{aligned}
		\partial_\tau\varepsilon_{\mathrm{in}}
		+
		(1-\chi_2)\partial_\tau\varepsilon_{\mathrm{out}}
		=
		&L_{\nu,\xi}[\varepsilon_{\mathrm{in}}]
		+
		A[\varepsilon_{\mathrm{in}}]
		+
		F_{\nu,\xi}
		+
		B[\varepsilon_{\mathrm{out}}]
		+
		Q_{\mathrm{in}}[\varepsilon]
		\\
		&+
		(1-\chi_2)
		\Bigl[
		L_\infty[\varepsilon_{\mathrm{out}}]
		+
		b_\varepsilon\cdot\nabla\varepsilon_{\mathrm{out}}
		+
		c_\varepsilon\varepsilon_{\mathrm{out}}
		\Bigr].
	\end{aligned}
\end{equation}

We are naturally led to define $(\varepsilon_{\mathrm{in}},\varepsilon_{\mathrm{out}})$ as the solution to the following inner--outer gluing system:
\begin{equation}\label{eq:inner_problem}
	\begin{cases}
		\partial_\tau\varepsilon_{\mathrm{in}}
		=
		L_{\nu,\xi}[\varepsilon_{\mathrm{in}}]
		+
		A[\varepsilon_{\mathrm{in}}]
		+
		F_{\nu,\xi}
		+
		B[\varepsilon_{\mathrm{out}}]
		+
		Q_{\mathrm{in}}[\varepsilon],
		\\[0.4em]
		\varepsilon_{\mathrm{in}}(\tau_0,z)
		=
		\varepsilon_0(z)\chi_4(z).
	\end{cases}
\end{equation}
and
\begin{equation}\label{eq:outer_problem}
	\begin{cases}
		\partial_\tau\varepsilon_{\mathrm{out}}
		=
		L_\infty[\varepsilon_{\mathrm{out}}]
		+
		b_\varepsilon\cdot\nabla\varepsilon_{\mathrm{out}}
		+
		c_\varepsilon\varepsilon_{\mathrm{out}},
		&
		(z,\tau)\in B_1(0)^c\times(\tau_0,\infty),
		\\[0.4em]
		\varepsilon_{\mathrm{out}}(\tau,z)=0,
		&
		z\in\partial B_1(0),
		\\[0.4em]
		\varepsilon_{\mathrm{out}}(\tau_0,z)
		=
		\varepsilon_0(z)(1-\chi_4(z)).
	\end{cases}
\end{equation}
where we recall the notation for the ball $B_1(0)=\{|z|\leq 1\}$. We verify that
\begin{equation}\label{eq:initial_data_compatibility}
	\varepsilon_{\mathrm{in}}(\tau_0)
	+
	(1-\chi_2)\varepsilon_{\mathrm{out}}(\tau_0)
	=
	\varepsilon_0\chi_4
	+
	(1-\chi_2)\varepsilon_0(1-\chi_4)
	=
	\varepsilon_0.
\end{equation}
Therefore, if the system \eqref{eq:inner_problem}--\eqref{eq:outer_problem} is solved, then $\varepsilon_{\mathrm{in}}+(1-\chi_2)\varepsilon_{\mathrm{out}}$ solves the original perturbation equation \eqref{eq:eps_simplified}, so by uniqueness, we indeed have
\[
\varepsilon
=
\varepsilon_{\mathrm{in}}
+
(1-\chi_2)\varepsilon_{\mathrm{out}}.
\]

Finally, by the pointwise bounds on $\varepsilon$ (Lemma \ref{selfsimilarqualieps}) and the far-field expansion of
$\nabla\Phi_{\tilde U_{\nu,\xi}}$ (Lemma \ref{lem:matched_errors_global}), we have, for any prescribed $\eta>0$, that there exists $\tau_0$ sufficiently large,  such that for all $\tau\geq \tau_0$ and $z\in \mathbb R^2$,
\begin{equation}\label{eq:b_eps_small}
	|b_\varepsilon(z,\tau)|
	\le
	\frac{\eta}{|z|},
	\qquad
	|\nabla\cdot b_\varepsilon(z,\tau)|
	\le
	\frac{\eta}{|z|^2},
\end{equation}
and
\begin{equation}\label{eq:c_eps_small}
	|c_\varepsilon(z,\tau)|
	\le
	\frac{\eta}{|z|^2}.
\end{equation}

\subsubsection{The outer problem with Dirichlet boundary condition}

In what follows we study the evolution \eqref{eq:outer_problem} of $\varepsilon_{\mathrm{out}}$, on
$$
\Omega=\{|z|> 1\}.
$$

\begin{proposition}[Coercivity of the perturbed exterior operator]\label{prop:ExteriorCoercivityPerturbed}
	Recall $\omega_\infty$ is defined by \eqref{def:omega-infty}. Then the following holds for all $\delta>0$ small enough. Assume that $b$ and $c$ are smooth functions on $\bar \Omega$ such that
	\[
	|b(z)|\le \frac{\eta}{|z|},
	\qquad
	|c(z)|\le \frac{\eta}{|z|^2}
	\qquad \text{for } z\in \bar \Omega,
	\]
	with \(\eta>0\) sufficiently small depending on $\delta$. Let \(u\in C^\infty_c(\bar \Omega )\) be such that \(u=0\) on \(\partial B_1\).  Then, one has
	\[
	\left\langle
	L_\infty u+b\cdot\nabla u+cu,u
	\right\rangle_{L^2_\infty(\Omega)}
	\le
	(-2+4\delta)\|u\|_{L^2_\infty(\Omega)}^2
	-
	\frac{\delta}{2}
	\|\nabla u\|_{L^2_\infty(\Omega)}^2.
	\]
\end{proposition}

\begin{proof}
	The coercivity of the unperturbed operator \(L_\infty\) follows exactly as in Proposition~\ref{prop:ExteriorCoercivity}. The same computations apply on $\Omega$, since the boundary term vanishes due to the Dirichlet condition \(u=0\) on \(\partial B_1\). In particular, for any \(0<\delta\ll1\),
	\begin{equation}\label{eq:coerc_ext_unperturbed}
		\langle L_\infty u,u\rangle_{L^2_\infty(\Omega)}
		\le
		(-2+3\delta)\|u\|_{L^2_\infty(\Omega)}^2
		-
		\delta\|\nabla u\|_{L^2_\infty(\Omega)}^2.
	\end{equation}
	It remains to estimate the perturbative terms. Using the assumptions on \(b\) and \(c\), and the fact that \(|z|\ge1\) in \(\Omega\), we obtain
	\[
	\left|
	\int_{\Omega}(b\cdot\nabla u)u\,\omega_\infty\,dz
	\right|+\left|
	\int_{\Omega}cu^2\,\omega_\infty\,dz
	\right|
	\le
	\eta \int_{\Omega} |\nabla u|\,|u|\,\omega_\infty\,dz+\eta \int_{\Omega} u^2  \omega_\infty dz
	\le
	\frac{\delta}{2}\|\nabla u\|_{L^2_\infty(\Omega)}^2
	+
	C_\delta\eta^2\|u\|_{L^2_\infty(\Omega)}^2,
	\]
	where we used Young's inequality. Choosing \(\eta>0\) sufficiently small yields the claimed inequality.
	
\end{proof}

\begin{proposition}[Contraction of the exterior evolution]\label{prop:ExteriorEvolutionContraction}
	For all $\eta>0$ small enough we have the following. Assume that for some $\tau_0*\geq 0$, $b$ and $c$ are smooth functions defined on $\bar \Omega \times [0,\infty)$ such that
	\[
	|b(z,\tau)|\le \frac{\eta}{|z|},
	\qquad
	|\nabla \cdot b(z,\tau)|+|c(z,\tau)|
	\le
	\frac{\eta}{|z|^2}
	\quad \text{in }\bar \Omega.
	\]
	for all $\tau\geq \tau_0*$. Consider the operator
	\[
	L(\tau)u
	=
	L_\infty u
	+
	b(z,\tau)\cdot\nabla u
	+
	c(z,\tau)u,
	\]
	with homogeneous Dirichlet boundary condition on $\partial B_1$, and denote by $S(\tau,\tau_0)$ its solution map. Then it satisfies
	\begin{equation} \label{bd:semi-group-contraction-outer}
	\|S(\tau,\tau_0)f\|_{L^1(\Omega,|z|^2dz)}
	\le
	e^{-(\tau-\tau_0)}
	\|f\|_{L^1(\Omega,|z|^2dz)},
	\qquad
	\tau\ge \tau_0.
	\end{equation}
\end{proposition}

\begin{proof}
	We first consider smooth compactly supported initial data \(u(\tau_0)=f\), so that the corresponding solution is classical. The general case follows by density and the contraction estimate below.
	
	We use a regularized version of the standard Kato inequality. Let
	\[
	\Psi_\delta(r)=\sqrt{r^2+\delta^2}-\delta,
	\qquad
	w_\delta=\Psi_\delta(u).
	\]
	Then \(w_\delta\to |u|\), \(0\le w_\delta\le |u|\), and \(\Psi_\delta''\ge0\). Writing
	\[
	L_\infty=A+\mathrm{Id},
	\qquad
	A=\Delta+\left(4\frac{z}{|z|^2}+\frac12 z\right)\cdot\nabla,
	\]
	the convexity of \(\Psi_\delta\) gives
	\[
	A w_\delta
	\ge
	\Psi_\delta'(u)Au.
	\]
	Multiplying the equation by \(\Psi_\delta'(u)\), we obtain
	\[
	\partial_\tau w_\delta
	\le
	L_\infty w_\delta
	+
	b\cdot\nabla w_\delta
	+
	|c|w_\delta
	+
	C\delta(1+|c|).
	\]
	
	Let
	\[
	\phi_R(z)=|z|^2\chi\left(\frac{|z|}{R}\right),
	\]
	with \(\chi\) a standard radial cutoff. Since \(u=0\) on \(\partial B_1\), also \(w_\delta=0\) on \(\partial B_1\), and no boundary contribution appears. Thus
	\[
	\frac{d}{d\tau}
	\int_\Omega w_\delta \phi_R
	\le
	\int_\Omega w_\delta L_\infty^\ast\phi_R
	-
	\int_\Omega w_\delta \nabla\cdot(b\phi_R)
	+
	\int_\Omega |c|w_\delta\phi_R
	+
	C\delta\int_\Omega (1+|c|)\phi_R.
	\]
	Here the adjoint is taken with respect to Lebesgue measure and is given by
	\[
	L_\infty^\ast\phi
	=
	\Delta\phi
	-
	\nabla\cdot\left[
	\left(
	4\frac{z}{|z|^2}
	+
	\frac12 z
	\right)\phi
	\right]
	+
	\phi.
	\]
	In particular,
	\[
	L_\infty^\ast(|z|^2)
	=
	-|z|^2-4.
	\]
	Passing first to the limit \(\delta\to0\), using dominated convergence on compact sets, and then letting \(R\to\infty\), we obtain
	\[
	\frac{d}{d\tau}
	\int_\Omega |u||z|^2
	\le
	-\int_\Omega |u||z|^2
	-
	4\int_\Omega |u|
	+
	C\eta\int_\Omega |u|\le
	-\int_\Omega |u||z|^2
	-
	2\int_\Omega |u|
	\]
	where the last inequality holds provided $\eta>0$ is small enough. This differential inequality implies the desired estimate \eqref{bd:semi-group-contraction-outer} by Gronwall. The extension to arbitrary data in \(L^1(\Omega,|z|^2dz)\) follows by density.
\end{proof}

\begin{proposition}[Outer moment and pointwise decay]\label{prop:OuterMomentPointwise}
	Let $\varepsilon_{\mathrm{out}}$ solve
	\begin{equation}\label{eq:OuterEqRiv}
		\begin{cases}
			\partial_\tau \varepsilon_{\mathrm{out}}
			=
			L_\infty[\varepsilon_{\mathrm{out}}]
			+
			b\cdot\nabla\varepsilon_{\mathrm{out}}
			+
			c\varepsilon_{\mathrm{out}},
			&
			(z,\tau)\in \Omega\times(\tau_0,\infty),
			\\[0.3em]
			\varepsilon_{\mathrm{out}}=0,
			&
			z\in\partial B_1,
		\end{cases}
	\end{equation}
	where we recall $L_\infty$ is given by \eqref{def:extlinoperator}, and where $b$ and $c$ are smooth functions defined on $\bar\Omega \times [\tau_0^*,\infty)$ that satisfy for all $\tau\geq \tau_0^*$,
	\[
	|b(z,\tau)|\le \frac{\eta}{|z|},
	\qquad
	|\nabla\cdot b(z,\tau)|+|c(z,\tau)|
	\le
	\frac{\eta}{|z|^2},
	\qquad z\in \bar \Omega,
	\]
	with $\eta>0$ sufficiently small. Assume that the initial datum at time $\tau_0\geq \tau_0^*$ is smooth and bounded on $\bar \Omega$ and satisfies
	\[
	\varepsilon_{\mathrm{out}}(\tau_0)\in
	L^\infty(\Omega)\cap L^1(\Omega,|z|^2dz).
	\]
	Then
	\begin{equation}\label{SecMomEst}
		\int_\Omega
		|\varepsilon_{\mathrm{out}}(z,\tau)|\,|z|^2\,dz
		=
		o(e^{-\tau}),
		\qquad
		\|\varepsilon_{\mathrm{out}}(\tau)\|_{L^\infty(\Omega)}
		=
		o(e^{-\tau}),
	\end{equation}
	and, for every fixed $R>1$,
	\begin{equation}\label{PointEst}
		\|\nabla \varepsilon_{\mathrm{out}}(\tau)\|_{L^\infty(\Omega\cap B_R)}
		=
		o_R(e^{-\tau}).
	\end{equation}
\end{proposition}

\begin{proof}
	We recall $\omega_\infty = |z|^4 e^{|z|^2/4}$ and set
	\[
	X:=L^1(\Omega,|z|^2dz),
	\]
	\textbf{Step 1.} \emph{Weighted $L^1$ decay.}
	We first prove the weighted $L^1$ decay. Fix $M\gg1$ and split
	\[
	\varepsilon_{\mathrm{out}}(\tau_0)
	=
	\varepsilon_{\mathrm{out}}(\tau_0)\chi_M
	+
	\varepsilon_{\mathrm{out}}(\tau_0)(1-\chi_M)
	=:f_1+f_2.
	\]
	Let $\varepsilon_1,\varepsilon_2$ be the corresponding solutions of
	\eqref{eq:OuterEqRiv}. By linearity, $\varepsilon_{\mathrm{out}}=\varepsilon_1+\varepsilon_2$. Since $f_1$ is compactly supported, $f_1\in L^2(\Omega,\omega_\infty dz)$. Using
	Proposition~\ref{prop:ExteriorCoercivityPerturbed} and taking $\eta$ small enough, we obtain
	\[
	\frac{d}{d\tau}
	\|\varepsilon_1(\tau)\|_{L^2(\Omega,\rho dz)}^2
	\le
	-3
	\|\varepsilon_1(\tau)\|_{L^2(\Omega,\rho dz)}^2.
	\]
	Hence
	\[
	\|\varepsilon_1(\tau)\|_{L^2(\Omega,\rho dz)}
	\lesssim_M
	e^{-\frac32(\tau-\tau_0)}.
	\]
	By Cauchy's inequality, and since $3/2>1$, we deduce
	\[
	\|\varepsilon_1(\tau)\|_X
	\lesssim
	\|\varepsilon_1(\tau)\|_{L^2(\Omega,\rho dz)}
	=
	o(e^{-\tau}).
	\]
	For $\varepsilon_2$, Proposition~\ref{prop:ExteriorEvolutionContraction} and the fact that $\|f_2\|_X\to 0$ as $M\to \infty$ give
	\[
	\|\varepsilon_2(\tau)\|_X
	\le
	e^{-(\tau-\tau_0)}
	\|f_2\|_X \lesssim o_{M\to \infty}(1)e^{-\tau}.
	\]
	Combining the two inequalities for $\varepsilon_1 $ and $\varepsilon_2$, as $M\gg1$ is arbitrary, shows the first bound in \eqref{SecMomEst}.
	
	\smallskip
	
	\noindent\textbf{Step 2.} \emph{$L^\infty$ and local gradient decay via parabolic smoothing.}
	We next prove the $L^\infty$ estimate. Since $\varepsilon_{\mathrm{out}}(\tau_0)$, $b$ and $c$ are smooth and bounded, by maximum principle, we have $\| \varepsilon_{\mathrm{out}}(\tau)\|_{L^\infty(\Omega)}\lesssim 1$ for $\tau\in [\tau_0,\tau_0+1]$. For $\tau\geq \tau_0+1$, applying Duhamel's formula on $[\tau-1,\tau]$, and using the weighted $L^1$ decay \eqref{SecMomEst} together with the short-time $L^1\to L^\infty$ smoothing estimate for the exterior evolution (obtained by removing the drift term in $L_\infty$ through the standard self-similar change of variables and then applying Gaussian heat kernel bounds and the maximum principle), we obtain
	\[
	\|\varepsilon_{\mathrm{out}}(\tau)\|_{L^\infty(\Omega)}
	=
	o(e^{-\tau}).
	\]
	The local gradient estimate \eqref{PointEst} follows similarly by applying local interior parabolic estimates on
	$[\tau-\frac12,\tau]\times(\Omega\cap B_{R'})$. Since all coefficients are
	bounded and smooth, and the source term is now controlled in $L^\infty$ by the above inequality.
	
\end{proof}

\subsection{Decay of the remainder in the parabolic zone}\label{subsec:inner}

We begin by collecting some preliminary lemmas. We recall the weight obtained by gluing together the interior and exterior weighted $L^{2}$ structures. Let $\zeta_* > 0$ be fixed, and recall the definition \eqref{def:omeganuxi} of the matched weight
\begin{align*}
	\omega(z)
	=
	\tilde{\gamma}(z)\,\chi_{\zeta_*}(z)
	+
	\Bigl(1-\chi_{\zeta_*}(z)\Bigr)
	\frac{1}{8}|z|^{4}e^{\frac{|z|^{2}}{4}}.
\end{align*}
We also recall the definition of the matched scalar product
\begin{align}
	\langle f ,g  \rangle_*
	:=
	\int_{\mathbb{R}^2} f g\,\omega\,dz
	-
	\nu^{2}\int_{\mathbb{R}^2} \nabla \Phi_{f}\cdot \nabla \Phi_{g}\,
	\chi_{R\nu,\xi}(z)\,dz.
\end{align}
We now revisit the quantitative modulation estimates established in Section \ref{sec:roughconv}, taking advantage of the inner--outer decomposition \eqref{eq:inner_outer_decomp}.
\subsubsection{Energy Estimates: Preliminaries}
\begin{lemma}\label{lem:quantmodlemma}
	The following estimates hold:
	\begin{align*}
		|\xi_{\tau}| + |\nu_{\tau}|
		\lesssim
		\frac{\|\nabla \varepsilon_{\mathrm{in}}\|_{L^{2}_{\omega}} + o(e^{-\tau})}{\nu},
	\end{align*}
	and
	\begin{align*}
		\left|1 - \frac{\tilde{\gamma}}{\gamma}\right|
		\lesssim
		\sup_{0<h<1}\frac{1}{h}\int_{\tau-h}^\tau \|\nabla \varepsilon_{\mathrm{in}}(\sigma)\|_{L^{2}_{\omega}}d\sigma
		+ o(e^{-\tau})
		+ o(|z-\xi|).
	\end{align*}

\end{lemma}
\begin{proof}
	Thanks to Lemmas \ref{lem:modulation_bounds} and \ref{tildegamma-gamma_quant}, the result follows directly once we will prove that
	\begin{align} \label{bd:technical-precision3}
		\mathcal{M}(\tau)
		&=
		\left(\int \tilde w_{1}^{2}(\nu^{2}+|z-\xi|^{2})dz\right)^{1/2}
		+ \|\tilde w_{2}\|_{L^{1}} + \nu^{2},\\
		\label{bd:technical-precision4} &\lesssim \| \nabla \varepsilon_{\mathrm{in}}\|_{L^2_\omega}+o(e^{-\tau}),
	\end{align}
	where we recall $\tilde w_1=\chi \tilde w$ and $\tilde w_2=(1-\chi)\tilde w$. Rewriting $\tilde w_1$ and $\tilde w_2$ in terms of $\varepsilon_{\mathrm{in}}$ and $\varepsilon_{\mathrm{out}}$, by combining \eqref{eq:eps_decomp} and \eqref{eq:inner_outer_decomp}, we see
	\begin{align*}
	& \tilde w_1= \chi U_{\nu,\xi}(\mathcal{X}_{\star}-1)+\chi (1-\alpha)\hat{U}_{\nu,\xi}+\chi\varepsilon_{\mathrm{in}},\\
	& \tilde w_2 = (1-\chi) U_{\nu,\xi}(\mathcal{X}_{\star}-1)+(1-\chi) (1-\alpha)\hat{U}_{\nu,\xi}+(1-\chi)\varepsilon_{\mathrm{in}}+(1-\chi_2)\varepsilon_{\mathrm{out}}.
	\end{align*}
	We first bound $\tilde w_1$. The first term is supported for $|z|\geq \zeta_*$ where $U_{\nu,\xi}\lesssim \nu^2|z|^{-4}= o(e^{-\tau})|z|^{-4}$ by \eqref{id:qualitative-cv-nu-xi-refined}. For the second we have $\int U_{\nu,\xi}^{2}(\nu^{2}+|z-\xi|^{2})dz\lesssim 1$ and $\alpha=1+O(\nu^2)=1+o(e^{-\tau})$ by Lemma \ref{lem:alpha}. Hence
	$$
	\int (|\chi U_{\nu,\xi}(\mathcal{X}_{\star}-1)|^2+|\chi (1-\alpha)\hat{U}_{\nu,\xi}|^2)(\nu^{2}+|z-\xi|^{2})dz=o(e^{-2\tau})
	$$
	For the third term, observing that $(\nu^{2}+|z-\xi|^{2}) \lesssim \frac{\omega}{\nu^{2}+|z-\xi|^{2}}$, applying the global Hardy-type inequality (Lemma \ref{lem:global_hardy_matched}) for the interior contribution shows
	$$
	\int |\chi \varepsilon_{\mathrm{in}}|^2(\nu^{2}+|z-\xi|^{2})dz\lesssim \| \nabla \varepsilon_{\mathrm{in}}\|^2_{L^2_\omega}.
	$$
	Combining the two inequalities above shows
	\begin{equation} \label{bd:technical-precision1}
	\left(\int \tilde w_{1}^{2}(\nu^{2}+|z-\xi|^{2})dz\right)^{1/2}\lesssim o(e^{-\tau})+\| \nabla \varepsilon_{\mathrm{in}}\|_{L^2_\omega}.
	\end{equation}
	We next bound $\tilde w_2$. By the aforementioned bounds on $U_{\nu,\xi}$ and $\alpha-1$ we have for the first two terms,
	$$
	\|(1-\chi) U_{\nu,\xi}(\mathcal{X}_{\star}-1)+(1-\chi) (1-\alpha)\hat{U}_{\nu,\xi}\|_{L^1}\lesssim \nu^2+|1-\alpha|=o(e^{-\tau}).
	$$
	For the third term, we use that $1-\chi$ is supported for $|z|\geq 1$, Cauchy-Schwarz and the Hardy-type inequality of Lemma \ref{lem:global_hardy_matched}. For the fourth, we simply apply \eqref{SecMomEst}. This shows
	$$
	\| (1-\chi)\varepsilon_{\mathrm{in}}+(1-\chi_2)\varepsilon_{\mathrm{out}}\|_{L^1}\lesssim \| \nabla \varepsilon_{\mathrm{in}}\|_{L^2_\omega}+o(e^{-\tau}).
	$$
	Combining these two inequalities shows
	\begin{equation} \label{bd:technical-precision2}
	\| \tilde w_1\|_{L^1}\lesssim \| \nabla \varepsilon_{\mathrm{in}}\|_{L^2_\omega}+o(e^{-\tau}).
	\end{equation}
	Injecting \eqref{bd:technical-precision1}, \eqref{bd:technical-precision2} and $\nu =o(e^{-\tau})$ (by \eqref{id:qualitative-cv-nu-xi-refined}) in \eqref{bd:technical-precision3} shows the desired inequality \eqref{bd:technical-precision4}.
	
\end{proof}

For the second lemma it is convenient to introduce the notation
\begin{equation}\label{eq:F1}
	F^{(1)}_{\nu,\xi}
	:=
	A_{\nu,\xi}^{[0]}
	+
	B_{\nu,\xi}^{[0]}
	+
	B_{\nu,\xi}^{[1]}.
\end{equation}
where we recall the definitions \eqref{eq:A0}, \eqref{eq:B0} and \eqref{eq:B1}.
\begin{lemma}\label{lem:energyestpre}
	For any $\delta>0$ one has
	\begin{align*}
		|\langle B[\varepsilon_{\mathrm{out}}], \varepsilon_{\mathrm{in}}  \rangle_*|
		&\lesssim o(e^{-2\tau})
		+
		\delta \|\nabla \varepsilon_{\mathrm{in}}\|^{2}_{L^{2}_{\omega}},
		\\
		|\langle F^{(1)}_{\nu,\xi},\varepsilon_{\mathrm{in}} \rangle_*|
		&\lesssim o(e^{-2\tau})+
		\delta \|\nabla \varepsilon_{\mathrm{in}}\|^{2}_{L^{2}_{\omega}},
		\\
		|\langle A[\varepsilon_{\mathrm{in}}], \varepsilon_{\mathrm{in}} \rangle_*|
		&\lesssim
		\delta \|\nabla \varepsilon_{\mathrm{in}}\|^{2}_{L^{2}_{\omega}}.
	\end{align*}
\end{lemma}

\begin{proof}

	We decompose for $f=B[\varepsilon_{\mathrm{out}}],F^{(1)}_{\nu,\xi},A[\varepsilon_{\mathrm{in}}]$,
	$$
	\langle f,\varepsilon_{\mathrm{in}}\rangle_*= \langle f,\varepsilon_{\mathrm{in}}\rangle_{L^2_\omega}-\nu^2 \int \nabla \Phi_{f}\cdot \nabla \Phi_{\varepsilon_{\mathrm{in}}}\chi\left(\frac{z-\xi}{R\nu}\right)dz.
	$$
	We will estimate the two terms separately, so that the desired estimate will follow from \eqref{bd:energy-estimate-technical1}, \eqref{bd:energy-estimate-technical2}, \eqref{bd:energy-estimate-technical3} and \eqref{bd:energy-estimate-technical4} we prove below.
	
	\smallskip
	
	\noindent\textbf{Step 1.} \emph{Pointwise bounds.}
	We recall
	\begin{align*}
		B[\varepsilon_{\mathrm{out}}]
		:=
		&-2\nabla\chi_2\cdot\nabla\varepsilon_{\mathrm{out}}
		-
		\varepsilon_{\mathrm{out}}\Delta\chi_2
		-
		\frac12(z\cdot\nabla\chi_2)\varepsilon_{\mathrm{out}}
		+
		\varepsilon_{\mathrm{out}}\nabla\chi_2\cdot\nabla\Phi_{\tilde U_{\nu,\xi}}
		-
		\nabla\tilde U_{\nu,\xi}\cdot
		\nabla\Phi_{(1-\chi_2)\varepsilon_{\mathrm{out}}}.
	\end{align*}
	Using Proposition~\ref{prop:OuterMomentPointwise}, all local terms are easily bounded. For the nonlocal term we use
	\[
	|\nabla \Phi_{(1-\chi_2)\varepsilon_{\mathrm{out}}}(z)|
	\lesssim
	\int_{\R^2}\frac{|(1-\chi_2(y))\varepsilon_{\mathrm{out}}(y)|}{|z-y|}\,dy,
	\]
	and split into $|z-y|\le 1$ and $|z-y|\ge 1$, using the decay of $\varepsilon_{\mathrm{out}}$ to obtain
	\[
	|\nabla \Phi_{(1-\chi_2)\varepsilon_{\mathrm{out}}}|
	\lesssim o(e^{-\tau}).
	\]
	Altogether,
	\begin{align}\label{eq:pointwiseB}
		|B[\varepsilon_{\mathrm{out}}](z)|
		\lesssim
		o(e^{-\tau})
		\Bigl[
		\frac{\nu^{2}}{(\nu^{2}+|z-\xi|^{2})^{5/2}}+1
		\Bigr]
		\langle z \rangle^2 e^{-|z|^2/4}.
	\end{align}
	Similarly, using the explicit expressions of $A_{\nu,\xi}^{[0]},B_{\nu,\xi}^{[0]},B_{\nu,\xi}^{[1]}$ and Lemma~\ref{lem:matched_errors_global}, we obtain
	\begin{align}\label{eq:pointwiseF}
		|F_{\nu,\xi}^{(1)}(z)|
		\lesssim_{\zeta_*}
		\begin{cases}
			o(e^{-\tau})\dfrac{\nu^{2}}{(\nu^{2}+|z-\xi|^{2})^{3}}, & |z|\le \frac{\zeta_*}{2},\\
			o(e^{-\tau}) e^{-\frac{|z|^{2}}{4}}, & |z|>\frac{\zeta_*}{2}.
		\end{cases}
	\end{align}
	\textbf{Step 2.} \emph{Weighted $L^2$ control.}
	Using Cauchy--Schwarz and Lemma~\ref{lem:global_hardy_matched}, we first estimate the local part
	\begin{align}
		\nonumber |\langle F_{\nu,\xi}^{(1)}+B[\varepsilon_{\mathrm{out}}],\varepsilon_{\mathrm{in}}\rangle_{L^2_{\omega}}|
		&\lesssim
		o(e^{-\tau})\nu^{2}
		\Bigl(
		\int_{\R^{2}}\frac{dz}{(\nu^{2}+|z-\xi|^{2})^{3}}
		\Bigr)^{1/2}
		\|\nabla \varepsilon_{\mathrm{in}}\|_{L^{2}_{\omega}}
		\\
		\label{bd:energy-estimate-technical1}&\lesssim
		o(e^{-\tau})
		\|\nabla \varepsilon_{\mathrm{in}}\|_{L^{2}_{\omega}},
	\end{align}
	where we used that $\int (\nu^{2}+|z-\xi|^{2})^{-3}dz \lesssim \nu^{-4}$. \newline 
	We now estimate the contribution of $A[\varepsilon_{\mathrm{in}}]$. Expanding the divergence terms,
	\begin{align*}
		A[\varepsilon_{\mathrm{in}}]
		=
		(1-\alpha)
		\Bigl[
		\nabla \cdot(\varepsilon_{\mathrm{in}}\nabla \Phi_{\hat U_{\nu,\xi}})
		+
		\nabla \cdot(\hat U_{\nu,\xi}\nabla \Phi_{\varepsilon_{\mathrm{in}}})
		\Bigr],
	\end{align*}
	and using $(1-\alpha)=o(e^{-\tau})$, we obtain
	\begin{align*}
		|\langle A[\varepsilon_{\mathrm{in}}], \varepsilon_{\mathrm{in}}\rangle_{L^2_{\omega}}|
		\lesssim
		o(e^{-\tau})
		\int |\nabla \hat{U}_{\nu,\xi}|
		|\nabla \Phi_{\varepsilon_{\mathrm{in}}}|
		|\varepsilon_{\mathrm{in}}|
		\omega dz
		+
		o(e^{-\tau})\|\nabla \varepsilon_{\mathrm{in}}\|^{2}_{L^2_{\omega}}.
	\end{align*}
	By Cauchy--Schwarz, decomposing the potential into inner and outer contributions, and using Lemma~\ref{lem:global_hardy_matched} together with Lemmas~\ref{lem:weightedL2poisson_inner} and \ref{lem:weightedL2poisson_farfield}, we infer
	\begin{align*}
		\int |\nabla \hat{U}_{\nu,\xi}|
		|\nabla \Phi_{\varepsilon_{\mathrm{in}}}|
		|\varepsilon_{\mathrm{in}}|
		\omega dz
		\lesssim
		\|\nabla \varepsilon_{\mathrm{in}}\|^{2}_{L^2_{\omega}},
	\end{align*}
	which yields
	\begin{equation} \label{bd:energy-estimate-technical2}
	|\langle A[\varepsilon_{\mathrm{in}}], \varepsilon_{\mathrm{in}}\rangle_{L^2_{\omega}}|\lesssim e^{-\tau}\|\nabla \varepsilon_{\mathrm{in}}\|^{2}_{L^2_{\omega}}.
	\end{equation}
	\newline
	\textbf{Step 3.} \emph{Potential part.}
	We estimate
	\begin{align*}
		\nu^{2}\left|\int \nabla \Phi_{f}\cdot \nabla \Phi_{\varepsilon_{\mathrm{in}}}
		\chi\!\left(\frac{z-\xi}{R\nu}\right)dz \right|
		\lesssim_{R}
		\nu^{2}
		\|\nabla \varepsilon_{\mathrm{in}}\|_{L^{2}_{\omega}}
		\Bigl(
		\int |\nabla \Phi_{f}|^{2}\chi\!\left(\frac{z-\xi}{R\nu}\right)dz
		\Bigr)^{1/2}.
	\end{align*}
	Using \eqref{eq:pointwiseB}--\eqref{eq:pointwiseF}, an inner--outer decomposition and Lemma~\ref{lem:weightedL2poisson_inner}, we obtain
	\[
	\int |\nabla \Phi_{f}|^{2}\chi\!\left(\frac{z-\xi}{R\nu}\right)dz
	\lesssim \nu^{-4}o(e^{-2\tau}),
	\]
	for $f=F^{(1)}_{\nu,\xi}, B[\varepsilon_{\mathrm{out}}]$, which yields
	\begin{equation} \label{bd:energy-estimate-technical3}
	\left|\nu^{2}\int \nabla \Phi_{F^{(1)}_{\nu,\xi}}\cdot \nabla \Phi_{\varepsilon_{\mathrm{in}}}\chi\!\left(\frac{z-\xi}{R\nu}\right)dz\right|+\left|\nu^{2}\int \nabla \Phi_{B[\varepsilon_{\mathrm{out}}]}\cdot \nabla \Phi_{\varepsilon_{\mathrm{in}}}\chi\!\left(\frac{z-\xi}{R\nu}\right)dz\right| \lesssim o(e^{-\tau})\|\nabla \varepsilon_{\mathrm{in}}\|_{L^{2}_{\omega}}.
	\end{equation}
	\newline 
	For $f=A[\varepsilon_{\mathrm{in}}]$, Lemma~\ref{lem:weightedL2poisson_inner} and $1-\alpha=o(e^{-\tau})$ give
    \begin{align}
    	\nonumber \int |\nabla \Phi_{A[\varepsilon_{\mathrm{in}}]}|^{2}\chi(\frac{z-\xi}{R\nu})dz &\lesssim_{R}\int |A[\varepsilon_{\mathrm{in}}]|^{2}(\nu^{2}+|z-\xi|^{2})dz\\
    	\label{bd:energy-estimate-technical4}&\lesssim e^{-\tau}\int |\nabla \varepsilon_{\mathrm{in}}|^{2}+\int |\nabla \hat{U}_{\nu,\xi}|^{2}|\nabla \Phi_{\varepsilon_{\mathrm{in}}}|^{2}(\nu^{2}+|z-\xi|^{2})dz+\int |\varepsilon_{\mathrm{in}}|^{2}U_{\nu,\xi}^{2}(\nu^{2}+|z-\xi|^{2})dz 
    \end{align}
	and using $\omega\gtrsim \nu^{4}$, we conclude that this contribution is controlled by
	$\delta \|\nabla \varepsilon_{\mathrm{in}}\|^{2}_{L^{2}_{\omega}}$. Collecting all contributions concludes the proof.
	
\end{proof}

The last two lemmas we introduce will play a crucial role in closing the energy estimates. Let us first introduce the following notation:
\begin{align*}
	\beta_{\nu,\xi} := \frac{\omega}{\gamma} \mathcal{X}_{\star}.
\end{align*}
In the following lemma we collect all the properties we need for this weight.

\begin{lemma}\label{lem:propbeta}
	Let $\tau$ be sufficiently large. Then we have
	\begin{align*}
		|\beta_{\nu,\xi} - 1| \lesssim |\mathfrak{m}_{\star} - 1| + \left|\frac{\tilde{\gamma}}{\gamma} - 1\right| \chi_{\zeta_*}(z) + o(1)\bigl(1 - \chi_{\zeta_*}(z)\bigr).
	\end{align*}
	In particular, $|\beta_{\nu,\xi}| \lesssim 1$, and if $|z - \xi| \le 2R\nu$ then $|\beta_{\nu,\xi} - 1| = o(1)$.
\end{lemma}

\begin{proof}
	By expanding the definition we obtain
	\begin{align*}
		\beta_{\nu,\xi}
		= \Bigl[\chi_{\zeta_*}(z) + \bigl(1 - \chi_{\zeta_*}(z)\bigr)e^{-\frac{|z|^{2}}{4}}\Bigr]
		\Bigl[\frac{\tilde{\gamma}}{\gamma} \chi_{\zeta_*}(z)
		+ \frac{1}{8}\frac{|z|^{4} e^{\frac{|z|^{2}}{4}}}{\gamma} \bigl(1 - \chi_{\zeta_*}(z)\bigr)\Bigr].
	\end{align*}
	The uniform boundedness is immediate. Moreover, if $|z - \xi| \le 2R\nu$, thanks to Lemma \ref{tildegamma-gamma} we have $|\beta_{\nu,\xi} - 1| = o(1)$.
	
	Now we simply observe that
	\begin{align*}
		\beta_{\nu,\xi}
		=& \Bigl[\chi_{\zeta_*}(z) + \bigl(1 - \chi_{\zeta_*}(z)\bigr)e^{-\frac{|z|^{2}}{4}}\Bigr]
		\Bigl[\frac{\tilde{\gamma}}{\gamma} \chi_{\zeta_*}(z)
		+ \frac{1}{8}\frac{|z|^{4} e^{\frac{|z|^{2}}{4}}}{\gamma} \bigl(1 - \chi_{\zeta_*}(z)\bigr)\Bigr] \\
		=& \Bigl[\chi_{\zeta_*}(z) + \bigl(1 - \chi_{\zeta_*}(z)\bigr)e^{-\frac{|z|^{2}}{4}}\Bigr]
		\Bigl[\chi_{\zeta_*}(z) + e^{\frac{|z|^{2}}{4}} \bigl(1 - \chi_{\zeta_*}(z)\bigr)\Bigr] \\
		& + \Bigl[\chi_{\zeta_*}(z) + \bigl(1 - \chi_{\zeta_*}(z)\bigr)e^{-\frac{|z|^{2}}{4}}\Bigr]  \Bigl[\left(\frac{\tilde{\gamma}}{\gamma} - 1\right)\chi_{\zeta_*}(z)
		+ \left(\frac{1}{8}\frac{|z|^{4}}{\gamma} - 1\right)e^{\frac{|z|^{2}}{4}} \bigl(1 - \chi_{\zeta_*}(z)\bigr)\Bigr].
	\end{align*}
	Observing that if $|z| \ge \zeta_*$ we have by \eqref{id:def-tilde-gamma},
	\begin{align*}
		\left|\frac{1}{8}\frac{|z|^{4}}{\gamma} - 1\right|
		\lesssim \left(\frac{|\xi|}{\zeta_*} + \frac{\nu^{2} + |\xi|^{2}}{\zeta_*^{2}}\right) = o(1),
	\end{align*}
	and recalling the definition of $\mathfrak{m}_{\star}$ introduced in Lemma \ref{lem:propmathfrakm}, we conclude the proof.
\end{proof}
In the next lemma we exploit integration by parts and show that all the resulting error terms can be absorbed.

\begin{lemma}\label{lem:energyintparts}
	Let $\tau$ be sufficiently large. The following estimates hold:
	\begin{align}\label{eq:energy_ibp_lambda}
		\Bigl|\int \nabla \Phi_{\Lambda_{\xi}U_{\nu,\xi}\mathcal{X}_{\star}}\cdot \nabla \Phi_{\varepsilon_{\mathrm{in}}}\,\chi_{R\nu,\xi}(z)\,dz
		- \int \Phi_{\Lambda_{\xi}U_{\nu,\xi}}\,\varepsilon_{\mathrm{in}}\,\beta_{\nu,\xi}\,dz\Bigr|
		\lesssim \frac{\sqrt{\ln R}}{R^{2}}\|\nabla \varepsilon_{\mathrm{in}}\|_{L^{2}_{\omega}},
	\end{align}
	\begin{align}\label{eq:energy_ibp_translation}
		\Bigl|\int \nabla \Phi_{\partial_{z_{i}}U_{\nu,\xi}\mathcal{X}_{\star}}\cdot \nabla \Phi_{\varepsilon_{\mathrm{in}}}\,\chi_{R\nu,\xi}(z)\,dz
		- \int \Phi_{\partial_{z_{i}}U_{\nu,\xi}}\,\varepsilon_{\mathrm{in}}\,\beta_{\nu,\xi}\,dz\Bigr|
		\lesssim \frac{\sqrt{\ln R}}{\nu R}\|\nabla \varepsilon_{\mathrm{in}}\|_{L^{2}_{\omega}},
	\end{align}
	\begin{align}\label{eq:energy_ibp_U}
		\Bigl|\int \nabla \Phi_{U_{\nu,\xi}\mathcal{X}_{\star}}\cdot \nabla \Phi_{\varepsilon_{\mathrm{in}}}\,\chi_{R\nu,\xi}(z)\,dz\Bigr|
		\lesssim \sqrt{\ln R}\|\nabla \varepsilon_{\mathrm{in}}\|_{L^{2}_{\omega}}.
	\end{align}
\end{lemma}

\begin{proof}
To ease the notation, in what follows $\chi$ denotes $\chi_{R\nu,\xi}(z)$.

\smallskip

	\noindent\textbf{Step 1.} \emph{Reduction to the uncut profiles.} Since the domain of integration is $|z-\xi|\le 2R\nu$, we have
	\begin{align*}
		\nabla \Phi_{\Lambda_{\xi}U_{\nu,\xi}\mathcal{X}_{\star}}
		= \nabla \Phi_{\Lambda_{\xi}U_{\nu,\xi}} + O(\nu^{2}), 
		\qquad
		\nabla \Phi_{\partial_{z_{i}}U_{\nu,\xi}\mathcal{X}_{\star}}
		= \nabla \ \Phi_{\partial_{z_{i}}U_{\nu,\xi}} + O(\nu^{2}).
	\end{align*}
	The corresponding error is negligible. Indeed,
	\begin{align*}
		\nu^{2}\int |\nabla \Phi_{\varepsilon_{\mathrm{in}}}|\chi\!dz
		\lesssim \nu^{2}
		\Bigl(\int \chi\!dz\Bigr)^{1/2}
		\Bigl(\int |\nabla \Phi_{\varepsilon_{\mathrm{in}}}|^{2}\chi\!dz\Bigr)^{1/2}
		\lesssim_{R} \nu^{3}\|\nabla \varepsilon_{\mathrm{in}}\|_{L^{2}_{\omega}},
	\end{align*}
	where we used the usual inner–outer decomposition $\varepsilon_{\mathrm{in}}=\chi\varepsilon_{\mathrm{in}}+(1-\chi)\varepsilon_{\mathrm{in}}$ together with Lemmas \ref{lem:weightedL2poisson_inner} and \ref{lem:weightedL2poisson_farfield}.
    
    \smallskip 
	\noindent\textbf{Step 2.} \emph{Integration by parts.} We write
	\begin{align*}
		\int \nabla \Phi_{\Lambda_{\xi}U_{\nu,\xi}}\cdot \nabla \Phi_{\varepsilon_{\mathrm{in}}}\chi
		=& \int \Phi_{\Lambda_{\xi}U_{\nu,\xi}}\,\varepsilon_{\mathrm{in}}\,\chi
		+ \int \Phi_{\Lambda_{\xi}U_{\nu,\xi}}\,\nabla \Phi_{\varepsilon_{\mathrm{in}}}\cdot \nabla \chi \\
		=& \int \Phi_{\Lambda_{\xi}U_{\nu,\xi}}\,\varepsilon_{\mathrm{in}}\,\beta_{\nu,\xi}\,\chi
		+ \int \Phi_{\Lambda_{\xi}U_{\nu,\xi}}\,\varepsilon_{\mathrm{in}}(1-\beta_{\nu,\xi})\,\chi \\
		&+ \int \Phi_{\Lambda_{\xi}U_{\nu,\xi}}\,\nabla \Phi_{\varepsilon_{\mathrm{in}}}\cdot \nabla \chi,
	\end{align*}
	and similarly for $\partial_{z_i}U_{\nu,\xi}$. We recall the explicit formulas
	\begin{align*}
		\Phi_{\Lambda_{\xi}U_{\nu,\xi}}=\frac{4\nu^{2}}{\nu^{2}+|z-\xi|^{2}},
		\qquad
		\Phi_{\partial_{z_{i}}U_{\nu,\xi}}=-4\frac{(z-\xi)_{i}}{\nu^{2}+|z-\xi|^{2}}.
	\end{align*}
	\textbf{Step 3.} \emph{Estimate of boundary terms.} For the cutoff error we have
	\begin{align*}
		\int \Phi_{\Lambda_{\xi}U_{\nu,\xi}}\nabla \Phi_{\varepsilon_{\mathrm{in}}}\cdot \nabla \chi
		\lesssim \frac{1}{R^{2}}
		\Bigl(\int|\nabla \chi|^{2}\Bigr)^{1/2}
		\Bigl(\int_{R\nu\le |z-\xi|\le 2R\nu}|\nabla \Phi_{\varepsilon_{\mathrm{in}}}|^{2}\Bigr)^{1/2}
		\lesssim \frac{\sqrt{\ln R}}{R^{2}}\|\nabla \varepsilon_{\mathrm{in}}\|_{L^{2}_{\omega}},
	\end{align*}
	and similarly
	\begin{align*}
		\int \Phi_{\partial_{z_{i}}U_{\nu,\xi}}\nabla \Phi_{\varepsilon_{\mathrm{in}}}\cdot \nabla \chi
		\lesssim \frac{\sqrt{\ln R}}{R\nu}\|\nabla \varepsilon_{\mathrm{in}}\|_{L^{2}_{\omega}}.
	\end{align*}
	\textbf{Step 4.} \emph{Estimate of $(1-\beta_{\nu,\xi})$ terms.} By Lemma \ref{lem:propbeta}, $|\beta_{\nu,\xi}-1|=o(1)$ in the region $|z-\xi|\le 2R\nu$. Hence,
	\begin{align*}
		\int \Phi_{\Lambda_{\xi}U_{\nu,\xi}}\varepsilon_{\mathrm{in}}(1-\beta_{\nu,\xi})\chi
		\lesssim o(1)\|\nabla \varepsilon_{\mathrm{in}}\|_{L^{2}_{\omega}},
	\end{align*}
	and
	\begin{align*}
		\int \Phi_{\partial_{z_{i}}U_{\nu,\xi}}\varepsilon_{\mathrm{in}}(1-\beta_{\nu,\xi})\chi
		\lesssim o(1)\frac{1}{\nu}\|\nabla \varepsilon_{\mathrm{in}}\|_{L^{2}_{\omega}}.
	\end{align*}
	\textbf{Step 5.} \emph{Exterior contribution.} Using the boundedness of $\beta_{\nu,\xi}$ and the global Hardy inequality (Lemma \ref{lem:global_hardy_matched}), we obtain
	\begin{align*}
		\int \Phi_{\Lambda_{\xi}U_{\nu,\xi}}\varepsilon_{\mathrm{in}}\beta_{\nu,\xi}(1-\chi)
		\lesssim \frac{1}{R^{2}}\|\nabla \varepsilon_{\mathrm{in}}\|_{L^{2}_{\omega}},
	\end{align*}
	and
	\begin{align*}
		\int \Phi_{\partial_{z_{i}}U_{\nu,\xi}}\varepsilon_{\mathrm{in}}\beta_{\nu,\xi}(1-\chi)
		\lesssim \frac{1}{R\nu}\|\nabla \varepsilon_{\mathrm{in}}\|_{L^{2}_{\omega}}.
	\end{align*}
	\textbf{Step 6.} \emph{Proof of \eqref{eq:energy_ibp_U}.} Finally, recalling Lemma \ref{lem:matched_errors_global}, we have
	\begin{align*}
		\Bigl|\int \nabla \Phi_{U_{\nu,\xi}\mathcal{X}_{\star}}\cdot \nabla \Phi_{\varepsilon_{\mathrm{in}}}\chi\Bigr|
		\lesssim \int \frac{1}{\nu+|z-\xi|}|\nabla \Phi_{\varepsilon_{\mathrm{in}}}|\chi
		\lesssim \sqrt{\ln R}\|\nabla \varepsilon_{\mathrm{in}}\|_{L^{2}_{\omega}}.
	\end{align*}
	
	Collecting all the estimates yields \eqref{eq:energy_ibp_lambda}, \eqref{eq:energy_ibp_translation}, and \eqref{eq:energy_ibp_U}.
\end{proof}

\subsubsection{Evolution of $\varepsilon_{\mathrm{in}}$}
We restate here, for clarity, the equation we need to control, namely \eqref{eq:inner_problem}:
\begin{equation*}
	\begin{cases}
		\partial_\tau\varepsilon_{\mathrm{in}}
		=
		L_{\nu,\xi}[\varepsilon_{\mathrm{in}}]
		+
		A[\varepsilon_{\mathrm{in}}]
		+
		F_{\nu,\xi}
		+
		B[\varepsilon_{\mathrm{out}}]
		+
		Q_{\mathrm{in}}[\varepsilon],
		\\[0.4em]
		\varepsilon_{\mathrm{in}}(\tau_0,z)
		=
		\varepsilon_0(z)\chi_4(z).
	\end{cases}
\end{equation*}

\begin{proposition}\label{prop:evolinnereps:pre}
	Let $\zeta_*>0$ be sufficiently small. Then for any $\delta>0$ there exist constants $C>0$ and $\tau_{0}$ such that for all $\tau\ge \tau_{0}$ one has
	\begin{align*}
		\partial_{\tau}\langle \varepsilon_{\mathrm{in}},\varepsilon_{\mathrm{in}} \rangle_*
		\le
		(-4+\delta)\langle \varepsilon_{\mathrm{in}},\varepsilon_{\mathrm{in}} \rangle_*
		-
		C\|\nabla \varepsilon_{\mathrm{in}}\|^{2}_{L^{2}_{\omega}}
		+
		o\!\left( \left(\sup_{0<h<1}\frac{1}{h}\int_{\tau-h}^\tau\|\nabla \varepsilon_{\mathrm{in}}(\sigma)\|_{L^{2}_{\omega}}d\sigma\right)^{2}\right)
		+
		o(e^{-2\tau}).
	\end{align*}
\end{proposition}

\begin{proof}
	We start by observing that
	\begin{align*}
		\langle \varepsilon_{\mathrm{in}},\varepsilon_{\mathrm{in}} \rangle_*
		=
		\int \varepsilon_{\mathrm{in}}^{2}\omega
		-
		\nu^{2}\int |\nabla \Phi_{\varepsilon_{\mathrm{in}}}|^{2}
		\chi\!\left(\frac{z-\xi}{R\nu}\right)dz.
	\end{align*}
	Taking the time derivative, we get
	\begin{align*}
		\partial_{\tau}\langle \varepsilon_{\mathrm{in}},\varepsilon_{\mathrm{in}} \rangle_*
		=&
		2\int \varepsilon_{\mathrm{in}}\partial_{\tau}\varepsilon_{\mathrm{in}}\omega
		-
		2\nu^{2}\int \nabla \Phi_{\varepsilon_{\mathrm{in}}}\cdot \nabla \Phi_{\partial_{\tau}\varepsilon_{\mathrm{in}}}
		\chi\!\left(\frac{z-\xi}{R\nu}\right)dz
		\\
		&+
		\int \varepsilon_{\mathrm{in}}^{2}\partial_{\tau}\omega
		-
		2\nu\nu_{\tau}\int |\nabla \Phi_{\varepsilon_{\mathrm{in}}}|^{2}
		\chi\!\left(\frac{z-\xi}{R\nu}\right)dz
		\\
		&-
		\nu^{2}\int |\nabla \Phi_{\varepsilon_{\mathrm{in}}}|^{2}
		\nabla_{y}\chi\!\left(\frac{z-\xi}{R\nu}\right)
		\cdot
		\left(
		-\frac{\nu_{\tau}}{\nu}\frac{z-\xi}{R\nu}
		-\frac{\xi_{\tau}}{R\nu}
		\right)dz.
	\end{align*}
	By the qualitative modulation estimates, Lemma \ref{lem:selfsimilarquali}, the global Hardy-type inequality, Lemma \ref{lem:global_hardy_matched}, and Lemma \ref{tildegamma-gamma}, which gives
	\begin{align*}
		\partial_{\tau}\omega
		=
		o\!\left(\frac{\omega}{\nu^{2}+|z-\xi|^{2}}\right),
	\end{align*}
	we obtain
	\begin{align*}
		\partial_{\tau}\langle \varepsilon_{\mathrm{in}},\varepsilon_{\mathrm{in}} \rangle_*
		=
		2\langle \partial_{\tau}\varepsilon_{\mathrm{in}},\varepsilon_{\mathrm{in}} \rangle_*
		+
		o\!\left(\|\nabla \varepsilon_{\mathrm{in}}\|_{L^{2}_{\omega}}^{2}\right).
	\end{align*}
	Substituting the equation for $\varepsilon_{\mathrm{in}}$ and recalling Lemma \ref{lem:energyestpre}, we obtain, for any $\delta>0$,
	\begin{align*}
		\partial_{\tau}\langle \varepsilon_{\mathrm{in}},\varepsilon_{\mathrm{in}} \rangle_*
		\leq&
		2\langle L_{\nu,\xi}[\varepsilon_{\mathrm{in}}], \varepsilon_{\mathrm{in}} \rangle_*
		+
		\delta \|\nabla \varepsilon_{\mathrm{in}}\|_{L^{2}_{\omega}}^{2}
		+
		o(e^{-2\tau})
		+
		2\alpha\left(\frac{\nu_{\tau}}{\nu}+\frac{1}{2}\right)
		\langle \Lambda_{\xi}U_{\nu,\xi}\mathcal{X}_{\star},\varepsilon_{\mathrm{in}} \rangle_*
		\\
		&-
		2\alpha_{\tau}
		\langle U_{\nu,\xi}\mathcal{X}_{\star},\varepsilon_{\mathrm{in}} \rangle_*+
		2\alpha\left(\xi_{\tau}+\frac{\xi}{2}\right)
		\cdot
		\langle \nabla U_{\nu,\xi}\mathcal{X}_{\star},\varepsilon_{\mathrm{in}} \rangle_*+
		2\langle Q_{\mathrm{in}}[\varepsilon],\varepsilon_{\mathrm{in}} \rangle_*.
	\end{align*}
    The desired inequality will then follow from the identity \eqref{estimate-energy-linear} for the linear term and the estimates \eqref{estimate-crossed-terms} and \eqref{estimate-energy-quad1}-\eqref{estimate-energy-quad2} for the modulation terms and the quadratic term we prove below, up to taking a large enough $R\gg1$, a small enough $0<\zeta_*\ll 1$ and then $\tau$ large enough.

    \smallskip 
	\noindent \textbf{Step 1.} \emph{The linear term.} Let us first consider the coercivity of the linearized operator. We observe that by the decompositions \eqref{self-similarKS-refined-decomposition} and \eqref{eq:ansatz}
	\begin{align}\label{id:varepsilon-in-tildew}
		\varepsilon_{\mathrm{in}}
		=
		\tilde{w}
		-(\tilde{U}_{\nu,\xi}-U_{\nu,\xi})
		-\varepsilon_{\mathrm{out}}(1-\chi_{2}).
	\end{align}
	By Proposition \ref{prop:global_matched_coercivity}, using \eqref{eq:ortoselfsimilar} and Lemma \ref{lem:matched_errors_global}, picking an $\eta_*\ll 2$, we get
	\begin{align*}
	\nu^{6}
		\langle \varepsilon_{\mathrm{in}}\, ,(\Lambda U)_{\nu,\xi}
		\,\chi_\eta\rangle^{2}_{L^2}& =\nu^{6}
		\langle \tilde w\, ,(\Lambda U)_{\nu,\xi}
		\,\chi_\eta\rangle^{2}_{L^2}=o(e^{-3\tau}\| \tilde w\|_{L^1}^2)=o(e^{-2\tau}).
	\end{align*}
	and similarly
	$$
	\sum_{i=1}^{2}\nu^{8}
		\langle \varepsilon_{\mathrm{in}}\, \partial_{z_{i}}U_{\nu,\xi}
		\,\chi_\eta \rangle^{2}_{L^2}=o(e^{-2\tau}).
	$$
	Therefore, for any $\delta_0>0$ small, and for $R$ sufficiently large, there exists $c_0>0$ such that
	\begin{align}\label{estimate-energy-linear}
		2\langle L_{\nu,\xi}[\varepsilon_{\mathrm{in}}],\varepsilon_{\mathrm{in}} \rangle_*
		\le
		(-4+2\delta_0)
		\langle \varepsilon_{\mathrm{in}},\varepsilon_{\mathrm{in}} \rangle_*
		-
		c_0\|\nabla \varepsilon_{\mathrm{in}}\|_{L^{2}_{\omega}}^{2}
		+
		o(e^{-2\tau}).
	\end{align}
    
    \smallskip
    
	\noindent \textbf{Step 2.} \emph{The modulation terms.} These crossed terms a priori have a problematic size, and we will rely crucially on the algebraic identities \eqref{id:algebraic-cancellation-energy-estimate} and the zero mass $\int \varepsilon=0$ to obtain a cancellation to leading order. We claim that
\begin{align} \label{estimate-crossed-terms}
		&\left|\alpha\left(\frac{\nu_{\tau}}{\nu}+\frac{1}{2}\right)
		\langle \Lambda_{\xi}U_{\nu,\xi}\mathcal{X}_{\star},\varepsilon_{\mathrm{in}} \rangle_*\right| +\left|\alpha_{\tau}
		\langle U_{\nu,\xi}\mathcal{X}_{\star},\varepsilon_{\mathrm{in}} \rangle_*\right|+\left|\alpha\left(\xi_{\tau}+\frac{\xi}{2}\right)
		\cdot
		\langle \nabla U_{\nu,\xi}\mathcal{X}_{\star},\varepsilon_{\mathrm{in}} \rangle_*\right|\\
       \nonumber &\lesssim o\left( \sup_{0<h<1}\frac{1}{h}\int_{\tau-h}^\tau \|\nabla \varepsilon_{\mathrm{in}}(\sigma)\|_{L^{2}_{\omega}}d\sigma\right)^2+o(e^{-2\tau})+(\zeta_*+\frac{\ln R}{R})\|\nabla \varepsilon_{\mathrm{in}}\|_{L^{2}_{\omega}}^2
\end{align}

    We recall the definition
	\begin{align*}
		\beta_{\nu,\xi}
		=
		\frac{\omega}{\gamma}\mathcal{X}_{\star}.
	\end{align*}
	By Lemma \ref{lem:energyintparts}, we have
	\begin{align*}
		\langle U_{\nu,\xi}\mathcal{X}_{\star},\varepsilon_{\mathrm{in}} \rangle_*
		=
		\int \varepsilon_{\mathrm{in}}U_{\nu,\xi}\gamma\beta_{\nu,\xi}
		+
		O\!\left(
		\nu^{2}\sqrt{\ln R}
		\|\nabla \varepsilon_{\mathrm{in}}\|_{L^{2}_{\omega}}
		\right),
	\end{align*}
	\begin{align*}
		\langle \nabla U_{\nu,\xi}\mathcal{X}_{\star},\varepsilon_{\mathrm{in}} \rangle_*
		=
		\int \varepsilon_{\mathrm{in}}\nabla U_{\nu,\xi}\gamma\beta_{\nu,\xi}
		-
		\nu^{2}\int \varepsilon_{\mathrm{in}}\Phi_{\nabla U_{\nu,\xi}}\beta_{\nu,\xi}\,dz
		+
		O\!\left(
		\nu\frac{\sqrt{\ln R}}{R}
		\|\nabla \varepsilon_{\mathrm{in}}\|_{L^{2}_{\omega}}
		\right),
	\end{align*}
	and
	\begin{align*}
		\langle \Lambda_{\xi}U_{\nu,\xi}\mathcal{X}_{\star},\varepsilon_{\mathrm{in}} \rangle_*
		=
		\int \varepsilon_{\mathrm{in}}\Lambda_{\xi}U_{\nu,\xi}\gamma\beta_{\nu,\xi}
		-
		\nu^{2}\int \varepsilon_{\mathrm{in}}\Phi_{\Lambda_{\xi}U_{\nu,\xi}}\beta_{\nu,\xi}\,dz
		+
		O\!\left(
		\nu^{2}\frac{\sqrt{\ln R}}{R^{2}}
		\|\nabla \varepsilon_{\mathrm{in}}\|_{L^{2}_{\omega}}
		\right).
	\end{align*}
	After rescaling the algebraic identities introduced in Proposition \ref{prop:RS}, we obtain
	\begin{align} \label{id:algebraic-cancellation-energy-estimate}
		\frac{\Lambda_{\xi}U_{\nu,\xi}}{U_{\nu,\xi}}
		-
		\Phi_{\Lambda_{\xi}U_{\nu,\xi}}
		=
		-2,
		\qquad
		\frac{\nabla U_{\nu,\xi}}{U_{\nu,\xi}}
		-
		\Phi_{\nabla U_{\nu,\xi}}
		=
		0.
	\end{align}
	Since $\gamma=\frac{\nu^{2}}{U_{\nu,\xi}}$, we infer
	\begin{align*}
		\langle U_{\nu,\xi}\mathcal{X}_{\star},\varepsilon_{\mathrm{in}} \rangle_*
		=
		\nu^{2}\int \varepsilon_{\mathrm{in}}\beta_{\nu,\xi}
		+
		O\!\left(
		\nu^{2}\sqrt{\ln R}
		\|\nabla \varepsilon_{\mathrm{in}}\|_{L^{2}_{\omega}}
		\right),
	\end{align*}
	\begin{align*}
		\langle \nabla U_{\nu,\xi}\mathcal{X}_{\star},\varepsilon_{\mathrm{in}} \rangle_*
		=
		O\!\left(
		\nu\frac{\sqrt{\ln R}}{R}
		\|\nabla \varepsilon_{\mathrm{in}}\|_{L^{2}_{\omega}}
		\right),
	\end{align*}
	and
	\begin{align*}
		\langle \Lambda_{\xi}U_{\nu,\xi}\mathcal{X}_{\star},\varepsilon_{\mathrm{in}} \rangle_*
		=
		-2\nu^{2}\int \varepsilon_{\mathrm{in}}\beta_{\nu,\xi}
		+
		O\!\left(
		\nu^{2}\frac{\sqrt{\ln R}}{R^{2}}
		\|\nabla \varepsilon_{\mathrm{in}}\|_{L^{2}_{\omega}}
		\right).
	\end{align*}
	Recalling the quantitative modulation estimate, Lemma \ref{lem:quantmodlemma}, and using $|\xi|+\nu=o(e^{-\tau/2})$, we obtain
	\begin{align*}
		&\left|\alpha\left(\xi_{\tau}+\frac{\xi}{2}\right)
		\cdot
		\langle \nabla U_{\nu,\xi}\mathcal{X}_{\star},\varepsilon_{\mathrm{in}} \rangle_*\right|
		\lesssim
		(|\xi_{\tau}|+|\xi|)
		\nu\frac{\sqrt{\ln R}}{R}
		\|\nabla \varepsilon_{\mathrm{in}}\|_{L^{2}_{\omega}}
		\\
		&\qquad \qquad \lesssim
		\nu |\xi|\frac{\sqrt{\ln R}}{R}
		\|\nabla \varepsilon_{\mathrm{in}}\|_{L^{2}_{\omega}}
		+
		\left(
		\|\nabla\varepsilon_{\mathrm{in}}\|_{L^{2}_{\omega}}
		+
		o(e^{-\tau})
		\right)
		\frac{\sqrt{\ln R}}{R}
		\|\nabla \varepsilon_{\mathrm{in}}\|_{L^{2}_{\omega}}
		\lesssim
		\frac{\sqrt{\ln R}}{R}
		\|\nabla \varepsilon_{\mathrm{in}}\|_{L^{2}_{\omega}}^{2}
		+
		o(e^{-2\tau}).
	\end{align*}
    This shows the desired bound for the third term in \eqref{estimate-crossed-terms}. Using similarly Lemma \ref{lem:quantmodlemma} and  $|\xi|+\nu=o(e^{-\tau/2})$,
	\begin{align*}
		&\alpha\left(\frac{\nu_{\tau}}{\nu}+\frac{1}{2}\right)
		\langle \Lambda_{\xi}U_{\nu,\xi}\mathcal{X}_{\star},\varepsilon_{\mathrm{in}} \rangle_* \\
        &=O\left(
		\frac{\|\nabla \varepsilon_{\mathrm{in}}\|_{L^{2}_{\omega}}}{\nu^{2}}
		+
		\frac{o(e^{-\tau})}{\nu^{2}}
		+
		1
		\right)
		\!\left(-2\nu^2 \int \varepsilon_{\mathrm{in}}\beta_{\nu,\xi}dz+O\left(
		\nu^{2}\frac{\sqrt{\ln R}}{R^{2}}
		\|\nabla \varepsilon_{\mathrm{in}}\|_{L^{2}_{\omega}}
		\right)\right)
		\\
		&\quad=(O(\|\nabla \varepsilon_{\mathrm{in}}\|_{L^{2}_{\omega}})+o(e^{-\tau})) \int \varepsilon_{\mathrm{in}}\beta_{\nu,\xi}dz+O\left(
		\frac{\sqrt{\ln R}}{R^{2}}
		\|\nabla \varepsilon_{\mathrm{in}}\|_{L^{2}_{\omega}}^{2}\right)
		+
		o(e^{-2\tau}).
	\end{align*}
    where we used $\| \varepsilon\|_{L^1}\lesssim \sqrt{|\ln \nu|}\| \nabla \varepsilon_{\mathrm{in}}\|_{L^2_{\omega}}$. By Lemma \ref{lem:alpha}, we also have
	\begin{align*}
		\left|\alpha_{\tau}
		\langle U_{\nu,\xi}\mathcal{X}_{\star},\varepsilon_{\mathrm{in}} \rangle_*\right|
		&=
		\left|\nu^{2}\alpha_{\tau}
		\int \varepsilon_{\mathrm{in}}\beta_{\nu,\xi}\right|
		+
		o(e^{-2\tau})
		+
		o\!\left(\|\nabla \varepsilon_{\mathrm{in}}\|_{L^{2}_{\omega}}^2\right)
		\\
		&\lesssim
		\nu^{2}
		\left(
		|\nu\nu_{\tau}|+\nu^{2}|\xi_{\tau}|
		\right)
		\left|
		\int \varepsilon_{\mathrm{in}}\beta_{\nu,\xi}
		\right|
		+
		o(e^{-2\tau})
		+
		o\!\left(\|\nabla \varepsilon_{\mathrm{in}}\|_{L^{2}_{\omega}}^2\right)
		\\
		&\lesssim
		\nu^{2}
		\left(
		\|\nabla \varepsilon_{\mathrm{in}}\|_{L^{2}_{\omega}}
		+
		o(e^{-2\tau})
		\right)
		\left|
		\int \varepsilon_{\mathrm{in}}\beta_{\nu,\xi}
		\right|
		+
		o(e^{-2\tau})
		+
		o\!\left(\|\nabla \varepsilon_{\mathrm{in}}\|_{L^{2}_{\omega}}^2\right)
		\\
		&\lesssim
		o(e^{-2\tau})
		+
		o\!\left(\|\nabla \varepsilon_{\mathrm{in}}\|_{L^{2}_{\omega}}^2\right).
	\end{align*}
    This shows the desired bound for the second term in \eqref{estimate-crossed-terms}. So there remains to finish bounding the first one. Recalling that
		$\varepsilon
		=
		\varepsilon_{\mathrm{in}}
		+
		(1-\chi_{2})\varepsilon_{\mathrm{out}}$, 
		and the mass cancellation $\int \varepsilon=0$,
	we can write, using Lemma \ref{lem:propbeta},
	\begin{align*}
		\left|\int \varepsilon_{\mathrm{in}}\beta_{\nu,\xi}\right|
		&=\left|
		-\int \varepsilon_{\mathrm{out}}(1-\chi_{2})
		+
		\int \varepsilon_{\mathrm{in}}(\beta_{\nu,\xi}-1)\right|\lesssim
		\int |\varepsilon_{\mathrm{in}}||1-\beta_{\nu,\xi}|
		+
		o(e^{-\tau})
		\\
		&\lesssim
		\int |\varepsilon_{\mathrm{in}}||\mathfrak{m}_{\star}-1|
		+
		\int |\varepsilon_{\mathrm{in}}|
		\left|
		\frac{\tilde{\gamma}}{\gamma}-1
		\right|
		\chi_{\zeta_*}(z)+
		o\!\left(
		\int |\varepsilon_{\mathrm{in}}|
		\left(
		1-\chi_{\zeta_*}(z)
		\right)
		\right).
	\end{align*}
	As a consequence, using again Lemma \ref{lem:quantmodlemma} and  $|\xi|+\nu=o(e^{-\tau/2})$
	\begin{align*}
		&\left|\alpha\left(\frac{\nu_{\tau}}{\nu}+\frac{1}{2}\right)
		\langle \Lambda_{\xi}U_{\nu,\xi}\mathcal{X}_{\star},\varepsilon_{\mathrm{in}} \rangle_*\right|
		\\
		&\quad\lesssim
		(O(\|\nabla \varepsilon_{\mathrm{in}}\|_{L^{2}_{\omega}})+o(e^{-\tau}))
		\Biggl(
		\int |\varepsilon_{\mathrm{in}}||\mathfrak{m}_{\star}-1|
		+
		\int |\varepsilon_{\mathrm{in}}|
		\left|
		\frac{\tilde{\gamma}}{\gamma}-1
		\right|
		\chi_{\zeta_*}(z)
		\Biggr)+
		o(e^{-2\tau})
		+
		\frac{\sqrt{\ln R}}{R^{2}}
		\|\nabla \varepsilon_{\mathrm{in}}\|_{L^{2}_{\omega}}^{2}.
	\end{align*}
	For the first term, recalling Lemma \ref{lem:propmathfrakm} we have
	\begin{align*}
		& \int |\varepsilon_{\mathrm{in}}| |\mathfrak{m}_{\star}-1|
		\lesssim
		\zeta_*^{2}
		\int_{\zeta_*\le |z|\le 2\zeta_*}
		|\varepsilon_{\mathrm{in}}|
		\\
		&\qquad \lesssim
		\zeta_*^{2}
		\left(
		\int_{\zeta_*\le |z|\le 2\zeta_*}
		\varepsilon_{\mathrm{in}}^{2}
		\frac{\omega}{\nu^{2}+|z-\xi|^{2}}
		\right)^{1/2}
		\left(
		\int_{\zeta_*\le |z|\le 2\zeta_*}
		\frac{\nu^{2}+|z-\xi|^{2}}{\omega}
		\right)^{1/2}
		\ \lesssim
		\zeta_*^{2}
		\|\nabla \varepsilon_{\mathrm{in}}\|_{L^{2}_{\omega}}.
	\end{align*}
    so that,
$$
(O(\|\nabla \varepsilon_{\mathrm{in}}\|_{L^{2}_{\omega}})+o(e^{-\tau}))
		\int |\varepsilon_{\mathrm{in}}||\mathfrak{m}_{\star}-1|\lesssim \zeta_*^2\|\nabla \varepsilon_{\mathrm{in}}\|_{L^{2}_{\omega}}+o(e^{-2\tau}).
$$
For the second term, by Lemma \ref{lem:quantmodlemma},
    \begin{align*}
    \int |\varepsilon_{\mathrm{in}}|\left|\frac{\tilde \gamma}{\gamma}-1\right|\chi_{\zeta_*}dz &\lesssim  \int |\varepsilon_{\mathrm{in}}|\left( \sup_{0<h<1}\frac{1}{h}\int_{\tau-h}^\tau \|\nabla \varepsilon_{\mathrm{in}}(\sigma)\|_{L^{2}_{\omega}}d\sigma+ o(e^{-\tau})+ o(|z-\xi|)\right)\chi_{\zeta_*}\\
    & \lesssim o\left( \sup_{0<h<1}\frac{1}{h}\int_{\tau-h}^\tau \|\nabla \varepsilon_{\mathrm{in}}(\sigma)\|_{L^{2}_{\omega}}d\sigma\right)+o(e^{-\tau})+\zeta_*\|\nabla \varepsilon_{\mathrm{in}}\|_{L^{2}_{\omega}}
    \end{align*}
    where we used $\| \varepsilon_{\mathrm{in}}\|_{L^1}\to 0$ and $|\int \varepsilon_{\mathrm{in}}|z-\xi|\chi_{\zeta_*}|\lesssim \zeta_*\|\nabla \varepsilon_{\mathrm{in}}\|_{L^{2}_{\omega}} $ by the Hardy inequality (Lemma \ref{lem:global_hardy_matched}). Hence
\begin{align*}
 &(O(\| \nabla \varepsilon_{\mathrm{in}}\|_{L^2_\omega})+o(e^{-\tau}))\int |\varepsilon_{\mathrm{in}}|\left|\frac{\tilde \gamma}{\gamma}-1\right|\chi_{\zeta_*}dz \\
 &\lesssim o\left( \sup_{0<h<1}\frac{1}{h}\int_{\tau-h}^\tau \|\nabla \varepsilon_{\mathrm{in}}(\sigma)\|_{L^{2}_{\omega}}d\sigma\right)^2+o(e^{-2\tau})+\zeta_*\|\nabla \varepsilon_{\mathrm{in}}\|_{L^{2}_{\omega}}^2.
\end{align*}
Combining, we obtain
\begin{align*}
		&\left|\alpha\left(\frac{\nu_{\tau}}{\nu}+\frac{1}{2}\right)
		\langle \Lambda_{\xi}U_{\nu,\xi}\mathcal{X}_{\star},\varepsilon_{\mathrm{in}} \rangle_*\right|\lesssim o\left( \sup_{0<h<1}\frac{1}{h}\int_{\tau-h}^\tau \|\nabla \varepsilon_{\mathrm{in}}(\sigma)\|_{L^{2}_{\omega}}d\sigma\right)^2+o(e^{-2\tau})+(\zeta_*+\frac{\ln R}{R^2})\|\nabla \varepsilon_{\mathrm{in}}\|_{L^{2}_{\omega}}^2
\end{align*}
This implies the first bound in \eqref{estimate-crossed-terms}.
\smallskip

\noindent \textbf{Step 3.} \emph{Quadratic terms.} We recall that
\begin{equation*}
	Q_{\mathrm{in}}[\varepsilon]
	=
	-\nabla\varepsilon_{\mathrm{in}}\cdot\nabla\Phi_\varepsilon
	+
	\varepsilon_{\mathrm{out}}\nabla\chi_2\cdot\nabla\Phi_\varepsilon
	+
	\varepsilon_{\mathrm{in}}\varepsilon.
\end{equation*}
Thanks to Lemma \ref{selfsimilarqualieps}, Proposition \ref{prop:OuterMomentPointwise}, and the global Hardy-type inequality, Lemma \ref{lem:hardy_rescaled_weighted}, we immediately obtain
\begin{align}\label{estimate-energy-quad1}
	\langle Q_{\mathrm{in}}[\varepsilon],\varepsilon_{\mathrm{in}}\rangle_{L^{2}_{\omega}}
	=
	o\!\left(\|\nabla \varepsilon_{\mathrm{in}}\|_{L^{2}_{\omega}}^{2}\right)
	+
	o(e^{-2\tau}).
\end{align}
Moreover, proceeding as in the absorption of the operator $A[\varepsilon_{\mathrm{in}}]$ in the proof of Lemma \ref{lem:energyestpre}, we also get
\begin{align}\label{estimate-energy-quad2}
	\nu^{2}
	\int
	\nabla \Phi_{Q_{\mathrm{in}}[\varepsilon]}
	\cdot
	\nabla \Phi_{\varepsilon_{\mathrm{in}}}
	\chi\!\left(\frac{z-\xi}{R\nu}\right)dz
	=
	o\!\left(\|\nabla \varepsilon_{\mathrm{in}}\|_{L^{2}_{\omega}}^{2}\right).
\end{align}
\end{proof}

As a consequence of the previous proposition, we obtain the following result.
\begin{proposition}\label{prop:globdynepsinn}

	Let $\tau$ be sufficiently large. Then
	\begin{align}\label{bd:varepsilon-out-matched-scalar}
		|\langle \varepsilon_{\mathrm{in}}, \varepsilon_{\mathrm{in}}  \rangle_*|+\int_{\tau}^\infty \| \nabla \varepsilon_{\mathrm{in}}\|_{L^2_\omega}^2 d\tau = o(e^{-2\tau})
	\end{align}
	and
	\begin{equation} \label{bd:varepsilon-out-L2omega}
	\|\varepsilon_{\mathrm{in}}\|_{L^{2}_{\omega}}
	= o(e^{-\tau}).
	\end{equation}
	
\end{proposition}
\begin{proof}

	The first step is to derive local-in-time bounds for 
	$\|\varepsilon_{\mathrm{in}}\|_{L^{2}_{\omega}}$ and 
	$\|\nabla \varepsilon_{\mathrm{in}}\|_{L^{2}_{\omega}}$.
	We recall that $\varepsilon_{\mathrm{in}}$ solves
	\begin{equation*}
		\begin{cases}
			\partial_\tau \varepsilon_{\mathrm{in}}
			=
			L_{\nu,\xi}[\varepsilon_{\mathrm{in}}]
			+
			A[\varepsilon_{\mathrm{in}}]
			+
			F_{\nu,\xi}
			+
			B[\varepsilon_{\mathrm{out}}]
			+
			Q_{\mathrm{in}}[\varepsilon],
			\\[0.4em]
			\varepsilon_{\mathrm{in}}(\tau_0,z)
			=
			\varepsilon_0(z)\chi_4(z),
		\end{cases}
	\end{equation*}
	where the initial data is smooth and compactly supported.
	Since $\varepsilon_{\mathrm{in}}=\varepsilon-\varepsilon_{\mathrm{out}}\in L^{1}(\mathbb{R}^{2})\cap L^{\infty}(\mathbb{R}^{2})$, we can rewrite the equation as
	\begin{align*}
		\partial_{\tau} \varepsilon_{\mathrm{in}}
		=
		\Delta \varepsilon_{\mathrm{in}}
		+
		\frac{z}{2}\cdot \nabla \varepsilon_{\mathrm{in}}
		+
		b(z,\tau)\cdot \nabla \varepsilon_{\mathrm{in}}
		+
		c(z,\tau)\varepsilon_{\mathrm{in}}
		+
		F(z,\tau),
	\end{align*}
	where for any $\tau \in [\tau_0,\tau_1]$,
	\begin{align*}
		|b(z,\tau)|+|c(z,\tau)|\lesssim_{\tau_{0},\tau_{1}}1\qquad  |F(z,\tau)| \lesssim_{\tau_0,\tau_1} \frac{e^{-|z|^{2}/4}}{1+|z|^{4}}.
	\end{align*}
Up to renormalization to absorb the drift term $z/2\cdot \nabla$, this is a second order parabolic equation with bounded coefficients, compactly supported initial data, and forcing with Gaussian decay. Therefore, standard parabolic estimates show that $\varepsilon_{\mathrm{in}}$ and $\nabla \varepsilon_{\mathrm{in}}$ have Gaussian decay in that they belong to $L^2(\langle z\rangle^4e^{|z|^2/4})$, so that the quantities $\langle \varepsilon_{\mathrm{in}},\varepsilon_{\mathrm{in}}\rangle_*$ and $\| \nabla\varepsilon_{\mathrm{in}}\|_{L^2_{\omega}}$ are well-defined.

	We now turn to the global dynamics of 
	$\langle \varepsilon_{\mathrm{in}},\varepsilon_{\mathrm{in}} \rangle_*$. We denote by $A_{f}$ the
one-sided Hardy--Littlewood maximal function truncated at scales
$h<1$ of a function $f$, defined by
\[
A_{f}(\tau)
:=
\sup_{0<h<1}\frac{1}{h}\int_{\tau-h}^{\tau} |f(\sigma)|\,d\sigma .
\]

Before proceeding, we recall the following elementary fact, which is a direct consequence of the boundedness of the maximal function in $L^2$:

	\begin{equation} \label{lem:hlmLemma}
	\|A_{f}\|_{L^{2}[\tau_{0},\tau_{1}]}
	\le C \|f\|_{L^{2}[\tau_{0}-1,\tau_{1}]}.
	\end{equation}

	From Proposition \ref{prop:evolinnereps:pre}, for $\tau \ge \tau_0$,
	\begin{align*}
		\partial_{\tau}\langle \varepsilon_{\mathrm{in}},\varepsilon_{\mathrm{in}} \rangle_*
		\le
		(-4+\delta)\langle \varepsilon_{\mathrm{in}},\varepsilon_{\mathrm{in}} \rangle_*
		-
		C\|\nabla \varepsilon_{\mathrm{in}}\|_{L^{2}_{\omega}}^{2}
		+
		o\!\left(\mathcal{A}^{2}_{\|\nabla \varepsilon_{\mathrm{in}}\|_{L^{2}_{\omega}}}\right)
		+
		o(e^{-2\tau}).
	\end{align*}
	Hence, for any $\eta>0$, there exists $\tau_0=\tau_0(\eta)$ such that
	\begin{align}\label{prop:innevolez}
		\partial_{\tau}\langle \varepsilon_{\mathrm{in}},\varepsilon_{\mathrm{in}} \rangle_*
		+
		(4-\delta)\langle \varepsilon_{\mathrm{in}},\varepsilon_{\mathrm{in}} \rangle_*
		+
		C\|\nabla \varepsilon_{\mathrm{in}}\|_{L^{2}_{\omega}}^{2}
		\le
		\eta\, \mathcal{A}^{2}_{\|\nabla \varepsilon_{\mathrm{in}}\|_{L^{2}_{\omega}}}
		+
		\eta e^{-2\tau}.
	\end{align}
	By the definition of $\mathcal{A}$,
	\begin{align}\label{ineqhlm}
		e^{\frac{4-\delta}{2}\tau}
		\mathcal{A}_{\|\nabla \varepsilon_{\mathrm{in}}\|_{L^{2}_{\omega}}}(\tau)
		&=
		\sup_{0<h<1}\frac{1}{h}
		\int_{\tau-h}^{\tau}
		e^{\frac{4-\delta}{2}(\tau-\sigma)}
		e^{\frac{4-\delta}{2}\sigma}
		\|\nabla \varepsilon_{\mathrm{in}}\|_{L^{2}_{\omega}}(\sigma)
		\, d\sigma
		\\
		&\lesssim
		\mathcal{A}_{e^{\frac{(4-\delta)\cdot}{2}}
			\|\nabla \varepsilon_{\mathrm{in}}\|_{L^{2}_{\omega}}}(\tau).
		\nonumber
	\end{align}
	Integrating \eqref{prop:innevolez} from $\tau_1 > \tau_0+1$ to $\tau$, we obtain
	\begin{align*}
		e^{(4-\delta)\tau}
		\langle \varepsilon_{\mathrm{in}},\varepsilon_{\mathrm{in}} \rangle_*
		&-
		e^{(4-\delta)\tau_1}
		\langle \varepsilon_{\mathrm{in}},\varepsilon_{\mathrm{in}} \rangle_*
		+
		C \int_{\tau_1}^{\tau}
		e^{(4-\delta)\sigma}
		\|\nabla \varepsilon_{\mathrm{in}}\|_{L^{2}_{\omega}}^{2}(\sigma)
		\, d\sigma
		\\
		&\le
		\eta \int_{\tau_1}^{\tau}
		e^{(4-\delta)\sigma}
		\mathcal{A}^{2}_{\|\nabla \varepsilon_{\mathrm{in}}\|_{L^{2}_{\omega}}}(\sigma)
		\, d\sigma
		+
		\eta \int_{\tau_1}^{\tau}
		e^{(4-\delta)\sigma} e^{-2\sigma}
		\, d\sigma.
	\end{align*}
	Using \eqref{ineqhlm} and the bound \ref{lem:hlmLemma} on the maximal function, we deduce
	\begin{align*}
		\langle \varepsilon_{\mathrm{in}},\varepsilon_{\mathrm{in}} \rangle_*+
		Ce^{-(4-\delta)\tau} \int_{\tau_1}^{\tau}
		e^{(4-\delta)\sigma}
		\|\nabla \varepsilon_{\mathrm{in}}\|_{L^{2}_{\omega}}^{2}(\sigma)
		\, d\sigma
		\lesssim
		\eta e^{-2\tau}
		+
		\eta e^{-(4-\delta)\tau}
		\int_{\tau_1-1}^{\tau_1}
		e^{(4-\delta)\sigma}
		\|\nabla \varepsilon_{\mathrm{in}}\|^{2}_{L^{2}_{\omega}}(\sigma)
		\, d\sigma.
	\end{align*}
	The first inequality \eqref{bd:varepsilon-out-matched-scalar} of the Lemma then follows, using that $\eta$ is arbitrary.

    There remains to show \eqref{bd:varepsilon-out-L2omega}. Then, by the coercivity of the matched scalar product (Proposition \ref{prop:matched_scalar_coercivity}), we have
    \begin{align*}
		& \|\varepsilon_{\mathrm{in}}\|^{2}_{L^{2}_{\omega}} \lesssim \langle \varepsilon_{\mathrm{in}},\varepsilon_{\mathrm{in}} \rangle_*+
		C\left[\nu^{2}
		\left(\int \varepsilon_{\mathrm{in}}\,\chi_{ \eta} \right)^{2}+\nu^{6}
		\left(\int \varepsilon_{\mathrm{in}}\, (\Lambda U)_{\nu,\xi}
		\,\chi_\eta \right)^{2}+\sum_{i=1}^{2}\nu^{8}
		\left(\int \varepsilon_{\mathrm{in}}\, \partial_{z_{i}}U_{\nu,\xi}
		\,\chi_\eta \right)^{2}\right]
	\end{align*}
    By \eqref{bd:varepsilon-out-matched-scalar} the first term is $o(e^{-2\tau})$. 
    This is the desired second bound \eqref{bd:varepsilon-out-L2omega}. Using the identity \eqref{id:varepsilon-in-tildew} for $\varepsilon_{\mathrm{in}}$, the cancellation $\int \varepsilon dz=0$ by \eqref{normalization-alpha} and the cancellation of the scalar products \eqref{eq:ortoselfsimilar-refined} for $\tilde w$, and the bounds \eqref{errorsmatched:eq1_new} and \eqref{SecMomEst}, the scalar products above are all $o(e^{-2\tau})$. This shows \eqref{bd:varepsilon-out-L2omega}.
    
\end{proof}

\subsection{Precise asymptotic behaviour and proof of the main theorem}\label{subsec:asymptotics}

Combining the ansatz \eqref{eq:ansatz} and the inner-outer decomposition \eqref{eq:inner_outer_decomp} we found that
\begin{equation} \label{proof-theorem-decomposition}
w=U_{\nu,\xi}+\tilde w=\tilde U_{\nu,\xi}+\varepsilon_{\mathrm{in}}+(1-\chi_{2})\varepsilon_{\mathrm{out}}
\end{equation}
where by \eqref{bd:varepsilon-out-L2omega} and the equivalence of the weight $8\omega=(\nu^2+|z-\xi|^2)^2e^{|z|^2/4}(1+o(1))$ (by \eqref{def:omeganuxi}, \eqref{id:qualitative-cv-nu-xi-refined} and Lemma \ref{tildegamma-gamma}) and Proposition \ref{prop:OuterMomentPointwise},
\begin{equation} \label{proof-theorem-bounds}
\int_{\mathbb R^2} \varepsilon_{\mathrm{in}}^2(\nu^2+ |z-\xi|^2)^2 e^{|z|^2/4}dz =o(e^{-2\tau}) \quad \mbox{and} \quad \int_{\mathbb R^2} (1-\chi_{2})|\varepsilon_{\mathrm{out}}| |z|^2dz=o(e^{-\tau}).
\end{equation}
The main Theorem \ref{main:th:main} will be a direct consequence of the precise formula \eqref{proof-theorem-identity-nu-final} for $\nu$ and of the bounds \eqref{bd:refined-L1} and \eqref{bd:refined-xi} for $\|\tilde w\|_{L^1}$ and $\xi$ we prove in this section. The $L^\infty$ bound follows from a standard application of parabolic regularization.

We start by computing the second momentum of the matched soliton.

\begin{lemma}\label{lem:secondmomentmatched}
	Assume that $\nu=o(e^{-\tau/2})$, $\xi=o(e^{-\tau/2})$ and $\alpha-1=o(e^{-\tau})$ as $\tau\to+\infty$. Then
	\begin{align}\label{eq:secondmomentmatched}
		& \int_{\mathbb{R}^{2}}\tilde{U}_{\nu,\xi}(z)|z|^{2}\,dz
		=
		16\pi \nu^{2}|\log \nu|
		+
		o(e^{-\tau}),\\
		& \label{eq:centreofmassmatched}
		\int_{\mathbb{R}^{2}} \tilde{U}_{\nu,\xi}(z)z\,dz
		= 8\pi \xi +o(e^{-3\tau/2}),
	\end{align}
\end{lemma}
\begin{proof}	
	Recall that the matched soliton is
	$$
	\tilde U_{\nu,\xi}=\alpha U_{\nu,\xi}\mathcal X_\star \quad \mbox{with} \quad U_{\nu,\xi}(z)=\frac{8\nu^2}{(\nu^2+|z-\xi|^2)^2}\quad \mbox{and}\quad \mathcal X_\star(z)=\chi_{\zeta_*}(z)+(1-\chi_{\zeta_*}(z))e^{-|z|^2/4}.
	$$
	From the bounds on the parameters, we see $U_{\nu,\xi}(z)=O(\nu^2)=o(e^{-\tau})$ for $|z|\geq \zeta_*/2$, so that
	$$
	\int_{\mathbb{R}^{2}}\tilde{U}_{\nu,\xi}(z)|z|^{2}\,dz = 8 \alpha \nu^2 \int_{\mathbb R^2} \frac{|z|^2}{(\nu^2+|z-\xi|^2)^2}\chi_{\zeta^*,\xi}(z)dz+o(e^{-\tau})
	$$
	and it remains for us to compute the first term in the right-hand side. We decompose it as
	\begin{align*}
	& \int \frac{|z-\xi|^2}{(\nu^2+|z-\xi|^2)^2}\chi_{\zeta^*,\xi}dz+2\int \frac{(z-\xi)\cdot \xi}{(\nu^2+|z-\xi|^2)^2}\chi_{\zeta^*,\xi}dz+\int \frac{|\xi|^2}{(\nu^2+|z-\xi|^2)^2}\chi_{\zeta^*,\xi}dz \\
	&\qquad = \int \frac{|y|^2}{(1+|y|^2)^2} \chi_{\zeta^*/\nu}(y)dy+0+|\xi|^2O(\nu^{-2})=2\pi|\log\nu|+O(1)+O(\nu^{-2}e^{-\tau}).
	\end{align*}
	where the second term vanishes because $\chi$ is radial. The identity \eqref{eq:secondmomentmatched} follows. Next, notice that $ \int U_{\nu,\xi}\mathcal X_\star(z-\xi)(z-\xi)dz=0$ by radiallity of $\mathcal X_\star$. Using this, the normalization $\int \tilde U_{\nu,\xi}dz=8\pi$, and the facts that $\xi=o(e^{-\tau/2})$ and $U_{\nu,\xi}(z)=O(\nu^2)=o(e^{-\tau})$ for $|z|\geq \zeta_*/2$ we get
	\begin{align*}
	& \int_{\mathbb{R}^{2}}\tilde{U}_{\nu,\xi}(z)z\,dz=8\pi \xi +\alpha \int_{\mathbb{R}^{2}} U_{\nu,\xi}(z)(\mathcal X_\star(z)-\mathcal X_\star(z-\xi))(z-\xi)\,dz= 8\pi \xi+O(e^{-3\tau/2}).
	\end{align*}
	This is \eqref{eq:centreofmassmatched}.
	
\end{proof}

We are now ready to give the asymptotic behaviour for the scale.

\begin{lemma}
	There holds as $\tau\to \infty$,
	\begin{equation}\label{proof-theorem-identity-nu-final}
	\nu(\tau)
	=
	\sqrt{\frac{\mu}{8\pi}}\,
	\frac{e^{-\tau/2}}{\sqrt{\tau}}
	(1+o(1)).
	\end{equation}
	
\end{lemma}

\begin{proof}

Expanding the second momentum \eqref{refined:second-momentum} thanks to the decomposition \eqref{proof-theorem-decomposition}, the expression \eqref{eq:secondmomentmatched} and the second bound in \eqref{proof-theorem-bounds}, we find
	\[
	\mu e^{-\tau}
	=
	16\pi \nu^{2}|\log\nu|
	+
	\int_{\mathbb{R}^{2}}\varepsilon_{\mathrm{in}}(z)|z|^{2}\,dz
	+
	o(e^{-\tau}).
	\]
	We now estimate the inner error contribution. Since $\varepsilon=\varepsilon_{\mathrm{in}}+(1-\chi_2)\varepsilon_{\mathrm{out}}$, the estimate \eqref{selfsimilarqualieps} and the second inequality in \eqref{proof-theorem-bounds} imply
	\begin{equation} \label{proof-theorem-bounds-2}
	\int| \varepsilon_{\mathrm{in}}|dz =o(1).
	\end{equation}
	In turn, the above bound \eqref{proof-theorem-bounds-2} and \eqref{proof-theorem-bounds} with Cauchy-Schwarz imply, since $\nu=o(e^{-\tau/2})$,
	\begin{equation} \label{proof-theorem-bounds-3}
	\int |\varepsilon_{\mathrm{in}}||z-\xi|dz = e^{-\tau/2}\int_{|z-\xi|<e^{-\tau/2}}|\varepsilon_{\mathrm{in}}|dz+e^{\tau/2}\int_{|z-\xi|>e^{-\tau/2}}|\varepsilon_{\mathrm{in}}|(|\nu+|z-\xi|)^2dz=o(e^{-\tau/2}).
	\end{equation}
	By decomposing $|z|^2=|\xi-z|^2+2z\cdot (z-\xi)+|\xi|^2$, using $|\xi|=o(e^{-\tau/2})$, \eqref{proof-theorem-bounds-2} and \eqref{proof-theorem-bounds-3} we obtain
	\[
	\int_{\mathbb{R}^{2}}\varepsilon_{\mathrm{in}}|z|^{2}\,dz
	=
	\int_{\mathbb{R}^{2}}\varepsilon_{\mathrm{in}}|z-\xi|^{2}\,dz
	+
	o(e^{-\tau/2})\cdot \int \varepsilon_{\mathrm{in}}(z-\xi)dz+o(e^{-\tau}) \int \varepsilon dz =o(e^{-\tau}).
	\]
	Therefore
	\[
	16\pi \nu^{2}|\log\nu|
	=
	\mu (1+o(1))e^{-\tau}.
	\]
	It remains to invert this relation. Since $\nu\to0$, the previous identity
	implies
	\[
	|\log\nu|
	=
	\frac{\tau}{2}(1+o(1)).
	\]
	The desired identity \eqref{proof-theorem-identity-nu-final} follows.

\end{proof}

\begin{lemma}
For any $\kappa\in (0,1)$, one has
\begin{equation} \label{bd:refined-L1}
\| \tilde w \|_{L^1}=O(e^{-\kappa \tau}).
\end{equation}
\end{lemma}

\begin{proof}

Notice first that by Cauchy-Schwarz, the first inequality in \eqref{proof-theorem-bounds} and \eqref{proof-theorem-identity-nu-final} implies $\| \varepsilon_{\mathrm{in}}\|_{L^1}\lesssim \nu^{-1} \| \varepsilon_{\mathrm{in}}\|_{L^2_{\omega}}\lesssim \sqrt{\tau}e^{-\tau/2}$. Combining this inequality with the second inequality in \eqref{proof-theorem-bounds} and the bound $|\alpha-1|=o(e^{-\tau})$ already shows
\begin{equation} \label{bd:refined-L1-tech1}
\| \tilde w \|_{L^1}=O(e^{-\kappa \tau})
\end{equation}
for all $\kappa \in (0,1/2)$. We will now show that if \eqref{bd:refined-L1} holds for some $\kappa\in (0,1)$, then the same inequality \eqref{bd:refined-L1} holds for all $\kappa'<(\kappa+1)/2$. This will then imply the desired inequality \eqref{bd:refined-L1} for all $\kappa \in (0,1)$.

We now prove this fact and assume \eqref{bd:refined-L1-tech1} holds for some $\kappa \in (0,1)$. By the decomposition \eqref{proof-theorem-decomposition}, and the bounds \eqref{proof-theorem-bounds} and \eqref{id:qualitative-cv-nu-xi-refined}, the exterior contribution is already under control,
$$
\| (1-\chi_{\eta})\tilde U_{\nu,\xi}\|_{L^1}+\| (1-\chi_{\eta})\tilde w\|_{L^1}=o( e^{-\tau})
$$
for any $\eta>0$. Since in the parabolic zone away from the origin $\{\eta<|z|<\eta^{-1}\}$, the Keller-Segel equation \eqref{self-similarKS-refined} for $w$ can be regarded as a parabolic equation with smooth and bounded coefficients thanks to \eqref{qualitative-tildew}-\eqref{id:qualitative-cv-nu-xi-refined}, applying standard parabolic regularization estimates shows
\begin{equation} \label{bd:refined-L1-tech-boundary}
|\nabla^{k} \tilde w(z,\tau)|=o_\eta(e^{-\tau}) \mbox{ for all }\eta<|z|<\eta^{-1},
\end{equation}
for $k=0,1,2$, for any $\eta>0$.

For the interior contribution, since by Cauchy-Schwarz and \eqref{proof-theorem-identity-nu-final} we have $\| \chi \tilde w\|_{L^1}\lesssim \sqrt{\ln t} \| (\nu+|z-\xi|)\chi \tilde w\|_{L^2}$, thanks to the Hardy-type inequality of Lemma \ref{lem:hardy_rescaled_weighted}) it will suffice for us to prove
\begin{equation} \label{bd:refined-L1-tech2}
\| \nabla (\chi \tilde w) \|_{L^2_{\gamma}}=O(e^{-\kappa' \tau})
\end{equation}
(up to taking a slightly larger $\kappa'$ to absorb the logarithmic loss).

We will consider the evolution inner variables, defined by \eqref{id:inner-variables}-\eqref{id:property-3bis}. Thanks to \eqref{proof-theorem-identity-nu-final} we have explicitly
$$
s= c_\infty t \ln t (1+o(1)) \quad \mbox{and} \quad \frac{ds}{d\tau}=\frac{1}{\nu^2}=s (1+o(1))
$$
where we write $c_\infty=8\pi\mu^{-1}$ to ease notation.

Changing variables, to prove the bound \eqref{bd:refined-L1-tech2} it suffices to prove
\begin{equation} \label{bd:refined-L1-tech3}
\int |\nabla (\chi_{\nu^{-1}}\tilde v)|^2(1+|y|^2)^2dy \lesssim t^{-2\kappa'}.
\end{equation}
Note that we shall write $t=t(s)$ to ease notation. The inequality \eqref{bd:varepsilon-out-matched-scalar}, combined with \eqref{bd:refined-L1-tech-boundary} and \eqref{id:qualitative-cv-nu-xi-refined}, implies after changing variables
\begin{equation} \label{bd:refined-L1-tech-dissipation-bound}
\int_s^\infty \int_{\mathbb R^2} |\nabla (\chi_{\nu^{-1}}\tilde v)|^2(1+|y|^2)^2dy \frac{ds}{s} =o(t^{-2}).
\end{equation}
We will use several times in the sequel that for $s\in [2^n,2^{n+1}]$, one has $s\approx 2^n$ and $t\approx 2^n/n$. By the mean value theorem, there exists a sequence of numbers $s_n\in (2^n,2^{n+1})$ such that
\begin{equation} \label{bd:refined-L1-tech-sequential}
 \int_{\mathbb R^2} |\nabla (\chi_{\nu^{-1}}\tilde v(s_n))|^2(1+|y|^2)^2dy  =o(t^{-2}_n).
\end{equation}
Therefore, the desired bound \eqref{bd:refined-L1-tech3} holds true at $s_n$, and in the sequel we will prove it for all $s\in [s_n,s_{n+1})$.

To prove this bound, we recall the evolution equation for $\tilde v$,
$$
\partial_s \tilde v =L[\tilde v]+\frac{\lambda_s}{\lambda}(\Lambda U+\Lambda \tilde v)+\frac{x_{\star,s}}{\lambda}\cdot (\nabla U+\nabla \tilde v)-\nabla (\tilde v\nabla \Phi_{\tilde v}),
$$
where $L$, the linearized operator around the stationary state is given by \eqref{id:formula-LU}. The a priori bound \eqref{bd:refined-L1-tech1} is
\begin{equation} \label{bd:refined-L1-tech-a-priori-L1-renormalized}
\| \tilde v \|_{L^1}\lesssim t^{-\kappa}.
\end{equation}
Using $\| \tilde v \|_{L^\infty}+\| \nabla \Phi_{\tilde v}\|_{L^\infty}=o(1)$ by \eqref{qualitative-tildew}, we can proceed exactly as in the proof of Lemma \ref{lem:modest} to establish bounds on the derivative of the modulation parameters, which gives
\begin{equation}\label{bd:refined-L1-tech-modulation-bound}
\frac{|\lambda_s|}{\lambda}+\frac{|x_{\star,s}|}{\lambda} \lesssim \| \tilde v \|_{L^1}\lesssim t^{-\kappa}.
\end{equation}
In turn, using the above bound, and the fact that the evolution equation for $\tilde v$ is, by the bounds of Theorem \ref{thn:soliton-resolution} for all derivatives, a second order parabolic equation with smooth and bounded coefficients (up to renormalization), applying standard parabolic regularization estimates gives
$$
\| \nabla^k \tilde v \|_{L^\infty}\lesssim t^{-\kappa} \qquad \mbox{for }k=0,1,2.
$$
We shall use the estimate
\[
\|\nabla^{k}\nabla\Phi_{\tilde v}\|_{L^\infty}
\lesssim
\|\nabla^{k-1}\tilde v\|_{L^1}
+
\|\nabla^k \tilde v\|_{L^\infty},
\qquad k=1,2,3,
\]
which follows from the cancellation of the principal value: after splitting the
kernel into \(|z-y|\le 1\) and \(|z-y|\ge 1\), the far-field contribution is
bounded by \(\|\nabla^{k-1}\tilde v\|_{L^1}\), while in the near field one
subtracts \(\nabla^{k-1}\tilde v(z)\) and obtains
\(\|\nabla^k\tilde v\|_{L^\infty}\).

We now introduce the localization
$$
h=\chi_{\nu^{-1}}\tilde v .
$$
To show the desired estimate \eqref{bd:refined-L1-tech3}, we will consider the quadratic form $\langle Lh,\mathcal M h\rangle$ studied in Proposition \ref{prop:coercivityoperator}. Indeed, since $\int \tilde v \Lambda U=0$, we infer $\int \Lambda U \tilde h=\int \Lambda U (\chi_{\nu^{-1}}-1)\tilde v$. The decays $|\Lambda U|=O(\langle y\rangle^{-4})$ and \eqref{bd:refined-L1-tech-a-priori-L1-renormalized} then imply $\int h\Lambda U=O(\nu^4t^{-\kappa})=o(t^{-2-\kappa})$. The same reasoning applies for estimating the scalar product $\int \tilde v \nabla U=0$, so we get
$$
\left|\int h  \Lambda U dy \right|+\left|\int h \nabla U dy \right|=o(t^{-2-\kappa}).
$$
By Proposition \ref{prop:coercivityoperator} and the inequalities above, we deduce that
\begin{equation} \label{bd:refined-L1-tech-equivalence-quadratic}
\int |\nabla h(s)|^2(1+|y|^2)^2dy \lesssim -\langle L[h],\mathcal M[h]\rangle_{L^{2}} +O(t^{-2-\kappa }).
\end{equation}
By \eqref{bd:refined-L1-tech-sequential}, we have
\begin{equation} \label{bd:refined-L1-tech-initial-quadratic}
\int |\nabla h(s_n)|^2(1+|y|^2)^2dy = o(t_n^{-2}).
\end{equation}
Combining \eqref{bd:refined-L1-tech-equivalence-quadratic} and \eqref{bd:refined-L1-tech-initial-quadratic}, and since $t\approx t_n$ on $[s_n,s_{n+1}]$, in order to prove the desired inequality \eqref{bd:refined-L1-tech3} it will then suffice for us to show
\begin{equation} \label{bd:refined-L1-tech-goal-quadratic}
\int_{s_n}^{s_{n+1}} \frac{d}{ds} (-\langle Lh,\mathcal M h\rangle) ds \leq O(t^{-2\kappa'}_n).
\end{equation}
The rest of the proof is now devoted to showing \eqref{bd:refined-L1-tech-goal-quadratic}. The evolution equation for $h$ is
$$
\partial_s h = L[h]+\frac{\lambda_s}{\lambda}(\Lambda U+\Lambda \tilde h)+\frac{x_{\star,s}}{\lambda}\cdot (\nabla U+\nabla \tilde h)-\nabla \cdot (\tilde h\nabla \Phi_{\tilde v})+\mathcal B
$$
where the boundary terms are
$$
\mathcal B=[\chi_{\nu^{-1}},L-\partial_s]\tilde v+\frac{\lambda_s}{\lambda}(\chi_{\nu^{-1}}-1) \Lambda U+(\chi_{\nu^{-1}}-1)\frac{x_{\star s}}{\lambda}\cdot  \nabla U-\frac{\lambda_s}{\lambda}y\cdot \nabla \chi_{\nu^{-1}} \tilde v-\frac{x_{\star ,s}}{\lambda}\cdot \nabla \chi_{\nu^{-1}} \tilde v+\tilde v \nabla \chi_{\nu^{-1}}\cdot \nabla \Phi_{\tilde v}.
$$
We recall the cancellations $L(\Lambda U)=L(\partial_{y_i}U)=0$ for $i=1,2$ and the nonnegativity $\langle L[h],\mathcal M[L[ h]]\rangle_{L^{2}}\geq 0$ from (2.19) in \cite{RS} (where the latter holds since $\int Lh=0$). Applying $L$ to the evolution equation, then multiplying by $-\mathcal M h$ and integrating shows that
$$
\partial_s \left( -\frac 12 \langle Lh,\mathcal M h\rangle \right)\leq   -\frac{\lambda_s}{\lambda}\langle Lh,\mathcal M \Lambda h\rangle-\langle Lh,\mathcal M \frac{x_{\star, s}}{\lambda}\cdot \nabla h\rangle+\langle Lh,\mathcal M(\nabla \cdot (\tilde h\nabla \Phi_{\tilde v}))\rangle-\langle Lh,\mathcal M \mathcal B\rangle.
$$
The desired inequality \eqref{bd:refined-L1-tech-goal-quadratic} then follows from the estimates \eqref{bd:refined-L1-tech-renormalization1}, \eqref{bd:refined-L1-tech-renormalization2}, \eqref{bd:refined-L1-tech-boundary1}, \eqref{bd:refined-L1-tech-boundary2} and \eqref{bd:refined-L1-tech-nonlinear} we show below for the terms in the right-hand side.

\smallskip
\noindent \underline{Bound for the renormalization terms}. A straightforward integration by parts gives
\[
\int \Delta h\, A\cdot\nabla h\, w
=
-\int w\,\nabla h\cdot(\nabla A)^T\nabla h
+\frac12\int |\nabla h|^2 \nabla\cdot(wA)
-\int (A\cdot\nabla h)(\nabla h\cdot\nabla w),
\]
and therefore
\begin{equation}\label{rateofconv:intpart}
\left|\int \Delta h\, A\cdot\nabla h\, w\right|
\lesssim
\int |\nabla h|^2
\bigl(
w|\nabla A|
+w|\nabla\!\cdot A|
+|A||\nabla w|
\bigr).
\end{equation}
We have by \eqref{id:formula-LU} after integrating by parts as above with $A=y$ and $w=U^{-1}$ and using Lemmas \ref{lemma:hardy} and \ref{lem:weightL2pote},
\begin{equation} \label{rateofconv:intpart2}
\left|\langle Lh,\mathcal M \Lambda h \rangle \right|
\lesssim
\int |\nabla h|^2(1+|y|^2)^2\,dy.
\end{equation}
Combining \eqref{rateofconv:intpart2} with \eqref{bd:refined-L1-tech-dissipation-bound} and \eqref{bd:refined-L1-tech-modulation-bound}, using that $s\approx s_n\approx t_n \ln t_n$ on $[s_n,s_{n+1}]$ and $t\approx t_n$, we get
\begin{equation} \label{bd:refined-L1-tech-renormalization1}
\int_{s_n}^{s_{n+1}} |\frac{\lambda_s}{\lambda}\langle Lh,\mathcal M \Lambda h\rangle| ds\lesssim s_n t_n^{-\kappa}\int_{s_n}^{s_{n+1}}  \int |\nabla h|^2(1+|y|^2)^2 dy\frac{ ds}{s}\lesssim s_n t_n^{-\kappa}t_n^{-2}\lesssim t_n^{-2\kappa'}
\end{equation}
since $2\kappa'<\kappa+1$.
The translation term can be estimated similarly, giving
\begin{equation} \label{bd:refined-L1-tech-renormalization2}
\int_{s_n}^{s_{n+1}} |\langle Lh,\mathcal M (\frac{x_{\star,s}}{\lambda} \cdot \nabla h)\rangle| ds \lesssim t_n^{-2\kappa'}.
\end{equation}

\noindent \underline{Bound for the boundary terms}. In the zone $\eta\nu^{-1}<|y|<\eta^{-1}\nu^{-1}$ for any $0<\eta\ll 1$, the bound \eqref{bd:refined-L1-tech-boundary} translates in inner variables as
\begin{align} \label{bd:refined-L1-tech-boundary-renormalized}
|\nabla^k \tilde v(y,s)|\lesssim \nu^{2+k} t^{-1}= t^{-2-k/2} (\sqrt{\ln t})^{-2-k}
\end{align}
for $k=0,1,2$. We expand the boundary term as
$$
\mathcal B=\mathcal B_1+\mathcal B_2
$$
where $\mathcal B_1$ gathers terms that are localized in the annulus $|y|\approx \nu^{-1}$:
\begin{align*}
\mathcal B_1
={}&
\tilde v\Big(
\partial_s\chi_{\nu^{-1}}
-\Delta\chi_{\nu^{-1}}
-\frac{\lambda_s}{\lambda}\,y\cdot\nabla\chi_{\nu^{-1}}
-\frac{x_{\star,s}}{\lambda}\cdot\nabla\chi_{\nu^{-1}}
\Big)\\
&-2\nabla\chi_{\nu^{-1}}\cdot\nabla\tilde v
+\tilde v\,\nabla\chi_{\nu^{-1}}\cdot
\big(\nabla\Phi_U+\nabla\Phi_{\tilde v}\big)
\end{align*}
and $\mathcal B_2$ contains the nonlocal commutator term together
with the renormalization terms supported outside the core region
$|y|\lesssim \nu^{-1}$:
\begin{align*}
\mathcal B_2
={}&
\nabla U\cdot
\Big(
\nabla\Phi_{\chi_{\nu^{-1}}\tilde v}
-\chi_{\nu^{-1}}\nabla\Phi_{\tilde v}
\Big)+
\frac{\lambda_s}{\lambda}
(\chi_{\nu^{-1}}-1)\Lambda U
+
(\chi_{\nu^{-1}}-1)
\frac{x_{\star,s}}{\lambda}\cdot\nabla U.
\end{align*}
By \eqref{bd:refined-L1-tech-modulation-bound} and \eqref{bd:refined-L1-tech-boundary-renormalized}, we have
$$
\nabla^k\mathcal B_1(y) = O(t^{-2-\kappa-k/2} (\sqrt{\ln t})^{-2-k} \mathbbm 1(\nu^{-1}<|y|<2\nu^{-1})).
$$
Integrating by parts thanks to \eqref{id:formula-LU}, then using Lemmas \ref{lemma:hardy} and \ref{lem:weightL2pote}, we find
\begin{align*}
|\langle Lh,\mathcal M \mathcal B_{1}\rangle| \lesssim& \left(\int |\nabla h|^2(1+|y|^2)^2 dy\right)^{1/2} \left(\int |\nabla \mathcal B_1|^2(1+|y|^2)^2 dy\right)^{1/2}\\
&\lesssim t^{-1-\kappa}\left(\int |\nabla h|^2(1+|y|^2)^2 dy\right)^{1/2}.
\end{align*}
By \eqref{bd:refined-L1-tech-dissipation-bound} and Cauchy-Schwarz, this implies
\begin{equation} \label{bd:refined-L1-tech-boundary1}
\int_{s_n}^{s_{n+1}} |\langle Lh,\mathcal M \mathcal B_{1}\rangle| ds \lesssim t_n^{-1-\kappa}(\ln t_{n})\lesssim t_{n}^{-2\kappa'}.
\end{equation}
The local terms appearing in $\mathcal B_2$ are estimated exactly as above.
It remains only to treat the terms produced by the commutator in the nonlocal
part. We first observe that, if $|z|\le \frac12\nu^{-1}$, then on the support of
$(\chi_{\nu^{-1}}-1)\tilde v$ one has $|y|\ge \nu^{-1}$, and hence
$|z-y|\ge \frac12\nu^{-1}$. Therefore,
\begin{align*}
 \left|\nabla \Phi_{(\chi_{\nu^{-1}}-1)\tilde v}(z)\right|
 &= \left|
 -\frac{1}{2\pi}
 \int_{\mathbb R^2}
 \frac{z-y}{|z-y|^2}
 (\chi_{\nu^{-1}}-1)\tilde v(y)\,dy
 \right|  \\
 &\le C\nu \int_{\mathbb R^2}|(\chi_{\nu^{-1}}-1)\tilde v(y)|\,dy
 \lesssim \nu t^{-\kappa}.
\end{align*}
For $|z|\ge \frac12\nu^{-1}$, the decay of the factor $\nabla U$ gives the
required gain. Treating the remaining terms as in the estimate of
$\mathcal B_1$, we obtain
\begin{equation}\label{bd:refined-L1-tech-boundary2}
\int_{s_n}^{s_{n+1}}
\left|\langle Lh,\mathcal M\mathcal B_2\rangle\right|\,ds
\lesssim t_n^{-1-\kappa}\log t_n
\lesssim t_n^{-2\kappa'} .
\end{equation}

\noindent \underline{Bound for the nonlinear term}. Since
\[
\nabla \cdot (h\nabla\Phi_{\tilde v})
=\nabla h\cdot\nabla\Phi_{\tilde v}-h\tilde v,
\]
the nonlinear term is estimated exactly as the normalization terms.
Using the bounds
\[
\|\nabla\Phi_{\tilde v}\|_{L^\infty}
+\|\tilde v\|_{L^\infty}
\lesssim t^{-\kappa},
\]
we obtain
\begin{equation}\label{bd:refined-L1-tech-nonlinear}
\left|
\left\langle
Lh,\mathcal M (\nabla\cdot(h\nabla\Phi_{\tilde v}))
\right\rangle
\right|
\lesssim
t^{-\kappa}
\int |\nabla h|^2(1+|y|^2)^2\,dy.
\end{equation}

\end{proof}

\begin{lemma}
For any $\kappa\in (0,1)$, one has
\begin{equation} \label{bd:refined-xi}
\xi =O(e^{-\kappa \tau}).
\end{equation}
\end{lemma}

\begin{proof}

Since the center of mass of the solution is zero by \eqref{refined:second-momentum}, by the decomposition \eqref{proof-theorem-decomposition} and the identity \eqref{eq:centreofmassmatched} we infer
$$
8\pi \xi = e^{-3\tau/2}+\int z \varepsilon_{\mathrm{in}}dz+\int z (1-\chi_2)\varepsilon_{\mathrm{out}}dz.
$$
The contribution of $\varepsilon_{\mathrm{in}}$ is $O(e^{-\kappa \tau})$ by \eqref{proof-theorem-bounds}, \eqref{proof-theorem-identity-nu-final} and \eqref{bd:refined-L1}, while that of $\varepsilon_{\mathrm{out}}$ is $o(e^{- \tau})$ by \eqref{proof-theorem-bounds}. This shows the estimate \eqref{bd:refined-xi}.

\end{proof}

\begin{appendix}
	\section{Quantitative $\varepsilon$-regularity}

\label{sec:epsilon-regularity}

\subsection{Statement and application to local $L^1-W^{1,\infty}$ smallness propagation}

The purpose of this section is to study solutions with possibly large total $L^1$ norm, but whose initial data have a small $L^1$ norm at some location. We will prove that the solution remains small locally over a universal timescale, via a quantitative estimate.

Such $\varepsilon$-regularity theory for solutions of the Keller–Segel system was studied for instance in \cite{SS}. 
It is worth emphasizing that $\varepsilon$-regularity results have played a crucial role in the analysis of other models, 
most notably the Navier–Stokes equations (see, for example, \cite{CKN, FLin, Va}; see also \cite{BP} and the references therein). 
The theorem stated below can be regarded as a generalization of Lemma~2.1 in \cite{SS} or Lemma~3.2 in \cite{SBook}. 
The main novelty is to establish smallness of the local $L^{\infty}$ norm, via a quantitative estimate. Our approach, building upon the local well-posedness provided by Theorem \ref{th:lwp:L1}, also relaxes the regularity assumptions on the initial data that were required in their result.

\begin{theorem}[Quantitative $\varepsilon$-regularity]\label{thm:quantitative-epsilon-reg}
	There exists  $\varepsilon_*>0$ such that for any $M_*>0$, there exists $T_{*}>0$ small enough such that the following holds. Consider any solution $u$ of the Keller-Segel equation \eqref{KS} with total mass $ \int u_0 dx=M\leq M_*$ and local mass
	$$
	\int_{B_8(0)} u_0 \, dx = \varepsilon_0\leq \varepsilon_*
	$$
	and let $T(u_0))$ denote its time of existence. Then
	\[
	\|u(t)\|_{L^\infty(B_{1}(0))}
	\lesssim \frac{\varepsilon_0}{t}+ M(1+M^2)t
	\qquad\text{for all } t \in (0,\min\{T_*,T(u_0)\}).
	\]
\end{theorem}

\begin{remark}

It is worth comparing with the linear heat equation, for which under the same assumptions there holds $\|u(t)\|_{L^\infty(B_{1}(0))}\lesssim \varepsilon_0 t^{-1}+ Mt^k$ for any integer $k$. Iterating the argument in the proof of Lemma \ref{lem:smallL1} would also yield a $t^k$ factor for the second term for $k\geq 2$, but we decided to take $k=1$ for simplicity.

\end{remark}

For clarity we defer the proof of Theorem \ref{thm:quantitative-epsilon-reg} to Section \ref{proof:thmuniespreg}.
As a consequence of Theorem \ref{thm:quantitative-epsilon-reg} we obtain the following result which states that smallness of the local $L^1$ mass propagates into smallness of the local $W^{1,\infty}$ norm.

\begin{corollary}[Local $L^1$--$W^{1,\infty}$ smallness propagation]\label{thm:L1Linfsmall}
	Let $u_0 \ge 0$ with $u_0 \in L^1(\mathbb{R}^2)$, and let $u$ be the corresponding solution of \eqref{KS} defined on $(0,T(u_0))$. \newline 
	Given any $\delta > 0$, there exist constants
	\[
	\sigma = \sigma(\delta,\|u_0\|_{L^1})>0,
	\qquad
	\varepsilon = \varepsilon(\sigma,\delta)>0,
	\]
	such that the following holds. \newline 
	For every $x_0 \in \mathbb{R}^2$, every $R>0$, and every $t_0 \ge 0$ satisfying
	$t_0+\sigma R^2 < T(u_0)$,
	if
	\[
	\int_{B_{8R}(x_0)} u(t_0,x)\,dx \le \varepsilon,
	\]
	then
	\[
	\sup_{t\in \left[t_0+\frac{\sigma}{2}R^2,\; t_0+\sigma R^2\right]}
	\left(
	R^3 \|\nabla u(t)\|_{L^\infty(B_{R/2}(x_0))}
	+
	R^2 \|u(t)\|_{L^\infty(B_{R/2}(x_0))}
	\right)
	\le \delta.
	\]
\end{corollary}

\begin{proof}[Proof of Corollary \ref{thm:L1Linfsmall}]
	By scaling invariance and translation invariance, it suffices to prove the statement in the case $x_0 = 0$, $t_0 = 0$, and $R = 1$. \newline 
	First, let $M=\int u_0dx$, and set $M_*=M$. let $\varepsilon_*$ and $T_*$ be as in Theorem \ref{thm:quantitative-epsilon-reg}. 
	Assume that $0 < \varepsilon < \varepsilon_*$, $0<\sigma<\min(T_*,T(u_0))$ and
	\[
	\int_{B_2(0)} u_0 \, dx \le \varepsilon.
	\]
	Then, by Theorem \ref{thm:quantitative-epsilon-reg}, we have
	\[
	\|u(t)\|_{L^\infty(B_{1})}
	\lesssim \frac{\varepsilon}{\sigma}+ M\sigma\leq \delta
	\qquad \text{for all } \sigma/4 < t < \sigma.
	\]
	where above we chose $\sigma$ small enough depending on $\delta$, $\varepsilon$ and $M$. Using the bound $\| \nabla \Phi_u\|_{L^\infty(B_{3/4}(0))}\lesssim \| u\|_{L^1(B_1(0)^c)}+\| u\|_{L^\infty(B_1(0))}$ implies
	\[
	\|u(t)\|_{L^\infty(B_{3/4}(0))}+\| \nabla \Phi_{u(t)}\|_{L^\infty(B_{3/4}(0))}
	\lesssim_M 1
	\qquad \text{for all } \sigma/4 < t < \sigma.
	\]
	We can then write the Keller-Segel equation $u_t=\Delta u+b \cdot \nabla u+cu$ with $b= -\nabla \Phi_u$ and $c=u$ that are bounded coefficients on $[\sigma/4,\sigma]\times B_{3/4(0)}$. After rescaling to the parabolic scale \(\sqrt{\sigma}\) and applying
	standard interior parabolic estimates then for all $\sigma/2<t<\sigma$ we have
	$$
	\| \nabla u(t)\|_{L^\infty(B_{1/2}(0))}\lesssim \sigma^{-1/2}\sup_{\sigma/4<t<\sigma} \| u(t)\|_{L^\infty(B_{3/4}(0))} \lesssim \frac{\varepsilon}{\sigma^{3/2}}+ M\sigma^{1/2}\leq \delta
	$$
	where again we chose $\sigma$ small enough depending on $\delta$, $\varepsilon$ and $M$. Hence
	\[
	\sup_{t\in \left[\frac{\sigma}{2},\; \sigma \right]}
	\left(
	 \|\nabla u(t)\|_{L^\infty(B_{1/2}(x_0))}
	+
	 \|u(t)\|_{L^\infty(B_{1/2}(x_0))}
	\right)
	\leq \delta .
	\]
	This is indeed the desired bound for $(x_0,t_0,R)=(0,0,1)$.
	
\end{proof}

\subsection{Proof of Theorem \ref{thm:quantitative-epsilon-reg}}\label{proof:thmuniespreg}

The starting point is the propagation of $L^{1}$ smallness, for which we refer to the proof of Lemma 3.2 in \cite{SBook} and the references therein.

\begin{lemma}[Propagation of local $L^1$ smallness]\label{lem:smallL1}
	There exists $C>0$ such that for any solution $u$ to \eqref{KS} with $u_0 \in L^1(\mathbb{R}^2)$, one has for all $ 0\leq t<T(u_{0})$,
	\[
	 \int_{B_7} u(x,t)\,dx
	\leq \int_{B_8} u_0(x)\,dx + Ct \int_{B_8} u_0(x)\,dx (1+\| u_0\|_{L^1})+Ct^2 \| u_0\|_{L^1}(1+\| u_0\|_{L^1}^2).
	\]
\end{lemma}

\begin{proof}
	We prove the inequality for smooth and compactly supported initial data, and the result then follows by a standard density argument. Let $\chi$ denote a smooth cut-off that is $1$ on $B_{15/2}(0)$ and $0$ on $B_8(0)^c$. Using the weak formulation of the equation, we obtain
	\[
	\frac{d}{dt}\int_{\mathbb{R}^2} u(x,t)\chi(x)\,dx
	= \int_{\mathbb{R}^2} u(x,t)\Delta \chi(x)\,dx
	+ \frac{1}{2}\int_{\mathbb{R}^2 \times \mathbb{R}^2} \rho_\chi(x,y)\, u(x,t)u(y,t)\,dx\,dy,
	\]
	where
	\[
	\rho_\chi(x,y)
	= -\frac{x-y}{|x-y|^2} \cdot \big(\nabla \chi(x) - \nabla \chi(y)\big).
	\]
	Both $\|\Delta \chi\|_\infty$ and $\|\rho_\chi\|_\infty$ are finite. Moreover, the support of $\Delta \chi$ is contained $\{15/2<|x|<8\}$ and the support of $\rho_\chi$ is contained in $\{15/2<|x|<8\}\times \mathbb R^2\cup \mathbb R^2\times \{15/2<|x|<8\}$. Using conservation of mass, we deduce that
	\[
	\left|\frac{d}{dt}\int_{\mathbb{R}^2} u(x,t)\chi(x)\,dx\right|
	\lesssim \| u(t)\|_{L^1(B_8(0))}(1+\| u_0\|_{L^1})\lesssim \| u_0\|_{L^1}(1+\| u_0\|_{L^1}).
	\]
	This implies, after integration,
	$$
	\| u(t)\|_{L^1(B_{15/2}(0))}\leq \| u(t)\|_{L^1(B_{8}(0))}+Ct\| u_0\|_{L^1}(1+\| u_0\|_{L^1}).
	$$
	We now repeat the argument, taking $\chi'$ a smooth cut-off that is $1$ on $B_{7}(0)$ and $0$ on $B_{15/2}(0)^c$, and obtain
	\[
	\left|\frac{d}{dt}\int_{\mathbb{R}^2} u(x,t)\chi'(x)\,dx\right|\lesssim [\| u(t)\|_{L^1(B_{8}(0))}+Ct\| u_0\|_{L^1}(1+\| u_0\|_{L^1})](1+\| u_0\|_{L^1}).
	\]
	The desired result follows after time integration.
		
\end{proof}

The second step is to obtain a uniform local regularization rate in $L^\infty$. Solutions of the Keller-Segel equation with initial data in $L^1$ satisfy $\| u\|_{L^\infty(\mathbb R^2)}\leq C(u_0) t^{-1}$ for small times. Here we show a local version of this bound on a uniform time interval and with a uniform constant.

\begin{proposition} \label{pr:uniform-regularization-bound}

There exists $\varepsilon_*>0$ such that for any $M_*>0$, there exist $t_*>0$ small enough such that for any solution $u$ to \eqref{KS} with $\| u\|_{L^1}\leq M_*$ and $\| u_0\|_{L^1(B_8(0))}\leq \varepsilon_*$, one has
$$
\| u(t)\|_{L^\infty(B_4(0))} \leq \frac{2}{t} \quad \mbox{for all $t\in [0,\min(t_*,T(u_0))$}.
$$

\end{proposition}

The proof of Proposition \ref{pr:uniform-regularization-bound} will require an intermediate Lemma. We shall rely on the known fact that a small total mass implies a small $L^\infty$ norm, via a blow-up type argument. To implement it, we define the following property describing a zone where the solution has a regular behaviour.

\begin{definition}[Property (P)]\label{def:PropP}
	Let $t \ge 0$ and $r>0$. We say that $u$ satisfies property (P) at $(r,t)$ if, setting
	\[
	\lambda(t,r) := \|u(t)\|_{L^\infty(B_r(0))}^{-1/2},
	\]
	one has
	\[
	\|u(t)\|_{L^\infty(B_r(0))} \ge \frac{1}{100}
	\sup_{s \in [t-\lambda(t,r)^2,\ t]}
	\|u(s)\|_{L^\infty(B_{r+\sqrt{\lambda(t,r)}}(0))}.
	\]
\end{definition}

\begin{lemma}[Control under property (P)]\label{lem:controlP}
	There exists $\varepsilon_*>0$ such that for any $M_*>0$, there exist $t_*>0$ small enough such that for any solution $u$ to \eqref{KS} with $\| u\|_{L^1}\leq M_*$ and $\| u_0\|_{L^1(B_8(0))}\leq \varepsilon_*$, then for every $t \in (0,\min(t_*,T(u_0)))$ and every $r \in [4,6]$, if $u$ satisfies {\rm (P)} at $(t,r)$, one has
	\[
	\|u(t)\|_{L^\infty(B_r(0))} \le \frac{1}{t}.
	\]
\end{lemma}
\begin{proof}

	Throughout the proof, balls are implicitly assumed to be closed balls. Let $\varepsilon_*>0$ small to be fixed later on. We reason by contradiction and assume the result does not hold true. There exists $M_*>0$ and a sequence of times $t_n\to 0$, radii $r_n\in [4,6]$ and solutions $u_n$ to \eqref{KS} with $\| u_{n,0}\|_{L^1}\leq M_*$ and $\| u_{n,0}\|_{L^1(B_8(0))}\leq \varepsilon_*$ such that $u_n$ satisfies Property {\rm (P)} at $(t,r)$ but
	\[
	\|u_n(t_n)\|_{L^\infty(B_{r_n}(0))} > \frac{1}{t_n}.
	\]
	By Lemma~\ref{lem:smallL1}, we obtain for all $t\in [0,t_n]$ that
	\[
	 \int_{B_7(0)} u_n(x,t)\,dx
	\leq\varepsilon_* + Ct_n\varepsilon_* (1+M_*)+Ct^2_n M_*(1+M_*^2)\leq 2\varepsilon_*,
	\]
	where the last inequality holds for $n$ large enough. Set
	\[
	\lambda_n:=  \|u_n(t_n)\|_{L^\infty(B_{r_n}(0))}^{-1/2}.
	\]
	Then
	\[
	\frac{\lambda_n^2}{t_n}
	=
	\frac{1}{t_n \|u_n(t_n)\|_{L^\infty(B_{r_n}(0))}}
	<1 \quad \mbox{and} \quad \lambda_n\to 0.
	\]
	Choose $x_n \in B_{r_n}(0)$ such that
	\[
	u_n(t_n,x_n)=\|u_n(t_n)\|_{L^\infty(B_{r_n}(0))}=\lambda_n^{-2}.
	\]
	This is possible since $t_n>0$ and, by parabolic regularization, $u_n(t_n,\cdot)$ is smooth. Define the rescaled functions
	\[
	v_n(s,y):=\lambda_n^2\,u_n\bigl(t_n-\lambda_n^2+\lambda_n^2 s,\ x_n+\lambda_n y\bigr),
	\qquad (s,y)\in[0,1]\times\mathbb{R}^2.
	\]
	By construction,
	\[
	v_n(1,0)=1.
	\]
	Note that since $t_n>\lambda_n^2$, we have indeed $t_n-\lambda_n^2  \ge 0$. Moreover,
	\[
	t_n-\lambda_n^2+\lambda_n^2 s \le t_n
	\qquad\text{for all } s\in[0,1].
	\]
	Therefore, for $n$ large one has for every $s\in[0,1]$,
	\[
	\int_{B_7(0)} u_n\bigl(x,t_n-\lambda_n^2+\lambda_n^2 s\bigr)\,dx
	\le 2\varepsilon_*.
	\]
	Using the change of variables $x=x_n+\lambda_n y$, since $|x_n|\leq 6$, we obtain
	\[
	\int_{B_{\frac{7}{\lambda_n}}(\frac{x_n}{\lambda_n}) }
	v_n(s,y)\,dy
	\le 2\varepsilon_*.
	\]
	Using Property {\rm (P)} at $(t_n,r_n)$, we have
	\[
	\sup_{t'\in[t_n-\lambda_n^2,t_n]}
	\|u_n(t')\|_{L^\infty(B_{r_n+\sqrt{\lambda_n}}(0))}
	\le
	100\,\|u_n(t_n)\|_{L^\infty(B_{r_n}(0))}
	=
	100\,\lambda_n^{-2}.
	\]
	Therefore,
	\[
	\sup_{s\in[0,1]}
	\|v_n(s)\|_{L^\infty\!\left(B_{\frac{r_n}{\lambda_n}+\frac{1}{\sqrt{\lambda_n}}} \left(-\frac{x_n}{\lambda_n}\right)\right)}
	\le 100.
	\]
	Since $|x_n|\le r_n\le 6$, we have
	\[
	B_{\frac{1}{\sqrt{\lambda_n}}}\!\left(0\right)
	\subset
	B_{\frac{r_n}{\lambda_n}+\frac{1}{\sqrt{\lambda_n}}}\!\left(-\frac{x_n}{\lambda_n}\right) \quad \mbox{and} \quad B_{\frac{1}{2\lambda_n}}\!\left(0\right)
	\subset
	B_{\frac{7}{\lambda_n}}\!\left(-\frac{x_n}{\lambda_n},\right).
	\]
	Consequently,
	\begin{equation} \label{bd:epsilon-regularity-vn}
	\sup_{s\in[0,1]}
	\|v_n(s)\|_{L^\infty(B_{1/\sqrt{\lambda_n}}(0))} \le 100,
	\qquad
	\sup_{s\in[0,1]}
	\|v_n(s)\|_{L^1(B_{1/(2\lambda_n)}(0))} \le 2\varepsilon_*.
	\end{equation}
	We next record the following compactness statement, whose proof is given after the proof of the lemma.
	\begin{claim}\label{claim:compactness}
		There exist a subsequence, still denoted by $\{v_n\}_n$, and a function
		$v_\infty \in C^2_{\mathrm{loc}}\bigl(\mathbb{R}^2\times[\tfrac12,1]\bigr)$,
		such that
		\[
		v_n \to v_\infty
		\quad\text{in } C^2_{\mathrm{loc}}\bigl(\mathbb{R}^2\times[\tfrac12,1]\bigr),
		, \qquad
		\nabla \Phi_{v_n} \to \nabla \Phi_{v_\infty}
		\quad\text{in } C_{\mathrm{loc}}\bigl(\mathbb{R}^2\times[\tfrac12,1]\bigr).
		\]
		Moreover, $v_\infty\ge 0$,
		\[
		\int_{\mathbb{R}^2} v_\infty(x,s)\,dx \le M_*
		\qquad\text{for every } s\in[\tfrac12,1],
		\]
		and $v_\infty$ solves the Keller--Segel equation on
		$\mathbb{R}^2\times[\tfrac12,1]$.
	\end{claim}
	Assuming the claim we can now complete the proof. On the one hand, since
	$v_n(1,0)=1$ for all $n$,
	it follows that
	\[
	v_\infty(1,0)=1.
	\]
	On the other hand, since $\lambda_n^{-1}\to \infty$, the second inequality in \eqref{bd:epsilon-regularity-vn} and Fatou's Lemma imply
	\[
	\|v_\infty(\tfrac12,\cdot)\|_{L^1(\mathbb{R}^2)}\leq 2\varepsilon_*.
	\]
	But this is a contradiction for $\varepsilon_*$ chosen small enough, because it is known that in that case any solution $w$ to \eqref{KS} with $\| w_0\|_{L^1(\mathbb R^2)}\leq 2\varepsilon_*$ satisfies $\| w\|_{L^\infty(\mathbb R^2)}\lesssim \varepsilon_*$, by a straightforward adaptation of the proof of Lemma \ref{lem:uniform-regularization-near-soliton} to the case $U=0$, see also \cite{EFJS}.

\medskip

\underline{Proof of Claim~\ref{claim:compactness}}
Fix $R>0$. Since $\lambda_n \to 0$, for $n$ sufficiently large we have
	$B_R(0) \subset B_{1/\sqrt{\lambda_n}}(0)$
	and 
	$B_R(0) \subset B_{1/(2\lambda_n)}(0)$.
	Therefore, the previous estimates apply on $B_R(0)$, namely
	\[
	\|v_n(t)\|_{L^\infty(B_R(0))} \le 100,
	\qquad
	\|v_n(t)\|_{L^1(B_R(0))} \le 1,
	\qquad
	\|v_n(t)\|_{L^1(\mathbb{R}^2)} \le K.
	\]
	In particular, using the representation of $\nabla \Phi_{v_n}$, it follows that there exists a $C>0$ depending only on $K$ such that
	\[
	\|\nabla \Phi_{v_n}(t)\|_{L^\infty(B_{R/2}(0))} \le C
	\qquad \text{for all } t\in[0,1].
	\]
We can thus write the equation for $v_n$ in the form
\[
\partial_t v_n = \Delta v_n + b_n \cdot \nabla v_n + c_n v_n,
\qquad
b_n := -\nabla \Phi_{v_n},
\quad
c_n := v_n.
\]
By the bounds above, $b_n$ and $c_n$ are uniformly bounded in $L^\infty([0,1]\times B_{R/2}(0))$.
Standard parabolic interior estimates then yield that for every $0<\tau<1$
\[
\sup_{(t,x)\in [\tau,1]\times B_{R/4}(0)} |\nabla v_n(t,x)| \le C(R,\tau).
\]
At this point we fix $v_n$ and regard $b_n$ and $c_n$ as given bounded coefficients. 
The equation is therefore linear with uniformly bounded coefficients, and interior $W^{2,1}_p$ estimates can be iterated. 
In particular, by parabolic Sobolev embedding, this yields uniform $C^{2}$ bounds on smaller cylinders:
\[
\|v_n\|_{C^{2}([\tfrac12,1]\times B_{R/8}(0))} \le C(R).
\]
	By a diagonal argument, up to extraction of a subsequence (still denoted by $v_n$), we obtain
	\[
	v_n \to v_\infty
	\qquad \text{in } C^2_{\mathrm{loc}}(\mathbb{R}^2 \times [\tfrac12,1]).
	\]
	Since $v_n \ge 0$ and $\|v_n(t)\|_{L^1(\mathbb{R}^2)} \le M_*$, Fatou's lemma yields
	\[
	v_\infty \ge 0,
	\qquad
	\int_{\mathbb{R}^2} v_\infty(x,t)\,dx \le M_*
	\quad \text{for all } t\in[\tfrac12,1].
	\]
	Moreover, by the convergence and the representation formula for $\nabla \Phi_{v_n}$, we have
	\[
	\nabla \Phi_{v_n} \to \nabla \Phi_{v_\infty}
	\qquad \text{in } C_{\mathrm{loc}}(\mathbb{R}^2 \times [\tfrac12,1]).
	\]
	Passing to the limit in the equation, we conclude that $v_\infty$ solves the Keller--Segel system on
	$\mathbb{R}^2 \times [\tfrac12,1]$.
\end{proof}

We can now prove Proposition \ref{pr:uniform-regularization-bound}.

\begin{proof}[Proof of Proposition \ref{pr:uniform-regularization-bound}.]
	We fix $\varepsilon_*>0$, $M^*>0$ and $t^*$ as in Lemma~\ref{lem:controlP}. We will possibly reduce the value of $t_*$ at the end of the proof. We argue by contradiction. Suppose that there exists $t_0 \in (0,\min(t_*,T(u_0)))$ such that, setting $r_0 = 4$,
	\[
	\lambda_0^{-2} := \|u(t_0)\|_{L^\infty(B_{r_0}(0))} > \frac{2}{t_0}.
	\]
	By Lemma \ref{lem:controlP}, property {\rm (P)} cannot hold at $(t_0,r_0)$, and by definition of {\rm (P)} there exists
	$t_1 \in [t_0-\lambda_0^2,\,t_0]$ such that, setting
	$r_1 := r_0 + \sqrt{\lambda_0}$,
	\[
	\lambda_1^{-2} := \|u(t_1)\|_{L^\infty(B_{r_1}(0))}
	> 100\,\lambda_0^{-2}.
	\]
	In particular, $\lambda_1 < \frac{\lambda_0}{10}$. Moreover, since
	$t_0 \in (0,t_*]$ and
	$\lambda_0^2 = \|u(t_0)\|_{L^\infty(B_{r_0}(0))}^{-1} \le \frac{t_0}{2}$,
	we obtain
	\[
	t_1 \ge t_0 - \lambda_0^2 \ge \frac{t_0}{2},
	\qquad
	t_1 \le t_0 \le t_*.
	\]
	Furthermore, since $r_0 = 4$ and $\lambda_0 \le \sqrt{t_*/2}$,
	we have
	\[
	r_1 = r_0 + \sqrt{\lambda_0}
	\le 4 + \sqrt{t_*/2} \le 6,
	\qquad
	r_1 \ge r_0 = 4.
	\]
	Hence
	$t_1 \in \Big[\tfrac{t_0}{4},\,t_0\Big]$,
	$r_1 \in [4,6]$. This implies in particular that $\| u(t_1)\|_{L^\infty(B_{r_1}(0))}>2/t_1$, so that by Lemma \ref{lem:controlP}, property {\rm (P)} cannot hold at $(t_1,r_1)$. Iterating this construction, we obtain sequences $\{t_k\}$, $\{r_k\}$, $\{\lambda_k\}$ such that for all $k\ge 0$:
	\begin{align*}
		&t_{k+1} \in [t_k-\lambda_k^2,\,t_k], \quad r_{k+1} = r_k + \sqrt{\lambda_k},\\
		& \lambda_{k+1} < \frac{\lambda_k}{10},\quad \lambda_k^{-2} = \|u(t_k)\|_{L^\infty(B_{r_k}(0))}.
	\end{align*}
	By induction we obtain
	\[
	t_k \ge t_0- \sum_{i=0}^{k-1} \lambda_i^2
	\ge t_0- \frac{t_0}{2}\sum_{i=0}^{\infty} 10^{-2i}
	\ge \frac{t_0}{4},
	\]
	and
	\[
	r_k \le 4 + \sum_{i=0}^{k-1} \sqrt{\lambda_i}
	\le 4 + \sqrt{\frac{t_0}{2}}\sum_{i=0}^{\infty} 10^{-i}
	\le 6
	\]
	where, using $t_0\leq t_*$, for these inequalities to hold we possibly took a smaller $t_*$. Hence, $t_k \in [t/4,\,t_0]$ and $r_k \in [4,6]$ for all $k$. On the other hand,
	\[
	\|u(t_k)\|_{L^\infty(B_{r_k}(0))} = \lambda_k^{-2} \to \infty.
	\]
	Since $u$ is bounded on $[t_0/4,t_0]\times B_6(0)$, this is a contradiction. Hence the result of the proposition.
	
\end{proof}

We are now ready to prove Theorem~\ref{thm:quantitative-epsilon-reg}. 
Our argument relies on a classical interior $L^1$--$L^\infty$ regularity estimate for parabolic equations; see \cite[Theorem~7.36, p.~186]{Lieberman}. For the reader's convenience, we state below the particular case needed in this paper, namely $p=1$, $n=2$, and nonnegative solutions.

\begin{theorem}[Interior $L^1$--$L^\infty$ estimate {[}Theorem 7.36, p.~186{]}]\label{thm:Lieberman}
	Let $u\ge 0$ solve
	\[
	\partial_t u = \Delta u + b(x,t)\cdot \nabla u + c(x,t) u
	\]
	in $Q_{2} := (t_0-4,t_0)\times B_{2}(x_0)$, with
	\[
	\|b\|_{L^\infty(Q_{2})} + \|c\|_{L^\infty(Q_{2})} \le M.
	\]
	Then
	\[
	\sup_{Q_1} u
	\le C(M)\, \int_{Q_{2}} u(x,t)\,dx\,dt.
	\]
\end{theorem}

\begin{proof}[Proof of Theorem~\ref{thm:quantitative-epsilon-reg}]

Let $\varepsilon_*$ be given by Proposition \ref{pr:uniform-regularization-bound}. Let $u$ be a solution to \eqref{KS} satisfying the assumptions of the Theorem. By Lemma \ref{lem:smallL1}, if $T_*\ll 1$ has been chosen small enough depending on $M_*$, we have
$$
	 \int_{B_7(0)} u(x,t)\,dx
	\lesssim \varepsilon_0 +t^2 C(M)
$$
for all $0\leq t\leq \min (T_*,T(u_0))$, where we recall $C(M)=M(1+M^2)$. By Proposition \ref{pr:uniform-regularization-bound}, by taking $T_*$ small enough we have
$$
\| u(t)\|_{L^\infty(B_4(0))} \leq \frac{2}{t} \quad \mbox{for all $t\in [0,\min(T_*,T(u_0))$}.
$$
	Observe that for $y\in B_2(0)$, and $t\in [0,\min(T_*,T(u_0))$, 
	\begin{align*}
	|\nabla \Phi_{u(t)}(y)| & \lesssim \int_{|y-z|<\sqrt{t}} \frac{dz}{|y-z|}u(z,t)+\int_{\sqrt{t}<|y-z|<2} \frac{dz}{|y-z|}u(z,t)+\int_{|y-z|>2} \frac{dz}{|y-z|}u(z,t) \\
	& \lesssim \sqrt{t}\| u(t)\|_{L^\infty(B_4(0))}+\frac{1}{\sqrt{t}}\| u(t)\|_{L^1(B_4(0))}+\| u(t)\|_{L^1}\lesssim \frac{1}{\sqrt{t}}
	\end{align*}
	provided once again $T_*$ is small enough depending on $M_*$. We will apply Theorem~\ref{thm:Lieberman} to a renormalization of the Keller-Segel equation
	\[
	\partial_t u=\Delta u+b\cdot \nabla u+c\,u,
	\qquad
	b=-\nabla\Phi_u,
	\qquad
	c=u.
	\]
	For every $(x,t)\in B_1(0)\times[0,\min(T_*,T(u_0))]$, set $R=\sqrt{t}/2$. we have $Q_{2R}(x,t)\subset B_2(0)\times[0,\min(T_*,T(u_0))]$. Define
	\[
	v(y,s):=R^2 u(x+Ry,t+R^2s),
	\qquad (y,s)\in Q_2:=(-4,0)\times B_2(0).
	\]
	Then $v$ solves
	\[
	\partial_s v=\Delta v+\tilde b\cdot \nabla v+\tilde c\,v,
	\]
	where
	\[
	\tilde b(y,s)=R\,b(x+Ry,t+R^2s),
	\qquad
	\tilde c(y,s)=R^2 c(x+Ry,t+R^2s).
	\]
	By the bounds established above,
	\[
	\|\tilde b\|_{L^\infty(Q_2)}+\|\tilde c\|_{L^\infty(Q_2)}\lesssim 1
	\]
	are uniformly bounded. Therefore, Theorem~\ref{thm:Lieberman} yields
	\[
	R^{2}u(x,t)=v(0,0)\le \sup_{Q_1} v \le C \int_{Q_2} v(y,s)\,dy\,ds.
	\]
	Changing variables back, using the explicit value of $R$, we obtain
	\[
	u(x,t)
	\lesssim \frac{1}{t^2}
	\int_{0}^{t}\int_{B_{2R}(x)} u(z,\tau)\,dz\,d\tau \lesssim \frac{1}{t^2} \int_0^t [ \varepsilon_0 +\tau^2 C(M)]d\tau \lesssim \frac{\varepsilon_0}{t}+tC(M).
	\]
	This proves the theorem.
\end{proof}

\section{Weighted estimates for the Poisson field}	\label{sec:poisson}

We collect here the weighted estimates for the Poisson field used throughout the paper. Their proof is omitted, since both estimates follow from the same decomposition into radial and non-radial components together with standard Fourier mode estimates. Similar arguments have appeared for instance in \cite{RS,CGMN0,CGMN1,CGMN2}. 
\begin{lemma}\label{lem:weightL2pote}
	Let \(u\in \mathcal{S}(\R^2)\), and let
\(W(y)=W(|y|)\ge 0\) be radial. Define
\begin{align*}
	K_{W}
	&:=
	\int_0^1 \rho |\ln \rho|\,W(\rho)\,d\rho
	+
	\int_1^\infty W(\rho)\frac{\ln(2+\rho)}{\rho}\,d\rho,\\
	K_{W,A}
	&:=
	\int_0^1 \rho |\ln \rho|\,W(\rho)\,d\rho
	+
	\int_1^\infty W(\rho)\rho^{1-A}\,d\rho.
\end{align*}
If \(A\ge 2\) and \(K_W<\infty\), then
\begin{align*}
	\int_{\R^2} |\nabla \Phi_u(y)|^2 W(|y|)\,dy
	\lesssim
	K_{W}
	\int_{\R^2} u^2(y)\langle y\rangle^A\,dy.
\end{align*}
If \(2<A<4\), \(\int_{\R^2}u=0\) and
\(K_{W,A}<\infty\), then
\begin{align*}
	\int_{\R^2} |\nabla \Phi_u(y)|^2 W(|y|)\,dy
	\lesssim
	K_{W,A}
	\int_{\R^2} u^2(y)\langle y\rangle^A\,dy.
\end{align*}
\end{lemma}
As a consequence of the previous lemma, we also obtain the following
rescaled version, which is often more convenient for applications
involving localization at scale \(\nu\) and center \(\xi\).
\begin{lemma}\label{lem:weightedL2poisson_inner}
	Let \(u\in \mathcal{S}(\R^2)\), let \(\nu>0\), \(\xi\in\R^2\), and let
	\[
	W(z):=W\!\left(\frac{|z-\xi|}{\nu}\right),
	\]
	where \(W=W(|\cdot|)\ge 0\) is radial. Define
	\begin{align*}
		K_{W}
		&:=
		\int_0^1 \rho |\ln \rho|\,W(\rho)\,d\rho
		+
		\int_1^\infty W(\rho)\frac{\ln(2+\rho)}{\rho}\,d\rho.
	\end{align*}
	 If
	\(A\ge 2\) and  \(K_W<\infty\), then
	\begin{align*}
		\int_{\R^2} |\nabla \Phi_u(z)|^2 W(z)\,dz
		\lesssim
		\nu^{2-A}
		K_{W}
		\int_{\R^2} u^2(z)\,(\nu^{2}+|z-\xi|^{2})^{A/2}\,dz .
	\end{align*}
\end{lemma}
Analogous estimates can also be obtained for functions localized away from the origin. In this case, the localization yields improved decay properties for the associated Poisson field, and the argument follows from the same angular decomposition together with the additional decay induced by the support condition.
\begin{lemma}\label{lem:weightedL2poisson_farfield}
	Let \(u\in\mathcal S(\R^2)\) be supported in \(\{|z|\ge\eta\}\)
for some \(\eta>0\), and let \(W(z)=W(|z|)\ge0\) be radial.
Define
\[
K_W
:=
\int_0^1 r|\ln r|\,W(r)\,dr
+
\int_1^\infty W(r)\frac{dr}{r}.
\]
If \(K_W<\infty\), then
\[
\int_{\R^2}|\nabla\Phi_u(z)|^2 W(|z|)\,dz
\lesssim_{\eta}
K_W
\int_{\R^2}u^2(z)\,|z|^4 e^{\frac{|z|^2}{4}}\,dz.
\]
\end{lemma}
\end{appendix}

\end{document}